\pdfoutput = 1 

\documentclass[11pt, a4paper]{article}

\usepackage[T1]{fontenc} 
\usepackage[utf8]{inputenc} 

\usepackage[noblocks]{authblk}

\title{An Additive-Noise Approximation to {K}eller--{S}egel--{D}ean--{K}awasaki Dynamics: {S}mall-Noise Results}
\author{Adrian Martini}
\affil{Institut f\"{u}r Mathematik, Technische Universit\"{a}t Berlin;\protect\\ \texttt{adrian.martini@tu-berlin.de}, ORCID iD: 0000-0001-9350-1338}

\author{Avi Mayorcas}
\affil{Department of Mathematical Sciences, University of Bath;\protect\\ \texttt{am2735@bath.ac.uk}, ORCID iD: 0000-0003-4133-9740}
\date{\today}

\usepackage{etoolbox} 

\newtoggle{diagrams}
\toggletrue{diagrams} 

\newtoggle{colours}
\toggletrue{colours} 

\newtoggle{details}
\togglefalse{details}

\usepackage[UKenglish]{babel} 
\usepackage{csquotes} 

\usepackage[
	babel=true,
	expansion=false,
	factor=1100
]{microtype} 

\usepackage[skip=10pt plus 1pt, indent=0pt]{parskip} 

\usepackage[all]{nowidow} 

\usepackage{amsmath,amssymb,amsthm} 
\usepackage{mathtools} 
\usepackage{amstext} 

\usepackage{dsfont} 
\usepackage{mathrsfs} 
\usepackage{stmaryrd} 
\usepackage{esint} 

\usepackage{halloweenmath}

\usepackage{xfrac}

\usepackage{subfiles} 

\usepackage{verbatim} 
\usepackage{mdframed} 

\PassOptionsToPackage{hyphens}{url} 

\usepackage{booktabs}

\usepackage{tikz} 

\usepackage{xcolor}

\definecolor{detailscolor}{RGB}{0, 153, 136} 

\iftoggle{colours}{%
	\definecolor{vectorcolor}{RGB}{238, 51, 119} 
	
	\definecolor{rootcolor}{RGB}{0, 119, 187} 
	}{%
	\definecolor{vectorcolor}{RGB}{180, 180, 180} 
	
	\definecolor{rootcolor}{RGB}{0, 0, 0} 
}%

\usepackage[
	hyperref=true,
	backend=biber,
	style=ext-alphabetic,
	useprefix=true,
	giveninits=true,
	doi=false,
	isbn=false,
	url=false,
	eprint=false
]{biblatex}

\DeclareUnicodeCharacter{0229}{\c{e}} 

\DeclareFieldFormat[article]{number}{\mkbibparens{#1}} 

\DeclareFieldFormat{issuedate}{#1} 

\DeclareFieldFormat[article]{title}{#1}
\DeclareFieldFormat[inproceedings]{title}{#1}
\DeclareFieldFormat[incollection]{title}{#1} 
\DeclareFieldFormat[misc]{title}{\mkbibemph{#1}} 

\DeclareFieldFormat{journaltitle}{\mkbibemph{#1}\addcomma} 

\renewbibmacro{in:}{} 

\usepackage{hyperref} 

\hypersetup{
	pdfauthor={Adrian Martini, Avi Mayorcas},
	pdftitle={An Additive-Noise Approximation to {K}eller--{S}egel--{D}ean--{K}awasaki Dynamics: {S}mall-Noise Results},
	pdfsubject={},
	pdfkeywords={},
	pdfproducer={LaTeX using hyperref},
	pdfcreator={pdfLaTeX from TeXLive},
	pageanchor={true},
	pdfborder={0 0 .33 [3 1]},
	colorlinks=true,
	linkcolor=blue,
	citecolor=blue,
	filecolor=blue,
	urlcolor=blue,
	linkbordercolor=teal,
	citebordercolor=teal,
	filebordercolor=teal,
	urlbordercolor=blue
}

\usepackage{geometry} 

\usepackage{fancyhdr}

\usepackage{ellipsis}

\usepackage{lmodern}

\usepackage{bm} 

\numberwithin{equation}{section}
\newtheorem{theorem}{Theorem}[section]
\newtheorem{corollary}[theorem]{Corollary}
\newtheorem{lemma}[theorem]{Lemma}
\newtheorem{proposition}[theorem]{Proposition}

\newtheorem{definition}[theorem]{Definition} 
\theoremstyle{remark}
\newtheorem{remark}[theorem]{Remark}

\let\oldproofname=\proofname
\renewcommand{\proofname}{\rm\bf{\oldproofname}}

\newtheoremstyle{sketch} 
	{\topsep}
	{\topsep}
	{\itshape}
	{0pt}
	{\bfseries}
	{.}
	{5pt plus 1pt minus 1pt}
	{\thmname{#1}~\thmnote{#3}}

\theoremstyle{sketch}

\newtheorem*{theoremsketch}{Sketch of Theorem}

\newcommand{\dd}{\mathop{}\!\mathrm{d}\mkern0.5mu} 
\newcommand{\euler}{\mathrm{e}\mkern0.7mu} 
\newcommand{\upi}{\mathrm{i}\mkern0.7mu} 
\newcommand{\uppi}{\pi}

\newcommand{\mcB}{\mathcal{B}}
\newcommand{\mcC}{\mathcal{C}}
\newcommand{\mcD}{\mathcal{D}}

\newcommand{\mcF}{\mathcal{F}}

\newcommand{\mcH}{\mathcal{H}}
\newcommand{\mcI}{\mathcal{I}}

\newcommand{\mcR}{\mathcal{R}}
\newcommand{\mcS}{\mathcal{S}}
\newcommand{\mcT}{\mathcal{T}}

\newcommand{\mcW}{\mathcal{W}}
\newcommand{\mcX}{\mathcal{X}}

\newcommand{\mbC}{\mathbb{C}}

\newcommand{\mbE}{\mathbb{E}}

\newcommand{\mbN}{\mathbb{N}}

\newcommand{\mbP}{\mathbb{P}}

\newcommand{\mbR}{\mathbb{R}}

\newcommand{\mbT}{\mathbb{T}}

\newcommand{\mbX}{\mathbb{X}}

\newcommand{\mbZ}{\mathbb{Z}}

\newcommand{\mfB}{\mathfrak{B}}

\newcommand{\mfK}{\mathfrak{K}}

\newcommand{\mfs}{\mathfrak{s}}

\newcommand{\msA}{\mathscr{A}}

\newcommand{\msC}{\mathscr{C}}

\newcommand{\msF}{\mathscr{F}}
\newcommand{\msG}{\mathscr{G}}
\newcommand{\msH}{\mathscr{H}}
\newcommand{\msI}{\mathscr{I}}
\newcommand{\msJ}{\mathscr{J}}

\newcommand{\msL}{\mathscr{L}}

\newcommand{\msP}{\mathscr{P}}

\DeclarePairedDelimiter{\ceil}{\lceil}{\rceil}
\DeclarePairedDelimiter{\floor}{\lfloor}{\rfloor}
\DeclarePairedDelimiter{\abs}{\lvert}{\rvert}
\DeclarePairedDelimiter{\norm}{\lVert}{\rVert}
\newcommand{\inner}[2]{\langle#1,#2\rangle}

\newcommand{\supp}{\textnormal{supp}}
\newcommand{\eps}{\varepsilon}

\newcommand{\tzero}{|_{t=0}}

\newcommand{\from}{\colon}
\newcommand{\defeq}{\coloneq} 
\newcommand{\eqdef}{\eqcolon} 

\newcommand{\pa}{\mathbin{\varolessthan}}
\newcommand{\re}{\mathbin{\varodot}}
\newcommand{\vdiv}{\nabla\cdot}
\newcommand{\om}{\omega}

\renewcommand{\hat}{\widehat}
\newcommand{\rksnoise}[2]{\mcX^{#1,#2}}
\newcommand{\rdet}{\rho_{\textnormal{det}}}
\newcommand{\srdet}{\sqrt{\rdet}} 
\newcommand{\het}{\sigma} 
\newcommand{\can}{\textnormal{can}}
\newcommand{\Enh}{\textnormal{Enh}}

\iftoggle{details}{
	\newenvironment{details}%
	{\color{detailscolor}\textbullet\textbf{Details: }\ignorespaces}%
	{\ignorespacesafterend}%
}{%
	\newenvironment{details}%
	{\comment}%
	{\endcomment}%
}%

\newcommand{\multiquad}[1][1]{\hspace*{#1em}\ignorespaces}

\newcommand{\per}{\textnormal{per}}
\newcommand{\loc}{\textnormal{loc}}
\newcommand{\comp}{\textnormal{c}} 

\newcommand{\gm}{\mu} 
\newcommand{\nban}{B} 
\newcommand{\rep}{\mcR_{\gm}} 
\newcommand{\cm}{\mcH_{\gm}} 
\newcommand{\ban}{E} 
\newcommand{\bban}{\mathbf{\ban}} 
\newcommand{\chaos}{\mcH} 

\newcommand{\rate}{\msI}
\newcommand{\multi}{\nu} 
\newcommand{\Law}{\textnormal{Law}}
\newcommand{\sym}{\textnormal{sym}}
\DeclareMathOperator{\Sym}{Sym}

\let\oldskull\skull
\def\skull{\mathord{\oldskull}}
\newcommand{\cem}{\skull}

\newcommand{\target}{F} 
\newcommand{\sol}{\textnormal{sol}}

\newcommand{\place}{\,\cdot\,}

\newcommand{\wt}{\textnormal{wt}}
\newcommand{\weight}{\textnormal{-}\wt}

\newcommand{\test}{\phi} 

\newcommand{\hompd}{\Gamma}

\newcommand{\rem}{\varpi} 

\newcommand{\chem}{\chi} 

\newcommand{\thresh}{\chem^{*}} 

\newcommand{\Trdet}{T^{*}} 

\newcommand{\mean}[1]{\overline{#1}} 

\newcommand{\embed}{\hookrightarrow}

\newcommand{\gnsper}{\textnormal{gns-per}}

\newcommand{\pred}{\textnormal{pred}}

\newcommand{\back}[1]{\tilde{#1}} 

\newcommand{\action}{\msA} 

\makeatletter
\newcommand{\mytextsize}{\f@size pt}
\makeatother

\newlength\RSu 
\newlength\RSwidth 
\newlength\RSbaseline 
\newtoggle{isoption}

\iftoggle{diagrams}{
	\iftoggle{colours}{
		\newcommand{\Checkmark}[1]{
			\togglefalse{isoption}
			\ifnumequal{#1}{1}{\toggletrue{isoption}\raisebox{-0.7ex}{\includegraphics[page = 10, width = 0.54cm*\ratio{\mytextsize}{10.95 pt}]{diagrams_coloured.pdf}\,}}{}
			\ifnumequal{#1}{10}{\toggletrue{isoption}\raisebox{-0.7ex}{\includegraphics[page = 11, width = 0.58cm*\ratio{\mytextsize}{10.95 pt}]{diagrams_coloured.pdf}\,}}{}
			\ifnumequal{#1}{2}{\toggletrue{isoption}\raisebox{-0.7ex}{\includegraphics[page = 12, width = 0.54cm*\ratio{\mytextsize}{10.95 pt}]{diagrams_coloured.pdf}\,}}{}
			\ifnumequal{#1}{20}{\toggletrue{isoption}\raisebox{-0.7ex}{\includegraphics[page = 13, width = 0.58cm*\ratio{\mytextsize}{10.95 pt}]{diagrams_coloured.pdf}\,}}{}
			\iftoggle{isoption}{}{\PackageError{Checkmark}{Undefined option to Checkmark: #1}{}}
		}%
		\newcommand{\PreThree}[1]{
			\togglefalse{isoption}
			\ifnumequal{#1}{1}{\toggletrue{isoption}\raisebox{-0.7ex}{\includegraphics[page = 28, width = 0.54cm*\ratio{\mytextsize}{10.95 pt}]{diagrams_coloured.pdf}\,}}{}
			\ifnumequal{#1}{10}{\toggletrue{isoption}\raisebox{-0.7ex}{\includegraphics[page = 29, width = 0.58cm*\ratio{\mytextsize}{10.95 pt}]{diagrams_coloured.pdf}\,}}{}
			\ifnumequal{#1}{2}{\toggletrue{isoption}\raisebox{-0.7ex}{\includegraphics[page = 30, width = 0.54cm*\ratio{\mytextsize}{10.95 pt}]{diagrams_coloured.pdf}\,}}{}
			\ifnumequal{#1}{20}{\toggletrue{isoption}\raisebox{-0.7ex}{\includegraphics[page = 31, width = 0.58cm*\ratio{\mytextsize}{10.95 pt}]{diagrams_coloured.pdf}\,}}{}
			\iftoggle{isoption}{}{\PackageError{PreThree}{Undefined option to PreThree: #1}{}}
		}%
		\newcommand{\ti}{\mathchoice
			{\,\includegraphics[page = 32, width = 0.1cm*\ratio{\mytextsize}{10.95 pt}]{diagrams_coloured.pdf}\,}
			{\,\includegraphics[page = 32, width = 0.1cm*\ratio{\mytextsize}{10.95 pt}]{diagrams_coloured.pdf}\,}
			{\scalebox{0.8}{\,\includegraphics[page = 32, width = 0.1cm*\ratio{\mytextsize}{10.95 pt}]{diagrams_coloured.pdf}\,}}
			{\scalebox{0.64}{\,\includegraphics[page = 32, width = 0.1cm*\ratio{\mytextsize}{10.95 pt}]{diagrams_coloured.pdf}\,}}
		}%
		\newcommand{\ty}{\mathchoice
			{\,\includegraphics[page = 33, width = 0.25cm*\ratio{\mytextsize}{10.95 pt}]{diagrams_coloured.pdf}\,}
			{\,\includegraphics[page = 33, width = 0.25cm*\ratio{\mytextsize}{10.95 pt}]{diagrams_coloured.pdf}\,}
			{\scalebox{0.8}{\,\includegraphics[page = 33, width = 0.25cm*\ratio{\mytextsize}{10.95 pt}]{diagrams_coloured.pdf}\,}}
			{\scalebox{0.64}{\,\includegraphics[page = 33, width = 0.25cm*\ratio{\mytextsize}{10.95 pt}]{diagrams_coloured.pdf}\,}}
		}%
		\newcommand{\tp}{\mathchoice
			{\,\includegraphics[page = 34, width = 0.36cm*\ratio{\mytextsize}{10.95 pt}]{diagrams_coloured.pdf}\,}
			{\,\includegraphics[page = 34, width = 0.36cm*\ratio{\mytextsize}{10.95 pt}]{diagrams_coloured.pdf}\,}
			{\scalebox{0.8}{\,\includegraphics[page = 34, width = 0.36cm*\ratio{\mytextsize}{10.95 pt}]{diagrams_coloured.pdf}\,}}
			{\scalebox{0.64}{\,\includegraphics[page = 34, width = 0.36cm*\ratio{\mytextsize}{10.95 pt}]{diagrams_coloured.pdf}\,}}
		}%
		\newcommand{\tc}{\mathchoice
			{\,\includegraphics[page = 35, width = 0.36cm*\ratio{\mytextsize}{10.95 pt}]{diagrams_coloured.pdf}\,}
			{\,\includegraphics[page = 35, width = 0.36cm*\ratio{\mytextsize}{10.95 pt}]{diagrams_coloured.pdf}\,}
			{\scalebox{0.8}{\,\includegraphics[page = 35, width = 0.36cm*\ratio{\mytextsize}{10.95 pt}]{diagrams_coloured.pdf}\,}}
			{\scalebox{0.64}{\,\includegraphics[page = 35, width = 0.36cm*\ratio{\mytextsize}{10.95 pt}]{diagrams_coloured.pdf}\,}}
		}%
		\newcommand{\tl}{\mathchoice
			{\,\includegraphics[page = 36, width = 0.18cm*\ratio{\mytextsize}{10.95 pt}]{diagrams_coloured.pdf}\,}
			{\,\includegraphics[page = 36, width = 0.18cm*\ratio{\mytextsize}{10.95 pt}]{diagrams_coloured.pdf}\,}
			{\scalebox{0.8}{\,\includegraphics[page = 36, width = 0.18cm*\ratio{\mytextsize}{10.95 pt}]{diagrams_coloured.pdf}\,}}
			{\scalebox{0.64}{\,\includegraphics[page = 36, width = 0.18cm*\ratio{\mytextsize}{10.95 pt}]{diagrams_coloured.pdf}\,}}
		}%
	}{
		\newcommand{\Checkmark}[1]{
			\togglefalse{isoption}
			\ifnumequal{#1}{1}{\toggletrue{isoption}\raisebox{-0.7ex}{\includegraphics[page = 10, width = 0.54cm*\ratio{\mytextsize}{10.95 pt}]{diagrams_grayscale.pdf}\,}}{}
			\ifnumequal{#1}{10}{\toggletrue{isoption}\raisebox{-0.7ex}{\includegraphics[page = 11, width = 0.58cm*\ratio{\mytextsize}{10.95 pt}]{diagrams_grayscale.pdf}\,}}{}
			\ifnumequal{#1}{2}{\toggletrue{isoption}\raisebox{-0.7ex}{\includegraphics[page = 12, width = 0.54cm*\ratio{\mytextsize}{10.95 pt}]{diagrams_grayscale.pdf}\,}}{}
			\ifnumequal{#1}{20}{\toggletrue{isoption}\raisebox{-0.7ex}{\includegraphics[page = 13, width = 0.58cm*\ratio{\mytextsize}{10.95 pt}]{diagrams_grayscale.pdf}\,}}{}
			\iftoggle{isoption}{}{\PackageError{Checkmark}{Undefined option to Checkmark: #1}{}}
		}%
		\newcommand{\PreThree}[1]{
			\togglefalse{isoption}
			\ifnumequal{#1}{1}{\toggletrue{isoption}\raisebox{-0.7ex}{\includegraphics[page = 28, width = 0.54cm*\ratio{\mytextsize}{10.95 pt}]{diagrams_grayscale.pdf}\,}}{}
			\ifnumequal{#1}{10}{\toggletrue{isoption}\raisebox{-0.7ex}{\includegraphics[page = 29, width = 0.58cm*\ratio{\mytextsize}{10.95 pt}]{diagrams_grayscale.pdf}\,}}{}
			\ifnumequal{#1}{2}{\toggletrue{isoption}\raisebox{-0.7ex}{\includegraphics[page = 30, width = 0.54cm*\ratio{\mytextsize}{10.95 pt}]{diagrams_grayscale.pdf}\,}}{}
			\ifnumequal{#1}{20}{\toggletrue{isoption}\raisebox{-0.7ex}{\includegraphics[page = 31, width = 0.58cm*\ratio{\mytextsize}{10.95 pt}]{diagrams_grayscale.pdf}\,}}{}
			\iftoggle{isoption}{}{\PackageError{PreThree}{Undefined option to PreThree: #1}{}}
		}%
		\newcommand{\ti}{\mathchoice
			{\,\includegraphics[page = 32, width = 0.1cm*\ratio{\mytextsize}{10.95 pt}]{diagrams_grayscale.pdf}\,}
			{\,\includegraphics[page = 32, width = 0.1cm*\ratio{\mytextsize}{10.95 pt}]{diagrams_grayscale.pdf}\,}
			{\scalebox{0.8}{\,\includegraphics[page = 32, width = 0.1cm*\ratio{\mytextsize}{10.95 pt}]{diagrams_grayscale.pdf}\,}}
			{\scalebox{0.64}{\,\includegraphics[page = 32, width = 0.1cm*\ratio{\mytextsize}{10.95 pt}]{diagrams_grayscale.pdf}\,}}
		}%
		\newcommand{\ty}{\mathchoice
			{\,\includegraphics[page = 33, width = 0.25cm*\ratio{\mytextsize}{10.95 pt}]{diagrams_grayscale.pdf}\,}
			{\,\includegraphics[page = 33, width = 0.25cm*\ratio{\mytextsize}{10.95 pt}]{diagrams_grayscale.pdf}\,}
			{\scalebox{0.8}{\,\includegraphics[page = 33, width = 0.25cm*\ratio{\mytextsize}{10.95 pt}]{diagrams_grayscale.pdf}\,}}
			{\scalebox{0.64}{\,\includegraphics[page = 33, width = 0.25cm*\ratio{\mytextsize}{10.95 pt}]{diagrams_grayscale.pdf}\,}}
		}%
		\newcommand{\tp}{\mathchoice
			{\,\includegraphics[page = 34, width = 0.36cm*\ratio{\mytextsize}{10.95 pt}]{diagrams_grayscale.pdf}\,}
			{\,\includegraphics[page = 34, width = 0.36cm*\ratio{\mytextsize}{10.95 pt}]{diagrams_grayscale.pdf}\,}
			{\scalebox{0.8}{\,\includegraphics[page = 34, width = 0.36cm*\ratio{\mytextsize}{10.95 pt}]{diagrams_grayscale.pdf}\,}}
			{\scalebox{0.64}{\,\includegraphics[page = 34, width = 0.36cm*\ratio{\mytextsize}{10.95 pt}]{diagrams_grayscale.pdf}\,}}
		}%
		\newcommand{\tc}{\mathchoice
			{\,\includegraphics[page = 35, width = 0.36cm*\ratio{\mytextsize}{10.95 pt}]{diagrams_grayscale.pdf}\,}
			{\,\includegraphics[page = 35, width = 0.36cm*\ratio{\mytextsize}{10.95 pt}]{diagrams_grayscale.pdf}\,}
			{\scalebox{0.8}{\,\includegraphics[page = 35, width = 0.36cm*\ratio{\mytextsize}{10.95 pt}]{diagrams_grayscale.pdf}\,}}
			{\scalebox{0.64}{\,\includegraphics[page = 35, width = 0.36cm*\ratio{\mytextsize}{10.95 pt}]{diagrams_grayscale.pdf}\,}}
		}%
		\newcommand{\tl}{\mathchoice
			{\,\includegraphics[page = 36, width = 0.18cm*\ratio{\mytextsize}{10.95 pt}]{diagrams_grayscale.pdf}\,}
			{\,\includegraphics[page = 36, width = 0.18cm*\ratio{\mytextsize}{10.95 pt}]{diagrams_grayscale.pdf}\,}
			{\scalebox{0.8}{\,\includegraphics[page = 36, width = 0.18cm*\ratio{\mytextsize}{10.95 pt}]{diagrams_grayscale.pdf}\,}}
			{\scalebox{0.64}{\,\includegraphics[page = 36, width = 0.18cm*\ratio{\mytextsize}{10.95 pt}]{diagrams_grayscale.pdf}\,}}
		}%
	}%
}{
	\newcommand{\Checkmark}[1]{\normalfont\textbf{Ch(#1)}}
	\newcommand{\PreThree}[1]{\normalfont\textbf{PT(#1)}}
	\newcommand{\ti}{\normalfont\textbf{ti}}
	\newcommand{\ty}{\normalfont\textbf{ty}}
	\newcommand{\tp}{\normalfont\textbf{tp}}
	\newcommand{\tc}{\normalfont\textbf{tc}}
	\newcommand{\tl}{\normalfont\textbf{tl}}
}%

\usetikzlibrary{shapes}
\usetikzlibrary{decorations.markings}

\tikzset{
	root/.style={circle, fill=rootcolor, inner sep=0ex, minimum size=1.1ex},
	dot/.style={circle, fill=black, inner sep=0ex, minimum size=0.55ex},
	rootvar/.style={circle,fill=rootcolor!10, draw=rootcolor, inner sep=0ex, minimum size=1.1ex},
	scalarkernel/.style={black, line width=\RSwidth, shorten <=0.2ex, decoration={markings,mark=at position 0.96 with {\arrow[scale=1.1,vectorcolor]{>}}}, postaction={decorate},shorten >=0.4ex},
	vectorkernel/.style={vectorcolor, line width=\RSwidth, shorten <=0.2ex, decoration={markings,mark=at position 0.96 with {\arrow[scale=1.1,vectorcolor]{>}}}, postaction={decorate},shorten >=0.4ex},
	vectorcircle/.style={circle, fill=none, draw=vectorcolor, inner sep=0ex, minimum size=0.66ex, line width=0.7\RSwidth, scale=1.15},
}%
\begin{document}

\maketitle

\begin{abstract}
	\noindent
	We study an additive-noise approximation to Keller--Segel--Dean--Kawasaki dynamics, which is proposed as an approximate model to the fluctuating hydrodynamics of chemotactically interacting particles around their mean-field limit. 
	As such, the interaction potential is given by the Green's function associated to Poisson's equation, which is singular around the origin.
	Two parameters play a key r\^{o}le in the approximation: the noise intensity $\eps$ which captures the amplitude of fluctuations (tending to zero as the effective system size tends to infinity) and the correlation length $\delta$ which represents the effective scale under consideration. 
	Let $\delta(\eps)\to0$ as $\eps\to0$.
	Under the relative scaling assumption $\lim_{\eps\to0}\eps\log(\delta(\eps)^{-1})=0$ we obtain analogues of law of large numbers and large deviation principles in irregular spaces of distributions using methods of singular stochastic partial differential equations. The same techniques also yield a central limit theorem under the relative scaling $\lim_{\eps\to0}\eps^{\sfrac{1}{2}}\log(\delta(\eps)^{-1})=0$.
	Assuming the more restrictive relative scaling $\lim_{\eps\to0}\eps^{\sfrac{1}{2}}\delta^{-\gamma-2}=0$ for some $\gamma\in(-1/2,0)$, we also obtain analogues of law of large numbers and large deviation principles in regular function spaces using a mixture of pathwise and probabilistic tools.
	We further describe consequences of these results relevant to applications of our approximation in studying continuum fluctuations of particle systems.
\end{abstract}
{%
}%
\microtypesetup{protrusion=false} 

\tableofcontents

\microtypesetup{protrusion=true}
\section{Introduction}\label{sec:introduction}
The goal of this work is to understand the behaviour of the so-called \emph{Keller--Segel--Dean--Kawasaki} equation,
\begin{equation}\label{eq:KSDK_intro_torus}
	\begin{cases}
		\begin{aligned}
			(\partial_{t}-\Delta)\rho&=-\chem\vdiv(\rho\nabla\Phi_{\rho})-\eps^{\sfrac{1}{2}}\vdiv(\sqrt{\rho}\boldsymbol{\xi}), &&~\text{in}~(0,T]\times\mbT^{d},\\
			-\Delta\Phi_{\rho}&=\rho-\inner{\rho}{1}_{L^{2}(\mbT^{d})},&&~\text{in}~(0,T]\times\mbT^{d},\\
			\rho\tzero&=\rho_{0},&&~\text{on}~\mbT^{d},
		\end{aligned}
	\end{cases}
\end{equation}
defined on the $d\in\mbN$-dimensional torus $\mbT^{d}=\mbR^{d}/\mbZ^{d}$,
where $\chem\in\mbR$ is the \emph{chemotactic sensitivity},
$\boldsymbol{\xi}$ denotes a vector-valued space-time white noise
and $\eps>0$ the \emph{noise intensity}. 
Assume that the initial data $\rho_{0}$ is non-negative and of unit mass.

The interest in~\eqref{eq:KSDK_intro_torus} stems from its connection to the system of weakly-interacting overdamped Langevin diffusions,
\begin{equation}\label{eq:IPS_KS_intro_torus}
	\dd X^{i}_{t}=\sqrt{2}\dd B^{i}_{t}+\frac{\chem}{N}\sum_{j=1,\, j\neq i}^{N}\nabla\msG(X^{i}_{t}-X^{j}_{t})\dd t,\qquad i=1,\ldots,N,
\end{equation} 
where $N\in\mbN$ denotes the population size, 
$(B^{i})_{i=1}^{N}$ a family of independent standard Brownian motions on $\mbT^{d}$ and $\msG$ the Green's function associated to Poisson's equation on $\mbT^{d}$ (so that $\Phi_{\rho}=\msG\ast\rho$.)  
For every $x\in[-1/2,1/2]^{d}$ one can decompose $\msG([x])=K(x)+K_{0}(x)$, where $[x]=x+\mbZ^{2}$ denotes the associated element of $\mbT^{d}$, $K$ the Green's function associated to Poisson's equation on the full space and $K_{0}$ a smooth correction to periodize $\msG$. 
In particular, it follows that $\msG$ is singular around the origin.
This combined with the singularity of $\Delta\msG$ causes the dynamics of~\eqref{eq:IPS_KS_intro_torus} to be highly non-trivial, see for example~\cite{fournier_tardy_25}.

\begin{details}
	Denote for every $x\in\mbR^{d}$ the associated element of $\mbT^{d}$ by $[x]=x+\mbZ^{d}$.
	Let $\msG$ be the Green's function associated to Poisson's equation on $\mbT^{d}$, i.e.\
	\begin{equation*}
		\msG\from\mbT^{d}\to\mbR,\qquad\msG([x])\defeq\sum_{\omega\in\mbZ^{d}\setminus\{0\}}\euler^{2\uppi\upi\inner{\omega}{x}}\frac{1}{\abs{2\uppi\omega}^{2}}.
	\end{equation*}
	It follows that formally $-\Delta\msG([x])=\sum_{\omega\in\mbZ^{d}\setminus\{0\}}\euler^{2\uppi\upi\inner{\omega}{x}}=\bigl(\sum_{z\in\mbZ^{d}}\delta_{0}(x+z)\bigr)-1$. 
	Let $K$ be the Green's function associated to Poisson's equation on the full space, that is, $-\Delta K(x)=\delta_{0}(x)$ for every $x\in\mbR^{d}$. 
	Assume $K_{0}$ is such that $-\Delta K_{0}=-1$ in $(-1/2,1/2)^{d}$ and such that $(K+K_{0})(x)=\msG([x])$ for every $x\in\partial[-1/2,1/2]^{d}$. Then, $-\Delta(K+K_{0})(x)=\delta_{0}(x)-1$ for every $x\in(-1/2,1/2)^{d}$ and by uniqueness it follows that $\msG([x])=K(x)+K_{0}(x)$ for every $x\in[-1/2,1/2]^{d}$.
\end{details}

Passing to the limit $N\to\infty$, one may expect the system~\eqref{eq:IPS_KS_intro_torus} to be a stochastic particle approximation to the solution $\rdet$ of the deterministic (parabolic-elliptic) \emph{(Patlak--)Keller--Segel} PDE of chemotaxis~\cite{patlak_53,keller_segel_70},
\begin{equation}\label{eq:rdet_intro}
	\begin{cases}
		\begin{aligned}
			(\partial_{t}-\Delta)\rdet&=-\chem\vdiv(\rdet\nabla\Phi_{\rdet}),&&~\text{in}~(0,T]\times\mbT^{d},\\
			-\Delta\Phi_{\rdet}&=\rdet-\inner{\rdet}{1}_{L^{2}(\mbT^{d})},&&~\text{in}~(0,T]\times\mbT^{d},\\
			\rdet\tzero&=\rho_{0},&&~\text{on}~\mbT^{d}.
		\end{aligned}
	\end{cases}
\end{equation}
This convergence has been verified in some cases depending on domain, dimension and $\chem$, see e.g.~\cite{fournier_jourdain_17,tardy_24,bresch_jabin_wang_23}.

On the other hand, for $\eps=2/N$, equation~\eqref{eq:KSDK_intro_torus} describes the continuum fluctuations of~\eqref{eq:IPS_KS_intro_torus} through the theory of fluctuating hydrodynamics (FHD)~\cite{dean_96,kawasaki_94, landau_lifshitz_87,spohn_91,giacomin_lebowitz_presutti_99,bouchet_gawedzki_nardini_16}, which is concerned with random fluctuations of large systems around their averaged behaviour as well as characterizing rare events of these systems. Large deviation principles (LDP) and central limit theorems (CLT) are concrete mathematical examples of such statements. Results of this type can often be leveraged to study phenomenological properties of the associated systems such as metastability, non-equilibrium phase transitions and long-range correlations, \cite{bertini_etal_15}. The theory is also related to robust simulation of large, random systems or those exhibiting small-scale unpredictability.

However, despite the numerous motivations to study~\eqref{eq:KSDK_intro_torus}, it provides significant analytic challenges. Not only is it supercritical in the sense of~\cite{hairer_14_RegStruct}, it was also shown by Konarovskyi, Lehmann and von~Renesse~\cite{konarovskyi_lehmann_vRenesse_19,konarovskyi_lehmann_vRenesse_20} that stochastic PDEs driven by \emph{Dean--Kawasaki noise} $\vdiv(\sqrt{\rho}\boldsymbol{\xi})$ are unstable. Indeed upon replacing $\msG$ by a smooth function, it has been shown in~\cite{konarovskyi_lehmann_vRenesse_20}, that a solution (in the sense of a martingale problem) exists if and only if the initial data is an empirical measure and $\eps=2/N$. Even though there are singular interactions that recover well-posedness (see e.g.\ the references in~\cite{konarovskyi_lehmann_vRenesse_20}), we do not expect the situation to be improved in the case of~\eqref{eq:KSDK_intro_torus}.

In recent years there has been great progress in finding good approximations to equations driven by Dean--Kawasaki noise~\cite{fehrman_gess_23_zero_range,dirr_fehrman_gess_20,wang_wu_zhang_24,wu_zhang_24,cornalba_fischer_23,cornalba_shardlow_23,cornalba_fischer_ingmanns_raithel_26,cornalba_fischer_25,djurdjevac_kremp_perkowski_24,djurdjevac_ji_perkowski_25}. 
However, the case of interactions that are both non-incompressible and singular beyond the Ladyzhenskaya–Prodi–Serrin condition, such as $\nabla\msG$, is still open.

Inspired by the physics literature~\cite{mahdisoltani_et_al_21}, we take a first step in this direction by considering an \emph{additive-noise approximation} $\rho^{(\eps)}_{\delta}$  given by
\begin{equation}\label{eq:rKS_intro}
	\begin{cases}
		\begin{aligned}
			(\partial_{t}-\Delta)\rho&=-\chem\vdiv(\rho\nabla\Phi_{\rho})-\eps^{\sfrac{1}{2}}\vdiv(\srdet\boldsymbol{\xi}^{\delta}), &&~\text{in}~(0,T]\times\mbT^{d},\\
			-\Delta\Phi_{\rho}&=\rho-\inner{\rho}{1}_{L^{2}(\mbT^{d})},&&~\text{in}~(0,T]\times\mbT^{d},\\
			\rho\tzero&=\rho_{0},&&~\text{on}~\mbT^{d},
		\end{aligned}
	\end{cases}
\end{equation}
where $\delta>0$ denotes the \emph{correlation length},
$\psi_{\delta}$ a smooth spatial mollifier of scale $\delta$,
$\boldsymbol{\xi}^{\delta}\defeq\psi_{\delta}\ast\boldsymbol{\xi}$ 
and $\rdet$ the solution to the deterministic Keller--Segel equation~\eqref{eq:rdet_intro}. 

While the additive-noise structure of~\eqref{eq:rKS_intro} does not satisfy a natural fluctuation-dissipation relation, i.e.\ it is no longer a gradient flow, we advocate it for two reasons. Firstly, as we will show in this paper, it provides the correct Gaussian fluctuations around the mean-field limit and secondly, at the time of writing, it has not been possible to incorporate truly non-linear noise, in the style of~\cite{fehrman_gess_24,fehrman_gess_23_zero_range,wang_wu_zhang_24,wu_zhang_24}, along with singular non-local interactions beyond the Ladyzhenskaya--Prodi--Serrin threshold.
\subsection{Main Results}
In the following, due to reasons of scaling sub-criticality~\cite{hairer_14_RegStruct}, we restrict our attention to two dimensions. In this situation, it is known for large $\chem$ that the solution $\rdet$ to the deterministic Keller--Segel equation can blow up in $L^{2}(\mbT^{2})$ in finite time, see e.g.~\cite[Thm.~8.1]{kiselev_xu_16_suppression}. Denote this deterministic blow-up time by $\Trdet$ and choose the time horizon $T<\Trdet$, so that the additive noise of~\eqref{eq:rKS_intro} is well-defined on $[0,T]$. 
Let $\mcC^{\alpha+1}(\mbT^{2})$ be the space of H\"{o}lder--Besov distributions of regularity $\alpha+1<-1$ (for the definition see Subsection~\ref{subsec:notation}.) By the regularity of space-time white noise in two dimensions, it follows that $\mcC^{\alpha+1}(\mbT^{2})$ is the natural regularity of solutions to~\eqref{eq:rKS_intro}.
However, due to the presence of the stochastic term in~\eqref{eq:rKS_intro}, it may still happen that $\rho^{(\eps)}_{\delta(\eps)}$ blows up in $\mcC^{\alpha+1}(\mbT^{2})$ before time $T$. Hence, we shall work in a space $(\mcC^{\alpha+1}(\mbT^{2}))^{\sol}_{T}$ of functions continuous in time and taking values in $\mcC^{\alpha+1}(\mbT^{2})$ that are allowed to blow up before time $T$ but once blown up are not allowed to return. For the precise details, see Subsection~\ref{subsec:notation}. We further denote by $\mcH^{\gamma}(\mbT^{2})$ the Sobolev space of regularity $\gamma\in\mbR$. The space $(\mcH^{\gamma}(\mbT^{2}))^{\sol}_{T}$ is then defined \emph{mutatis mutandis}.
\begin{theoremsketch}[{\ref{thm:LLN_rough} \textnormal{(LLN)}}]
	Let $\rho_{0}\in\mcH^{2}(\mbT^{2})$ be such that $\rho_{0}>0$ and $\inner{\rho_{0}}{1}_{L^{2}}=1$, $\rdet$ be the weak solution to~\eqref{eq:rdet_intro} with initial data $\rho_{0}$ and chemotactic sensitivity $\chem\in\mbR$, and $\Trdet$ be its maximal time of existence. Assume $\delta(\eps)\to0$ and $\eps\log(\delta(\eps)^{-1})\to0$ as $\eps\to0$. Then for all $T<\Trdet$ and $\alpha\in(-9/4,-2)$, it follows that $\rho^{(\eps)}_{\delta(\eps)}\to\rdet$ in $(\mcC^{\alpha+1}(\mbT^{2}))^{\sol}_{T}$ in probability as $\eps\to0$.
\end{theoremsketch}
\begin{proof}[Sketch of Proof.]
	Passing to the paracontrolled decomposition~\cite{gubinelli_15_GIP}, we can show that the solution map is locally Lipschitz continuous in a vector of stochastic objects $\mbX^{(\eps)}_{\delta(\eps)}$ known as the noise enhancement. Using the relative scaling $\eps\log(\delta(\eps)^{-1})\to0$ , we can deduce that $\mbX^{(\eps)}_{\delta(\eps)}\to0$ in probability as $\eps\to0$. The claim then follows by the continuous mapping theorem. For the full proof see Theorem~\ref{thm:LLN_rough}.
\end{proof}
Theorem~\ref{thm:LLN_rough} is consistent with the propagation of chaos for~\eqref{eq:IPS_KS_intro_torus} which was shown to converge to the mean-field limit~\eqref{eq:rdet_intro} for appropriate initial data and $\chem$ sufficiently small, see~\cite{fournier_jourdain_17,tardy_24,bresch_jabin_wang_19,bresch_jabin_wang_20,bresch_jabin_wang_23,decourcel_rosenzweig_serfaty_25_attractive}.

The same continuities of the solution map also allow us to establish a CLT.
\begin{theoremsketch}[{\ref{thm:CLT} \textnormal{(CLT)}}]
	Let $\rho_{0}\in\mcH^{2}(\mbT^{2})$ be such that $\rho_{0}>0$ and $\inner{\rho_{0}}{1}_{L^{2}}=1$, $\rdet$ be the weak solution to~\eqref{eq:rdet_intro} with initial data $\rho_{0}$ and chemotactic sensitivity $\chem\in\mbR$, and $\Trdet$ be its maximal time of existence. Assume $\delta(\eps)\to0$ and $\eps^{\sfrac{1}{2}}\log(\delta(\eps)^{-1})\to0$ as $\eps\to0$. Then for all $T<\Trdet$ and  $\alpha\in(-9/4,-2)$, it follows that $\eps^{-\sfrac{1}{2}}(\rho^{(\eps)}_{\delta(\eps)}-\rdet)\to v$ in $(\mcC^{\alpha+1}(\mbT^{2}))^{\sol}_{T}$ in probability as $\eps\to0$, where $v\in C_{T}\mcC^{\alpha+1}(\mbT^{2})$ is the unique mild solution to
	\begin{equation}\label{eq:genOU_intro}
		\begin{cases}
			\begin{aligned}
				(\partial_{t}-\Delta)v&=-\chem\vdiv(v\nabla\Phi_{\rdet})-\chem\vdiv(\rdet\nabla\Phi_{v})-\vdiv(\srdet\boldsymbol{\xi}),\\
				v\tzero&=0.
			\end{aligned}
		\end{cases}
	\end{equation}
	We refer to this $v$ as the generalized Ornstein--Uhlenbeck process.
\end{theoremsketch}
\begin{proof}
	For a proof see Theorem~\ref{thm:CLT}.
\end{proof}
In the repulsive case ($\chem<0$) the CLT for the particle system~\eqref{eq:IPS_KS_intro_torus} has been shown in~\cite{cecchin_nikolaev_25} with leading-order fluctuations converging to the generalized Ornstein--Uhlenbeck process~\eqref{eq:genOU_intro}. In the attractive case ($\chem>0$), a CLT for~\eqref{eq:IPS_KS_intro_torus} is to our knowledge still open. However, if~\eqref{eq:IPS_KS_intro_torus} allows a CLT (possibly for a certain range of small $\chem$), we would expect the leading-order fluctuations to converge to~\eqref{eq:genOU_intro}. 
See also~\cite{wang_zhao_zhu_23} for a CLT of the point-vortex approximation to the two-dimensional incompressible Navier--Stokes equation with leading-order fluctuations converging to a generalized Ornstein--Uhlenbeck process driven by $\nabla^{\perp}\Phi$.
Consequently, the combination of Theorems~\ref{thm:LLN_rough}~\&~\ref{thm:CLT} leads us to expect that our additive-noise approximation is precise up to first order.

Another first-order approximation is given by \emph{linear fluctuating hydrodynamics} $\rdet+\eps^{\sfrac{1}{2}}v$, which, however, is not able to model noise-induced aggregation when started from uniform initial data $\rho_{0}\equiv1$. On the other hand, it is clear that~\eqref{eq:rKS_intro} can blow up to a Dirac mass due to the fluctuations interacting with the non-linearity, even when started from uniformity.

Finally, we can also show an LDP by adapting the techniques of~\cite{friz_victoir_07,hairer_weber_15}.
\begin{theoremsketch}[{\ref{thm:LDP_rough} \textnormal{(LDP)}}]
	Let $\rho_{0}\in\mcH^{2}(\mbT^{2})$ be such that $\rho_{0}>0$ and $\inner{\rho_{0}}{1}_{L^{2}}=1$, $\rdet$ be the weak solution to~\eqref{eq:rdet_intro} with initial data $\rho_{0}$ and chemotactic sensitivity $\chem\in\mbR$, and $\Trdet$ be its maximal time of existence. Assume $\delta(\eps)\to0$ and $\eps\log(\delta(\eps)^{-1})\to0$ as $\eps\to0$. Then for all $T<\Trdet$ and $\alpha\in(-9/4,-2)$, it follows that $(\rho_{\delta(\eps)}^{(\eps)})_{\eps>0}$ satisfies an LDP in $(\mcC^{\alpha+1}(\mbT^{2}))^{\sol}_{T}$ with speed $\eps$ and good rate function
	\begin{equation}\label{eq:rate_function_rKS_intro}
		\rate(\rho)\defeq\inf\Bigl\{\frac{1}{2}\norm{h}_{L^{2}([0,T]\times\mbT^{2};\mbR^{2})}^{2}:(\partial_{t}-\Delta)\rho=-\chem\vdiv(\rho\nabla\Phi_{\rho})-\vdiv(\srdet h),~\rho\tzero=\rho_{0}\Bigr\}.
	\end{equation}
\end{theoremsketch}
\begin{proof}[Sketch of Proof.]
	An adaptation of~\cite{hairer_weber_15} combined with the relative scaling allows us to deduce an LDP for the noise enhancement $(\mbX^{(\eps)}_{\delta(\eps)})_{\eps>0}$, which we can then transfer to $(\rho^{(\eps)}_{\delta(\eps)})_{\eps>0}$ via the contraction principle. For the full proof see Theorem~\ref{thm:LDP_rough}.
\end{proof}
\begin{remark}
	Results similar to Theorem~\ref{thm:LDP_rough} were also established in~\cite{cerrai_debussche_19_Phi_2n_d,cerrai_debussche_19_SNS} for both the $\Phi^{2n}_{d}$~equation and the two-dimensional stochastic Navier--Stokes equation by combining careful energy estimates with the weak-convergence approach. We do not rule out the existence of another proof for Theorem~\ref{thm:LDP_rough} based on these techniques, but do not pursue this line of enquiry any further in this paper.
\end{remark}
The LDP for the particle system~\eqref{eq:IPS_KS_intro_torus} is to our knowledge still open. In the case of Lipschitz interactions, a first result was given by~\cite{dawson_gaertner_87}, see also~\cite{hardin_leble_saff_serfaty_18_LDP_riesz,coghi_deuschel_friz_maurelli_20,hoeksema_holding_maurelli_tse_24,chen_ge_25} for LDPs under more general assumptions on the regularity of the interaction.

After formal manipulations, using the idea that the \emph{Dean–Kawasaki noise} $\vdiv(\sqrt{\rho}\boldsymbol{\xi})$ appearing in~\eqref{eq:KSDK_intro_torus} is $\delta$-correlated with respect to the local metric tensor of the Otto–Wasserstein metric~\cite{otto_01}, one can bring the rate function of~\cite{dawson_gaertner_87} into the following form:
\begin{equation}\label{eq:rate_function_IPS_intro}
	\action(\rho)\defeq\inf\Bigl\{\frac{1}{2}\norm{h}_{L^{2}([0,T]\times\mbT^{2};\mbR^{2})}^{2}:(\partial_{t}-\Delta)\rho=-\chem\vdiv(\rho\nabla\Phi_{\rho})-\vdiv(\sqrt{\rho}h),~\rho\tzero=\rho_{0}\Bigr\}.
\end{equation}
Comparing~\eqref{eq:rate_function_rKS_intro} and~\eqref{eq:rate_function_IPS_intro}, we find that the approximation error lies in the so-called \emph{skeleton equation} appearing inside the infimum. The additive-noise approximation~\eqref{eq:rKS_intro} induces a skeleton that is additive in the Cameron--Martin element $h$, whereas the interacting particle system~\eqref{eq:IPS_KS_intro_torus} is associated to a fully nonlinear skeleton.

Let $\alpha\in(-9/4,-2)$ and denote by $T^{\cem}_{\mcC^{\alpha+1}}[\rho^{(\eps)}_{\delta(\eps)}]$ the blow-up time of $\rho^{(\eps)}_{\delta(\eps)}$ in $\mcC^{\alpha+1}(\mbT^{2})$. In the next result we establish that the probability of observing a blow-up in $\mcC^{\alpha+1}(\mbT^{2})$ before time $T<\Trdet$ is exponentially small in~$\eps^{-1}$. 
\begin{theoremsketch}[{\ref{thm:blow_up_time_estimate}}]
	Let $\rho_{0}\in\mcH^{2}(\mbT^{2})$ be such that $\rho_{0}>0$ and $\inner{\rho_{0}}{1}_{L^{2}}=1$, $\rdet$ be the weak solution to~\eqref{eq:rdet_intro} with initial data $\rho_{0}$ and chemotactic sensitivity $\chem\in\mbR$, and $\Trdet$ be its maximal time of existence. Assume $\delta(\eps)\to0$ and $\eps\log(\delta(\eps)^{-1})\to0$ as $\eps\to0$. Then for all $S\leq T<\Trdet$ and $\alpha\in(-9/4,-2)$, there exists a constant $c>0$ such that
	\begin{equation*}
		\limsup_{\eps\to0}\eps\log\mbP\Bigl(T^{\cem}_{\mcC^{\alpha+1}}[\rho^{(\eps)}_{\delta(\eps)}]\leq S\Bigr)\leq-c.
	\end{equation*}
\end{theoremsketch}
\begin{proof}[Sketch of Proof.]
	The proof follows by Theorem~\ref{thm:LDP_rough} combined with the lower semicontinuity of $\rho\mapsto T^{\cem}_{\mcC^{\alpha+1}}[\rho]$ and the continuity of the skeleton equation in the Cameron--Martin element. For the full proof see Theorem~\ref{thm:blow_up_time_estimate}.
\end{proof}
While Theorem~\ref{thm:blow_up_time_estimate} rules out the possibility of observig a blow-up with high probability, it can still occur that~\eqref{eq:rKS_intro} produces negative values, which is a major modelling error of the additive-noise approximation. To control this event, we impose a more restrictive relative scaling, $\eps^{\sfrac{1}{2}}\delta(\eps)^{-2}\to0$ as $\eps\to0$.
We denote the negative part of every $u\in L^{2}(\mbT^{2})$ by $u^{-}\defeq\min\{0,u\}$ and control $(\rho^{(\eps)}_{\delta(\eps)})^{-}$ in~$L^{2}(\mbT^{2})$.
\begin{theoremsketch}[{\ref{thm:negative_values}}]
	Let $\rho_{0}\in C(\mbT^{2})$ be such that $\rho_{0}>0$ and $\inner{\rho_{0}}{1}_{L^{2}}=1$, $\rdet$ be the weak solution to~\eqref{eq:rdet_intro} with initial data $\rho_{0}$ and chemotactic sensitivity $\chem\in\mbR$, and $\Trdet$ be its maximal time of existence. Assume $\delta(\eps)\to0$ and $\eps^{\sfrac{1}{2}}\delta(\eps)^{-2}\to0$ as $\eps\to0$. Then for all $T<\Trdet$, $L>0$ and $\lambda>0$, there exists some $\mu>0$ such that
	\begin{equation*}
		\limsup_{\eps\to0}\eps(1+\delta(\eps)^{-2})^{2}\log\mbP\Bigl(\norm{(\rho^{(\eps)}_{\delta(\eps)})^{-}}_{C_{S_{L}}L^{2}}\geq\lambda\Bigr)\lesssim-\mu^{2}\norm{\srdet}_{C_{T}L^{2}\cap L_{T}^{2}\mcH^{1}}^{-2},
	\end{equation*}
	where 
	\begin{equation*}
		S_{L}\defeq T\wedge L\wedge\inf\Bigl\{t\in[0,T]:\norm{\rho^{(\eps)}_{\delta(\eps)}(t)}_{L^{2}}>L\Bigr\}
	\end{equation*}
	denotes the first time that the $L^{2}(\mbT^{2})$-norm of $\rho^{(\eps)}_{\delta(\eps)}$ exceeds the level $L$.
\end{theoremsketch}
\begin{proof}[Sketch of Proof.]
	We first pass to the mild formulation of~\eqref{eq:rKS_intro} given by
	\begin{equation}\label{eq:rKS_intro_mild}
		\rho=P\rho_{0}-\chem\vdiv\mcI[\rho\nabla\Phi_{\rho}]-\eps^{\sfrac{1}{2}}\ti^{\delta},
	\end{equation}
	where $P_{t}\defeq\euler^{t\Delta}$ denotes the heat semigroup, $\mcI[f]_{t}\defeq\int_{0}^{t}P_{t-s}f_{s}\dd s$ the resolution of the heat equation and $\ti^{\delta}\defeq\vdiv\mcI[\srdet\boldsymbol{\xi}^{\delta}]$ the (additive) stochastic heat equation.
	The bound then follows by an infinite-dimensional sub-Gaussian upper deviation inequality on $\norm{\ti^{\delta(\eps)}}_{C_{T}L^{2}\cap L_{T}^{2}\mcH^{1}}$ (see Lemma~\ref{lem:lolli_regular_upper_deviation}), combined with the continuity of $\rho^{(\eps)}_{\delta(\eps)}$ in $\eps^{\sfrac{1}{2}}\ti^{\delta(\eps)}$ and the continuity of $u\mapsto u^{-}$ on $L^{2}(\mbT^{2})$. For the full proof see Theorem~\ref{thm:negative_values}.
\end{proof}
Under the same restrictive relative scaling, one can also show an LLN~\&~LDP in spaces of higher regularity.
\begin{theoremsketch}[{\ref{thm:LLN_regular} \textnormal{(LLN)}}]
	Let $\rho_{0}\in C(\mbT^{2})$ be such that $\rho_{0}>0$ and $\inner{\rho_{0}}{1}_{L^{2}}=1$, $\rdet$ be the weak solution to~\eqref{eq:rdet_intro} with initial data $\rho_{0}$ and chemotactic sensitivity $\chem\in\mbR$, and $\Trdet$ be its maximal time of existence. Assume $\delta(\eps)\to0$ and $\eps^{\sfrac{1}{2}}\delta(\eps)^{-\gamma-2}\to0$ as $\eps\to0$ for some $\gamma\in(-1/2,0]$. Then for all $T<\Trdet$, it follows that $\rho^{(\eps)}_{\delta(\eps)}\to\rdet$ in $(\mcH^{\gamma}(\mbT^{2}))^{\sol}_{T}$ in probability as $\eps\to0$.
\end{theoremsketch}
\begin{proof}[Sketch of Proof]
	Using the relative scaling of $(\eps,\delta(\eps))$, we can deduce that $\eps^{\sfrac{1}{2}}\ti^{\delta(\eps)}\to0$ in $C_{T}\mcH^{\gamma}(\mbT^{2})\cap L_{T}^{2}\mcH^{\gamma+1}(\mbT^{2})$ in probability as $\eps\to0$. The claim then follows by the continuous mapping theorem and the continuity of $\rho^{(\eps)}_{\delta(\eps)}$ in $\eps^{\sfrac{1}{2}}\ti^{\delta(\eps)}$. For the full proof see Theorem~\ref{thm:LLN_regular}.
\end{proof}
\begin{theoremsketch}[{\ref{thm:LDP_regular} \textnormal{(LDP)}}]
	Let $\rho_{0}\in C(\mbT^{2})$ be such that $\rho_{0}>0$ and $\inner{\rho_{0}}{1}_{L^{2}}=1$, $\rdet$ be the weak solution to~\eqref{eq:rdet_intro} with initial data $\rho_{0}$ and chemotactic sensitivity $\chem\in\mbR$, and $\Trdet$ be its maximal time of existence. Assume $\delta(\eps)\to0$ and $\eps^{\sfrac{1}{2}}\delta(\eps)^{-\gamma-2}\to0$ as $\eps\to0$ for some $\gamma\in(-1/2,0)$. Then for all $T<\Trdet$, it follows that the sequence $(\rho^{(\eps)}_{\delta(\eps)})_{\eps>0}$ satisfies an LDP in $(\mcH^{\gamma}(\mbT^{2}))^{\sol}_{T}$ with speed $\eps$ and good rate function $\rate$.
\end{theoremsketch}
\begin{proof}[Sketch of Proof]
	An application of the weak-convergence approach to large deviations, the Bou\'{e}--Dupuis formula~\cite{budhiraja_dupuis_maroulas_08} and the relative scaling yields an LDP for $(\eps^{\sfrac{1}{2}}\ti^{\delta(\eps)})_{\eps>0}$, which we can then transfer to $(\rho^{(\eps)}_{\delta(\eps)})_{\eps>0}$ via the contraction principle. For the full proof see Theorem~\ref{thm:LDP_regular}.
\end{proof}
In the present setting there is no analogue of the CLT (Theorem~\ref{thm:CLT}) since the leading-order fluctuations $\eps^{-\sfrac{1}{2}}(\rho^{(\eps)}_{\delta(\eps)}-\rdet)$ are driven by $\ti^{\delta(\eps)}$, which does not converge in spaces $\mcH^{\gamma}(\mbT^{2})$ of regularity $\gamma>-1/2$ as $\eps\to0$.
\begin{remark}[Extensions to More General Non-Local Interactions]
	We naturally expect many of the the arguments of the current paper and~\cite{martini_mayorcas_25} to generalize beyond the two-dimensional (attractive and repulsive) Coulomb potentials considered herein.
	Without detailing all possible variants of the theorems above, we highlight that the two main features we appeal to are sub-criticality and symmetries of our interaction. Concerning the former property, if one considers a more general interaction $C_{T}\mcC^{\beta}(\mbT^{d}) \ni \rho \mapsto \nabla \Psi_{\rho} \in \mcC^{\beta+\gamma}(\mbT^{d})$ bounded and linear, for $\beta, \gamma \in \mbR$, one would be required to check that the product $\rho \nabla \Psi_{\rho}$ is either classically well-defined or \emph{sub-critical} in the sense of singular SPDE, i.e.\ that it can be expected to have \emph{homogeneity} lower than that of almost sure realisations of the noise. On the other hand, symmetry of the Coulomb potential plays a r\^ole for us in showing that the required renormalisation fields diverge less quickly than one would expect from simple power counting.
	A similar phenomenon is observed in stochastic fluid equations, see for example~\cite{daprato_debussche_02}.
\end{remark}
\subsection{Discussion of Relative Scalings}\label{subsec:scaling_discussion}
The more we make use of the mollification in~\eqref{eq:rKS_intro}, the more regularity we can assume of the noise~$\boldsymbol{\xi}^{\delta}$ at the price of a faster divergence as $\delta\to0$, which in turn one then needs to compensate with the noise intensity $\eps$.
Hence, throughout this paper we consider different relative scalings between $\eps$ and $\delta$.
The general relative scaling assumes $\eps\log(\delta(\eps)^{-1})\to0$ and applies in H\"{o}lder--Besov spaces $\mcC^{\alpha+1}(\mbT^{2})$ for all $\alpha\in(-9/4,-2)$,
whereas the restrictive relative scaling imposes that $\eps^{\sfrac{1}{2}}\delta(\eps)^{-\gamma-2}\to0$ as $\eps,\delta(\eps)\to0$ for some $\gamma\in(-1/2,0]$ and applies in the Sobolev space $\mcH^{\gamma}(\mbT^{2})$.
Theorems~\ref{thm:LLN_rough},~\ref{thm:CLT},~\ref{thm:LDP_rough}~and~\ref{thm:blow_up_time_estimate} rely on the general relative scaling and are collectively referred to as the \emph{rough setting}, since they make minimal use of the mollifier in~\eqref{eq:rKS_intro}.
Theorems~\ref{thm:LLN_regular},~\ref{thm:LDP_regular}~and~\ref{thm:negative_values} rely on the restrictive relative scaling and are collectively referred to as the \emph{regular setting}. See Table~\ref{tab:relative_scaling_table} for a summary.

\begin{table}[tbph]
	\centering
	\begin{tabular}{cccc}
		\textbf{Name} & \textbf{Scaling} &  \textbf{Regularity} & \textbf{Results}\\
		\toprule
		General  & $\eps\log(\delta(\eps)^{-1})\to0$ & $\mcC^{\alpha+1}(\mbT^{2})$ & Thms.~\ref{thm:LLN_rough},~\ref{thm:CLT},~\ref{thm:LDP_rough}~\&~\ref{thm:blow_up_time_estimate} \\
		& as $\eps,\delta(\eps)\to0$ & $\alpha\in(-9/4,-2)$ & (rough setting, Sec.~\ref{sec:rough_theory}) \\
		\midrule
		Restrictive  & $\eps^{\sfrac{1}{2}}\delta(\eps)^{-\gamma-2}\to0$ & $\mcH^{\gamma}(\mbT^{2})$ & Thms.~\ref{thm:LLN_regular},~\ref{thm:LDP_regular}~\&~\ref{thm:negative_values} \\
		& as $\eps,\delta(\eps)\to0$ & $\gamma\in(-1/2,0]$ & (regular setting,  Sec.~\ref{sec:regular_theory}) \\
		\bottomrule
	\end{tabular}
	\caption{Summary of the different relative scalings used throughout this paper.}
	\label{tab:relative_scaling_table}
\end{table}

Given the motivation from fluctuating hydrodynamics (see~\eqref{eq:IPS_KS_intro_torus}) one natural scaling relation to hope for is based on the intra-particle distance between uniformly distributed particles. If it were to hold that at each $t\geq 0$ the $(X^{i}_{t})_{i=1}^N$ were uniformly i.i.d.\ distributed (or indeed uniformly separated e.g.\ on a lattice) then the average intra-particle distance in two dimensions is proportional to $N^{-\sfrac{1}{2}}$ (more generally $N^{-\sfrac{1}{d}}$). Hence, the natural scaling to aim for in this setting would be   
\begin{equation}\label{eq:relative_scaling_natural}
	\delta(N)=N^{-\sfrac{1}{2}},\quad\eps(N)=2/N\quad\implies\quad\lim_{N\to\infty}\eps(N)^{\sfrac{1}{2}}\delta(N)^{-1} =\sqrt{2}.
\end{equation}
On the other hand, for general interactions there is no reason to expect particles to be either uniformly distributed or uniformly separated at positive times. To wit in the case of the Keller--Segel particle system, for $\chem>8\uppi$ one expects all particles to concentrate at one location at some random time, see~\cite{fournier_tardy_25}. The heuristic \eqref{eq:relative_scaling_natural} is more naturally applicable to particle systems set on a lattice such as exclusion and zero-range processes~\cite{dirr_fehrman_gess_20,fehrman_gess_23_zero_range}
or for interacting diffusions with mutually repulsive potential~\cite{serfaty_23_gaussian_fluct,serfaty_17_coulomb_review,serfaty_20_coulomb_mean_field}.

 While our restrictive scaling regime, $\lim_{\eps \to 0}\eps^{\sfrac{1}{2}}\delta(\eps)^{-\gamma-2}=0$, and the natural scaling, \eqref{eq:relative_scaling_natural}, are mutually exclusive,
 the general scaling, $\lim_{\eps \to 0}\eps \log(\delta(\eps)^{-1})= 0$, is more permissive and always allows for~\eqref{eq:relative_scaling_natural}.
\paragraph{Structure of the Paper:}
In the rest of this section (Subsection~\ref{subsec:notation}) we introduce notations and conventions used throughout the paper.
In Section~\ref{sec:det_PDE} we establish the existence and regularity of the deterministic, periodic Keller--Segel equation. 
In Section~\ref{sec:regular_theory} we prove the LLN in the regular setting (Theorem~\ref{thm:LLN_regular}). In the same section we also obtain the LDP in the regular setting (Theorem~\ref{thm:LDP_regular}). As an application of the theory developed in Section~\ref{sec:regular_theory} we prove in Subsection~\ref{subsec:negative_values} that solutions to our additive-noise approximation become negative with exponentially small probability (Theorem~\ref{thm:negative_values}).
In Section~\ref{sec:rough_theory} we prove the LLN in the rough setting (Theorem~\ref{thm:LLN_rough}). In the same section we also obtain a CLT (Theorem~\ref{thm:CLT}) and the LDP in the rough setting (Theorem~\ref{thm:LDP_rough}). As an application of the theory developed in Section~\ref{sec:rough_theory} we prove in Subsection~\ref{subsec:blow_up} that solutions to our additive-noise approximation blow up with exponentially small probability (Theorem~\ref{thm:blow_up_time_estimate}).
Finally, we collect a number of necessary results in a handful of appendices. Appendix~\ref{app:toolbox} recaps properties of various function spaces. Appendix~\ref{app:periodic_wn} covers the definition and necessary properties of periodic space-time white noise. Appendices~\ref{app:freidlin_wentzell},~\ref{app:enh_driven_by_h} and~\ref{app:lolli_regular} cover necessary extensions of Friedlin--Wentzell large deviations theory to noise enhancements, analysis of the noise enhancement driven by a Cameron--Martin element and regularity of solutions to the additive stochastic heat equation relevant in our case.
\paragraph{Acknowledgements:}
We would like to express our gratitude to B.~Fehrman, B.~Gess, M.~Gubinelli, N.~Perkowski and H.~Weber for helpful discussions during the writing of this paper.
A.~Martini was supported by the Engineering and Physical Sciences Research Council Doctoral Training Partnerships [grant number EP/R513295/1], by the Lamb \& Flag Scholarship of St John's College, Oxford, and by the European Union [ERC, FluCo, grant agreement No.~101088488]. Views and opinions expressed are however those of the author(s) only and do not necessarily reflect those of the European Union or of the European Research Council. Neither the European Union nor the granting authority can be held responsible for them.
A.~Mayorcas was supported by the DFG research unit FOR2402 and through an extended research invitation by F.~Flandoli to SNS Pisa supported by ERC Advanced Grant no.~101053472.
\subsection{Notations and Conventions}\label{subsec:notation}
We write $\mbN$ for the natural numbers excluding zero, $\mbN_{0}\defeq\mbN\cup\{0\}$ and $\mbN_{-1}\defeq\mbN_{0}\cup\{-1\}$. We define the two-dimensional torus by $\mbT^{2}\defeq\mbR^{2}/\mbZ^{2}$. Throughout, $\abs{\place}$ will indicate the norm $\abs{x}=\bigl(\sum_{i=1}^{2}\lvert x_{i}\rvert^{2}\bigr)^{1/2}$ on $\mbR^{2}$ or the induced metric on $\mbT^{2}$, depending on the context.
For $r>0$ we use the notation $B(0,r)\defeq\{x\in\mbR^{2}:\abs{x}<r\}$. From now on we will write $\inner{a}{b}$ to denote the inner product on any Hilbert space which we either specify or leave clear from the context. For $k,n\in\mbN$ we denote by $C^{k}(\mbT^{2};\mbR^{n})$ (resp.\ $C^{\infty}(\mbT^{2};\mbR^{n})$) the space of $k$-times continuously differentiable (resp.\ smooth) functions on the torus with values in $\mbR^{n}$. We write $\mcS'(\mbT^{2};\mbR)$ for the dual of $C^{\infty}(\mbT^{2};\mbR)$ and $\mcS'(\mbT^{2};\mbR^{n})$ for vectors of elements in $\mcS'(\mbT^{2};\mbR)$. Denote by $\mcS'_{\per}(\mbR^{2};\mbR^{n})$ the space of periodic distributions and (with an abuse of notation) $\mcS'_{\per}(\mbR\times\mbR^{2};\mbR^{n})$ the space of tempered distributions that are periodic in the space variable. Let $\mcS'(\mbR\times\mbT^{2};\mbR^{n})$ be the space of tempered distributions on $\mbR\times\mbT^{2}$.

For $f\in C^\infty(\mbT^2;\mbR)$ (resp.\ complex sequences $(\zeta(\om))_{\om\in\mbZ^{2}}$ with $\overline{\zeta(-\om)}=\zeta(\om)$ that decay faster than any polynomial) we define its Fourier transform (resp.\ inverse Fourier transform) by the expression,
\begin{equation*}
	\msF f(\omega)\defeq\int_{\mbT^2}\euler^{-2\uppi\upi\inner{\om}{x}}f(x)\dd x,\qquad \msF^{-1} \zeta(x)\defeq\sum_{\om\in\mbZ^2}\euler^{2\uppi\upi\inner{\om}{x}}\zeta(\om).
\end{equation*} 
This is extended componentwise to vector-valued functions, by density to $f\in L^p(\mbT^2;\mbR^n)$ for $p\in [1,\infty)$ and by duality to $f\in \mcS'(\mbT^2;\mbR^n)$. Where convenient we use the shorthand $\hat{f}(\omega)\defeq\msF f(\omega)$.

We define the Bessel potential space $\mcH^{\alpha}(\mbT^{2};\mbR^{n})$ of regularity $\alpha\in\mbR$ as the space of periodic distributions $u\in\mcS'(\mbT^{2};\mbR^{n})$ such that
\begin{equation}\label{eq:Bessel_norm}
	\norm{u}_{\mcH^{\alpha}}\defeq\Bigl(\sum_{\om\in\mbZ^{2}}(1+\abs{2\uppi\om}^{2})^{\alpha}\abs{\hat{u}(\om)}^{2}\Bigr)^{1/2}<\infty.
\end{equation}
We denote the mean by $\mean{u}\defeq\inner{u}{1}_{L^{2}(\mbT^{2})}$ for each $u\in L^{2}(\mbT^{2};\mbR)=\mcH^{0}(\mbT^{2};\mbR)$.

We often work in the scale of Besov and H\"{o}lder--Besov spaces whose definitions we summarize. For their basic properties, see e.g.~\cite[App.~A]{martini_mayorcas_25} or the references therein. We let:
\begin{equation*}
	\begin{split}
		&\varrho_{-1},\varrho_{0}\in C^{\infty}(\mbR^{2};[0,1])~\text{be radially symmetric and such that}\\
		&\supp(\varrho_{-1})\subset B(0,1/2),\qquad\supp(\varrho_{0})\subset\{x\in\mbR^{2}:9/32 \leq\abs{x}\leq 1\},\\
		&\sum_{k=-1}^{\infty}\varrho_{k}(x)=1\quad\text{for all}~x\in\mbR^{2}\quad\text{where}~\varrho_{k}(x)\defeq\varrho_{0}(2^{-k}x)\quad\text{for each}~k\in\mbN.
	\end{split}
\end{equation*}
This defines a dyadic partition of unity. Given $k\geq-1$ we write $\Delta_{k}u\defeq\msF^{-1}(\varrho_{k}\msF u)$ for the associated Littlewood--Paley block and given $\alpha\in\mbR$, $p,q\in [1,\infty]$, we define the Besov-norm $\norm{u}_{\mcB^{\alpha}_{p,q}(\mbT^2;\mbR^{n})}\defeq\norm{(2^{k\alpha}\norm{\Delta_{k}u}_{L^{p}(\mbT^{2};\mbR^{n})})_{k\in\mbN_{-1}}}_{\ell^{q}}$ for $u\in\mcS'(\mbT^{2};\mbR^{n})$. We use $\mcB^\alpha_{p,q}(\mbT^2;\mbR^n)$ to denote the completion of $C^{\infty}(\mbT^{2};\mbR^{n})$ under $\norm{\place}_{\mcB^\alpha_{p,q}(\mbT^2;\mbR^{n})}$ and the shorthand $\mcC^{\alpha}(\mbT^2;\mbR^n)\defeq\mcB^\alpha_{\infty,\infty}(\mbT^{2};\mbR^{n})$. The space $\mcB_{2,2}^{\alpha}(\mbT^{2};\mbR^{n})$ coincides with $\mcH^{\alpha}(\mbT^{2};\mbR^{n})$.
\begin{details}
	See \cite[Thm.~23.2]{vanzuijlen_22}.
\end{details}
For $\alpha\in\mbR$ and $p,q\in [1,\infty]$ we use the notation $\mcB^{\alpha-}_{p,q}(\mbT^{2})\defeq\cap_{\alpha'<\alpha}\mcB^{\alpha'}_{p,q}(\mbT^{2})$. When the context is clear we will remove the target space so as to lighten notation.

For $u,v\in C^{\infty}(\mbT^{2};\mbR)$ we define the paraproduct $\pa$ and resonant product $\re$ by
\begin{equation*}
	u\pa v\defeq\sum_{k\geq-1}\sum_{l=-1}^{k-2}\Delta_{l}u\Delta_{k}v,\qquad u\re v\defeq\sum_{\substack{k,l\in\mbN_{-1}\\\abs{k-l}\leq1}}\Delta_{l}u\Delta_{k}v
\end{equation*}
and extend them to the scale of Besov distributions by Bony's lemma~\cite[Thm.~2.82~\&~Thm.~2.85]{bahouri_chemin_danchin_11} and generalizations thereof.

We define the action of the heat semigroup on $f \in L^1(\mbT^2;\mbR)$ by the flow,
\begin{equation*}
	[0,\infty)\ni t\mapsto P_{t} f \defeq\msF^{-1}(\euler^{-t\abs{2\uppi\place}^{2}}\hat{f}(\place))=\msH_t\ast f
\end{equation*}
where the convolution against the heat kernel $\msH_{t}$ on $\mbT^{2}$ is defined by
\begin{equation*}
	\msH_{t}\ast f=\int_{\mbT^{2}}\msH_{t}(\place-y)f(y)\dd y=\int_{(0,1)^{2}}\msH_{t}^{\per}(\place-y)f^{\per}(y)\dd y
\end{equation*}
with $f^{\per}\from\mbR^{2}\to\mbR$ the periodization of $f\from\mbT^{2}\to\mbR$ and $\msH^{\per}_{t}$ given by
\begin{equation*}
	\begin{split}
		\msH_{t}^{\per}(x)&\defeq\frac{1}{4\uppi t}\sum_{n\in\mbZ^2}\euler^{-\frac{\abs{x-n}^2}{4t}}=\sum_{\om\in\mbZ^{2}}\euler^{2\uppi\upi\inner{\om}{x}}\euler^{-t\abs{2\uppi\om}^2}\qquad\text{for}~t>0,\\
		\msH_{0}^{\per}&\defeq\sum_{n\in\mbZ^{2}}\delta_{n},\qquad\text{and}\qquad\msH_{t}^{\per}(x)\defeq0\qquad\text{for}~t<0.
	\end{split}
\end{equation*}
For $f\from[0,T]\times\mbT^{2}\to \mbR$, we define the resolution of the heat equation as
\begin{equation*}
	\mcI[f]_{t}\defeq\int_{0}^{t}P_{t-s}f_s\dd s=\int_{0}^{t}\msH_{t-s}\ast f_s\dd s\qquad\text{where}~t\in[0,T].
\end{equation*}
We define the solution to the elliptic equation $-\Delta\Phi_{f}=f-\mean{f}$ by
\begin{equation*}
	\Phi_{f}\defeq\msG\ast f,\qquad\text{where}\quad f\in\mcS'(\mbT^{2};\mbR)\quad\text{and}
\end{equation*}
\begin{equation*}
	\msG(x)\defeq\sum_{\om\in\mbZ^{2}\setminus\{0\}}\euler^{2\uppi\upi\inner{\om}{x}}\frac{1}{\abs{2\uppi\om}^{2}},\qquad(\msG\ast f)(x)=\sum_{\om\in\mbZ^{2}\setminus\{0\}}\euler^{2\uppi\upi\inner{\om}{x}}\frac{1}{\abs{2\uppi\om}^{2}}\hat{f}(\om).
\end{equation*}
We denote for every $\om=(\om^{1},\om^{2})\in\mbZ^{2}$, $t\in\mbR$ and $j=1,2$ the Fourier multipliers $H^{j}_{t}(\om)\defeq2\uppi\upi\om^{j}\exp(-t\abs{2\uppi\om}^{2})\mathds{1}_{t\geq0}$ and $G^{j}(\om)\defeq2\uppi\upi\om^{j}\abs{2\uppi\om}^{-2}\mathds{1}_{\om\neq0}$.

Given a Banach space $E$ and an interval $I\subseteq[0,\infty)$ we write $C_{I}E\defeq C(I;E)$ for the space of continuous maps $f\from I\to E$. For compact $I$ we equip $C_{I}E$ with the norm $\norm{f}_{C_{I}E}\defeq\sup_{t\in I} \norm{f_t}_{E}$. For $\kappa\in(0,1)$, we define $C^{\kappa}_{I}E\defeq C^{\kappa}(I;E)$ as the completion of $C^{\infty}(I;E)$ under the norm $\norm{f}_{C_{I}^{\kappa}E}\defeq\norm{f}_{C_{I}E}+\sup_{s\neq t\in I}\frac{\norm{f_{t}-f_{s}}_{E}}{\abs{t-s}^{\kappa}}$, where $C^{\infty}(I;E)$ denotes the space of smooth functions in the interior of $I$, which, along with all of their derivatives, can be continuously extended to functions in $C_{I}E$. For $T>0$, we use the shorthand $C_{T}E=C_{[0,T]}E$, $C_{T}^{\kappa} E=C_{[0,T]}^{\kappa}E$ and understand $C_{T}^{0}E=C_{T}E$.
%
%
Given $\eta\geq0$, we denote $f_{\eta\weight}(t)\defeq(1\wedge t)^{\eta}f_{t}$ for each $f\from(0,T]\to E$ and $t\in(0,T]$. We define the Banach space
\begin{equation*}
	C_{\eta;T}E\defeq\Big\{f\from(0,T]\to E:~f_{\eta\weight}(0)\defeq\lim_{t\to0}f_{\eta\weight}(t)\in E,~f_{\eta\weight}\in C_{(0,T]}E\Big\}
\end{equation*}
equipped with the norm
\begin{equation*}
	\norm{f}_{C_{\eta;T}E}\defeq\sup_{t\in(0,T]}(1\wedge t)^{\eta}\norm{f_{t}}_{E}
\end{equation*}
and refer to $\eta$ as the weight at $0$. We also make use of the notation $\msL_{T}^{\kappa}\mcB_{p,q}^{\alpha}(\mbT^{2})\defeq C_{T}^{\kappa}\mcB_{p,q}^{\alpha-2\kappa}(\mbT^{2})\cap C_{T}\mcB_{p,q}^{\alpha}(\mbT^{2})$ to denote an interpolation space, on which we define
\begin{equation*}
	\norm{\place}_{\msL_{T}^{\kappa}\mcB_{p,q}^{\alpha}}=\max\{\norm{\place}_{C_{T}^{\kappa}\mcB_{p,q}^{\alpha-2\kappa}},\norm{\place}_{C_{T}\mcB_{p,q}^{\alpha}}\}.
\end{equation*}
We further understand $\msL^{0}_{T}\mcC^{\alpha}(\mbT^{2})=C_{T}\mcC^{\alpha}(\mbT^{2})$. See Subsection~\ref{subsec:function_spaces} for properties of those spaces.

Let $T>0$, $\alpha\in(-5/2,-2)$ and $\kappa\in(0,1/2)$. To carry out the rough analysis of Section~\ref{sec:rough_theory}, we adapt the pathwise approach developed in~\cite{martini_mayorcas_25}. These methods rely on the following non-linear space of enhanced rough noises $\rksnoise{\alpha}{\kappa}_{T}$, defined as the closure of the subset
\begin{equation*}
	\begin{split}
		&\Bigl\{\Theta(v,f):(v,f)\in\msL^{\kappa}_{T}\mcC^{\alpha+2}(\mbT^{2};\mbR)\times\msL^{\kappa}_{T}\mcC^{2\alpha+5}(\mbT^{2};\mbR),~v_{0}=f_{0}=0\Bigr\}\\
		&\qquad\subset\msL^{\kappa}_{T}\mcC^{\alpha+1}(\mbT^{2};\mbR)\times\msL^{\kappa}_{T}\mcC^{2\alpha+4}(\mbT^{2};\mbR)\times\msL^{\kappa}_{T}\mcC^{3\alpha+6}(\mbT^{2};\mbR^{2})\times\msL^{\kappa}_{T}\mcC^{2\alpha+4}(\mbT^{2};\mbR^{2\times2}),
	\end{split}
\end{equation*}
where the map
\begin{align*}
	\Theta\from&(\msL^{\kappa}_{T}\mcC^{\alpha+2}(\mbT^{2};\mbR)\times\msL^{\kappa}_{T}\mcC^{2\alpha+5}(\mbT^{2};\mbR))\\
	&\to\msL^{\kappa}_{T}\mcC^{\alpha+1}(\mbT^{2};\mbR)\times\msL^{\kappa}_{T}\mcC^{2\alpha+4}(\mbT^{2};\mbR)\times\msL^{\kappa}_{T}\mcC^{3\alpha+6}(\mbT^{2};\mbR^{2})\times\msL^{\kappa}_{T}\mcC^{2\alpha+4}(\mbT^{2};\mbR^{2\times2}),\\
	&(v,f)\mapsto\Theta(v,f),
\end{align*}
is given by 
\begin{align*}
	Y&\defeq\vdiv\mcI[{v}\nabla\Phi_{v}]-f,\\
	\Theta(v,f)&\defeq(v,Y,Y\re\nabla\Phi_{v}+\nabla\Phi_{Y}\re v,\nabla \mcI[v]\re\nabla\Phi_{v}+\nabla^2\mcI[\Phi_{v}]\re v)
\end{align*}
with resonant products acting component-wise. We denote a generic element of this closure by $\mbX=(\ti,\ty,\tp,\tc)\in\rksnoise{\alpha}{\kappa}_{T}$ and equip it with the metric induced by the norm
\begin{equation*}
	\norm{\mbX}_{\rksnoise{\alpha}{\kappa}_{T}}\defeq\max\{\norm{\ti}_{\msL^{\kappa}_{T}\mcC^{\alpha+1}},\norm{\ty}_{\msL^{\kappa}_{T}\mcC^{2\alpha+4}},\norm{\tp}_{\msL^{\kappa}_{T}\mcC^{3\alpha+6}},\norm{\tc}_{\msL^{\kappa}_{T}\mcC^{2\alpha+4}}\}.
\end{equation*}
Given a normed vector space $(\target,\norm{\place}_{\target})$ and a cemetery state $\cem$, we define $\target^{\cem}\defeq\target\cup\{\cem\}$ equipped with the topology generated by the open balls of $\target$ and sets of the form $\{u:\norm{u}_{\target}>R\}\cup\{\cem\}$ for all $R>0$. We define the space $\target^{\sol}_{T}$ by 
\begin{equation*}
	\target^{\sol}_{T}\defeq\Bigl\{f\in C([0,T];\target^{\cem}):f|_{(T^{\cem}_{\target}[f],T]}\equiv\cem\Bigr\},\quad\text{where}\quad T^{\cem}_{\target}[f]\defeq T\wedge\inf\{t\in[0,T]:f(t)=\cem\},
\end{equation*}
which consists of the continuous functions on $[0,T]$ with values in $\target^{\cem}$ that cannot return from the cemetery state. There exists a metric $D_{T}^{\target}$ on $\target^{\sol}_{T}$ such that $D_{T}^{\target}(f_{n},f)\to0$ if and only if for every $L>0$,
\begin{equation}\label{eq:FsolT_metric}
	\lim_{n\to\infty}\sup_{t\in[0,T_{L}]}\norm{f_{n}(t)-f(t)}_{\target}=0,
\end{equation}
where
\begin{equation*}
	T_{L}\defeq T\wedge L\wedge\inf\{t\in[0,T]:\norm{f_{n}(t)}_{\target}>L~\text{or}~\norm{f(t)}_{\target}>L\},
\end{equation*}
cf.~\cite[Lem.~1.2]{chandra_chevyrev_hairer_shen_22}~\&~\cite[Lem.~3.18]{martini_mayorcas_25}.

Let $C^{2}_{1}((0,T]\times\mbT^{2})$ be the space of functions $f\in C((0,T]\times\mbT^{2})$ such that $\partial_{t}f$, $\partial_{x_{i}}f$ and $\partial_{x_{i}}\partial_{x_{j}}f$ exist in $C((0,T)\times\mbT^{2})$ and can be extended to functions in $C((0,T]\times\mbT^{2})$ for all $i,j\in\{1,2\}$.

We do not distinguish between the spaces $L^{2}([0,T]^{n}\times\mbT^{2n}\times\{1,2\}^{n};\mbR)$, $L^{2}([0,T]^{n}\times\mbT^{2n};\mbR^{2n})$ and $L^{2}([0,T]^{n};L^{2}(\mbT^{2n};\mbR^{2n}))$, which are identical up to isometric isomorphisms.

We further denote $L^{2}_{T}\mcH^{\alpha}(\mbT^{2})\defeq L^{2}([0,T];\mcH^{\alpha}(\mbT^{2};\mbR))$ and $W^{1,2}_{T}\mcH^{-1}(\mbT^{2})=W^{1,2}([0,T];\mcH^{-1}(\mbT^{2}))$ the space of functions in $L^{2}_{T}\mcH^{-1}(\mbT^{2})$ with weak derivative (in time) belonging to $L^{2}_{T}\mcH^{-1}(\mbT^{2})$, cf.~\cite[Subsec.~5.9.2]{evans_98_PDE}. On the intersection $C_{T}\mcH^{\alpha}(\mbT^{2})\cap L_{T}^{2}\mcH^{\alpha+1}(\mbT^{2})$ we define the norm $\norm{\place}_{C_{T}\mcH^{\alpha}\cap L_{T}^{2}\mcH^{\alpha+1}}\defeq\max\{\norm{\place}_{C_{T}\mcH^{\alpha}},\norm{\place}_{L_{T}^{2}\mcH^{\alpha+1}}\}$.

We write $\lesssim$ to indicate that an inequality holds up to a constant depending on quantities that we do not keep track of or are fixed throughout. When we do wish to emphasise the dependence on certain quantities $\alpha$, $p$, $d$, we either write $\lesssim_{\alpha,p,d}$ or define $C\defeq C(\alpha,p,d)>0$ and write $\leq C$.

Let $(\Omega,\mcF,\mbP)$ be a probability space that is large enough to support two independent space-time white noises $\xi^{1},\xi^{2}$ on $[0,T]\times\mbT^{2}$ as specified in Subsection~\ref{subsec:existence_and_regularity}. Given $\boldsymbol{\xi}\defeq(\xi^{1},\xi^{2})$, we can define a family of complex Brownian motions by $W^{j}(t,\om)\defeq\boldsymbol{\xi}(\mathds{1}_{[0,t]}(\place)\otimes\euler^{-2\uppi\upi\inner{\om}{\place}}\otimes\delta_{j,\place})$ where $t\geq0$, $\om\in\mbZ^{2}$ and $j=1,2$, see Lemma~\ref{lem:space_time_white_noise_yields_complex_BM}. This probability space will be fixed, so that whenever a property holds almost surely, it will do so with respect to $\mbP$.

Let $n\in\mbN$, $T>0$ and $[0,T]^{n}_{>}\defeq\{(u_1,\ldots,u_n)\in[0,T]^{n}:u_{1}>u_{2}>\ldots>u_{n}\}$. We define the iterated It\^{o} integral, acting on $\phi\in L^{2}([0,T]^{n}_{>}\times\mbT^{2n};\mbR^{2n})$, by
\begin{equation}\label{eq:def_iterated_Ito_Fourier}
	I^{n}(\phi)\defeq\sum_{\om_{1},\ldots,\om_{n}\in\mbZ^{2}}\sum_{j_{1},\ldots,j_{n}=1}^{2}\int_{0}^{T}\dd W^{j_{1}}(u_{1},\om_{1})\ldots\int_{0}^{u_{n-1}}\dd W^{j_{n}}(u_{n},\om_{n})\hat{\phi}(u_{1},-\om_{1},j_{1},\ldots,u_{n},-\om_{n},j_{n}),
\end{equation}
where
\begin{equation*}
	\hat{\phi}(u_{1},\om_{1},j_{1},\ldots,u_{n},\om_{n},j_{n})\defeq\int_{(\mbT^{2})^{n}}\euler^{-2\uppi\upi(\inner{\om_{1}}{x_{1}}+\ldots+\inner{\om_{n}}{x_{n}})}h(u_{1},x_{1},j_{1},\ldots,u_{n},x_{n},j_{n})\dd x_{1}\ldots\dd x_{n}.
\end{equation*}
For $n=0$, it is conventional to set $I^{0}\defeq\textnormal{id}_{\mbR}$.

We define the cut-off 
\begin{equation}\label{eq:def_cut_off}
	\begin{split}
		&\varphi\in C^{\infty}(\mbR^{2})~\text{to be of compact support},\\
		&\supp(\varphi)\subset B(0,1),~\text{even}~\text{and such that}~\varphi(0)=1.
	\end{split}
\end{equation}
Given $\varphi$ satisfying~\eqref{eq:def_cut_off}, we define a sequence of mollifiers $(\psi_{\delta})_{\delta>0}$ on $\mbT^{2}$ by
\begin{equation}\label{eq:def_mollifiers}
	\psi_{\delta}(x)\defeq\sum_{\om\in\mbZ^{2}}\euler^{2\uppi\upi\inner{\om}{x}}\varphi(\delta\om).
\end{equation}
\section{Analysis of the Deterministic Keller--Segel Equation}\label{sec:det_PDE}
In this section we show the existence, uniqueness and regularity of weak solutions $f$ to the deterministic Keller--Segel equation,
\begin{equation}\label{eq:Determ_KS}
	\begin{cases}
		\begin{aligned}
			(\partial_{t}-\Delta)f&=-\chem\vdiv(f\nabla\Phi_{f}),&&~\text{in}~(0,T]\times\mbT^{2},\\
			-\Delta\Phi_{f}&=f-\mean{f}, &&~\text{in}~(0,T]\times\mbT^{2},\\
			f\tzero&=f_{0}, &&~\text{on}~\mbT^{2},
		\end{aligned}
	\end{cases}
\end{equation}
where we denote the mean of each $u\in L^{2}(\mbT^{2})$ by $\mean{u}=\inner{u}{1}_{L^{2}}$. We further establish the regularity of $\sqrt{f}$ in a suitable Bessel potential space (Theorem~\ref{thm:Determ_KS_Well_Posedness} \& Lemma~\ref{lem:regularity_srdet}).

Let us first define what we understand to be a weak solution to~\eqref{eq:Determ_KS}.
\begin{definition}\label{def:Determ_KS_weak_solution}
	Let $f_{0}\in L^{2}(\mbT^{2})$, $\chem\in\mbR$ and $T>0$. Then, $f\in C_{T}L^{2}(\mbT^{2})\cap L_{T}^{2}\mcH^{1}(\mbT^{2})$ is a weak solution to~\eqref{eq:Determ_KS} on $[0,T]$ if for all $t\in[0,T]$ and $\psi\in\mcH^{1}(\mbT^{2})$,
	\begin{equation*}
		\inner{f_{t}}{\psi}_{L^{2}}=\inner{f_{0}}{\psi}_{L^{2}}+\chem\int_{0}^{t}\inner{f_{s}\nabla\Phi_{f_{s}}}{\nabla\psi}_{L^{2}}\dd s.
	\end{equation*}
\end{definition}
\subsection{Existence and Regularity}\label{subsec:Determ_KS_existence_regularity}
We assume that the initial data $f_{0}$ lies in $L^{2}(\mbT^{2})$, is non-negative and of unit mass. We refer to~\cite{perthame_04, hillen_painter_09_users, horstmann_03_KS_I,horstmann_04_KS_II} and Remark~\ref{rem:sharp_apriori} for a wider ranging survey of the derivation and properties of the Cauchy problem~\eqref{eq:Determ_KS} and related PDE.

Our main emphasis is to recall that the global well-posedness of~\eqref{eq:Determ_KS} depends on the sign and size of $\chem$. For small $\chem$ we can employ the periodic Gagliardo--Nirenberg--Sobolev inequality~\eqref{eq:GNS_Per} to show that the solution will be globally well-posed (cf.\ Theorem~\ref{thm:Determ_KS_Well_Posedness}), whereas for large $\chem$ it is known that the solution may blow up in finite time~\cite[Thm.~8.1]{kiselev_xu_16_suppression}.

In stating the main result of this subsection (Theorem~\ref{thm:Determ_KS_Well_Posedness}) we use the notation $C_{\gnsper}(2,3,1)$ to denote the optimal constant such that the periodic Gagliardo–Nirenberg–Sobolev inequality~\eqref{eq:GNS_Per} holds, which is properly elucidated in Subsection~\ref{subsec:bounds} below.
\begin{theorem}\label{thm:Determ_KS_Well_Posedness}
	Let $f_{0}\in L^{2}(\mbT^{2})$ be such that $f_{0}\geq0$ and $\mean{f_{0}}=1$, and $\chem\in\mbR$. Then the following hold:
	\begin{enumerate}
		\item\label{it:Local_Weak_Sol}
		There exists a $\Trdet=\Trdet(\chem,\norm{f_{0}}_{L^{2}})\in(0,\infty]$ and a unique map $f\from[0,\Trdet)\to L^{2}(\mbT^{2})$ such that $f\in C_{T}L^{2}(\mbT^{2})\cap L^{2}_{T}\mcH^{1}(\mbT^{2})$ for every $T\in(0,\Trdet)$, $f\tzero=f_{0}$ and $f$ solves~\eqref{eq:Determ_KS} in the weak sense on $[0,T]$ (see Definition~\ref{def:Determ_KS_weak_solution}). Furthermore for each $t\in[0,\Trdet)$ this solution is non-negative almost everywhere and $\norm{f_{t}}_{L^{1}}=1$. Finally $\Trdet$ is such that either $\Trdet=\infty$ or $\lim_{t\nearrow \Trdet}\norm{f_{t}}_{L^{2}}=\infty$, and we refer to it as the maximal time of existence (for $f$ in $L^{2}(\mbT^{2})$.)
		\item\label{it:Global_Weak_Sol}
		If $\chem\leq\thresh\defeq C_{\gnsper}(2,3,1)^{-3}$, then one has $\Trdet=\infty$.
		\item\label{it:Regular_Weak_Sol}
		It holds that $f\in C^{2}_{1}((0,T]\times\mbT^{2})$ for every $T\in(0,\Trdet)$. Further if $f_{0}\in\mcH^{\gamma}(\mbT^{2})$ for some $\gamma>0$, then $f\in C_{T}\mcH^{\gamma}(\mbT^{2})\cap L_{T}^{2}\mcH^{\gamma+1}(\mbT^{2})$.
	\end{enumerate}
\end{theorem}
\begin{proof}
	Since these statements are mostly amalgams of known results we only provide a sketched proof with references for full details.
	
	The proof of Claim~\ref{it:Local_Weak_Sol} is carried out through a Galerkin approximation and is classical. In the following we give a brief sketch: Testing the equation against itself, we obtain the a priori energy identity
	\begin{equation}\label{eq:energy_identity_1}
		\norm{f_{t}}^{2}_{L^{2}}=\norm{f_{0}}^{2}_{L^{2}}-2\int_{0}^{t}\norm{\nabla f_{s}}^{2}_{L^{2}}\dd s+\chem\int_{0}^{t}\norm{f_{s}}^{3}_{L^{3}}\dd s-\chem\int_{0}^{t}\norm{f_{s}}^{2}_{L^{2}}\dd s.
	\end{equation}
	We identify
	\begin{equation*}
		\norm{f_{s}}_{L^{3}}^{3}-\norm{f_{s}}_{L^{2}}^{2}=\inner{f_{s}-\mean{f_{s}}}{f_{s}^{2}}_{L^{2}}
	\end{equation*}
	and bound for some $C>0$ by the Cauchy--Schwarz inequality, the Sobolev embedding $\mcH^{1/2}(\mbT^{2})\embed L^{4}(\mbT^{2})$ (cf.~\eqref{eq:Sobolev_embedding_I}) and real interpolation,
	\begin{equation*}
		\abs{\inner{f_{s}-\mean{f_{s}}}{f_{s}^{2}}_{L^{2}}}\leq\norm{f_{s}-\mean{f_{s}}}_{L^{2}}\norm{f_{s}}_{L^{4}}^{2}\leq\norm{f_{s}}_{L^{2}}\norm{f_{s}}_{L^{4}}^{2}\leq C\norm{f_{s}}_{L^{2}}\norm{f_{s}}_{\mcH^{1/2}}^{2}\leq C\norm{f_{s}}_{L^{2}}^{2}\norm{f_{s}}_{\mcH^{1}}.
	\end{equation*}
	Therefore, it follows by~\eqref{eq:energy_identity_1} and an application of Young's inequality that for each $\vartheta>0$,
	\begin{equation*}
		\norm{f_{t}}^{2}_{L^{2}}\leq\norm{f_{0}}^{2}_{L^{2}}-2\int_{0}^{t}\norm{\nabla f_{s}}^{2}_{L^{2}}\dd s+\frac{C^{2}\chem^{2}}{2\vartheta}\int_{0}^{t}\norm{f_{s}}_{L^{2}}^{4}\dd s+\frac{\vartheta}{2}\int_{0}^{t}\norm{f_{s}}_{\mcH^{1}}^{2}\dd s.
	\end{equation*}
	Using the identity $\norm{f_{s}}_{\mcH^{1}}^{2}=\norm{f_{s}}_{L^{2}}^{2}+\norm{\nabla f_{s}}_{L^{2}}^{2}$, we obtain
	\begin{equation*}
		\norm{f_{t}}^{2}_{L^{2}}\leq\norm{f_{0}}^{2}_{L^{2}}-\Bigl(2-\frac{\vartheta}{2}\Bigr)\int_{0}^{t}\norm{\nabla f_{s}}^{2}_{L^{2}}\dd s+\frac{C^{2}\chem^{2}}{2\vartheta}\int_{0}^{t}\norm{f_{s}}_{L^{2}}^{4}\dd s+\frac{\vartheta}{2}\int_{0}^{t}\norm{f_{s}}_{L^{2}}^{2}\dd s.
	\end{equation*}
	We can then choose $\vartheta=2$ to deduce
	\begin{equation*}
		\norm{f_{t}}^{2}_{L^{2}}\leq\norm{f_{0}}^{2}_{L^{2}}-\int_{0}^{t}\norm{\nabla f_{s}}^{2}_{L^{2}}\dd s+\frac{C^{2}\chem^{2}}{4}\int_{0}^{t}\norm{f_{s}}_{L^{2}}^{4}\dd s+\int_{0}^{t}\norm{f_{s}}_{L^{2}}^{2}\dd s.
	\end{equation*}
	A comparison argument to the ODE $\Dot{x}(t)=C^{2}\chem^{2}x(t)^{2}/4+x(t)$ shows that there exists a time horizon $\Trdet$ such that $t\mapsto\norm{f_{t}}_{L^{2}}^{2}$ and $t\mapsto\int_{0}^{t}\norm{\nabla f_{s}}^{2}_{L^{2}}\dd s$ are locally bounded on $[0,\Trdet)$.
	
	To prove uniqueness, let $g$ be another weak solution with $g_{0}=f_{0}$. By a similar calculation, one finds a priori that
	\begin{equation*}
		\begin{split}
			\frac{1}{2}\frac{\dd}{\dd t}\norm{f_{t}-g_{t}}^{2}_{L^{2}}+\norm{\nabla(f_{t}-g_{t})}_{L^{2}}^{2}&=\chem\Bigl\langle f_{t}\nabla\Phi_{f_{t}-g_{t}}+(f_{t}-g_{t})\nabla\Phi_{g_{t}},\nabla(f_{t}-g_{t})\Bigr\rangle_{L^{2}}\\
			&\leq\abs{\chem}\Bigl(\norm{f_{t}\nabla\Phi_{f_{t}-g_{t}}}_{L^{2}}+\norm{(f_{t}-g_{t})\nabla\Phi_{g_{t}}}_{L^{2}}\Bigr)\norm{\nabla(f_{t}-g_{t})}_{L^{2}}\\
			&\lesssim\abs{\chem}\Bigl(\norm{f_{t}}_{\mcH^{1}}\norm{\nabla\Phi_{f_{t}-g_{t}}}_{\mcH^{1}}+\norm{f_{t}-g_{t}}_{L^{2}}\norm{\nabla\Phi_{g_{t}}}_{\mcH^{2}}\Bigr)\norm{\nabla(f_{t}-g_{t})}_{L^{2}}\\
			&\lesssim\abs{\chem}\Bigl(\norm{f_{t}}_{\mcH^{1}}+\norm{g_{t}}_{\mcH^{1}}\Bigr)\norm{f_{t}-g_{t}}_{L^{2}}\norm{\nabla(f_{t}-g_{t})}_{L^{2}},
		\end{split}
	\end{equation*}
	where in the second inequality we used that $uv\in L^{2}(\mbT^{2})$ for each $(u,v)\in\mcH^{1}(\mbT^{2})\times\mcH^{1}(\mbT^{2})$ or $(u,v)\in L^{2}(\mbT^{2})\times\mcH^{2}(\mbT^{2})$, see~\cite[Thm.~7.4]{behzadan_holst_21}. A combination of Young's and Gronwall's inequalities then implies that there exists some $C>0$ such that
	\begin{equation*}
		\norm{f_{t}-g_{t}}_{L^{2}}^{2}\leq\norm{f_{0}-g_{0}}_{L^{2}}^{2}\exp\Bigl(C\chem^{2}\int_{0}^{t}\norm{f_{s}}_{\mcH^{1}}^{2}+\norm{g_{s}}_{\mcH^{1}}^{2}\dd s\Bigr),
	\end{equation*}
	from which uniqueness follows on the interval $[0,\Trdet)$.
	\begin{details}
		\paragraph{Detailed Proof of Theorem~\ref{thm:Determ_KS_Well_Posedness}.}
		Our goal is to solve~\eqref{eq:Determ_KS} with the variational approach of~\cite{liu_roeckner_13_local_global}. We work with the Gelfand triple $\mcH^{\gamma+1}(\mbT^{2})\subset\mcH^{\gamma}(\mbT^{2})\subset\mcH^{\gamma-1}(\mbT^{2})$ for $\gamma\geq0$; in other words, if we explicit the Riesz isomorphism, $\mcH^{\gamma+1}(\mbT^{2})\subset\mcH^{\gamma}(\mbT^{2})\stackrel{(1-\Delta)^{\gamma}}{\longleftrightarrow}\mcH^{-\gamma}(\mbT^{2})\subset\mcH^{-\gamma-1}(\mbT^{2})$.
		
		In the following, we denote by $\inner{u}{v}_{\mcH^{-\gamma-1},\mcH^{\gamma+1}}$ the duality pairing between $u\in\mcH^{-\gamma-1}(\mbT^{2})$ and  $v\in\mcH^{\gamma+1}(\mbT^{2})$. For $u\in\mcH^{\gamma-1}(\mbT^{2})$, we further denote by $\inner{u}{v}_{\mcH^{\gamma-1},\mcH^{\gamma+1}}$ the pairing $\inner{u}{v}_{\mcH^{\gamma-1},\mcH^{\gamma+1}}=\inner{(1-\Delta)^{\gamma}u}{v}_{\mcH^{-\gamma-1},\mcH^{\gamma+1}}$. We also denote by $\inner{\place}{\place}_{\mcH^{\gamma}}$ the inner product on $\mcH^{\gamma}(\mbT^{2})$.
		
		An application of Riesz's representation theorem yields for all $u\in\mcH^{\gamma-1}(\mbT^{2})$ and $v\in\mcH^{\gamma+1}(\mbT^{2})$,
		\begin{equation*}
			\inner{u}{v}_{\mcH^{\gamma-1},\mcH^{\gamma+1}}=\inner{(1-\Delta)^{\gamma}u}{v}_{\mcH^{-\gamma-1},\mcH^{\gamma+1}}=\inner{(1-\Delta)^{-\frac{\gamma+1}{2}}(1-\Delta)^{\gamma}u}{(1-\Delta)^{\frac{\gamma+1}{2}}v}_{L^{2}}.
		\end{equation*}
		If further $u\in\mcH^{\gamma}(\mbT^{2})$, then by integration by parts,
		\begin{equation*}
			\inner{u}{v}_{\mcH^{\gamma-1},\mcH^{\gamma+1}}=\inner{(1-\Delta)^{\gamma/2}u}{(1-\Delta)^{\gamma/2}v}_{L^{2}}=\inner{u}{v}_{\mcH^{\gamma}}.
		\end{equation*}
		\begin{lemma}\label{lem:IBP_dual_pairing}
			Let $u\in \mcH^{\gamma}(\mbT^{2};\mbR^{2})$ and $v\in\mcH^{\gamma+1}(\mbT^{2};\mbR)$, then
			\begin{equation*}
				\inner{\vdiv u}{v}_{\mcH^{\gamma-1},\mcH^{\gamma+1}}=-\inner{u}{\nabla v}_{\mcH^{\gamma}}.
			\end{equation*}
		\end{lemma}
		\begin{proof}
			Let $(u_{n})_{n\in\mbN}$ be such that $u_{n}\in C^{\infty}(\mbT^{2};\mbR^{2})$ and $u_{n}\to u\in\mcH^{\gamma}(\mbT^{2};\mbR^{2})$. It follows by integration by parts
			\begin{equation*}
				\inner{\vdiv u_{n}}{v}_{\mcH^{\gamma-1},\mcH^{\gamma+1}}=\inner{\vdiv u_{n}}{v}_{\mcH^{\gamma}}=-\inner{u_{n}}{\nabla v}_{\mcH^{\gamma}}.
			\end{equation*}
			Both sides are linear and bounded,
			\begin{equation*}
				\begin{split}
					&\abs{\inner{\vdiv u_{n}}{v}_{\mcH^{\gamma-1},\mcH^{\gamma+1}}}\leq\norm{\vdiv u_{n}}_{\mcH^{\gamma-1}}\norm{v}_{\mcH^{\gamma+1}}\lesssim\norm{u_{n}}_{\mcH^{\gamma}}\norm{v}_{\mcH^{\gamma+1}},\\
					&\abs{\inner{u_{n}}{\nabla v}_{\mcH^{\gamma}}}\leq\norm{u_{n}}_{\mcH^{\gamma}}\norm{\nabla v}_{\mcH^{\gamma}}\lesssim\norm{u_{n}}_{\mcH^{\gamma}}\norm{v}_{\mcH^{\gamma+1}},
				\end{split}
			\end{equation*}
			hence, we can pass to the limit, which yields the claim.
		\end{proof}
		We show the existence of a weak solution to~\eqref{eq:Determ_KS} by an application of~\cite[Thm.~1.1]{liu_roeckner_13_local_global}, which in turn relies on a Galerkin approximation. First we show that the operator driving the dynamics of~\eqref{eq:Determ_KS} satisfies~\cite[(H1)--(H4)]{liu_roeckner_13_local_global}, i.e.\ that it is hemi-continuous, locally monotone, generalized coercive and that it satisfies a growth condition.
		\begin{definition}\label{def:operator_properties}
			Let $\gamma\geq0$ and $T>0$, we say that a map $A\from[0,T]\times\mcH^{\gamma+1}(\mbT^{2})\to\mcH^{\gamma-1}(\mbT^{2})$ is hemi-continuous, locally monotone, generalized coercive and that it satisfies a growth condition, if there exists a positive function $f\in L^{1}([0,T];\mbR)$ such that the following hold:
			\begin{itemize}
				\item Hemi-continuity: For all $t\in[0,T]$ and $u,v,w\in\mcH^{\gamma+1}(\mbT^{2})$, it holds that the map $\mbR\ni\theta\mapsto\inner{A(t,u+\theta w)}{v}_{\mcH^{\gamma-1},\mcH^{\gamma+1}}\in\mbR$ is continuous.
				\item Local monotonicity: For all $t\in[0,T]$ and $u,v\in\mcH^{\gamma+1}(\mbT^{2})$, it holds that
				\begin{equation*}
					\inner{A(t,u)-A(t,v)}{u-v}_{\mcH^{\gamma-1},\mcH^{\gamma+1}}\lesssim(f(t)+\norm{u}_{\mcH^{\gamma+1}}^{2}+\norm{v}_{\mcH^{\gamma+1}}^{2})\norm{u-v}_{\mcH^{\gamma}}^{2}.
				\end{equation*}
				\item Generalized coercivity: For all $t\in[0,T]$ and $u\in\mcH^{\gamma+1}(\mbT^{2})$, it holds that
				\begin{equation*}
					\inner{A(t,u)}{u}_{\mcH^{\gamma-1},\mcH^{\gamma+1}}\lesssim-\norm{u}_{\mcH^{\gamma+1}}^{2}+g(\norm{u}_{\mcH^{\gamma}}^{2})+f(t),
				\end{equation*} 
				where $g\from[0,\infty)\to[0,\infty)$ is a non-decreasing, continuous function.
				\item Growth: For all $t\in[0,T]$ and $u\in\mcH^{\gamma+1}(\mbT^{2})$, it holds that
				\begin{equation*}
					\norm{A(t,u)}_{\mcH^{\gamma-1}}\lesssim(f(t)^{1/2}+\norm{u}_{\mcH^{\gamma+1}})(1+\norm{u}_{\mcH^{\gamma}}).
				\end{equation*}
			\end{itemize}
		\end{definition}
		\begin{lemma}\label{lem:operator_properties}
			Let $T>0$, $\chem\in\mbR$ and $h\in L_{T}^{2}\mcH^{\gamma}(\mbT^{2};\mbR^{2})$, the map $A\from[0,T]\times\mcH^{\gamma+1}(\mbT^{2})\to\mcH^{\gamma-1}(\mbT^{2})$ given by $A(t,u)\defeq\Delta u-\chem\vdiv(u\nabla\Phi_{u})-\vdiv h_{t}$ is hemi-continuous, locally monotone, generalized coercive and satisfies a growth condition.
		\end{lemma}
		\begin{proof}
			We show each property separately.
			
			\emph{Hemi-continuity:} Let $t\in[0,T]$, $u,v,w\in\mcH^{\gamma+1}(\mbT^{2})$ and $\theta_{1},\theta_{2}\in\mbR$. We obtain
			\begin{equation*}
				\begin{split}
					&\inner{A(t,u+\theta_{1}w)-A(t,u+\theta_{2}w)}{v}_{\mcH^{\gamma-1},\mcH^{\gamma+1}}\\
					&=-(\theta_{1}-\theta_{2})\inner{\nabla w}{\nabla v}_{\mcH^{\gamma}}+\chem\inner{(u+\theta_{1}w)\nabla\Phi_{u+\theta_{1}w}-(u+\theta_{2}w)\nabla\Phi_{u+\theta_{2}w}}{\nabla v}_{\mcH^{\gamma}}\\
					&=-(\theta_{1}-\theta_{2})\inner{\nabla w}{\nabla v}_{\mcH^{\gamma}}+\chem(\theta_{1}-\theta_{2})\inner{u\nabla\Phi_{w}}{\nabla v}_{\mcH^{\gamma}}\\
					&\quad+\chem(\theta_{1}-\theta_{2})\inner{w\nabla\Phi_{u}}{\nabla v}_{\mcH^{\gamma}}+\chem(\theta_{1}^{2}-\theta_{2}^{2})\inner{w\nabla\Phi_{w}}{\nabla v}_{\mcH^{\gamma}}.
				\end{split}
			\end{equation*}
			Using that $u,v,w\in\mcH^{\gamma+1}(\mbT^{2})$, we bound
			\begin{equation*}
				\begin{split}
					\inner{\nabla w}{\nabla v}_{\mcH^{\gamma}}&\leq\norm{\nabla w}_{\mcH^{\gamma}}\norm{\nabla v}_{\mcH^{\gamma}}\lesssim\norm{w}_{\mcH^{\gamma+1}}\norm{v}_{\mcH^{\gamma+1}},\\
					\inner{u\nabla\Phi_{w}}{\nabla v}_{\mcH^{\gamma}}&\leq\norm{u\nabla\Phi_{w}}_{\mcH^{\gamma}}\norm{\nabla v}_{\mcH^{\gamma}}\lesssim\norm{u\nabla\Phi_{w}}_{\mcH^{\gamma}}\norm{v}_{\mcH^{\gamma+1}}.
				\end{split}
			\end{equation*}
			To control the term $\norm{u\nabla\Phi_{w}}_{\mcH^{\gamma}}$, we use the product estimate~\eqref{eq:product_estimate_Sobolev_I} which yields
			\begin{equation*}
				\inner{u\nabla\Phi_{w}}{\nabla v}_{\mcH^{\gamma}}\lesssim\norm{u}_{\mcH^{\gamma+1}}\norm{\nabla\Phi_{w}}_{\mcH^{\gamma+1}}\norm{v}_{\mcH^{\gamma+1}}\lesssim\norm{u}_{\mcH^{\gamma+1}}\norm{w}_{\mcH^{\gamma}}\norm{v}_{\mcH^{\gamma+1}}.
			\end{equation*}
			The remaining terms can be bounded similarly. This shows that the map $\mbR\ni\theta\mapsto\inner{A(t,u+\theta w)}{v}_{\mcH^{\gamma-1},\mcH^{\gamma+1}}$ is locally Lipschitz continuous, hence continuous. This yields the hemi-continuity of $A$.
			
			\emph{Local monotonicity:} Let $t\in[0,T]$ and $u,v\in\mcH^{\gamma+1}(\mbT^{2})$, we obtain by the Cauchy--Schwarz inequality,
			\begin{equation*}
				\begin{split}
					\inner{A(t,u)-A(t,v)}{u-v}_{\mcH^{\gamma-1},\mcH^{\gamma+1}}&=-\norm{\nabla(u-v)}_{\mcH^{\gamma}}^{2}+\chem\inner{u\nabla\Phi_{u}-v\nabla\Phi_{v}}{\nabla(u-v)}_{\mcH^{\gamma}}\\
					&\leq-\norm{\nabla(u-v)}_{\mcH^{\gamma}}^{2}+\abs{\chem}\norm{\nabla(u-v)}_{\mcH^{\gamma}}\norm{u\nabla\Phi_{u}-v\nabla\Phi_{v}}_{\mcH^{\gamma}}
				\end{split}
			\end{equation*}
			and by the Young inequality for all $\vartheta>0$,
			\begin{equation*}
				\begin{split}
					&\abs{\chem}\norm{\nabla(u-v)}_{\mcH^{\gamma}}\norm{u\nabla\Phi_{u}-v\nabla\Phi_{v}}_{\mcH^{\gamma}}\leq\frac{\chem^{2}\vartheta}{2}\norm{\nabla(u-v)}_{\mcH^{\gamma}}^{2}+\frac{1}{2\vartheta}\norm{u\nabla\Phi_{u}-v\nabla\Phi_{v}}_{\mcH^{\gamma}}^{2}.
				\end{split}
			\end{equation*}
			Therefore,
			\begin{equation*}
				\inner{A(t,u)-A(t,v)}{u-v}_{\mcH^{\gamma-1},\mcH^{\gamma+1}}+\Bigl(1-\frac{\chem^{2}\vartheta}{2}\Bigr)\norm{\nabla(u-v)}_{\mcH^{\gamma}}^{2}\leq\frac{1}{2\vartheta}\norm{u\nabla\Phi_{u}-v\nabla\Phi_{v}}_{\mcH^{\gamma}}^{2}.
			\end{equation*}
			We use the triangle inequality to bound
			\begin{equation*}
				\norm{u\nabla\Phi_{u}-v\nabla\Phi_{v}}_{\mcH^{\gamma}}\leq\norm{u\nabla\Phi_{u-v}}_{\mcH^{\gamma}}+\norm{(u-v)\nabla\Phi_{v}}_{\mcH^{\gamma}}.
			\end{equation*}
			We can now apply the product estimate~\eqref{eq:product_estimate_Sobolev_I} to control the first summand
			\begin{equation*}
				\norm{u\nabla\Phi_{u-v}}_{\mcH^{\gamma}}\lesssim\norm{u}_{\mcH^{\gamma+1}}\norm{\nabla\Phi_{u-v}}_{\mcH^{\gamma+1}}\lesssim\norm{u}_{\mcH^{\gamma+1}}\norm{u-v}_{\mcH^{\gamma}}.
			\end{equation*}
			To control the second summand, we use the product estimate~\eqref{eq:product_estimate_Sobolev_II} which yields
			\begin{equation*}
				\norm{(u-v)\nabla\Phi_{v}}_{\mcH^{\gamma}}\leq\norm{u-v}_{\mcH^{\gamma}}\norm{\nabla\Phi_{v}}_{\mcH^{\gamma+2}}\lesssim\norm{u-v}_{\mcH^{\gamma}}\norm{v}_{\mcH^{\gamma+1}}.
			\end{equation*}
			Therefore,
			\begin{equation*}
				\inner{A(t,u)-A(t,v)}{u-v}_{\mcH^{\gamma-1},\mcH^{\gamma+1}}+\Bigl(1-\frac{\chem^{2}\vartheta}{2}\Bigr)\norm{\nabla(u-v)}_{\mcH^{\gamma}}^{2}\lesssim\frac{1}{\vartheta}(\norm{u}_{\mcH^{\gamma+1}}^{2}+\norm{v}_{\mcH^{\gamma+1}}^{2})\norm{u-v}_{\mcH^{\gamma}}^{2}.
			\end{equation*}
			Upon choosing $\vartheta\leq2/\chem^{2}$, this yields the local monotonicity of $A$.
			
			\emph{Generalized coercivity:} To show the generalized coercivity of $A$, we distinguish the cases $\gamma=0$ and $\gamma>0$. First let us assume $\gamma=0$. Let $t\in[0,T]$ and $u\in\mcH^{1}(\mbT^{2})$, integrating by parts we obtain 
			\begin{equation}\label{eq:coercivity_IBP}
				\begin{split}
					&\inner{A(t,u)}{u}_{\mcH^{-1},\mcH^{1}}-\inner{h_{t}}{\nabla u}_{L^{2}}\\
					&\qquad=-\norm{\nabla u}_{L^{2}}^{2}+\chem\inner{u\nabla\Phi_{u}}{\nabla u}_{L^{2}}=-\norm{\nabla u}_{L^{2}}^{2}+\frac{\chem}{2}\inner{\nabla\Phi_{u}}{\nabla u^{2}}_{L^{2}}\\
					&\qquad=-\norm{\nabla u}_{L^{2}}^{2}-\frac{\chem}{2}\inner{\Delta\Phi_{u}}{u^{2}}_{L^{2}}=-\norm{\nabla u}_{L^{2}}^{2}+\frac{\chem}{2}\inner{u-\mean{u}}{u^{2}}_{L^{2}}.
				\end{split}
			\end{equation}
			We apply the Cauchy--Schwarz inequality to bound
			\begin{equation*}
				\abs{\inner{u-\mean{u}}{u^{2}}_{L^{2}}}\leq\norm{u-\mean{u}}_{L^{2}}\norm{u^{2}}_{L^{2}}.
			\end{equation*}			
			To estimate the first factor, we use
			\begin{equation*}
				\norm{u-\mean{u}}_{L^{2}}^{2}=\norm{u}_{L^{2}}^{2}-\mean{u}^{2}\leq\norm{u}_{L^{2}}^{2}.
			\end{equation*}
			To estimate the second factor, we use the Sobolev embedding $\mcH^{1/2}(\mbT^{2})\embed L^{4}(\mbT^{2})$ (cf.~\eqref{eq:Sobolev_embedding_I}) followed by real interpolation to bound
			\begin{equation*}
				\norm{u^{2}}_{L^{2}}=\norm{u}_{L^{4}}^{2}\lesssim\norm{u}_{\mcH^{1/2}}^{2}\leq\norm{u}_{\mcH^{1}}\norm{u}_{L^{2}}.
			\end{equation*}
			Real interpolation, in this instance, is just another application of the Cauchy--Schwarz inequality:
			\begin{equation*}
				\norm{u}_{\mcH^{1/2}}^{2}=\sum_{\om\in\mbZ^{2}}(1+\abs{2\uppi\om}^{2})^{1/2}\abs{\hat{u}(\om)}\abs{\hat{u}(\om)}\leq\norm{u}_{\mcH^{1}}\norm{u}_{L^{2}}.
			\end{equation*}
			To summarize, it follows by~\eqref{eq:coercivity_IBP} that there exists a constant $C>0$ such that
			\begin{equation*}
				\inner{A(t,u)}{u}_{\mcH^{-1},\mcH^{1}}-\inner{h_{t}}{\nabla u}_{L^{2}}\leq-\norm{\nabla u}_{L^{2}}^{2}+C\abs{\chem}\norm{u}_{L^{2}}^{2}\norm{u}_{\mcH^{1}}.
			\end{equation*}
			An application of the Cauchy--Schwarz inequality combined with Young's inequality yields for all $\vartheta>0$,
			\begin{equation*}
				\inner{A(t,u)}{u}_{\mcH^{-1},\mcH^{1}}+\Bigl(1-\frac{\vartheta}{2}\Bigr)\norm{\nabla u}_{L^{2}}^{2}\leq C\abs{\chem}\norm{u}_{L^{2}}^{2}\norm{u}_{\mcH^{1}}+\frac{1}{2\vartheta}\norm{h_{t}}_{L^{2}}^{2}.
			\end{equation*}
			Using that
			\begin{equation*}
				\begin{split}
					\norm{u}_{\mcH^{1}}^{2}&=\sum_{\om\in\mbZ^{2}}(1+\abs{2\uppi\om}^{2})\abs{\hat{u}(\om)}^{2}=\norm{u}_{L^{2}}^{2}+\sum_{\om\in\mbZ^{2}}\sum_{j=1}^{2}\abs{2\uppi\om^{j}}^{2}\abs{\hat{u}(\om)}^{2}\\
					&=\norm{u}_{L^{2}}^{2}+\sum_{j=1}^{2}\norm{\partial_{j}u}_{L^{2}}^{2}=\norm{u}_{L^{2}}^{2}+\norm{\nabla u}_{L^{2}}^{2},
				\end{split}
			\end{equation*}
			we obtain
			\begin{equation*}
				\inner{A(t,u)}{u}_{\mcH^{-1},\mcH^{1}}+\Bigl(1-\frac{\vartheta}{2}\Bigr)\norm{u}_{\mcH^{1}}^{2}\leq C\abs{\chem}\norm{u}_{L^{2}}^{2}\norm{u}_{\mcH^{1}}+\frac{1}{2\vartheta}\norm{h_{t}}_{L^{2}}^{2}+\Bigl(1-\frac{\vartheta}{2}\Bigr)\norm{u}_{L^{2}}^{2}.
			\end{equation*}
			Another application of Young's inequality yields
			\begin{equation*}
				\inner{A(t,u)}{u}_{\mcH^{-1},\mcH^{1}}+(1-\vartheta)\norm{u}_{\mcH^{1}}^{2}\leq\frac{C^{2}\chem^{2}}{2\vartheta}\norm{u}_{L^{2}}^{4}+\frac{1}{2\vartheta}\norm{h_{t}}_{L^{2}}^{2}+\Bigl(1-\frac{\vartheta}{2}\Bigr)\norm{u}_{L^{2}}^{2}.
			\end{equation*}
			Choosing $\vartheta<1$, we obtain the generalized coercivity of $A$ with $g(x)=\chem^{2}x^{2}+x$ and $f(t)=\norm{h_{t}}_{L^{2}}^{2}$.
			
			Next, let us assume $\gamma>0$. We obtain as in~\eqref{eq:coercivity_IBP},
			\begin{equation*}
				\inner{A(t,u)}{u}_{\mcH^{\gamma-1},\mcH^{\gamma+1}}-\inner{h_{t}}{\nabla u}_{\mcH^{\gamma}}=-\norm{\nabla u}_{\mcH^{\gamma}}^{2}+\chem\inner{u\nabla\Phi_{u}}{\nabla u}_{\mcH^{\gamma}}.
			\end{equation*}
			An application of the Cauchy--Schwarz inequality yields
			\begin{equation*}
				\abs{\inner{u\nabla\Phi_{u}}{\nabla u}_{\mcH^{\gamma}}}\leq\norm{u\nabla\Phi_{u}}_{\mcH^{\gamma}}\norm{\nabla u}_{\mcH^{\gamma}}.
			\end{equation*}
			To control the term $\norm{u\nabla\Phi_{u}}_{\mcH^{\gamma}}$, we use the product estimate~\eqref{eq:product_estimate_Sobolev_III} to bound
			\begin{equation*}
				\norm{u\nabla\Phi_{u}}_{\mcH^{\gamma}}\lesssim\norm{u}_{\mcH^{\gamma}}\norm{\nabla\Phi_{u}}_{\mcH^{\gamma+1}}\lesssim\norm{u}_{\mcH^{\gamma}}^{2}.
			\end{equation*}
			Therefore, there exists a constant $C>0$ such that
			\begin{equation*}
				\inner{A(t,u)}{u}_{\mcH^{\gamma-1},\mcH^{\gamma+1}}-\inner{h_{t}}{\nabla u}_{\mcH^{\gamma}}\leq-\norm{\nabla u}_{\mcH^{\gamma}}^{2}+C\abs{\chem}\norm{u}_{\mcH^{\gamma}}^{2}\norm{\nabla u}_{\mcH^{\gamma}}.
			\end{equation*}
			and we can apply the Cauchy--Schwarz inequality combined with Young's inequality to show that for all $\vartheta>0$,
			\begin{equation*}
				\inner{A(t,u)}{u}_{\mcH^{\gamma-1},\mcH^{\gamma+1}}+(1-\vartheta)\norm{\nabla u}_{\mcH^{\gamma}}^{2}\leq \frac{C^{2}\chem^{2}}{2\vartheta}\norm{u}_{\mcH^{\gamma}}^{4}+\frac{1}{2\vartheta}\norm{h_{t}}_{\mcH^{\gamma}}^{2}.
			\end{equation*}
			Using that 
			\begin{equation*}
				\norm{\nabla u}_{\mcH^{\gamma}}^{2}=\sum_{j=1}^{2}\sum_{\om\in\mbZ^{2}}\abs{2\uppi\om^{j}}^{2}\abs{\hat{u}(\om)}^{2}=\sum_{\om\in\mbZ^{2}}\abs{2\uppi\om}^{2}\abs{\hat{u}(\om)}^{2}=\norm{u}_{\mcH^{\gamma+1}}^{2}-\norm{u}_{\mcH^{\gamma}}^{2},
			\end{equation*}
			we obtain
			\begin{equation*}
				\inner{A(t,u)}{u}_{\mcH^{\gamma-1},\mcH^{\gamma+1}}+(1-\vartheta)\norm{u}_{\mcH^{\gamma+1}}^{2}\leq \frac{C^{2}\chem^{2}}{2\vartheta}\norm{u}_{\mcH^{\gamma}}^{4}+\frac{1}{2\vartheta}\norm{h_{t}}_{\mcH^{\gamma}}^{2}+(1-\vartheta)\norm{u}_{\mcH^{\gamma}}^{2}.
			\end{equation*}
			Choosing $\vartheta<1$, we obtain the generalized coercivity of $A$ with $g(x)=\chem^{2}x^{2}+x$ and $f(t)=\norm{h_{t}}_{\mcH^{\gamma}}^{2}$.
			
			\emph{Growth:} Let $t\in[0,T]$ and $u\in\mcH^{\gamma+1}(\mbT^{2})$, we obtain by the product estimate~\eqref{eq:product_estimate_Sobolev_I}
			\begin{equation*}
				\begin{split}
					\norm{A(t,u)}_{\mcH^{\gamma-1}}&=\norm{\Delta u-\chem\vdiv(u\nabla\Phi_{u})-\vdiv h_{t}}_{\mcH^{\gamma-1}}\lesssim\norm{u}_{\mcH^{\gamma+1}}+\abs{\chem}\norm{u\nabla\Phi_{u}}_{\mcH^{\gamma}}+\norm{h_{t}}_{\mcH^{\gamma}}\\
					&\lesssim\norm{u}_{\mcH^{\gamma+1}}+\abs{\chem}\norm{u}_{\mcH^{\gamma+1}}\norm{\nabla\Phi_{u}}_{\mcH^{\gamma+1}}+\norm{h_{t}}_{\mcH^{\gamma}}\leq(\norm{h_{t}}_{\mcH^{\gamma}}+\norm{u}_{\mcH^{\gamma+1}})(1+\abs{\chem}\norm{u}_{\mcH^{\gamma}}).
				\end{split}
			\end{equation*}
			This yields a growth condition on $A$.
			
			Hence, $A$ satisfies the properties given in Definition~\ref{def:operator_properties}.
		\end{proof}
		We can now apply~\cite[Thm.~1.1]{liu_roeckner_13_local_global} to construct a local weak solution to~\eqref{eq:Determ_KS}:
		\begin{proposition}\label{prop:Determ_KS_local_weak_existence}
			For each $\chem\in\mbR$, $\gamma\geq0$ and $f_{0}\in\mcH^{\gamma}(\mbT^{2})$, there exists some $\bar{T}>0$ such that~\eqref{eq:Determ_KS} has a weak solution on $[0,\bar{T}]$, i.e.\ there exists a unique $f\in C_{\bar{T}}\mcH^{\gamma}(\mbT^{2})\cap L^{2}_{\bar{T}}\mcH^{\gamma+1}(\mbT^{2})\cap W^{1,2}_{\bar{T}}\mcH^{\gamma-1}(\mbT^{2})$ that satisfies for each $t\in[0,\bar{T}]$ and $\test\in\mcH^{\gamma+1}(\mbT^{2})$,
			\begin{equation*}
				\inner{f_{t}}{\test}_{\mcH^{\gamma}}-\inner{f_{0}}{\test}_{\mcH^{\gamma}}=\int_{0}^{t}\inner{\Delta f_{s}-\chem\vdiv(f_{s}\nabla\Phi_{f_{s}})}{\test}_{\mcH^{\gamma-1},\mcH^{\gamma+1}}\dd s.
			\end{equation*}
			The local time of existence $\bar{T}$ depends only on $\gamma$, $\chem$ and $\norm{f_{0}}_{\mcH^{\gamma}}$.
		\end{proposition}
		\begin{proof}
			An application of Lemma~\ref{lem:operator_properties} with $\het=0$ shows that the operator $A\from\mcH^{\gamma+1}(\mbT^{2})\to\mcH^{\gamma-1}(\mbT^{2})$ defined by $u\mapsto A(u)\defeq\Delta u-\chem\vdiv(u\nabla\Phi_{u})$ is hemi-continuous, locally monotone, generalized coercive and satisfies a growth condition. The claim then follows by~\cite[Thm.~1.1]{liu_roeckner_13_local_global}, where we used that the embedding $\mcH^{\gamma+1}(\mbT^{2})\subset\mcH^{\gamma}(\mbT^{2})$ is compact (Lemma~\ref{lem:Besov_compact}). It follows by~\cite[Rem.~1.2]{liu_roeckner_13_local_global} that $\bar{T}$ only depends on $\gamma$, $\chem$ and $\norm{f_{0}}_{\mcH^{\gamma}}$.
		\end{proof}
		Having constructed a local weak solution, we can re-start~\eqref{eq:Determ_KS} to establish a maximal time of existence.
		\begin{theorem}\label{thm:Determ_KS_maximal}
			For each $\chem\in\mbR$, $\gamma\geq0$ and $f_{0}\in\mcH^{\gamma}(\mbT^{2})$, there exists a $\Trdet\in(0,\infty]$ and a unique weak solution $f\from[0,\Trdet)\to\mcH^{\gamma}(\mbT^{2})$ to~\eqref{eq:Determ_KS} with initial data $f_{0}$ and chemotactic sensitivity $\chem$ such that $f\in C_{T}\mcH^{\gamma}(\mbT^{2})\cap L^{2}_{T}\mcH^{\gamma+1}(\mbT^{2})\cap W^{1,2}_{T}\mcH^{\gamma-1}(\mbT^{2})$ for every $T\in(0,\Trdet)$. The time of existence $\Trdet$ depends only on $\chem$ and $\norm{f_{0}}_{L^{2}}$; in particular it is independent of $\gamma$ and coincides with the maximal time of existence in $L^{2}(\mbT^{2})$. It holds that either $\Trdet=\infty$ or $\lim_{t\nearrow\Trdet}\norm{f_{t}}_{L^{2}}=\infty$.
		\end{theorem}
		\begin{proof}
			Let $\chem\in\mbR$, $\gamma\geq0$ and $f_{0}\in\mcH^{\gamma}(\mbT^{2})$, we define the time horizon
			\begin{equation*}
				\begin{split}
					\Trdet_{\gamma}\defeq\sup\Bigl\{T>0:~&\text{there exists a weak solution}\\
					&f\in C_{T}\mcH^{\gamma}(\mbT^{2})\cap L^{2}_{T}\mcH^{\gamma+1}(\mbT^{2})\cap W^{1,2}_{T}\mcH^{\gamma-1}(\mbT^{2})\text{ to~\eqref{eq:Determ_KS}}\Bigr\},
				\end{split}
			\end{equation*}
			which is well-defined and non-zero by Proposition~\ref{prop:Determ_KS_local_weak_existence}. We first show that either 
			\begin{equation*}
				\Trdet_{\gamma}=\infty\quad\text{or}\quad\lim_{t\nearrow\Trdet_{\gamma}}\norm{f_{t}}_{\mcH^{\gamma}}=\infty.
			\end{equation*}
			Assume $\Trdet_{\gamma}<\infty$ and that there exists a sequence $(t_{n})_{n\in\mbN}$ such that $t_{n}\to\Trdet_{\gamma}$ and $R\defeq\sup_{n\in\mbN}\norm{f_{t_{n}}}_{\mcH^{\gamma}}<\infty$. Let $n\in\mbN$ be such that $t_{n}<\Trdet<t_{n}+\bar{T}(\chem,R)$, where $\bar{T}=\bar{T}(\chem,R)$ is as in Proposition~\ref{prop:Determ_KS_local_weak_existence}. We can then re-start~\eqref{eq:Determ_KS} from $f_{t_{n}}$ and run it until time $t_{n}+\bar{T}(\chem,R)$, contradicting the definition of $\Trdet_{\gamma}$.
			
			Next we show that $\Trdet_{\gamma}$ depends only on $\chem$ and $\norm{f_{0}}_{L^{2}}$. Using that $\norm{f_{t}}_{L^{2}}\leq\norm{f_{t}}_{\mcH^{\gamma}}$ for every $t\leq\Trdet_{\gamma}$, we obtain $\Trdet_{\gamma}\leq\Trdet_{0}$. Furthermore, it follows by a bootstrapping argument (cf.~\eqref{eq:bootstrap_IC}--\eqref{eq:bootstrap_advection_II} below), that for every $T<\Trdet_{0}$ it holds that $f\in C((0,T];\mcH^{\gamma}(\mbT^{2}))$, which implies $\Trdet_{0}\leq\Trdet_{\gamma}$. Therefore $\Trdet_{\gamma}=\Trdet_{0}\eqdef\Trdet$ which yields the claim.
		\end{proof}
	\end{details}
	\begin{details}
		We define three notions of weak solution:
		\begin{enumerate}
			\item $f\in C_{T}\mcH^{\gamma}(\mbT^{2})\cap L_{T}^{2}\mcH^{\gamma+1}(\mbT^{2})\cap W^{1,2}_{T}\mcH^{\gamma-1}(\mbT^{2})$ s.t.\ for a.e.\ $t\in [0,T]$ and all $\test\in\mcH^{\gamma+1}(\mbT^{2})$,
			\begin{equation}\label{eq:Weak_Sol_Def_1}
				\inner{\partial_{t}f_{t}}{\test}_{\mcH^{\gamma-1},\mcH^{\gamma+1}}=\inner{\Delta f_{t}-\chem\vdiv (f_{t}\nabla\Phi_{f_{t}})}{\test}_{\mcH^{\gamma-1},\mcH^{\gamma+1}},
			\end{equation}
			\item $f\in C_{T}\mcH^{\gamma}(\mbT^{2})\cap L_{T}^{2}\mcH^{\gamma+1}(\mbT^{2})$ s.t.\ for a.e.\ $t\in [0,T]$ and all $\test\in\mcH^{\gamma+1}(\mbT^{2})$,
			\begin{equation}\label{eq:Weak_Sol_Def_2}
				\inner{f_{t}}{\test}_{\mcH^{\gamma}}-\inner{f_{0}}{\test}_{\mcH^{\gamma}}=\int_{0}^{t}\inner{\Delta f_{s}-\chem\vdiv(f_{s}\nabla\Phi_{f_{s}})}{\test}_{\mcH^{\gamma-1},\mcH^{\gamma+1}}\dd s,
			\end{equation}
			\item $f\in C_{T}\mcH^{\gamma}(\mbT^{2})\cap L_{T}^{2}\mcH^{\gamma+1}(\mbT^{2})$ s.t.\ for a.e.\ $t\in[0,T]$ and all $\test\in L^{2}_{T}\mcH^{\gamma+1}(\mbT^{2})\cap W^{1,2}_{T}\mcH^{\gamma-1}(\mbT^{2})$,
			\begin{equation}\label{eq:Weak_Sol_Def_3}
				\inner{f_{t}}{\test_{t}}_{\mcH^{\gamma}}-\inner{f_{0}}{\test_{0}}_{\mcH^{\gamma}}=\int_{0}^{t}\inner{\partial_{s}\test_{s}}{f_{s}}_{\mcH^{\gamma-1},\mcH^{\gamma+1}}+\inner{\Delta f_{s}-\chem\vdiv(f_{s}\nabla\Phi_{f_{s}})}{\test_{s}}_{\mcH^{\gamma-1},\mcH^{\gamma+1}}\dd s.
			\end{equation}
		\end{enumerate}
		We will show that these three definitions are in fact equivalent. The point of doing these definitions is that the Galerkin solutions is of the form~\eqref{eq:Weak_Sol_Def_2}, see Proposition~\ref{prop:Determ_KS_local_weak_existence}, however, the energy estimate will make use of~\eqref{eq:Weak_Sol_Def_3} and the identity~\eqref{eq:Weak_Sol_Def_1} shows that weak solutions are in the class $W^{1,2}((0,T);\mcH^{\gamma-1}(\mbT^{2}))$.
		
		The equivalence between~\eqref{eq:Weak_Sol_Def_2} and~\eqref{eq:Weak_Sol_Def_3} when $\test$ is not a function of time is clear; to show that~\eqref{eq:Weak_Sol_Def_2} implies~\eqref{eq:Weak_Sol_Def_3} when $\test$ is time-dependent one leverages the equivalence between~\eqref{eq:Weak_Sol_Def_2} and~\eqref{eq:Weak_Sol_Def_1} followed by an application of~\cite[Subsec.~5.9.2,~Proof~of~Thm.~3]{evans_98_PDE} to move from~\eqref{eq:Weak_Sol_Def_1} to~\eqref{eq:Weak_Sol_Def_3} using the identity
		\begin{equation*}
			\frac{\dd}{\dd t}\inner{f_{t}}{\test_{t}}_{\mcH^{\gamma}}=\frac{1}{2}\frac{\dd}{\dd t}(\norm{f_{t}+\test_{t}}_{\mcH^{\gamma}}^{2}-\norm{f_{t}}_{\mcH^{\gamma}}^{2}-\norm{\test_{t}}_{\mcH^{\gamma}}^{2})=\inner{\partial_{t}\test_{t}}{f_{t}}_{\mcH^{\gamma-1},\mcH^{\gamma+1}}+\inner{\partial_{t}f_{t}}{\test_{t}}_{\mcH^{\gamma-1},\mcH^{\gamma+1}}.
		\end{equation*}
		The equivalence between~\eqref{eq:Weak_Sol_Def_2} and~\eqref{eq:Weak_Sol_Def_1} is easy in one direction by the identity,
		\begin{equation*}
			f_{t}-f_{0}=\int_{0}^{t}\partial_{s}f_{s}\dd s,
		\end{equation*}
		which holds in a functional sense with the right hand side interpreted as a Bochner integral~\cite[Subsec.~5.9.2,~Thm.~2]{evans_98_PDE}. The opposite direction can be shown using the Lebesgue differentiation theorem: Let $\test\in\mcH^{\gamma+1}(\mbT^{2})$, we obtain by~\eqref{eq:Weak_Sol_Def_2} for a.e.\ $t\in[0,T]$,
		\begin{equation*}
			\begin{split}
				\lim_{h\to0}\frac{\inner{f_{t+h}}{\test}_{\mcH^{\gamma}}-\inner{f_{t}}{\test}_{\mcH^{\gamma}}}{h}&=\lim_{h\to0}\frac{1}{h}\int_{t}^{t+h}\inner{\Delta 	f_{s}-\chem\vdiv(f_{s}\nabla\Phi_{f_{s}})}{\test}_{\mcH^{\gamma-1},\mcH^{\gamma+1}}\dd s\\
				&=\inner{\Delta f_{t}-\chem\vdiv(f_{t}\nabla\Phi_{f_{t}})}{\test}_{\mcH^{\gamma-1},\mcH^{\gamma+1}}.
			\end{split}
		\end{equation*}
		Let $\theta\in C_{c}^{\infty}((0,T);\mbR)$ and assume that $h>0$ is such that $\supp(\theta)\subset[0,T-h]$, then
		\begin{equation*}
			\begin{split}
				&\int_{0}^{T}\inner{f_{t}}{\test}_{\mcH^{\gamma}}\frac{\theta(t-h)-\theta(t)}{h}\dd t\\
				&=\int_{0}^{T}\frac{\inner{f_{t+h}}{\test}_{\mcH^{\gamma}}-\inner{f_{t}}{\test}_{\mcH^{\gamma}}}{h}\theta(t)\dd t+\frac{1}{h}\int_{-h}^{0}\inner{f_{t+h}}{\test}_{\mcH^{\gamma}}\theta(t)\dd t-\frac{1}{h}\int_{T-h}^{T}\inner{f_{t+h}}{\test}_{\mcH^{\gamma}}\theta(t)\dd t\\
				&=\int_{0}^{T}\frac{\inner{f_{t+h}}{\test}_{\mcH^{\gamma}}-\inner{f_{t}}{\test}_{\mcH^{\gamma}}}{h}\theta(t)\dd t.
			\end{split}
		\end{equation*}
		Dividing the above by $h$ and taking the limit $h\to0$, we obtain
		\begin{equation*}
			\begin{split}
				\int_{0}^{T}\inner{f_{t}}{\test}_{\mcH^{\gamma}}\partial_{t}\theta(t)\dd t&=\int_{0}^{T}\inner{f_{t}}{\test}_{\mcH^{\gamma}}\lim_{h\to0}\frac{\theta(t)-\theta(t-h)}{h}\dd t\\
				&=-\int_{0}^{T}\inner{\Delta f_{t}-\chem\vdiv(f_{t}\nabla\Phi_{f_{t}})}{\test}_{\mcH^{\gamma-1},\mcH^{\gamma+1}}\theta(t)\dd t.
			\end{split}
		\end{equation*}
		This shows that $\inner{f}{\test}_{\mcH^{\gamma-1},\mcH^{\gamma+1}}=\inner{f}{\test}_{\mcH^{\gamma}}$ is weakly differentiable with weak derivative
		\begin{equation*}
			\partial_{t}\inner{f_{t}}{\test}_{\mcH^{\gamma-1},\mcH^{\gamma+1}}=\inner{\Delta f_{t}-\chem\vdiv(f_{t}\nabla\Phi_{f_{t}})}{\test}_{\mcH^{\gamma-1},\mcH^{\gamma+1}}.
		\end{equation*}
		Using that the Bochner integral commutes with linear maps~\cite[Chpt.~X,~Thm.~2.11]{amann_escher_09}, we obtain for all $\phi\in\mcH^{\gamma+1}(\mbT^{2})$,
		\begin{equation*}
			\Bigl\langle\int_{0}^{T}f_{t}\partial_{t}\theta(t)\dd t,\phi\Bigr\rangle_{\mcH^{\gamma-1},\mcH^{\gamma+1}}=-\Bigl\langle\int_{0}^{T}(\Delta f_{t}-\chem\vdiv(f_{t}\nabla\Phi_{f_{t}}))\theta(t)\dd t,\phi\Bigr\rangle_{\mcH^{\gamma-1},\mcH^{\gamma+1}}.
		\end{equation*}
		Since $\mcH^{\gamma+1}(\mbT^{2})$ separates points of $\mcH^{\gamma-1}(\mbT^{2})$, it follows that
		\begin{equation*}
			\int_{0}^{T}f_{t}\partial_{t}\theta(t)\dd t=-\int_{0}^{T}(\Delta f_{t}-\chem\vdiv(f_{t}\nabla\Phi_{f_{t}}))\theta(t)\dd t;
		\end{equation*}
		hence $f$ is weakly differentiable with weak derivative $\partial_{t}f_{t}=\Delta f_{t}-\chem\vdiv(f_{t}\nabla\Phi_{f_{t}})$. This yields the implication~\eqref{eq:Weak_Sol_Def_2} to~\eqref{eq:Weak_Sol_Def_1}.
	\end{details}
	
	Taking the constant function $1$ as a test function shows that the mean of the solution is preserved. The almost everywhere non-negativity can be shown by testing with (a smooth approximant to) the negative part of the solution $f_{t}^{-}\defeq\min\{0,f_{t}\}$ and taking limits by the dominated convergence theorem. For a similar proof in the two-dimensional plane, see~\cite[Prop.~3.4]{mayorcas_tomasevic_22}.
	\begin{details}
		\paragraph{Proof that $\norm{f_{0}^{-}}_{L^{2}(\mbT^{2})}=0$ implies $\norm{f_{t}^{-}}_{L^{2}(\mbT^{2})}=0$.}
		First let us assume that $f_{0}\in\mcH^{1}(\mbT^{2})$, it follows by Theorem~\ref{thm:Determ_KS_maximal} that there exists a weak solution $f\in C_{T}\mcH^{1}(\mbT^{2})\cap L_{T}^{2}\mcH^{2}(\mbT^{2})\cap W^{1,2}_{T}L^{2}(\mbT^{2})$ to~\eqref{eq:Determ_KS} for all $T<\Trdet$, where $\Trdet$ denotes the blow-up time in $L^{2}(\mbT^{2})$.
		
		An application of~\cite[Ex.~4.15~(b)]{arendt_kreuter_18} yields that $f^{-}\in W^{1,2}_{T}L^{2}(\mbT^{2})$ with $\partial_{t}f^{-}_{t}=\partial_{t}f_{t}\mathds{1}_{\{x\in\mbT^{2}:f_{t}(x)<0\}}$. Using the linearity and boundedness of the embedding $L^{2}(\mbT^{2})\subset\mcH^{-1}(\mbT^{2})$ as well as the definition of the Bochner--Lebesgue spaces~\cite[Chpt.~X,~Sec.~4]{amann_escher_09}, it follows that $f^{-}\in W^{1,2}((0,T);\mcH^{-1}(\mbT^{2}))$. Similarly, it also follows that $f\in C_{T}L^{2}(\mbT^{2})\cap L_{T}^{2}\mcH^{1}(\mbT^{2})$. Using that $x\mapsto x^{-}$ is Lipschitz continuous on $\mbR$, i.e.\
		\begin{equation}\label{eq:negative_part_Lipschitz}
			\begin{split}
				\abs{x^{-}-y^{-}}
				&=
				\begin{cases}
					\begin{aligned}
						&\abs{x-y}\quad&\text{if}~x<0,~y<0,\\
						&\abs{x}\leq\abs{x-y}\quad&\text{if}~x<0,~y\geq0,\\
						&0\leq\abs{x-y}\quad&\text{if}~x\geq0,~y\geq0,
					\end{aligned}
				\end{cases}\\
				&\leq\abs{x-y},
			\end{split}
		\end{equation}
		we obtain $\norm{u^{-}-v^{-}}_{L^{2}}^{2}\leq\norm{u-v}_{L^{2}}^{2}$; hence $u\mapsto u^{-}$ is Lipschitz continuous on $L^{2}(\mbT^{2})$, which yields $f^{-}\in C_{T}L^{2}(\mbT^{2})$. Using that $u^{-}\in\mcH^{1}(\mbT^{2})$ and $\nabla u^{-}=\mathds{1}_{u<0}\nabla u$ for each $u\in\mcH^{1}(\mbT^{2})$ (cf.~\cite[Ex.~4.15~(a)]{arendt_kreuter_18}), we obtain
		\begin{equation*}
			\norm{u^{-}}_{\mcH^{1}}^{2}\lesssim\norm{u^{-}}_{L^{2}}^{2}+\norm{\nabla u^{-}}_{L^{2}}^{2}\leq\norm{u}_{L^{2}}^{2}+\norm{\nabla u}_{L^{2}}^{2}\lesssim\norm{u}_{\mcH^{1}}^{2},
		\end{equation*}
		hence $f^{-}\in L^{2}_{T}\mcH^{1}(\mbT^{2})$. All in all, we showed that $f^{-}\in C_{T}L^{2}(\mbT^{2})\cap L_{T}^{2}\mcH^{1}(\mbT^{2})\cap W^{1,2}_{T}\mcH^{-1}(\mbT^{2})$.
		
		An application of~\eqref{eq:Weak_Sol_Def_3} with $\gamma=0$ and $\test=f^{-}$ followed by integration by parts (Lemma~\ref{lem:IBP_dual_pairing}) yields for a.e.\ $t\in[0,T]$,
		\begin{equation*}
			\norm{f_{t}^{-}}_{L^{2}}^{2}-\norm{f_{0}^{-}}_{L^{2}}^{2}=\int_{0}^{t}\inner{\partial_{s}f_{s}^{-}}{f_{s}}_{\mcH^{-1},\mcH^{1}}-\inner{\nabla f_{s}-\chem f_{s}\nabla\Phi_{f_{s}}}{{\nabla f_{s}^{-}}}_{L^{2}}\dd s.
		\end{equation*}
		To analyse the first term, we use that $f^{-}\in W^{1,2}_{T}L^{2}(\mbT^{2})$ to identify
		\begin{equation*}
			\begin{split}
				\inner{\partial_{s}f_{s}^{-}}{f_{s}}_{\mcH^{-1},\mcH^{1}}&=\inner{\partial_{s}f_{s}^{-}}{f_{s}}_{L^{2}}=\inner{\partial_{s}f_{s}}{\mathds{1}_{f_{s}<0}f_{s}}_{L^{2}}=\inner{\Delta f_{s}-\chem\vdiv(f_{s}\nabla\Phi_{f_{s}})}{f_{s}^{-}}_{L^{2}}\\
				&=-\inner{\nabla f_{s}-\chem f_{s}\nabla\Phi_{f_{s}}}{\nabla f_{s}^{-}}_{L^{2}}=-\inner{\nabla f_{s}-\chem f_{s}\nabla\Phi_{f_{s}}}{\mathds{1}_{f_{s}<0}\nabla f_{s}}_{L^{2}},
			\end{split}
		\end{equation*}
		which yields
		\begin{equation*}
			\begin{split}
				\norm{f_{t}^{-}}_{L^{2}}^{2}-\norm{f_{0}^{-}}_{L^{2}}^{2}&=-2\int_{0}^{t}\int_{\{x\in\mbT^{2}:f_{s}(x)<0\}}\abs{\nabla f_{s}(x)}^{2}\dd x\dd s \\
				&\quad+2\chem\int_{0}^{t}\int_{\{x\in\mbT^{2}:f_{s}(x)<0\}}f_{s}(x)\nabla\Phi_{f_{s}(x)}\nabla f_{s}(x)\dd x\dd s.
			\end{split}
		\end{equation*}
		In the second summand, we apply Young's inequality to bound for every $\vartheta>0$,
		\begin{equation*}
			\begin{split}
				2\chem\int_{\{x\in\mbT^{2}:f_{s}(x)<0\}}f_{s}(x)\nabla\Phi_{f_{s}(x)}\nabla f_{s}(x)\dd x&\leq\frac{\chem^{2}}{\vartheta}\int_{\{x\in\mbT^{2}:f_{s}(x)<0\}}\abs{\nabla f_{s}(x)}^{2}\dd x\\
				&\quad+\vartheta\int_{\{x\in\mbT^{2}:f_{s}(x)<0\}}\abs{f_{s}(x)\nabla\Phi_{f_{s}(x)}}^{2}\dd x,
			\end{split}
		\end{equation*}
		which, if we choose $\vartheta=\chem^{2}$, yields
		\begin{equation*}
			\begin{split}
				\norm{f_{t}^{-}}_{L^{2}}^{2}-\norm{f_{0}^{-}}_{L^{2}}^{2}&\leq-\int_{0}^{t}\int_{\{x\in\mbT^{2}:f_{s}(x)<0\}}\abs{\nabla f_{s}(x)}^{2}\dd x\dd s \\
				&\quad+\chem^{2}\int_{0}^{t}\int_{\{x\in\mbT^{2}:f_{s}(x)<0\}}\abs{f_{s}(x)\nabla\Phi_{f_{s}(x)}}^{2}\dd x\dd s.
			\end{split}
		\end{equation*}
		Sobolev's embedding $\mcH^{2}(\mbT^{2})\embed L^{\infty}(\mbT^{2})$ (cf.~\eqref{eq:Sobolev_embedding_II}) yields for some $C>0$,
		\begin{equation*}
			\abs{f_{s}(x)\nabla\Phi_{f_{s}(x)}}^{2}\leq\abs{f_{s}(x)}^{2}\norm{\nabla\Phi_{f_{s}}}_{L^{\infty}}^{2}\leq C\abs{f_{s}(x)}^{2}\norm{f_{s}}_{\mcH^{1}}^{2},
		\end{equation*}
		which implies
		\begin{equation*}
			\norm{f_{t}^{-}}_{L^{2}}^{2}-\norm{f_{0}^{-}}_{L^{2}}^{2}\leq-\int_{0}^{t}\int_{\{x\in\mbT^{2}:f_{s}(x)<0\}}\abs{\nabla f_{s}(x)}^{2}\dd x\dd s+C\chem^{2}\int_{0}^{t}\norm{f_{s}^{-}}_{L^{2}}^{2}\norm{f_{s}}_{\mcH^{1}}^{2}\dd s.
		\end{equation*}
		Assume $\norm{f_{0}^{-}}_{L^{2}}=0$, Gronwall's inequality then implies $\norm{f_{t}^{-}}_{L^{2}}=0$ for all $t\in[0,T]$.
		
		To prove the claim for $f_{0}\in L^{2}(\mbT^{2})$, let $(f_{0}^{n})_{n\in\mbN}$ be a sequence such that $f_{0}^{n}\in\mcH^{1}(\mbT^{2})$ for every $n\in\mbN$ and $f_{0}^{n}\to f_{0}\in L^{2}(\mbT^{2})$ as $n\to\infty$. Let $(f^{n})_{n\in\mbN}$ be the associated sequence of solutions to~\eqref{eq:Determ_KS} and let $f$ be the solution to~\eqref{eq:Determ_KS} with initial data $f_{0}$. Using that $(\norm{f_{0}^{n}}_{L^{2}})_{n\in\mbN}$ is convergent, hence bounded, we can use~\cite[Rem.~1.2]{liu_roeckner_13_local_global}, to conclude that there exists some $\bar{T}>0$ such that both $(f^{n})_{n\in\mbN}$ and $f$ exist on $[0,\bar{T}]$. Further, by~\cite[Lem.~2.4]{liu_roeckner_13_local_global}, we can control $\norm{f^{n}}_{L^{2}_{\bar{T}}\mcH^{1}}$ uniformly in $n\in\mbN$; hence,~\cite[Thm.~1.2]{liu_roeckner_13_local_global} implies $f^{n}\to f$ in $C_{\bar{T}}L^{2}(\mbT^{2})$. Denote the maximal times of existence in $L^{2}(\mbT^{2})$ by $(\Trdet[f^{n}])_{n\in\mbN}$ and $\Trdet[f]$. It follows that $\Trdet$ is lower semicontinuous (cf.~Lemma~\ref{lem:blow_up_lsc}), i.e.\ $\Trdet[f]\leq\liminf_{n\to\infty}\Trdet[f^{n}]$. Therefore, for every $T<\Trdet[f]$ there exists some $N\in\mbN$ such that $T<\Trdet[f^{n}]$ and $\norm{f^{n}}_{C_{T}L^{2}}<\infty$ uniformly for all $n\geq N$. Iterating the local convergence above, it follows that $f^{n}\to f$ in $C_{T}L^{2}(\mbT^{2})$ for all $T<\Trdet[f]$. In particular, for every $t\in[0,T]$,
		\begin{equation*}
			\norm{f_{t}}_{L^{2}}=\lim_{n\to\infty}\norm{f^{n}_{t}}_{L^{2}}=0.
		\end{equation*}
	\end{details}
	A stronger notion of positivity of solutions given initial data that are continuous and bounded away from $0$ is shown in Lemma~\ref{lem:Determ_KS_positivity}.
	
	To prove Claim~\ref{it:Global_Weak_Sol} we use that for $T<\Trdet$ and weak solutions $f\in C_{T}L^{2}(\mbT^{2})\cap L^{2}_{T}\mcH^{1}(\mbT^{2})$ one has for each $t\leq T$ the identity~\eqref{eq:energy_identity_1}.
	\begin{details}
		To show~\eqref{eq:energy_identity_1}, it suffices to apply~\eqref{eq:Weak_Sol_Def_3} with $\gamma=0$ and $\test=f$ followed by integration by parts (Lemma~\ref{lem:IBP_dual_pairing}) to deduce
		\begin{equation*}
			\norm{f_{t}}_{L^{2}}^{2}-\norm{f_{0}}_{L^{2}}^{2}=-2\int_{0}^{t}\inner{\nabla f_{s}-\chem f_{s}\nabla\Phi_{f_{s}}}{\nabla f_{s}}_{L^{2}}\dd s,
		\end{equation*}
		which we can further simplify by using the identity
		\begin{equation*}
			\inner{f_{s}\nabla\Phi_{f_{s}}}{\nabla f_{s}}_{L^{2}}=\frac{1}{2}\inner{\nabla\Phi_{f_{s}}}{\nabla f_{s}^{2}}_{L^{2}}=-\frac{1}{2}\inner{\Delta\Phi_{f_{s}}}{f_{s}^{2}}_{L^{2}}=\frac{1}{2}\norm{f_{s}}_{L^{3}}^{3}-\frac{1}{2}\norm{f_{s}}_{L^{2}}^{2}
		\end{equation*}
		to obtain
		\begin{equation*}
			\norm{f_{t}}_{L^{2}}^{2}-\norm{f_{0}}_{L^{2}}^{2}=-2\int_{0}^{t}\norm{\nabla f_{s}}_{L^{2}}^{2}\dd s+\chem\int_{0}^{t}\norm{f_{s}}_{L^{3}}^{3}\dd s-\chem\int_{0}^{t}\norm{f_{s}}_{L^{2}}^{2}\dd s,
		\end{equation*}
		which yields~\eqref{eq:energy_identity_1}.
	\end{details}
	
	We first note that if $\chem\leq0$ then by Gronwall's inequality applied to~\eqref{eq:energy_identity_1} we have a direct bound for $\norm{f_{t}}^{2}_{L^{2}}$ which is sufficient to show that $\norm{f_{t}}_{L^{2}}<\infty$ for all $t<\infty$.
	\begin{details}
		Using that $\chem\leq0$, we obtain from~\eqref{eq:energy_identity_1},
		\begin{equation*}
			\norm{f_{t}}^{2}_{L^{2}}\leq\norm{f_{0}}^{2}_{L^{2}}-\chem\int_{0}^{t}\norm{f_{s}}^{2}_{L^{2}}\dd s.
		\end{equation*}
		An application of Gronwall's inequality yields the explicit estimate
		\begin{equation*}
			\norm{f_{t}}^2_{L^{2}}\leq\norm{f_{0}}^2_{L^{2}}\exp(-\chem t).
		\end{equation*}
	\end{details}
	
	In the case $\chem>0$, we need an energy identity for the re-centred solution $f_{t}-\mean{f_{t}}$. Noting that $\mean{f_{t}}\equiv1$, we obtain
	\begin{equation*}
		\norm{f_{t}-\mean{f_{t}}}_{L^{2}}^{2}=\norm{f_{t}}_{L^{2}}^{2}-1
	\end{equation*}
	and
	\begin{equation*}
		\norm{f_{t}}_{L^{3}}^{3}-\norm{f_{t}}_{L^{2}}^{2}=\norm{f_{t}-\mean{f_{t}}}_{L^{3}}^{3}+2\norm{f_{t}-\mean{f_{t}}}_{L^{2}}^{2}.
	\end{equation*}
	Hence, we can deduce from~\eqref{eq:energy_identity_1}, 
	\begin{equation*}
		\norm{f_{t}-\mean{f_{t}}}^{2}_{L^{2}}+1=\norm{f_{0}-\mean{f_{0}}}^{2}_{L^{2}}+1-2\int_{0}^{t}\norm{\nabla f_{s}}^{2}_{L^{2}}\dd s+\chem\int_{0}^{t}\norm{f_{s}-\mean{f_{s}}}^{3}_{L^{3}}\dd s+2\chem\int_{0}^{t}\norm{f_{s}-\mean{f_{s}}}^{2}_{L^{2}}\dd s,
	\end{equation*}
	which yields
	\begin{equation}\label{eq:energy_identity_2}
		\norm{f_{t}-\mean{f_{t}}}^{2}_{L^{2}}=\norm{f_{0}-\mean{f_{0}}}^{2}_{L^{2}}-2\int_{0}^{t}\norm{\nabla f_{s}}^{2}_{L^{2}}\dd s+\chem\int_{0}^{t}\norm{f_{s}-\mean{f_{s}}}^{3}_{L^{3}}\dd s+2\chem\int_{0}^{t}\norm{f_{s}-\mean{f_{s}}}^{2}_{L^{2}}\dd s.
	\end{equation}
	An application of the periodic Gagliardo--Nirenberg--Sobolev inequality~\eqref{eq:GNS_Per} with $d=2$, $q=3$ and $r=1$ yields
	\begin{equation*}
		\norm{f_{s}-\mean{f_{s}}}^{3}_{L^{3}}\leq C_{\gnsper}(2,3,1)^{3}\norm{\nabla f_{s}}^{2}_{L^{2}}\norm{f_{s}-\mean{f_{s}}}_{L^{1}}\leq2C_{\gnsper}(2,3,1)^{3}\norm{\nabla f_{s}}^{2}_{L^{2}},
	\end{equation*}
	where in the second step we used Minkowski's inequality to estimate $\norm{f_{s}-\mean{f_{s}}}_{L^{1}}\leq2$. 
	\begin{details}
		This bound seems to be reasonable sharp. Indeed, if $f_{s}$ approximates a rectangle of width $h$ and height $1/h$, then $\mean{f_{s}}=1$ and $\norm{f-\mean{f}_{s}}_{L^{1}}=(1-h)+h(\frac{1}{h}-1)=2-2h\to2$ as $h\to0$.
	\end{details}
	Inserting this into~\eqref{eq:energy_identity_2} we find the inequality
	\begin{equation*}
		\norm{f_{t}-\mean{f_{t}}}^{2}_{L^{2}}\leq\norm{f_{0}-\mean{f_{0}}}^{2}_{L^{2}}-2\Bigl(1-\chem C_{\gnsper}(2,3,1)^{3}\Bigr)\int_{0}^{t}\norm{\nabla f_{s}}^{2}_{L^{2}}\dd s+2\chem\int_{0}^{t}\norm{f_{s}-\mean{f_{s}}}^{2}_{L^{2}}\dd s.
	\end{equation*}
	If $\chem\leq\thresh\defeq C_{\gnsper}(2,3,1)^{-3}$ then 
	\begin{equation*}
		\norm{f_{t}-\mean{f_{t}}}^{2}_{L^{2}}\leq\norm{f_{0}-\mean{f_{0}}}^{2}_{L^{2}}+2\chem\int_{0}^{t}\norm{f_{s}-\mean{f_{s}}}^{2}_{L^{2}}\dd s
	\end{equation*}
	and we can apply Gronwall's inequality to deduce
	\begin{equation*}
		\norm{f_{t}-\mean{f_{t}}}^{2}_{L^{2}}\leq\norm{f_{0}-\mean{f_{0}}}^{2}_{L^{2}}\exp(2\chem t),
	\end{equation*}
	which implies $\Trdet=\infty$.
	\begin{details}
		where we used that $\norm{f_{t}}_{L^{2}}^{2}=\norm{f_{t}-\mean{f_{t}}}_{L^{2}}^{2}+1$.
	\end{details}
	
	To prove Claim~\ref{it:Regular_Weak_Sol} we first use the fact that weak solutions to~\eqref{eq:Determ_KS} are mild, 
	\begin{equation}\label{eq:Determ_KS_mild}
		f=Pf_{0}-\chem\vdiv\mcI[f\nabla\Phi_{f}],
	\end{equation}
	which can be shown by testing the equation against the time-reversed heat kernel.
	\begin{details}
		\paragraph{Proof that the weak solution is a mild solution.}
		We adapt~\cite[Proof of Prop.~2.4, Step~3]{blath_hammer_nie_23}.
		Let $\varphi\in C^{\infty}(\mbR^{2})$ be of compact support, even and such that $\varphi(0)=1$. For every $\delta>0$ we define $\psi_{\delta}$ as in Subsection~\ref{subsec:notation}.
		
		Let $T<\Trdet$ and define for $t\in[0,T]$,
		\begin{equation*}
			\test^{\delta}_{t}(x)\defeq\sum_{\om\in\mbZ^{2}}\euler^{2\uppi\upi\inner{\om}{x}}\varphi(\delta\om)\euler^{-\abs{T-t}\abs{2\uppi\om}^{2}}=(\psi_{\delta}\ast\msH_{T-t})(x).
		\end{equation*}
		This test function lies in $L^{2}_{T}\mcH^{1}(\mbT^{2})$, since
		\begin{equation*}
			\begin{split}
				\int_{0}^{T}\norm{\test^{\delta}_{t}}_{\mcH^{1}}^{2}\dd t&=\sum_{\om\in\mbZ^{2}}(1+\abs{2\uppi\om}^{2})\varphi(\delta\om)^{2}\int_{0}^{T}\euler^{-2\abs{T-t}\abs{2\uppi\om}^{2}}\dd t\\
				&=T+\frac{1}{2}\sum_{\om\in\mbZ^{2}\setminus\{0\}}\varphi(\delta\om)^{2}(1+\abs{2\uppi\om}^{2})\abs{2\uppi\om}^{-2}(1-\euler^{-2T\abs{2\uppi\om}^{2}})<\infty.
			\end{split}
		\end{equation*}
		Its derivative is given by
		\begin{equation*}
			\partial_{t}\test^{\delta}_{t}(x)=\sum_{\om\in\mbZ^{2}\setminus\{0\}}\euler^{2\uppi\upi\inner{\om}{x}}\varphi(\delta\om)\abs{2\uppi\om}^{2}\euler^{-\abs{T-t}\abs{2\uppi\om}^{2}}
		\end{equation*}
		and lies in $L^{2}_{T}\mcH^{-1}(\mbT^{2})$, since
		\begin{equation*}
			\begin{split}
				\int_{0}^{T}\norm{\partial_{t}\test^{\delta}_{t}}_{\mcH^{-1}}^{2}\dd t&=\sum_{\om\in\mbZ^{2}\setminus\{0\}}(1+\abs{2\uppi\om}^{2})^{-1}\varphi(\delta\om)^{2}\abs{2\uppi\om}^{4}\int_{0}^{T}\euler^{-2\abs{T-t}\abs{2\uppi\om}^{2}}\dd t\\
				&=\frac{1}{2}\sum_{\om\in\mbZ^{2}\setminus\{0\}}\varphi(\delta\om)^{2}(1+\abs{2\uppi\om}^{2})^{-1}\abs{2\uppi\om}^{2}(1-\euler^{-2T\abs{2\uppi\om}^{2}})<\infty;
			\end{split}
		\end{equation*}
		hence $\test^{\delta}\in W^{1,2}_{T}\mcH^{-1}(\mbT^{2})$.
		
		For a given $x\in\mbT^{2}$, we apply~\eqref{eq:Weak_Sol_Def_3} with $\gamma=0$ and $\test^{\delta}(x-\place)$ to obtain
		\begin{equation*}
			\begin{split}
				&\inner{f_{T}}{\test^{\delta}_{T}(x-\place)}_{L^{2}}-\inner{f_{0}}{\test^{\delta}_{0}(x-\place)}_{L^{2}}\\
				&=\int_{0}^{T}\inner{\partial_{s}\test^{\delta}_{s}(x-\place)}{f_{s}}_{\mcH^{-1},\mcH^{1}}\dd s+\int_{0}^{T}\inner{\Delta f_{s}-\chem\vdiv(f_{s}\nabla\Phi_{f_{s}})}{\test^{\delta}_{s}(x-\place)}_{\mcH^{-1},\mcH^{1}}\dd s.
			\end{split}
		\end{equation*}
		Note that $(\partial_{s}+\Delta)\test^{\delta}_{s}=0$, hence integration by parts (Lemma~\ref{lem:IBP_dual_pairing}) yields
		\begin{equation*}
			\inner{\partial_{s}\test^{\delta}_{s}(x-\place)}{f_{s}}_{\mcH^{-1},\mcH^{1}}=-\inner{\Delta\test^{\delta}_{s}(x-\place)}{f_{s}}_{\mcH^{-1},\mcH^{1}}=-\inner{\nabla\test^{\delta}_{s}(x-\place)}{\nabla f_{s}}_{L^{2}},
		\end{equation*}
		which implies
		\begin{equation}\label{eq:weak_to_mild_mollified}
			\inner{f_{T}}{\test^{\delta}_{T}(x-\place)}_{L^{2}}-\inner{f_{0}}{\test^{\delta}_{0}(x-\place)}_{L^{2}}=-\chem\int_{0}^{T}\inner{\vdiv(f_{s}\nabla\Phi_{f_{s}})}{\test^{\delta}_{s}(x-\place)}_{\mcH^{-1},\mcH^{1}}\dd s.
		\end{equation}
		Using that $\test^{\delta}_{t}(x)=(\psi_{\delta}\ast\msH_{T-t})(x)$, we obtain
		\begin{equation*}
			\inner{f_{T}}{\test^{\delta}_{T}(x-\place)}_{L^{2}}=(\psi_{\delta}\ast f_{T})(x),\quad \inner{f_{0}}{\test^{\delta}_{0}(x-\place)}_{L^{2}}=P_{T}(\psi_{\delta}\ast f_{0})(x),
		\end{equation*}
		and
		\begin{equation*}
			\int_{0}^{T}\inner{\vdiv(f_{s}\nabla\Phi_{f_{s}})}{\test^{\delta}_{s}(x-\place)}_{\mcH^{-1},\mcH^{1}}\dd s=\vdiv\mcI[\psi_{\delta}\ast(f\nabla\Phi_{f})]_{T}(x),
		\end{equation*}
		where we used~\eqref{eq:product_estimate_Sobolev_III} to ensure $f_{s}\nabla\Phi_{f_{s}}\in\mcH^{1}(\mbT^{2})$ so that $\inner{\vdiv(f_{s}\nabla\Phi_{f_{s}})}{\test^{\delta}_{s}(x-\place)}_{\mcH^{-1},\mcH^{1}}=\inner{\vdiv(f_{s}\nabla\Phi_{f_{s}})}{\test^{\delta}_{s}(x-\place)}_{L^{2}}$.
		
		To take $\delta\to0$ in~\eqref{eq:weak_to_mild_mollified}, we apply the dominated convergence theorem. By Parseval's theorem, each term in~\eqref{eq:weak_to_mild_mollified} is dominated:
		\begin{equation*}
			\norm{\psi_{\delta}\ast f_{T}-f_{T}}_{L^{2}}^{2}=\sum_{\om\in\mbZ^{2}}(\varphi(\delta\om)-1)^{2}\abs{\hat{f_{T}}(\om)}^{2}\lesssim\norm{f_{T}}_{L^{2}}^{2},
		\end{equation*}
		\begin{equation*}
			\norm{P_{T}(\psi_{\delta}\ast f_{0})-P_{T}f_{0}}_{L^{2}}^{2}=\sum_{\om\in\mbZ^{2}}\euler^{-2T\abs{2\uppi\om}^{2}}(\varphi(\delta\om)-1)^{2}\abs{\hat{f_{0}}(\om)}^{2}\lesssim\norm{P_{T}f_{0}}_{L^{2}}^{2}
		\end{equation*}
		and
		\begin{equation*}
			\begin{split}
				\norm{\vdiv\mcI[\psi_{\delta}\ast(f\nabla\Phi_{f})]_{T}-\vdiv\mcI[f\nabla\Phi_{f}]_{T}}_{L^{2}}^{2}&=\sum_{\om\in\mbZ^{2}}(1-\varphi(\delta\om))^{2}\Bigl\lvert2\uppi\upi\om\int_{0}^{T}\euler^{-\abs{T-s}\abs{2\uppi\om}^{2}}\msF(f_{s}\nabla\Phi_{f_{s}})(\om)\dd s\Bigr\rvert^{2}\\
				&\lesssim\norm{\vdiv\mcI[f\nabla\Phi_{f}]_{T}}_{L^{2}}^{2}.
			\end{split}
		\end{equation*}
		Let $\vartheta\in(0,1)$, we can control the right hand sides above by Schauder's estimate (Lemma~\ref{lem:Schauder_Sobolev}),
		\begin{equation*}
			\norm{P_{T}f_{0}}_{L^{2}}^{2}\lesssim\norm{f_{0}}_{L^{2}}^{2},\qquad\norm{\vdiv\mcI[f\nabla\Phi_{f}]}_{C_{T}L^{2}}\lesssim\norm{\mcI[f\nabla\Phi_{f}]}_{C_{T}\mcH^{1}}\lesssim(1\vee T^{\frac{1-\vartheta}{2}})\norm{f\nabla\Phi_{f}}_{C_{T}\mcH^{-\vartheta}}
		\end{equation*}
		and the product estimate~\eqref{eq:product_estimate_Sobolev_IV} which yields
		\begin{equation*}
			\norm{f\nabla\Phi_{f}}_{C_{T}\mcH^{-\vartheta}}\lesssim\norm{f}_{C_{T}L^{2}}^{2}.
		\end{equation*}
		Hence, we can pass to the limit $\delta\to0$ in~\eqref{eq:weak_to_mild_mollified} to obtain the mild equation
		\begin{equation*}
			f_{T}=P_{T}f_{0}-\chem\vdiv\mcI[f\nabla\Phi_{f}]_{T}.
		\end{equation*}
	\end{details}
	Using~\eqref{eq:Determ_KS_mild} we can bootstrap the space regularity of $f$ at positive times to establish $f\in C((0,T];\mcH^{4}(\mbT^{2}))$, which by Sobolev's embedding $\mcH^{4}(\mbT^{2})\embed C^{2}(\mbT^{2})$ is enough to deduce $\partial_{x_{i}}\partial_{x_{j}}f\in C((0,T]\times\mbT^{2})$ for every $i,j\in\{1,2\}$. Furthermore, we can use the same space regularity to deduce that the weak time derivative $\partial_{t}f$ lies in $C((0,T];\mcH^{2}(\mbT^{2}))$, which together with the embedding $\mcH^{2}(\mbT^{2})\embed C(\mbT^{2})$ (cf.~\eqref{eq:Sobolev_embedding_II}) yields $\partial_{t}f\in C((0,T]\times\mbT^{2})$. The fundamental theorem of calculus then allows us to argue that the weak time derivative coincides with the (usual) time derivative, which yields $f\in C^{2}_{1}((0,T]\times\mbT^{2})$.
	\begin{details}
		\paragraph{Proof that $f\in C^{2}_{1}((0,T]\times\mbT^{2})$.}
		Let $\eta\geq0$, $\gamma\geq0$ and $f_{0}\in\mcH^{\gamma}(\mbT^{2})$. It follows by Schauder's estimate~\cite[Lem.~A.5]{martini_mayorcas_25}, that
		\begin{equation}\label{eq:bootstrap_IC}
			\norm{Pf_{0}}_{C_{\eta;T}\mcH^{\gamma+2\eta}}\lesssim\norm{f_{0}}_{\mcH^{\gamma}},
		\end{equation}
		and in particular $Pf_{0}\in C_{\eta;T}\mcH^{\gamma+2\eta}(\mbT^{2})$, where the continuity of $[0,T]\ni t\mapsto(1\wedge t)^{\eta}P_{t}f_{0}\in\mcH^{\gamma+2\eta}(\mbT^{2})$ follows by an approximation argument as in~\cite[Lem.~A.6]{martini_mayorcas_25}. 
		
		We further obtain by Lemma~\ref{lem:Schauder_Sobolev} and the product estimate~\eqref{eq:product_estimate_Sobolev_IV} that for every $\vartheta>0$,
		\begin{equation}\label{eq:bootstrap_advection_I}
			\norm{\vdiv\mcI[f\nabla\Phi_{f}]}_{C_{T}\mcH^{1-\vartheta}}\lesssim\norm{\mcI[f\nabla\Phi_{f}]}_{C_{T}\mcH^{2-\vartheta}}\lesssim(1\vee T^{\frac{\vartheta}{4}})\norm{f\nabla\Phi_{f}}_{C_{T}\mcH^{-\vartheta/2}}\lesssim_{T}\norm{f}_{C_{T}L^{2}}^{2}
		\end{equation}
		and in particular $\vdiv\mcI[f\nabla\Phi_{f}]\in C_{T}\mcH^{1-\vartheta}(\mbT^{2})$.
		
		Let $f_{0}\in L^{2}(\mbT^{2})$, combining~\eqref{eq:bootstrap_IC} and~\eqref{eq:bootstrap_advection_I} we obtain $f\in C_{1/2;T}\mcH^{1-\vartheta}(\mbT^{2})\embed C((0,T];\mcH^{1-\vartheta}(\mbT^{2})$. Hence for every $t_{0}\in(0,T)$, it follows that $[0,T-t_{0}]\ni t\mapsto g_{t}\defeq f|_{[t_{0},T]}(t+t_{0})$ solves~\eqref{eq:Determ_KS_mild} with initial data $f_{t_{0}}\in\mcH^{1-\vartheta}(\mbT^{2})$.
		
		An application of Lemma~\ref{lem:Schauder_Sobolev} and the product estimate~\eqref{eq:product_estimate_Sobolev_III} yields for every $\gamma>0$ and $\vartheta>0$,
		\begin{equation}\label{eq:bootstrap_advection_II}
			\norm{\vdiv\mcI[g\nabla\Phi_{g}]}_{C_{T-t_{0}}\mcH^{\gamma+1-\vartheta}}\lesssim\norm{\mcI[g\nabla\Phi_{g}]}_{C_{T-t_{0}}\mcH^{\gamma+2-\vartheta}}\lesssim(1\vee (T-t_{0})^{\frac{\vartheta}{2}})\norm{g\nabla\Phi_{g}}_{C_{T-t_{0}}\mcH^{\gamma}}\lesssim_{T-t_{0}}\norm{g}_{C_{T-t_{0}}\mcH^{\gamma}}^{2}
		\end{equation}
		and in particular $\vdiv\mcI[g\nabla\Phi_{g}]\in C_{T-t_{0}}\mcH^{\gamma+1-\vartheta}(\mbT^{2})$.
		
		Taking $\vartheta<1$ and $\gamma=1-\vartheta>0$, we obtain by~\eqref{eq:bootstrap_IC} and~\eqref{eq:bootstrap_advection_II} that $g\in C_{1/2;T-t_{0}}\mcH^{2(1-\vartheta)}(\mbT^{2})\embed C((0,T-t_{0}];\mcH^{2(1-\vartheta)}(\mbT^{2}))$ and hence $f|_{[t_{0},T]}\in C((t_{0},T];\mcH^{2(1-\vartheta)}(\mbT^{2}))$.
		
		Iterating this argument on small time intervals and using that $t_{0}$ was arbitrary, we can deduce that $f\in C((0,T];\mcH^{4}(\mbT^{2}))$, which embeds into $C((0,T];C^{2}(\mbT^{2}))$ (cf.~\cite[Thm.~23.5]{vanzuijlen_22}). In particular $\partial_{x_{i}}\partial_{x_{j}}f\in C((0,T];C(\mbT^{2}))$ for every $i,j\in\{1,2\}$. To establish $\partial_{x_{i}}\partial_{x_{j}}f\in C((0,T]\times\mbT^{2})$, it suffices to apply the triangle inequality (cf.~\eqref{eq:continuity_triangle_argument} below).
		
		Using that $\partial_{t}f=\Delta f-\chem\vdiv(f\nabla\Phi_{f})$ and $f\in C((0,T];\mcH^{4}(\mbT^{2}))$, we obtain by Sobolev's embedding~\eqref{eq:Sobolev_embedding_II},
		\begin{equation*}
			\partial_{t}f\in C((0,T];\mcH^{2}(\mbT^{2}))\embed C((0,T];C(\mbT^{2})).
		\end{equation*}
		To establish $\partial_{t}f\in C((0,T]\times\mbT^{2})$, we can apply the triangle inequality to bound for every $(t,x),(s,y)\in(0,T]\times\mbT^{2}$,
		\begin{equation}\label{eq:continuity_triangle_argument}
			\begin{split}
				\abs{\partial_{t}f(t,x)-\partial_{t}f(s,y)}&\leq\abs{\partial_{t}f(t,x)-\partial_{t}f(t,y)}+\abs{\partial_{t}f(t,y)-\partial_{t}f(s,y)}\\
				&\leq\abs{\partial_{t}f(t,x)-\partial_{t}f(t,y)}+\norm{\partial_{t}f_{t}-\partial_{t}f_{s}}_{L^{\infty}},
			\end{split}
		\end{equation}
		where we may control the first summand with $\partial_{t}f_{t}\in C(\mbT^{2})$ and the second summand with $\partial_{t}f\in C((0,T];C(\mbT^{2}))$, which yields $\partial_{t}f\in C((0,T]\times\mbT^{2})$. By~\cite[Subsec.~5.9.2,~Thm.~2]{evans_98_PDE} it follows that
		\begin{equation*}
			f_{t}(x)-f_{t_{0}}(x)=\int_{t_{0}}^{t}\partial_{s}f_{s}(x)\dd s,
		\end{equation*}
		which implies by the (classical) fundamental theorem of calculus that $f(x)$ is differentiable in time for every $x\in\mbT^{2}$ with derivative $\partial_{t}f(x)$.
	\end{details}
	
	To establish $f\in C_{T}\mcH^{\gamma}(\mbT^{2})\cap L_{T}^{2}\mcH^{\gamma+1}(\mbT^{2})$ for every $\gamma>0$, $f_{0}\in\mcH^{\gamma}(\mbT^{2})$ and $T<\Trdet$ , it suffices to estimate each term in the mild formulation~\eqref{eq:Determ_KS_mild} by using Schauder's estimates (Lemma~\ref{lem:Schauder_Sobolev}) combined with product estimates in Bessel potential spaces.
	\begin{details}
		\paragraph{Proof of $f\in C_{T}\mcH^{\gamma}(\mbT^{2})\cap L_{T}^{2}\mcH^{\gamma+1}(\mbT^{2})$ if $f_{0}\in\mcH^{\gamma}(\mbT^{2})$.}
		The claim is a direct consequence of Theorem~\ref{thm:Determ_KS_maximal}. Here we show that $f\in C_{T}\mcH^{\gamma}(\mbT^{2})\cap L_{T}^{2}\mcH^{\gamma+1}(\mbT^{2})$ also follows from bootstrapping the regularity $f\in C_{T}L^{2}(\mbT^{2})\cap L_{T}^{2}\mcH^{1}(\mbT^{2})$ via the mild formulation~\eqref{eq:Determ_KS_mild}.
		
		Since $f_{0}\in\mcH^{\gamma}(\mbT^{2})$, we obtain by Lemma~\ref{lem:Schauder_Sobolev} that $Pf_{0}\in C_{T}\mcH^{\gamma}(\mbT^{2})\cap L_{T}^{2}\mcH^{\gamma+1}(\mbT^{2})$. To bootstrap the advection, we use $f\in C_{T}L^{2}(\mbT^{2})$ and~\eqref{eq:bootstrap_advection_I} to deduce $\vdiv\mcI[f\nabla\Phi_{f}]\in C_{T}\mcH^{1-\vartheta}(\mbT^{2})$ for every $\vartheta\in(0,1)$, which yields $f\in C_{T}\mcH^{\gamma\wedge(1-\vartheta)}(\mbT^{2})$. If $\gamma\leq1-\vartheta$, we're done, otherwise we can iterate~\eqref{eq:bootstrap_advection_II} with $t_{0}=0$ to obtain $f\in C_{T}\mcH^{\gamma}(\mbT^{2})$.
		
		Using that $f\in C_{T}L^{2}(\mbT^{2})\cap L_{T}^{2}\mcH^{1}(\mbT^{2})$, we obtain by an application of Lemma~\ref{lem:Schauder_Sobolev} and the product estimate~\eqref{eq:product_estimate_Sobolev_IV},
		\begin{equation*}
			\norm{\vdiv\mcI[f\nabla\Phi_{f}]}_{L_{T}^{2}\mcH^{1-\vartheta}}\lesssim\norm{\mcI[f\nabla\Phi_{f}]}_{L_{T}^{2}\mcH^{2-\vartheta}}\lesssim(T+1)\norm{f\nabla\Phi_{f}}_{L_{T}^{2}\mcH^{-\vartheta}}\lesssim_{T}\norm{f}_{C_{T}L^{2}}\norm{f}_{L_{T}^{2}\mcH^{1}},
		\end{equation*}
		which implies $f\in L_{T}^{2}\mcH^{\gamma\wedge(1-\vartheta)}(\mbT^{2})$. If $\gamma\leq1-\vartheta$, we're done, otherwise we can iterate~\eqref{eq:bootstrap_advection_L2_II} (see below) to obtain $f\in L_{T}^{2}\mcH^{\gamma+1}(\mbT^{2})$.
		
		Let $\gamma'>0$, combining Lemma~\ref{lem:Schauder_Sobolev} with the product estimate~\eqref{eq:product_estimate_Sobolev_II} we deduce
		\begin{equation}\label{eq:bootstrap_advection_L2_II}
			\norm{\vdiv\mcI[f\nabla\Phi_{f}]}_{L_{T}^{2}\mcH^{\gamma'+1}}\lesssim\norm{\mcI[f\nabla\Phi_{f}]}_{L_{T}^{2}\mcH^{\gamma'+2}}\lesssim(T+1)\norm{f\nabla\Phi_{f}}_{L_{T}^{2}\mcH^{\gamma'}}\lesssim_{T}\norm{f}_{C_{T}\mcH^{\gamma'}}\norm{f}_{L_{T}^{2}\mcH^{\gamma'+1}}.
		\end{equation}
	\end{details}
	This yields Claim~\ref{it:Regular_Weak_Sol}.
\end{proof}
\begin{remark}[Sharp a priori bounds]\label{rem:sharp_apriori}
	On the whole space instead of the torus, one can make use of the sharp Gagliardo--Nirenberg--Sobolev inequality found in~\cite{weinstein_83_schrodinger} to show that the entropy is strictly non-increasing and the solution is globally well-posed provided $\chem<4\uppi(1.86225\ldots)$, see~\cite[Sec.~2.2]{blanchet_dolbeault_perthame_06}. Still on the whole space, this range can be further improved to allow for all $\chem<8\uppi$ by instead analysing the evolution of the free energy and replacing the Gagliardo--Nirenberg--Sobolev inequality with the logarithmic Hardy--Littlewood--Sobolev inequality, see~\cite{blanchet_dolbeault_perthame_06}. This result is in fact sharp. A straightforward argument, see~\cite[Sec.~2.1]{blanchet_dolbeault_perthame_06}, analysing the evolution of the second moment of solutions demonstrates that weak and very weak solutions to~\eqref{eq:Determ_KS} must blow up in finite time provided $\chem>8\uppi$ with a blow-up time converging to $+\infty$ as $\chem$ approaches $8\uppi$. See also more recent results that treat more general initial data~\cite{wei_18,fournier_tardy_22_gwp}. The picture on the torus is more complicated~\cite[Thm.~8.1]{kiselev_xu_16_suppression}; however, for our purposes we do not require a sharp result, only well-posedness of the PDE~\eqref{eq:Determ_KS} for some range of $\chem$.
\end{remark}
Next, we show that the solution to~\eqref{eq:Determ_KS} started from continuous initial data is continuous everywhere.
\begin{lemma}\label{lem:Determ_KS_continuity}
	Let $f_{0}\in C(\mbT^{2})$ be such that $f_{0}\geq0$ and $\mean{f_{0}}=1$, $f$ be the weak solution to~\eqref{eq:Determ_KS} with initial data $f_{0}$ and chemotactic sensitivity $\chem\in\mbR$, and $\Trdet$ be its maximal time of existence (see Theorem~\ref{thm:Determ_KS_Well_Posedness}). Then $f\in C([0,T]\times\mbT^{2})$ for every $T<\Trdet$ and $\Trdet$ coincides with the maximal time of existence in $C(\mbT^{2})$.
\end{lemma}
\begin{proof}
	Using that $f_{0}\in C(\mbT^{2})\subset L^{2}(\mbT^{2})$, we can apply Theorem~\ref{thm:Determ_KS_Well_Posedness} to find the unique weak solution $f\in C((0,T]\times\mbT^{2})$; hence it suffices to establish the continuity of $f$ at time $0$, i.e.\ on a small time interval $[0,\bar{T}]$ where $\bar{T}\in(0,T\wedge1]$. By the proof of Theorem~\ref{thm:Determ_KS_Well_Posedness} (cf.~\eqref{eq:Determ_KS_mild}), it follows that the weak solution is also mild, hence we can use Young's convolution estimate and Schauder's lemma to bound for every $\vartheta\in(0,1)$,
	 \begin{equation*}
	 	\norm{f_{t}}_{C(\mbT^{2})}\leq\norm{f_{0}}_{C(\mbT^{2})}+C\abs{\chem}\norm{f}_{C_{\bar{T}}L^{2}}\int_{0}^{t}\abs{t-s}^{-\frac{1+\vartheta}{2}}\norm{f_{s}}_{C(\mbT^{2})}\dd s,
	 \end{equation*}
	where $C$ is some constant. Denote $\beta\defeq\frac{2}{1-\vartheta}$, we can now apply the fractional Gronwall estimate~\cite[Rem.~3.7]{webb_19} to deduce
	\begin{equation*}
		\norm{f_{t}}_{C_{\bar{T}}C(\mbT^{2})}\leq\beta\norm{f_{0}}_{C(\mbT^{2})}\exp\Bigl(C_{\beta}\abs{\chem}^{\beta}\norm{f}_{C_{\bar{T}}L^{2}}^{\beta}t\Bigr),
	\end{equation*}
	where $C_{\beta}>0$ is a constant that only depends on $\beta$. This a priori estimate allows us to use an approximation argument to show that $f$ is continuous at time $0$ which implies $f\in C([0,T]\times\mbT^{2})$.
	\begin{details}
		\paragraph{Proof that $f\in C([0,T]\times\mbT^{2})$.}
		We already know by Theorem~\ref{thm:Determ_KS_Well_Posedness} that $f\in C((0,T]\times\mbT^{2})$, hence it suffices to show that $f$ is continuous at time $0$. Assume $f\in C_{\bar{T}}C(\mbT^{2})$ for some $\bar{T}\in(0,T]$, we can then use the triangle inequality to estimate for every $(0,x),(s,y)\in[0,\bar{T}]\times\mbT^{2}$,
		\begin{equation*}
			\abs{f(0,x)-f(s,y)}\leq\abs{f(0,x)-f(0,y)}+\abs{f(0,y)-f(s,y)}\leq\abs{f(0,x)-f(0,y)}+\norm{f_{0}-f_{s}}_{L^{\infty}},
		\end{equation*}
		which, using that $f_{0}\in C(\mbT^{2})$ and $f\in C_{\bar{T}}C(\mbT^{2})$, yields $f\in C([0,\bar{T}]\times\mbT^{2})$. Hence, it suffices to show $f\in C_{\bar{T}}C(\mbT^{2})$.
		
		To show $f\in C_{\bar{T}}C(\mbT^{2})$, we use the mild formulation~\eqref{eq:Determ_KS_mild}. We first derive an a priori bound on the $C_{\bar{T}}C(\mbT^{2})$-norm of $f$, which we then use to deduce continuity by an approximation argument.
		
		We can control the supremum-norm of $P_{t}f_{0}$ by
		\begin{equation*}
			\norm{P_{t}f_{0}}_{C(\mbT^{2})}\leq\norm{f_{0}}_{C(\mbT^{2})},
		\end{equation*}
		where for $t>0$ we use Young's convolution inequality on $\mbT^{2}$ (cf.~\cite[Thm.~33.7]{vanzuijlen_22}) and for $t=0$ the definition $P_{0}f_{0}=f_{0}$.
		
		The mild formulation~\eqref{eq:Determ_KS_mild} combined with the embedding $\mcB_{\infty,1}^{0}(\mbT^{2})\embed C(\mbT^{2})$ (cf.~\cite[Thm.~23.5]{vanzuijlen_22}) allows us to estimate for $t\leq\bar{T}$,
		\begin{equation*}
			\norm{f_{t}}_{C(\mbT^{2})}\leq\norm{f_{0}}_{C(\mbT^{2})}+C\abs{\chem}\int_{0}^{t}\norm{P_{t-s}(f_{s}\nabla\Phi_{f_{s}})}_{\mcB_{\infty,1}^{1}}\dd s,
		\end{equation*}
		where $C>0$ is constant that may change from line to line. Let $\vartheta>0$, we use the product estimate~\eqref{eq:product_estimate_VI} to deduce
		\begin{equation*}
			\norm{f_{s}\nabla\Phi_{f_{s}}}_{\mcB_{\infty,1}^{-\vartheta}}\lesssim\norm{f_{s}}_{C(\mbT^{2})}\norm{f_{s}}_{L^{2}},
		\end{equation*}
		which we can combine with Schauder's estimate~\cite[Lem.~A.5]{martini_mayorcas_25}, i.e.\
		\begin{equation*}
			\norm{P_{t-s}(f_{s}\nabla\Phi_{f_{s}})}_{\mcB_{\infty,1}^{1}}\lesssim(1\vee\abs{t-s}^{-\frac{1+\vartheta}{2}})\norm{f_{s}\nabla\Phi_{f_{s}}}_{\mcB_{\infty,1}^{-\vartheta}},
		\end{equation*}
		to control for every $\bar{T}\in(0,1\wedge T]$,
		\begin{equation}\label{eq:frac_Gronwall_0}
			\norm{f_{t}}_{C(\mbT^{2})}\leq\norm{f_{0}}_{C(\mbT^{2})}+C\abs{\chem}\norm{f}_{C_{\bar{T}}L^{2}}\int_{0}^{t}\abs{t-s}^{-\frac{1+\vartheta}{2}}\norm{f_{s}}_{C(\mbT^{2})}\dd s.
		\end{equation}
		The bound~\eqref{eq:frac_Gronwall_0} alone is not strong enough to exclude the possibility that $f\notin C_{\bar{T}}C(\mbT^{2})$ (indeed in that case~\eqref{eq:frac_Gronwall_0} would simply read $\infty\leq\infty$.) However, it allows us to approximate and pass to the limit.
		
		To deduce $f\in C_{\bar{T}}C(\mbT^{2})$, we use the completeness of $C_{\bar{T}}C(\mbT^{2})$ under the supremum-norm. By~\cite[Lem.~30.2]{vanzuijlen_22} we can consider each $f_{0}\in C(\mbT^{2})$ as an element of $C_{\per}(\mbR^{d})$ (modulo an isometric isomorphism.) It then follows by~\cite[App.~C.4,~Thm.~6]{evans_98_PDE} that there exist smooth, periodic $(f_{0}^{n})_{n\in\mbN}$ such that $f_{0}^{n}\to f_{0}$ in $C(\mbT^{2})$. We denote by $(f^{n})_{n\in\mbN}$ the sequence of weak solutions to~\eqref{eq:Determ_KS} with initial data $(f_{0}^{n})_{n\in\mbN}$. Using that $(\norm{f_{0}^{n}}_{L^{2}})_{n\in\mbN}$ is bounded, we can use~\cite[Rem.~1.2]{liu_roeckner_13_local_global} to conclude that there exists some $\bar{T}\in(0,1\wedge T]$ such that each $f^{n}$ exists on $[0,\bar{T}]$. Further, by~\cite[Lem.~2.4]{liu_roeckner_13_local_global}, we can control $\norm{f^{n}}_{L^{2}_{\bar{T}}\mcH^{1}}$ uniformly in $n\in\mbN$; hence,~\cite[Thm.~1.2]{liu_roeckner_13_local_global} implies $f^{n}\to f$ in $C_{\bar{T}}L^{2}(\mbT^{2})$. Using that $f_{0}^{n}\in\mcH^{2}(\mbT^{2})$, we may deduce by Theorem~\ref{thm:Determ_KS_maximal} and the Sobolev embedding~\eqref{eq:Sobolev_embedding_II} that $f^{n}\in C_{\bar{T}}\mcH^{2}(\mbT^{2})\embed C_{\bar{T}}C(\mbT^{2})$. Hence to show $f\in C_{\bar{T}}C(\mbT^{2})$, it suffices to show that $(f^{n})_{n\in\mbN}$ is a Cauchy sequence in $C_{\bar{T}}C(\mbT^{2})$.
		
		Assume $\frac{1+\vartheta}{2}<1$, i.e.\ $\vartheta<1$, and denote $\beta\defeq\frac{2}{1-\vartheta}\in(2,\infty)$. We can now use $f^{n}\in C_{\bar{T}}C(\mbT^{2})$, \eqref{eq:frac_Gronwall_0} (applied to $f^{n}$) and the fractional Gronwall estimate~\cite[Rem.~3.7]{webb_19} to deduce uniformly for every $n\in\mbN$ and $t\in[0,\bar{T}]$,
		\begin{equation*}
			\norm{f_{t}^{n}}_{C(\mbT^{2})}\leq\beta\norm{f_{0}^{n}}_{C(\mbT^{2})}\exp\Bigl(C_{\beta}\abs{\chem}^{\beta}\norm{f^{n}}_{C_{\bar{T}}L^{2}}^{\beta}t\Bigr),
		\end{equation*}
		where $C_{\beta}>0$ is a constant that only depends on $\beta$. Taking the supremum over $[0,\bar{T}]$, we arrive at
		\begin{equation}\label{eq:frac_Gronwall_I_approx}
			\norm{f^{n}}_{C_{\bar{T}}C(\mbT^{2})}\leq\beta\norm{f_{0}^{n}}_{C(\mbT^{2})}\exp\Bigl(C_{\beta}\abs{\chem}^{\beta}\norm{f^{n}}_{C_{\bar{T}}L^{2}}^{\beta}\bar{T}\Bigr).
		\end{equation}
		Using that $\norm{f^{n}}_{C_{\bar{T}}L^{2}}$ is bounded uniformly in $n\in\mbN$, we can apply~\eqref{eq:frac_Gronwall_I_approx} to control $\norm{f^{n}}_{C_{\bar{T}}C(\mbT^{2})}$ uniformly in $n\in\mbN$.
		
		Using the bilinearity of~\eqref{eq:Determ_KS_mild} and following the derivation of~\eqref{eq:frac_Gronwall_0}, we obtain
		\begin{equation}\label{eq:determ_KS_continuity_bilinear}
			\begin{split}
				\norm{f_{t}^{n}-f_{t}^{m}}_{C(\mbT^{2})}\leq\norm{f_{0}^{n}-f_{0}^{m}}_{C(\mbT^{2})}&+C\abs{\chem}\norm{f^{n}-f^{m}}_{C_{\bar{T}}L^{2}}\int_{0}^{t}\abs{t-s}^{-\frac{1+\vartheta}{2}}\norm{f_{s}^{m}}_{C(\mbT^{2})}\dd s\\
				&+C\abs{\chem}\norm{f^{n}}_{C_{\bar{T}}L^{2}}\int_{0}^{t}\abs{t-s}^{-\frac{1+\vartheta}{2}}\norm{f_{s}^{n}-f_{s}^{m}}_{C(\mbT^{2})}\dd s.
			\end{split}
		\end{equation}
		We can then apply the fractional Gronwall estimate~\cite[Rem.~3.7]{webb_19} and the convergences $f_{0}^{n}\to f\in C(\mbT^{2})$ and $f^{n}\to f\in C_{\bar{T}}L^{2}(\mbT^{2})$ to conclude that $(f^{n})_{n\in\mbN}$ is a Cauchy sequence in $C_{\bar{T}}C(\mbT^{2})$. Using the completeness of $C_{\bar{T}}C(\mbT^{2})$, it follows that $f\in C_{\bar{T}}C(\mbT^{2})$. In particular, we can now use~\eqref{eq:frac_Gronwall_0} and the fractional Gronwall estimate~\cite[Rem.~3.7]{webb_19} to deduce
		\begin{equation}\label{eq:frac_Gronwall_I}
			\norm{f}_{C_{\bar{T}}C(\mbT^{2})}\leq\beta\norm{f_{0}}_{C(\mbT^{2})}\exp\Bigl(C_{\beta}\abs{\chem}^{\beta}\norm{f}_{C_{\bar{T}}L^{2}}^{\beta}\bar{T}\Bigr).
		\end{equation}
		This yields the claim.
	\end{details}
\end{proof}
\begin{details}
	By iterating the local estimate~\eqref{eq:frac_Gronwall_I}, we can control the supremum norm of $f$ on $[0,T]\times\mbT^{2}$ for every $T<\Trdet$.
	\begin{lemma}
		Let $f_{0}\in C(\mbT^{2})$ be such that $f_{0}\geq0$ and $\mean{f_{0}}=1$, $f$ be the weak solution to~\eqref{eq:Determ_KS} with initial data $f_{0}$ and chemotactic sensitivity $\chem\in\mbR$, and $\Trdet$ be its maximal time of existence $\Trdet$ (see Theorem~\ref{thm:Determ_KS_Well_Posedness}). Let $T<\Trdet$ then for each $\beta\in(2,\infty)$ there exists a constant $C_{\beta}>0$ such that
		\begin{equation}\label{eq:Determ_KS_uniform_bound}
			\norm{f}_{C_{T}C(\mbT^{2})}\leq\beta^{T+1}\norm{f_{0}}_{C(\mbT^{2})}\exp\Bigl(C_{\beta}\abs{\chem}^{\beta}\norm{f}_{C_{T}L^{2}}^{\beta}T\Bigr).
		\end{equation}
	\end{lemma}
	\begin{proof}
		To deduce~\eqref{eq:Determ_KS_uniform_bound}, we iterate the local-in-time estimate~\eqref{eq:frac_Gronwall_I}. Let $t_{0}<T$, we can represent the dynamics on $[t_{0},T]$ as
		\begin{equation*}
			f_{t_{0}+t}=P_{t}f_{t_{0}}-\chem\int_{0}^{t}\vdiv P_{t-s}(f_{t_{0}+s}\nabla\Phi_{f_{t_{0}+s}})\dd s,\quad\text{where}\quad t\in[0,T-t_{0}].
		\end{equation*}
		Hence, we can estimate by~\eqref{eq:frac_Gronwall_I} for every $\bar{T}\in(0,1\wedge(T-t_{0})]$,
		\begin{equation}\label{eq:frac_Gronwall_II}
			\norm{f}_{C_{[t_{0},t_{0}+\bar{T}]}C(\mbT^{2})}\leq\beta\norm{f_{t_{0}}}_{C(\mbT^{2})}\exp\Bigl(C_{\beta}\abs{\chem}^{\beta}\norm{f}_{C_{[t_{0},t_{0}+\bar{T}]}L^{2}}^{\beta}\bar{T}\Bigr).
		\end{equation}
		Let $\bar{T}=T/\ceil{T}\in(0,1\wedge T]$. We claim that for each $k\in\{1,\ldots,\ceil{T}\}$,
		\begin{equation}\label{eq:frac_Gronwall_III}
			\norm{f}_{C_{[0,k\bar{T}]}C(\mbT^{2})}\leq\beta^{k}\norm{f_{0}}_{C(\mbT^{2})}\exp\Bigl(C_{\beta}\abs{\chem}^{\beta}\norm{f}_{C_{[0,k\bar{T}]}L^{2}}^{\beta}k\bar{T}\Bigr).
		\end{equation}
		We proceed inductively. The claim~\eqref{eq:frac_Gronwall_III} with $k=1$ holds true by~\eqref{eq:frac_Gronwall_I}. For the induction step, assume~\eqref{eq:frac_Gronwall_III} holds with $k\in\{1,\ldots,\ceil{T}-1\}$. We bound by~\eqref{eq:frac_Gronwall_II} and~\eqref{eq:frac_Gronwall_III},
		\begin{equation*}
			\begin{split}
				\norm{f}_{C_{[k\bar{T},(k+1)\bar{T}]}C(\mbT^{2})}&\leq\beta\norm{f_{k\bar{T}}}_{C(\mbT^{2})}\exp\Bigl(C_{\beta}\abs{\chem}^{\beta}\norm{f}_{C_{[k\bar{T},(k+1)\bar{T}]}L^{2}}^{\beta}\bar{T}\Bigr)\\
				&\leq\beta^{k+1}\norm{f_{0}}_{C(\mbT^{2})}\exp\Bigl(C_{\beta}\abs{\chem}^{\beta}\norm{f}_{C_{[0,k\bar{T}]}L^{2}}^{\beta}k\bar{T}\Bigr)\exp\Bigl(C_{\beta}\abs{\chem}^{\beta}\norm{f}_{C_{[k\bar{T},(k+1)\bar{T}]}L^{2}}^{\beta}\bar{T}\Bigr)\\
				&\leq\beta^{k+1}\norm{f_{0}}_{C(\mbT^{2})}\exp\Bigl(C_{\beta}\abs{\chem}^{\beta}\norm{f}_{C_{[0,(k+1)\bar{T}]}L^{2}}^{\beta}(k+1)\bar{T}\Bigr).
			\end{split}
		\end{equation*}
		It then follows by the above and~\eqref{eq:frac_Gronwall_III},
		\begin{equation*}
			\begin{split}
				\norm{f}_{C_{[0,(k+1)\bar{T}]}C(\mbT^{2})}&=\max\{\norm{f}_{C_{[0,k\bar{T}]}C(\mbT^{2})},\norm{f}_{C_{[k\bar{T},(k+1)\bar{T}]}C(\mbT^{2})}\}\\
				&\leq\beta^{k+1}\norm{f_{0}}_{C(\mbT^{2})}\exp\Bigl(C_{\beta}\abs{\chem}^{\beta}\norm{f}_{C_{[0,(k+1)\bar{T}]}L^{2}}^{\beta}(k+1)\bar{T}\Bigr),
			\end{split}
		\end{equation*}
		which yields the inductive step.
		
		We can now apply~\eqref{eq:frac_Gronwall_III} with $k=\ceil{T}$ to deduce the bound
		\begin{equation*}
			\norm{f}_{C_{T}C(\mbT^{2})}\leq\beta^{T+1}\norm{f_{0}}_{C(\mbT^{2})}\exp\Bigl(C_{\beta}\abs{\chem}^{\beta}\norm{f}_{C_{T}L^{2}}^{\beta}T\Bigr),
		\end{equation*}
		which yields the claim.
	\end{proof}
\end{details}
\subsection{Boundedness Away From Zero}\label{subsec:bounded_away_from_0}
In this subsection we show that the solution to~\eqref{eq:Determ_KS} constructed in Theorem~\ref{thm:Determ_KS_Well_Posedness} is bounded away from zero uniformly in $[0,T]\times\mbT^{2}$ for every $T<\Trdet$, if its initial data is assumed to be continuous and positive.
\begin{lemma}\label{lem:Determ_KS_positivity}
	Let $f_{0}\in C(\mbT^{2})$ be such that $f_{0}>0$ and $\mean{f_{0}}=1$, $f$ be the weak solution to~\eqref{eq:Determ_KS} with initial data $f_{0}$ and chemotactic sensitivity $\chem\in\mbR$, and $\Trdet$ be its maximal time of existence (see Lemma~\ref{lem:Determ_KS_continuity}). Then for all $T<\Trdet$ there exists a $c>0$ such that $f(t,x)\geq c$ for all $(t,x)\in[0,T]\times\mbT^{2}$.
\end{lemma}
\begin{proof}
	It follows by Theorem~\ref{thm:Determ_KS_Well_Posedness} and Lemma~\ref{lem:Determ_KS_continuity} that $f\in C^{2}_{1}((0,T]\times\mbT^{2})\cap C([0,T]\times\mbT^{2})$ is a classical solution to~\eqref{eq:Determ_KS}. Denote $b\defeq-\chem\nabla\Phi_{f}$, it then follows by elliptic regularity~\cite[Lem.~A.8]{martini_mayorcas_25} that $b\in C([0,T]\times\mbT^{2};\mbR^{2})$ and $\vdiv b=\chem(f-\mean{f})\in C([0,T]\times\mbT^{2};\mbR)$. Hence we can apply Proposition~\ref{prop:diff_adv_positivity} to deduce that there exist some $c>0$ such that $f(t,x)\geq c$ for all $(t,x)\in[0,T]\times\mbT^{2}$.
	\begin{details}
		\paragraph{Proof that $b\in C([0,T]\times\mbT^{2})$.}
		Using the embedding $C(\mbT^{2})\embed\mcC^{0}(\mbT^{2})$ (cf.~\cite[Thm.~23.5]{vanzuijlen_22}), we obtain $f\in C_{T}C(\mbT^{2})\embed C_{T}\mcC^{0}(\mbT^{2})$, which implies $b=-\chem\nabla\Phi_{f}\in C_{T}\mcC^{1}(\mbT^{2})\embed C_{T}C(\mbT^{2})$ by~\cite[Lem.~A.8]{martini_mayorcas_25}. The continuity in space-time then follows by the triangle inequality.
	\end{details}
\end{proof}
\subsection{Regularity of the Square Root}\label{subsec:regularity_srdet}
In this subsection we establish the regularity of the square root of the solution to the deterministic Keller--Segel equation~\eqref{eq:Determ_KS} (Lemma~\ref{lem:regularity_srdet}).
\begin{lemma}\label{lem:regularity_srdet}
	Let $f_{0}\in C(\mbT^{2})$ be such that $f_{0}\geq0$ and $\mean{f_{0}}=1$, $f$ be the weak solution to~\eqref{eq:Determ_KS} with initial data $f_{0}$ and chemotactic sensitivity $\chem\in\mbR$, and $\Trdet$ be its maximal time of existence (see Lemma~\ref{lem:Determ_KS_continuity}). Then for all $T<\Trdet$ it holds that $\sqrt{f}\in C([0,T]\times\mbT^{2})$. Assume further $f_{0}>0$ and $f_{0}\in\mcH^{\gamma}(\mbT^{2})\cap C(\mbT^{2})$ for some  $\gamma\geq0$, then it holds that $\sqrt{f}\in C_{\eta;T}\mcH^{\gamma+2\eta}(\mbT^{2})\cap L_{T}^{2}\mcH^{\gamma+1}(\mbT^{2})$ for every $\eta\in[0,1/2)$.
\end{lemma}
\begin{proof}
	Given $f_{0}\in C(\mbT^{2})$ such that $f_{0}\geq0$ and $\mean{f_{0}}=1$, we can apply Lemma~\ref{lem:Determ_KS_continuity} to construct $f\in C([0,T]\times\mbT^{2})$ for all $T<\Trdet$. Using the continuity of the square root in the non-negative reals, it follows that $\sqrt{f}\in C([0,T]\times\mbT^{2})$.
	
	If further $f_{0}\in\mcH^{\gamma}(\mbT^{2})\cap C(\mbT^{2})$ for some $\gamma\geq0$, we can apply Theorem~\ref{thm:Determ_KS_Well_Posedness} and Lemma~\ref{lem:Determ_KS_continuity} to construct $f\in C_{T}\mcH^{\gamma}(\mbT^{2})\cap L_{T}^{2}\mcH^{\gamma+1}(\mbT^{2})\cap C([0,T]\times\mbT^{2})$. Recall from~\eqref{eq:Determ_KS_mild} the mild formulation 
	\begin{equation*}
		f=Pf_{0}-\chem\vdiv\mcI[f\nabla\Phi_{f}].
	\end{equation*}
	We can bound the contribution from the initial data by Schauder's estimate~\cite[Lem.~A.5]{martini_mayorcas_25},
	\begin{equation*}
		\norm{Pf_{0}}_{C_{\eta;T}\mcH^{\gamma+2\eta}}\lesssim\norm{f_{0}}_{\mcH^{\gamma}}
	\end{equation*}
	and the contribution from the nonlinearity by Lemma~\ref{lem:Schauder_Sobolev} combined with a product estimate,
	\begin{equation*}
		\norm{\vdiv\mcI[f\nabla\Phi_{f}]}_{C_{T}\mcH^{\gamma+2\eta}}\lesssim\norm{\mcI[f\nabla\Phi_{f}]}_{C_{T}\mcH^{\gamma+1+2\eta}}\lesssim_{T}\norm{f\nabla\Phi_{f}}_{C_{T}\mcH^{\gamma}}\lesssim\norm{f}_{C_{T}\mcH^{\gamma}}^{2},\quad\text{if}~\gamma>0
	\end{equation*}
	and
	\begin{equation*}
		\norm{\vdiv\mcI[f\nabla\Phi_{f}]}_{C_{T}\mcH^{2\eta}}\lesssim\norm{\mcI[f\nabla\Phi_{f}]}_{C_{T}\mcH^{1+2\eta}}\lesssim_{T}\norm{f\nabla\Phi_{f}}_{C_{T}\mcH^{-(1-2\eta)/2}}\lesssim\norm{f}_{C_{T}L^{2}}^{2},\quad\text{if}~\gamma=0.
	\end{equation*}
	Hence, we obtain $f\in C_{\eta;T}\mcH^{\gamma+2\eta}(\mbT^{2})\cap L_{T}^{2}\mcH^{\gamma+1}(\mbT^{2})\cap C([0,T]\times\mbT^{2})$ for all $T<\Trdet$.
	If in addition $f_{0}>0$, it also follows by Lemma~\ref{lem:Determ_KS_positivity} that there exists a constant $c>0$ such that $f\geq c$ on $[0,T]\times\mbT^{2}$. To prove $\sqrt{f}\in C_{\eta;T}\mcH^{\gamma+2\eta}(\mbT^{2})$, it then suffices to apply~\cite[Cor.~2.91]{bahouri_chemin_danchin_11} combined with a smooth cut-off. To prove $\sqrt{f}\in L_{T}^{2}\mcH^{\gamma+1}(\mbT^{2})$, we apply Lemma~\ref{lem:chain_rule_fractional} with $F=\sqrt{\place}$ to deduce the existence of a constant $C\defeq C(\gamma+1, F,\norm{f}_{C_{T}L^{\infty}}, c^{-1})>0$ such that
	\begin{equation*}
		\bigl\lVert\sqrt{f}\bigr\rVert_{L_{T}^{2}\mcH^{\gamma+1}}^{2}=\int_{0}^{T}\bigl\lVert\sqrt{f_{t}}\bigr\rVert_{\mcH^{\gamma+1}}^{2}\dd t\leq\int_{0}^{T}C^{2}\norm{f_{t}}_{\mcH^{\gamma+1}}^{2}\dd t\leq C^{2}\norm{f}_{L_{T}^{2}\mcH^{\gamma+1}}^{2},
	\end{equation*}
	which yields the claim.
\end{proof}
\section{Regular Theory for Small Noise Intensities}\label{sec:regular_theory}
In this section we consider solutions $\rho^{(\eps)}_{\delta(\eps)}$ to the additive-noise approximation~\eqref{eq:rKS_intro_mild}. As discussed in Subsection~\ref{subsec:scaling_discussion}, the choice of state space affects the relative scaling between $\eps$ and $\delta(\eps)$, which becomes more restrictive as we consider increasingly regular spaces. Here we consider the regular setting, i.e.\ the Bessel potential space $\mcH^{\gamma}(\mbT^{2})$ of regularity $\gamma\in[-1,0]$ and the relative scaling $\eps^{\sfrac{1}{2}}\delta(\eps)^{-\gamma-2}\to0$ as $\eps,\delta(\eps)\to0$; for the rough setting see Section~\ref{sec:rough_theory}. 

We can deduce increasingly stronger results by decreasing the range of $\gamma$. The most general range is $\gamma\in[-1,0]$, for which we establish the well-posedness of the stochastic heat equation $\ti^{\delta}$ (Proposition~\ref{prop:lolli_regular_srdet}) and a law of large numbers for $(\eps^{\sfrac{1}{2}}\ti^{\delta(\eps)})_{\eps>0}$  (Lemma~\ref{lem:LLN_regular_lolli}). The lower bound $\gamma\geq-1$ is used to make sense of the product $\srdet\boldsymbol{\xi}^{\delta}$ that appears in $\ti^{\delta}$, see the proof of Proposition~\ref{prop:lolli_regular_srdet} for details.
The smaller range $\gamma\in(-1/2,0]$ allows us to establish the well-posedness of $\rho^{(\eps)}_{\delta(\eps)}$ in $\mcH^{\gamma}(\mbT^{2})$ (Lemma~\ref{lem:rKS_well_posedness_regular}) and a law of large numbers (Theorem~\ref{thm:LLN_regular}).
To establish large deviation principles, we require some compactness, which we achieve by assuming $\gamma<1$. As such, we obtain a large deviation principle for $(\eps^{\sfrac{1}{2}}\ti^{\delta(\eps)})_{\eps>0}$ in $\mcH^{\gamma}(\mbT^{2})$ in the range of regularities $\gamma\in[-1,0)$ (Theorem~\ref{thm:LDP_regular_lolli}) and one for $(\rho^{(\eps)}_{\delta(\eps)})_{\eps>0}$ in the range $\gamma\in(-1/2,0)$ (Theorem~\ref{thm:LDP_regular}).

Throughout this section we consider initial data $\rho_{0}\in C(\mbT^{2})$ such that $\rho_{0}>0$ and $\mean{\rho_{0}}=1$. In particular, this ensures that $\sqrt{\rdet}\in C([0,T]\times\mbT^{2})$ (see Lemma~\ref{lem:regularity_srdet}) and that $\ti^{\delta}\in C_{T}L^{2}(\mbT^{2})\cap L_{T}^{2}\mcH^{1}(\mbT^{2})$ (see Proposition~\ref{prop:lolli_regular_srdet}).
\subsection{Well-Posedness}\label{subsec:well_posedness_regular}
In this subsection we construct solutions to~\eqref{eq:rKS_intro_mild} with initial data in $C(\mbT^{2})$ taking values in the Bessel potential space $\mcH^{\gamma}(\mbT^{2})$ of regularity $\gamma\in(-1/2,0]$ (Lemma~\ref{lem:rKS_well_posedness_regular}).

Recall that $(\mcH^{\gamma}(\mbT^{2}))^{\sol}_{T}$ denotes the space of continuous functions that may blow up in finite time (see Subsection~\ref{subsec:notation}).
\begin{lemma}\label{lem:rKS_well_posedness_regular}
	Let $\rho_{0}\in C(\mbT^{2})$ be such that $\rho_{0}>0$ and $\mean{\rho_{0}}=1$, $\rdet$ be the weak solution to~\eqref{eq:Determ_KS} with initial data $\rho_{0}$ and chemotactic sensitivity $\chem\in\mbR$, and $\Trdet$ be its maximal time of existence (see Lemma~\ref{lem:Determ_KS_continuity}). Then for every $T<\Trdet$, $\eps>0$, $\delta>0$ and $\gamma\in(-1/2,0]$, there exists a unique solution $\rho^{(\eps)}_{\delta}$ to~\eqref{eq:rKS_intro_mild} such that, almost surely, $\rho^{(\eps)}_{\delta}\in(\mcH^{\gamma}(\mbT^{2}))^{\sol}_{T}$ and $\rho^{(\eps)}_{\delta}\in C_{S}\mcH^{\gamma}(\mbT^{2})\cap L_{S}^{2}\mcH^{\gamma+1}(\mbT^{2})$ for every $S<T^{\cem}_{\mcH^{\gamma}}[\rho^{(\eps)}_{\delta}]$ (where $T^{\cem}_{\mcH^{\gamma}}[\rho^{(\eps)}_{\delta}]\in(0,T]$ denotes the maximal time of existence of $\rho^{(\eps)}_{\delta}$.) In particular, the solution map is almost surely locally Lipschitz continuous in the noise $\eps^{\sfrac{1}{2}}\ti^{\delta(\eps)}\in C_{T}\mcH^{\gamma}(\mbT^{2})\cap L_{T}^{2}\mcH^{\gamma+1}(\mbT^{2})$.
\end{lemma}
\begin{proof}
	An application of Proposition~\ref{prop:lolli_regular_srdet} yields $\ti^{\delta}\in C_{T}L^{2}(\mbT^{2})\cap L_{T}^{2}\mcH^{1}(\mbT^{2})$ almost surely for each $\delta>0$. Given a realisation of the noise, $\eps>0$, $\gamma\in(-1/2,0]$ and $\bar{T}$ sufficiently small, we can construct a local solution $\rho^{(\eps)}_{\delta}$ to~\eqref{eq:rKS_intro_mild} as a fixed point in $C_{\bar{T}}\mcH^{\gamma}(\mbT^{2})\cap L_{\bar{T}}^{2}\mcH^{\gamma+1}(\mbT^{2})$. Re-starting the equation, we may then iterate the construction to obtain a solution in $(\mcH^{\gamma}(\mbT^{2}))^{\sol}_{T}$. Since the noise is additive in~\eqref{eq:rKS_intro_mild}, it follows that the solution map is continuous in $\eps^{\sfrac{1}{2}}\ti^{\delta(\eps)}\in C_{T}\mcH^{\gamma}(\mbT^{2})\cap L_{T}^{2}\mcH^{\gamma+1}(\mbT^{2})$.
	\begin{details}
		\paragraph{Detailed Construction.}
		Let $\gamma'\in(-1,\gamma)$, $\bar{T}\in(0,T]$ and define $\Psi$, acting on $u\in C_{\bar{T}}\mcH^{\gamma}(\mbT^{2})\cap L_{\bar{T}}^{2}\mcH^{\gamma'+1}(\mbT^{2})$, by
		\begin{equation*}
			\Psi(u)=P\rho_{0}-\chem\vdiv\mcI[u\nabla\Phi_{u}]-\eps^{\sfrac{1}{2}}\ti^{\delta(\eps)}.
		\end{equation*}
	We combine Lemma~\ref{lem:Schauder_Sobolev}, Claims~\ref{it:Schauder_Sobolev_I}~\&~\ref{it:Schauder_Sobolev_II}, with~\eqref{eq:product_estimate_Sobolev_V} (for $\gamma\in(-1/2,0)$) resp.~\eqref{eq:product_estimate_Sobolev_IV} (for $\gamma=0$) to bound for every $\vartheta\in(0,\gamma+1)$,
		\begin{equation}\label{eq:bound_a_priori_regular_I}
			\begin{split}
				\norm{\Psi(u)}_{C_{\bar{T}}\mcH^{\gamma}}&\lesssim\norm{\rho_{0}}_{\mcH^{\gamma}}+\abs{\chem}(\bar{T}+\bar{T}^{\frac{1+\gamma-\vartheta}{2}})\norm{u\nabla\Phi_{u}}_{C_{\bar{T}}\mcH^{2\gamma-\vartheta}}+\eps^{\sfrac{1}{2}}\norm{\ti^{\delta(\eps)}}_{C_{\bar{T}}\mcH^{\gamma}}\\
				&\lesssim\norm{\rho_{0}}_{\mcH^{\gamma}}+\abs{\chem}(\bar{T}+\bar{T}^{\frac{1+\gamma-\vartheta}{2}})\norm{u}_{C_{\bar{T}}\mcH^{\gamma}}^{2}+\eps^{\sfrac{1}{2}}\norm{\ti^{\delta(\eps)}}_{C_{\bar{T}}\mcH^{\gamma}}
			\end{split}
		\end{equation}
		and Lemma~\ref{lem:Schauder_Sobolev}, Claims~\ref{it:Schauder_Sobolev_I}~\&~\ref{it:Schauder_Sobolev_III}, with~\eqref{eq:product_estimate_Sobolev_VII} to bound
		\begin{equation}\label{eq:bound_a_priori_regular_II}
			\begin{split}
				\norm{\Psi(u)}_{L_{\bar{T}}^{2}\mcH^{\gamma'+1}}&\lesssim\norm{\rho_{0}}_{\mcH^{\gamma}}+\abs{\chem}(\bar{T}+\bar{T}^{\gamma-\gamma'})\norm{u\nabla\Phi_{u}}_{L_{\bar{T}}^{2}\mcH^{\gamma}}+\eps^{\sfrac{1}{2}}\norm{\ti^{\delta(\eps)}}_{L_{\bar{T}}^{2}\mcH^{\gamma+1}}\\
				&\lesssim\norm{\rho_{0}}_{\mcH^{\gamma}}+\abs{\chem}(\bar{T}+\bar{T}^{\gamma-\gamma'})\norm{u}_{C_{\bar{T}}\mcH^{\gamma}}\norm{u}_{L_{\bar{T}}^{2}\mcH^{\gamma'+1}}+\eps^{\sfrac{1}{2}}\norm{\ti^{\delta(\eps)}}_{L_{\bar{T}}^{2}\mcH^{\gamma+1}}.
			\end{split}
		\end{equation}
		Let $C$ be the implicit constant of~\eqref{eq:bound_a_priori_regular_I}~\&~\eqref{eq:bound_a_priori_regular_II} and choose
		\begin{equation*}
			R>C(\norm{\rho_{0}}_{\mcH^{\gamma}}+1+\eps^{\sfrac{1}{2}}\norm{\ti^{\delta(\eps)}}_{C_{T}\mcH^{\gamma}\cap L_{T}^{2}\mcH^{\gamma+1}}).
		\end{equation*}
		Define 
		\begin{equation*}
			\mfB_{R,\bar{T}}\defeq\{u\in C_{\bar{T}}\mcH^{\gamma}(\mbT^{2})\cap L_{\bar{T}}^{2}\mcH^{\gamma'+1}(\mbT^{2}):\norm{u}_{C_{\bar{T}}\mcH^{\gamma}\cap L_{\bar{T}}^{2}\mcH^{\gamma'+1}}<R\}
		\end{equation*}
		and let $\bar{T}=\bar{T}(R,\gamma,\vartheta,\gamma',\chem)\leq T$ be sufficiently small such that for all $u\in\mfB_{R,\bar{T}}$,
		\begin{equation*}
			\norm{\Psi(u)}_{C_{\bar{T}}\mcH^{\gamma}\cap L_{\bar{T}}^{2}\mcH^{\gamma'+1}}<R.
		\end{equation*}
		It follows that $\Psi$ is a self-mapping on $\mfB_{R,\bar{T}}$. To show that $\Psi$ is a contraction, let $u,v\in\mfB_{R,\bar{T}}$ and use bilinearity to bound
		\begin{equation}\label{eq:bound_a_priori_regular_bilinearity}
			\begin{split}
				\norm{\Psi(u)-\Psi(v)}_{C_{\bar{T}}\mcH^{\gamma}}&\lesssim \abs{\chem}(\bar{T}+\bar{T}^{\frac{1+\gamma-\vartheta}{2}})R\norm{u-v}_{C_{\bar{T}}\mcH^{\gamma}}\\
				\norm{\Psi(u)-\Psi(v)}_{L_{\bar{T}}^{2}\mcH^{\gamma'+1}}&\lesssim \abs{\chem}(\bar{T}+\bar{T}^{\gamma-\gamma'})R\norm{u-v}_{L_{\bar{T}}^{2}\mcH^{\gamma'+1}}+\abs{\chem}(\bar{T}+\bar{T}^{\gamma-\gamma'})R\norm{u-v}_{C_{\bar{T}}\mcH^{\gamma}}.
			\end{split}
		\end{equation}
		Choosing $\bar{T}$ sufficiently small depending on $R$, $\gamma$, $\vartheta$, $\gamma'$ and $\chem$, we obtain a contraction. Hence, by Banach's fixed-point theorem there exists a unique $\rho\in\mfB_{R,\bar{T}}$ that solves~\eqref{eq:rKS_intro_mild}. 
		
		Next we show that this fixed point is unique in $C_{\bar{T}}\mcH^{\gamma}(\mbT^{2})\cap L_{\bar{T}}^{2}\mcH^{\gamma'+1}(\mbT^{2})$. Let $\rho_{1}$ and $\rho_{2}$ be two putative solutions to~\eqref{eq:rKS_intro_mild}. It follows by~\eqref{eq:bound_a_priori_regular_bilinearity} that for $\bar{S}\leq\bar{T}$,
		\begin{equation*}
			\norm{\rho_{1}-\rho_{2}}_{C_{\bar{S}}\mcH^{\gamma}}\lesssim \abs{\chem}(\bar{S}+\bar{S}^{\frac{1+\gamma-\vartheta}{2}})(\norm{\rho_{1}}_{C_{\bar{T}}\mcH^{\gamma}}+\norm{\rho_{2}}_{C_{\bar{T}}\mcH^{\gamma}})\norm{\rho_{1}-\rho_{2}}_{C_{\bar{S}}\mcH^{\gamma}},
		\end{equation*}
		which implies for $\bar{S}$ sufficiently small that $\rho_{1}=\rho_{2}$ in $C_{\bar{S}}\mcH^{\gamma}(\mbT^{2})$. An iteration argument then yields $\rho_{1}=\rho_{2}$ in $C_{\bar{T}}\mcH^{\gamma}(\mbT^{2})$. Similarly,
		\begin{equation*}
			\begin{split}
				\norm{\rho_{1}-\rho_{2}}_{L_{\bar{S}}^{2}\mcH^{\gamma'+1}}&\lesssim\abs{\chem}(\bar{S}+\bar{S}^{\gamma-\gamma'})\norm{\rho_{1}-\rho_{2}}_{C_{\bar{T}}\mcH^{\gamma}}\norm{\rho_{1}}_{L_{\bar{T}}^{2}\mcH^{\gamma'+1}}+\abs{\chem}(\bar{S}+\bar{S}^{\gamma-\gamma'})\norm{\rho_{2}}_{C_{\bar{T}}\mcH^{\gamma}}\norm{\rho_{1}-\rho_{2}}_{L_{\bar{S}}^{2}\mcH^{\gamma'+1}}\\
				&=\abs{\chem}(\bar{S}+\bar{S}^{\gamma-\gamma'})\norm{\rho_{2}}_{C_{\bar{T}}\mcH^{\gamma}}\norm{\rho_{1}-\rho_{2}}_{L_{\bar{S}}^{2}\mcH^{\gamma'+1}},
			\end{split}
		\end{equation*}
		from which we obtain $\rho_{1}=\rho_{2}$ in $L_{\bar{T}}^{2}\mcH^{\gamma'+1}(\mbT^{2})$.
		
		Let $\rho$ be the unique solution in $C_{\bar{T}}\mcH^{\gamma}(\mbT^{2})\cap L_{\bar{T}}^{2}\mcH^{\gamma'+1}(\mbT^{2})$, to show that $\rho\in L_{\bar{T}}^{2}\mcH^{\gamma+1}(\mbT^{2})$, we use an estimate analogous to~\eqref{eq:bound_a_priori_regular_II}, 
		\begin{equation*}
			\norm{\rho}_{L_{\bar{T}}^{2}\mcH^{\gamma+1}}\lesssim\norm{\rho_{0}}_{\mcH^{\gamma}}+\abs{\chem}(\bar{T}+1)\norm{\rho}_{C_{\bar{T}}\mcH^{\gamma}}\norm{\rho}_{L_{\bar{T}}^{2}\mcH^{\gamma'+1}}+\eps^{\sfrac{1}{2}}\norm{\ti^{\delta(\eps)}}_{L_{\bar{T}}^{2}\mcH^{\gamma+1}},
		\end{equation*}
		which yields $\rho\in L_{\bar{T}}^{2}\mcH^{\gamma+1}(\mbT^{2})$.
		
		We can iterate the construction to obtain a solution in $(\mcH^{\gamma}(\mbT^{2}))^{\sol}_{T}$ such that for every $S<T^{\cem}[\rho]$ it holds that $\rho\in L_{S}^{2}\mcH^{\gamma+1}(\mbT^{2})$.
	\end{details}
\end{proof}
\subsection{Law of Large Numbers}\label{subsec:LLN_regular}
In this subsection we establish a (weak) law of large numbers for solutions $\rho^{(\eps)}_{\delta(\eps)}$ to~\eqref{eq:rKS_intro_mild} in $(\mcH^{\gamma}(\mbT^{2}))^{\sol}_{T}$ under the relative scaling $\delta(\eps)\to0$ and $\eps^{\sfrac{1}{2}}\delta(\eps)^{-\gamma-2}\to0$ as $\eps\to0$ for some $\gamma\in(-1/2,0]$ (Theorem~\ref{thm:LLN_regular}). It is convenient to first prove a law of large numbers for $(\eps^{\sfrac{1}{2}}\ti^{\delta(\eps)})_{\eps>0}$ (Lemma~\ref{lem:LLN_regular_lolli}), which we can then transfer to $(\rho^{(\eps)}_{\delta(\eps)})_{\eps>0}$ by the continuous mapping theorem (Theorem~\ref{thm:LLN_regular}).
\begin{lemma}\label{lem:LLN_regular_lolli}
	Let $\rho_{0}\in C(\mbT^{2})$ be such that $\rho_{0}>0$ and $\mean{\rho_{0}}=1$, $\rdet$ be the weak solution to~\eqref{eq:Determ_KS} with initial data $\rho_{0}$ and chemotactic sensitivity $\chem\in\mbR$, and $\Trdet$ be its maximal time of existence (see Lemma~\ref{lem:Determ_KS_continuity}). Assume $\delta(\eps)\to0$ and $\eps^{\sfrac{1}{2}}\delta(\eps)^{-\gamma-2}\to0$ as $\eps\to0$ for some $\gamma\in[-1,0]$. Then for all $T<\Trdet$, it follows that $\eps^{\sfrac{1}{2}}\ti^{\delta(\eps)}\to0$ in $C_{T}\mcH^{\gamma}(\mbT^{2})\cap L_{T}^{2}\mcH^{\gamma+1}(\mbT^{2})$ in probability as $\eps\to0$.
\end{lemma}
\begin{proof}
	An application of Proposition~\ref{prop:lolli_regular_srdet} combined with the relative scaling $\eps^{\sfrac{1}{2}}\delta(\eps)^{\gamma-2}\to0$ yields
	\begin{equation*}
		\lim_{\eps\to0}\eps^{\sfrac{1}{2}}\mbE\Bigl[\norm{\ti^{\delta}}_{C_{T}\mcH^{\gamma}\cap L_{T}^{2}\mcH^{\gamma+1}}^{2}\Bigr]^{1/2}\lesssim\norm{\srdet}_{C_{T}L^{2}\cap L_{T}^{2}\mcH^{1}}\lim_{\eps\to0}\eps^{\sfrac{1}{2}}(1+\delta(\eps)^{-\gamma-2})=0,
	\end{equation*}
	which implies convergence in probability by Markov's inequality.
\end{proof}
The following theorem demonstrates an analogue of the law of large numbers for $(\rho^{(\eps)}_{\delta(\eps)})_{\eps>0}$ in $(\mcH^{\gamma}(\mbT^{2}))^{\sol}_{T}$. Note that in contrast to Lemma~\ref{lem:LLN_regular_lolli} we have to restrict the range of $\gamma$ to $\gamma\in(-1/2,0]$ so that we have enough regularity to make sense of
the non-linearity in~\eqref{eq:rKS_intro_mild} (see also Lemma~\ref{lem:rKS_well_posedness_regular}).
\begin{theorem}\label{thm:LLN_regular}
	Let $\rho_{0}\in C(\mbT^{2})$ be such that $\rho_{0}>0$ and $\mean{\rho_{0}}=1$, $\rdet$ be the weak solution to~\eqref{eq:Determ_KS} with initial data $\rho_{0}$ and chemotactic sensitivity $\chem\in\mbR$, and $\Trdet$ be its maximal time of existence (see Lemma~\ref{lem:Determ_KS_continuity}). Assume $\delta(\eps)\to0$ and $\eps^{\sfrac{1}{2}}\delta(\eps)^{-\gamma-2}\to0$ as $\eps\to0$ for some $\gamma\in(-1/2,0]$. Then for all $T<\Trdet$, it follows that $\rho^{(\eps)}_{\delta(\eps)}\to\rdet$ in $(\mcH^{\gamma}(\mbT^{2}))^{\sol}_{T}$ in probability as $\eps\to0$.
\end{theorem}
\begin{proof}
	It follows by Lemma~\ref{lem:LLN_regular_lolli} and the relative scaling $\eps^{\sfrac{1}{2}}\delta(\eps)^{-\gamma-2}\to0$ that $\eps^{\sfrac{1}{2}}\ti^{\delta(\eps)}\to0$ in $C_{T}\mcH^{\gamma}(\mbT^{2})\cap L_{T}^{2}\mcH^{\gamma+1}(\mbT^{2})$ in probability as $\eps\to0$. Using the local Lipschitz continuity of $\rho^{(\eps)}_{\delta(\eps)}$ in $\eps^{\sfrac{1}{2}}\ti^{\delta(\eps)}$ (Lemma~\ref{lem:rKS_well_posedness_regular}), it follows by the continuous mapping theorem that $\rho^{(\eps)}_{\delta(\eps)}\to\rdet$ in $(\mcH^{\gamma}(\mbT^{2}))^{\sol}_{T}$ in probability as $\eps\to0$.
\end{proof}
\subsection{Large Deviation Principle}\label{subsec:LDP_regular}
In this subsection we we apply the weak-convergence approach to large deviations~\cite{budhiraja_dupuis_maroulas_08} to establish a large deviation principle for solutions $\rho^{(\eps)}_{\delta(\eps)}$ to~\eqref{eq:rKS_intro_mild} in $(\mcH^{\gamma}(\mbT^{2}))^{\sol}_{T}$ under the relative scaling $\delta(\eps)\to0$ and $\eps^{\sfrac{1}{2}}\delta(\eps)^{-\gamma-2}\to0$ as $\eps\to0$ for some $\gamma\in(-1/2,0)$ (Theorem~\ref{thm:LDP_regular}). As in Subsection~\ref{subsec:LLN_regular}, we first consider the noise $(\eps^{\sfrac{1}{2}}\ti^{\delta(\eps)})_{\eps>0}$ driving $(\rho^{(\eps)}_{\delta(\eps)})_{\eps>0}$ (cf.\ Lemma~\ref{lem:rKS_well_posedness_regular}) for which we establish a large deviation principle (Theorem~\ref{thm:LDP_regular_lolli}) that we can then transfer to $(\rho^{(\eps)}_{\delta(\eps)})_{\eps>0}$ via the contraction principle (Theorem~\ref{thm:LDP_regular}).
\begin{theorem}\label{thm:LDP_regular_lolli}
	Let $\rho_{0}\in C(\mbT^{2})$ be such that $\rho_{0}>0$ and $\mean{\rho_{0}}=1$, $\rdet$ be the weak solution to~\eqref{eq:Determ_KS} with initial data $\rho_{0}$ and chemotactic sensitivity $\chem\in\mbR$, and $\Trdet$ be its maximal time of existence (see Lemma~\ref{lem:Determ_KS_continuity}). Assume $\delta(\eps)\to0$ and $\eps^{\sfrac{1}{2}}\delta(\eps)^{-\gamma-2}\to0$ as $\eps\to0$ for some $\gamma\in[-1,0)$. Then for all $T<\Trdet$, it follows that the sequence $(\eps^{\sfrac{1}{2}}\ti^{\delta(\eps)})_{\eps>0}$ satisfies a large deviation principle in $C_{T}\mcH^{\gamma}(\mbT^{2})\cap L_{T}^{2}\mcH^{\gamma+1}(\mbT^{2})$ with speed $\eps$ and good rate function
	\begin{equation*}
		\rate_{\ti}(s)\defeq\inf\Bigl\{\frac{1}{2}\norm{h}^{2}_{L^{2}([0,T]\times\mbT^{2};\mbR^{2})}:\ti^{h}=s\Bigr\},\qquad\ti^{h}\defeq\vdiv\mcI[\srdet h].
	\end{equation*}
	For the existence and regularity of $\ti^{h}$, see also Lemmas~\ref{lem:lolli_h}~\&~\ref{lem:lolli_h_L2H1}.
\end{theorem}
\begin{proof}
	It follows from~\eqref{eq:lolli_real} that $\ti^{\delta}$ is a functional of finitely many Brownian motions with the number of Brownian drivers increasing to $\infty$ as $\delta\to0$. Hence to deduce a large deviation principle we apply~\cite[Thm.~6]{budhiraja_dupuis_maroulas_08}, which is based on an infinite-dimensional version of the Bou\'{e}--Dupuis formula.
	
	First let us introduce some notation. We define for every $M<\infty$ the closed ball
	\begin{equation*}
	 	S^{M}\defeq\{h\in L^{2}([0,T]\times\mbT^{2};\mbR^{2}):\norm{h}_{L^{2}([0,T]\times\mbT^{2};\mbR^{2})}\leq M\}
	\end{equation*}
	equipped with the weak topology, which is metrizable (resp.\ compact) by the separability (resp.\ reflexivity) of $L^{2}([0,T]\times\mbT^{2};\mbR^{2})$ (see~\cite[Thm.~3.29~\&~Thm.~3.18]{brezis_11}). We further denote
	\begin{equation*}
		L^{2}_{\pred}\defeq\Bigl\{\phi\from\Omega\times[0,T]\to L^{2}(\mbT^{2};\mbR^{2})~\text{predictable}:\norm{\phi}_{L^{2}([0,T]\times\mbT^{2};\mbR^{2})}<\infty~\text{a.s.}\Bigr\}
	\end{equation*}
	and
	\begin{equation*}
		L^{2,M}_{\pred}\defeq\{\phi\in L^{2}_{\pred}:\phi\in S^{M}~\text{a.s.}\}.
	\end{equation*}
	For every sequence $(u^{(\eps)})_{\eps>0}$ such that $u^{(\eps)}\in L^{2,M}_{\pred}$ for each $\eps>0$ and $u^{(\eps)}\rightharpoonup u$ in $S^{M}$ in law as $\eps\to0$ (i.e.\ as $S^{M}$-valued random variables), define $u^{(\eps)}_{\delta}\defeq u^{(\eps)}\ast\psi_{\delta}$, where $\psi_{\delta}$ is as in~\eqref{eq:def_mollifiers}.
	
	To apply~\cite[Thm.~6]{budhiraja_dupuis_maroulas_08} in our context, we need to show that for every $M>0$, $(u^{(\eps)})_{\eps>0}$ and $\gamma\in[-1,0)$:
	\begin{enumerate}
		\item\label{it:var_LDP_compact}
		 The set $\Gamma_{M}\defeq\{\ti^{h}:h\in S^{M}\}\subset C_{T}\mcH^{\gamma}(\mbT^{2})\cap L_{T}^{2}\mcH^{\gamma+1}(\mbT^{2})$ is compact in the strong topology induced by the norm $\norm{\place}_{C_{T}\mcH^{\gamma}\cap L_{T}^{2}\mcH^{\gamma+1}}$.
		\item\label{it:var_LDP_conv}
	 	It holds that $\eps^{\sfrac{1}{2}}\ti^{\delta(\eps)}+\ti^{u^{(\eps)}_{\delta(\eps)}}\to\ti^{u}$ in $C_{T}\mcH^{\gamma}(\mbT^{2})\cap L_{T}^{2}\mcH^{\gamma+1}(\mbT^{2})$ in law as $\eps\to0$.
	\end{enumerate}
	Let us first show~\ref{it:var_LDP_compact}. For the existence and regularity of $\ti^{h}=\vdiv\mcI[\srdet h]$, see Appendix~\ref{app:enh_driven_by_h}. The weak derivative of $\ti^{h}$ is given by $\partial_{t}\ti^{h}\defeq\Delta\ti^{h}+\vdiv(\srdet h)$, hence Lemma~\ref{lem:lolli_h_L2H1} implies for every $h\in S^{M}$ the bound
	\begin{equation}\label{eq:lolli_h_Aubin_Lions_bound}
		\max\{\norm{\ti^{h}}_{L_{T}^{2}\mcH^{1}},\norm{\partial_{t}\ti^{h}}_{L_{T}^{2}\mcH^{-1}}\}\lesssim\norm{\srdet}_{C([0,T]\times\mbT^{2})}\norm{h}_{L^{2}([0,T]\times\mbT^{2};\mbR^{2})}\leq\norm{\srdet}_{C([0,T]\times\mbT^{2})}M.
	\end{equation}
	For every $R>0$ define the set 
	\begin{equation*}
		A_{R}\defeq\Bigl\{u\in L^{2}_{T}\mcH^{1}(\mbT^{2})\cap W^{1,2}_{T}\mcH^{-1}(\mbT^{2}):\norm{u}_{L^{2}_{T}\mcH^{1}}+\norm{\partial_{t}u}_{L^{2}_{T}\mcH^{-1}}\leq R\Bigr\},
	\end{equation*}
	it then follows by~\eqref{eq:lolli_h_Aubin_Lions_bound} that $\Gamma_{M}\subset A_{R}$ for $R$ sufficiently large depending on $\norm{\srdet}_{C([0,T]\times\mbT^{2})}M$. An application of~\cite[Sec.~5.9,~Thm.~3]{evans_98_PDE} combined with the Aubin--Lions lemma yields that the set $A_{R}$ is totally bounded in $C_{T}\mcH^{\gamma}(\mbT^{2})\cap L_{T}^{2}\mcH^{\gamma+1}(\mbT^{2})$, where we used the compact embeddings $L^{2}(\mbT^{2})\embed\mcH^{\gamma}(\mbT^{2})$ and $\mcH^{1}(\mbT^{2})\embed\mcH^{\gamma+1}(\mbT^{2})$ (see Lemma~\ref{lem:Besov_compact}). Using that every subset of a totally bounded set is again totally bounded, it follows that $\Gamma_{M}$ is itself totally bounded in $C_{T}\mcH^{\gamma}(\mbT^{2})\cap L_{T}^{2}\mcH^{\gamma+1}(\mbT^{2})$. Hence to show that $\Gamma_{M}$ is compact it suffices to show that it is closed.
	\begin{details}
		A subset of  a totally bounded set is totally bounded. Further, a complete, totally bounded space is compact and closed subsets of complete metric spaces are complete.
	\end{details}
	
	Let $(\ti^{h_{n}})_{n\in\mbN}$ be a sequence such that $\ti^{h_{n}}\in\Gamma_{M}$ for every $n\in\mbN$ and $\ti^{h_{n}}\to x\in C_{T}\mcH^{\gamma}(\mbT^{2})\cap L_{T}^{2}\mcH^{\gamma+1}(\mbT^{2})$ as $n\to\infty$. Using the compactness of $S^{M}$, we can extract a subsequence of $(h_{n})_{n\in\mbN}$ (without re-labelling) such that $h_{n}\rightharpoonup h\in S^{M}$. The linear map $h\mapsto\ti^{h}\in C_{T}\mcH^{\gamma}(\mbT^{2})\cap L_{T}^{2}\mcH^{\gamma+1}(\mbT^{2})$ is weak-strong continuous by Lemma~\ref{lem:lolli_h_L2H1} and the compactness of the embedding $L_{T}^{2}\mcH^{1}(\mbT^{2})\cap W^{1,2}_{T}\mcH^{-1}(\mbT^{2})\subset C_{T}\mcH^{\gamma}(\mbT^{2})\cap L_{T}^{2}\mcH^{\gamma+1}(\mbT^{2})$.
	\begin{details}
		The linear map $h\mapsto\ti^{h}\in L_{T}^{2}\mcH^{1}(\mbT^{2})\cap W^{1,2}_{T}\mcH^{-1}(\mbT^{2})$ is strong-strong continuous by Lemma~\ref{lem:lolli_h_L2H1}, hence weak-weak continuous. Indeed, let $T\from X\to Y$ be linear and strong-strong continuous. For each $f\in Y'$, it follows that $f\circ T\from X\to\mbR$ is continuous and linear, hence $f\circ T\in X'$. Let $(x_{n})_{n\in\mbN}$ be a sequence such that $x_{n}\in X$ for every $n\in\mbN$ and $x_{n}\rightharpoonup x\in X$ as $n\to\infty$, then by the definition of weak convergence, $(f\circ T)(x_{n})\to (f\circ T)(x)$. This implies the weak-weak continuity of $T$.
		
		Furthermore, the embedding $L_{T}^{2}\mcH^{1}(\mbT^{2})\cap W^{1,2}_{T}\mcH^{-1}(\mbT^{2})\subset C_{T}\mcH^{\gamma}(\mbT^{2})\cap L_{T}^{2}\mcH^{\gamma+1}(\mbT^{2})$ is compact by~\cite[Sec.~5.9,~Thm.~3]{evans_98_PDE} and the Aubin--Lions lemma; and every compact operator is completely continuous.
	\end{details}
	Therefore $x=\ti^{h}\in\Gamma_{M}$, which shows that $\Gamma_{M}$ is closed hence compact in $C_{T}\mcH^{\gamma}(\mbT^{2})\cap L_{T}^{2}\mcH^{\gamma+1}(\mbT^{2})$, proving~\ref{it:var_LDP_compact}.
	
	Next we show~\ref{it:var_LDP_conv}. Let $M<\infty$ and $(u^{(\eps)})_{\eps>0}$ be a sequence such that $u^{(\eps)}\in L^{2,M}_{\pred}$ for every $\eps>0$ and $u^{(\eps)}\rightharpoonup u$ in $S^{M}$ in law as $\eps\to0$. Recall that $u^{(\eps)}_{\delta(\eps)}=u^{(\eps)}\ast\psi_{\delta(\eps)}$. An application of the Skorokhod representation theorem~\cite[Chpt.~3,~Thm.~1.8]{ethier_kurtz_86} yields a probability space $(\tilde{\Omega},\tilde{\mcF},\tilde{\mbP})$ that supports a sequence of random variables $(\tilde{u}^{(\eps)})_{\eps>0}$ and $\tilde{u}$ such that $\text{Law}(\tilde{u}^{(\eps)})=\text{Law}(u^{(\eps)})$ for every $\eps>0$, $\text{Law}(\tilde{u})=\text{Law}(u)$ and $\tilde{u}^{(\eps)}\rightharpoonup\tilde{u}$ in $S^{M}$ almost surely as $\eps\to0$. It then follows by classical continuity properties of the weak dual pairing~\cite[Prop.~3.5]{brezis_11} that for every $\test\in L^{2}([0,T]\times\mbT^{2};\mbR^{2})$,
	\begin{equation*}
		\inner{\test}{\tilde{u}^{(\eps)}_{\delta(\eps)}}_{L^{2}([0,T]\times\mbT^{2};\mbR^{2})}\to\inner{\test}{\tilde{u}}_{L^{2}([0,T]\times\mbT^{2};\mbR^{2})}\quad\tilde{\mbP}\text{-almost surely},
	\end{equation*}
	which can be used to show that $\tilde{u}^{(\eps)}_{\delta(\eps)}\rightharpoonup\tilde{u}$ in $S^{M}$ in law and, therefore, $u^{(\eps)}_{\delta(\eps)}\rightharpoonup u$ in $S^{M}$ in law as $\eps\to0$.
	\begin{details}
		\paragraph{Convergence in law of $u^{(\eps)}_{\delta(\eps)}$.}
		An application of the Skorokhod representation theorem~\cite[Chpt.~3,~Thm.~1.8]{ethier_kurtz_86} yields a probability space $(\tilde{\Omega},\tilde{\mcF},\tilde{\mbP})$ that supports a sequence of random variables $(\tilde{u}^{(\eps)})_{\eps>0}$ and $\tilde{u}$ such that $\text{Law}(\tilde{u}^{(\eps)})=\text{Law}(u^{(\eps)})$ for every $\eps>0$, $\text{Law}(\tilde{u})=\text{Law}(u)$ and $\tilde{u}^{(\eps)}\rightharpoonup\tilde{u}$ in $S^{M}$ almost surely as $\eps\to0$.
		
		Let $\test\in L^{2}([0,T]\times\mbT^{2};\mbR^{2})$, it follows that
		\begin{equation*}
			\inner{\test}{\tilde{u}^{(\eps)}_{\delta(\eps)}}_{L^{2}([0,T]\times\mbT^{2};\mbR^{2})}=\inner{\test\ast\psi_{\delta(\eps)}}{\tilde{u}^{(\eps)}}_{L^{2}([0,T]\times\mbT^{2};\mbR^{2})}\to\inner{\test}{\tilde{u}}_{L^{2}([0,T]\times\mbT^{2};\mbR^{2})}\quad\text{almost surely}
		\end{equation*}
		where in the first equality we applied Fubini's theorem and in the second equality~\cite[Prop.~3.5]{brezis_11}, using that $\test\ast\psi_{\delta(\eps)}\to\test$ strongly in $L^{2}([0,T]\times\mbT^{2};\mbR^{2})$ and $\tilde{u}^{(\eps)}\rightharpoonup\tilde{u}$ in $S^{M}$ almost surely.
		
		To show that $\text{Law}(u^{(\eps)}_{\delta(\eps)})\to\text{Law}(u)$ in law as $\eps\to0$, it suffices to find a convergence determining set of functions on $S^{M}$ (cf.~\cite[Chpt.~3,~Sec.~4]{ethier_kurtz_86}) such that for each element $f$,
		\begin{equation*}
			\tilde{\mbE}[f(\tilde{u}^{(\eps)}_{\delta(\eps)})]\to\tilde{\mbE}[f(\tilde{u})].
		\end{equation*}
		Let $x\in L^{2}([0,T]\times\mbT^{2};\mbR^{2})$, the collection of sets
		\begin{equation*}
			\begin{split}
				V(\test_{1},\ldots,\test_{k};\vartheta)&=\Bigl\{y\in L^{2}([0,T]\times\mbT^{2};\mbR^{2}):\abs{\inner{y-x}{\phi_{i}}_{L^{2}([0,T]\times\mbT^{2};\mbR^{2})}}<\vartheta~\text{for all}~i=1,\ldots,k\Bigr\},\\
				&\multiquad[12]\text{with}\quad k\in\mbN,\quad\phi_{1},\ldots,\phi_{k}\in L^{2}([0,T]\times\mbT^{2};\mbR^{2}),\quad\vartheta>0,
			\end{split}
		\end{equation*}
		forms a neighbourhood basis of $x$ in the weak topology of $L^{2}([0,T]\times\mbT^{2};\mbR^{2})$ (see~\cite[Prop.~3.4]{brezis_11}).
		
		Let $B(x,r)$ be the open ball of radius $r>0$ in $S^{M}$ and denote $B(x,r)^{c}=S^{M}\setminus B(x,r)$. Using that $B(x,r)$ is a neighbourhood of $x$ in $S^{M}$, there exists a neighbourhood $N$ of $x$ in $L^{2}([0,T]\times\mbT^{2};\mbR^{2})$ such that $B(x,r)=N\cap S^{M}$. It follows that there exist some $\test_{1},\ldots,\test_{k}\in L^{2}([0,T]\times\mbT^{2};\mbR^{2})$ and $\vartheta>0$ such that $V(\test_{1},\ldots,\test_{k};\vartheta)\subseteq N$, which implies $S^{M}\setminus(V(\test_{1},\ldots,\test_{k};\vartheta)\cap S^{M})\supseteq B(x,r)^{c}$. Hence for every $y\in B(x,r)^{c}$, it follows that $y\notin V(\test_{1},\ldots,\test_{k};\vartheta)$ and in particular that there exists some $\test_{i}$ such that $\abs{\inner{y-x}{\test_{i}}_{L^{2}([0,T]\times\mbT^{2};\mbR^{2})}}\geq\vartheta$.
		
		In particular for every $x\in L^{2}([0,T]\times\mbT^{2};\mbR^{2})$ and $r>0$ we can find some $\test_{1},\ldots,\test_{k}\in L^{2}([0,T]\times\mbT^{2};\mbR^{2})$ such that
		\begin{equation*}
			\inf_{y\in B(x,r)^{c}}\max_{1\leq i\leq k}\abs{\inner{x-y}{\test_{i}}_{L^{2}([0,T]\times\mbT^{2};\mbR^{2})}}>0,
		\end{equation*}
		which proves that the set
		\begin{equation*}
			L=\{\inner{\test}{\place}_{L^{2}([0,T]\times\mbT^{2};\mbR^{2})}:\test\in L^{2}([0,T]\times\mbT^{2};\mbR^{2})\}
		\end{equation*}
		is strongly separating points on $S^{M}$ in the sense of~\cite[Chpt.~3,~Sec.~4]{ethier_kurtz_86}. Using that $S^{M}$ is complete and separable, it follows by~\cite[Chpt.~3,~Thm.~4.5]{ethier_kurtz_86} that the polynomials with variables in $L$ are convergence determining.
		
		Therefore to prove that $u^{(\eps)}_{\delta(\eps)}\rightharpoonup u$ in $S^{M}$ in law, it suffices to show that for every $k\in\mbN$ and $\test\in L^{2}([0,T]\times\mbT^{2};\mbR^{2})$,
		\begin{equation*}
			\tilde{\mbE}[\abs{\inner{\test}{\tilde{u}^{(\eps)}_{\delta(\eps)}}_{L^{2}([0,T]\times\mbT^{2};\mbR^{2})}}^{k}]\to\tilde{\mbE}[\abs{\inner{\test}{\tilde{u}}_{L^{2}([0,T]\times\mbT^{2};\mbR^{2})}}^{k}].
		\end{equation*}
		We obtain the claim by using the almost sure convergence of $\inner{\test}{\tilde{u}^{(\eps)}_{\delta(\eps)}}_{L^{2}([0,T]\times\mbT^{2};\mbR^{2})}\to\inner{\test}{\tilde{u}}_{L^{2}([0,T]\times\mbT^{2};\mbR^{2})}$ and the dominated convergence theorem with dominant
		\begin{equation*}
			\abs{\inner{\test}{\tilde{u}^{(\eps)}_{\delta(\eps)}}_{L^{2}([0,T]\times\mbT^{2};\mbR^{2})}}\leq\norm{\test\ast\psi_{\delta(\eps)}}_{L^{2}([0,T]\times\mbT^{2};\mbR^{2})}\norm{\tilde{u}^{(\eps)}}_{L^{2}([0,T]\times\mbT^{2};\mbR^{2})}\lesssim M^{1/2},
		\end{equation*}
		where we applied Parseval's theorem to deduce uniformly in $\delta>0$ and $t\in[0,T]$,
		\begin{equation*}
			\norm{\test_{t}\ast\psi_{\delta}}_{L^{2}(\mbT^{2};\mbR^{2})}^{2}=\sum_{\om\in\mbZ^{2}}\varphi(\delta\om)^{2}\abs{\hat{\test_{t}}(\om)}^{2}\lesssim\norm{\varphi}_{L^{\infty}(\mbT^{2})}^{2}\norm{\test_{t}}_{L^{2}(\mbT^{2};\mbR^{2})}^{2}.
		\end{equation*}
	\end{details}
	This combined with the weak-strong continuity of $S^{M}\ni h\to\ti^{h}\in\Gamma_{M}$ and the continuous mapping theorem yields $\ti^{u^{(\eps)}_{\delta(\eps)}}\to\ti^{u}$ in $C_{T}\mcH^{\gamma}(\mbT^{2})\cap L_{T}^{2}\mcH^{\gamma+1}(\mbT^{2})$ in law.
	
	Proposition~\ref{prop:lolli_regular_srdet} and the relative scaling $\eps^{\sfrac{1}{2}}\delta(\eps)^{-\gamma-2}\to0$ yield $\eps^{\sfrac{1}{2}}\ti^{\delta(\eps)}\to0\in C_{T}\mcH^{\gamma}(\mbT^{2})\cap L_{T}^{2}\mcH^{\gamma+1}(\mbT^{2})$ in $L^{p}(\mbP)$ as $\eps\to0$ for every $p\in[1,\infty]$, which implies convergence in probability by Markov's inequality. Therefore we can deduce by Slutsky's theorem that $\eps^{\sfrac{1}{2}}\ti^{\delta(\eps)}+\ti^{u^{(\eps)}_{\delta(\eps)}}\to\ti^{u}$ in $C_{T}\mcH^{\gamma}(\mbT^{2})\cap L_{T}^{2}\mcH^{\gamma+1}(\mbT^{2})$ in law, which shows~\ref{it:var_LDP_conv}. 
	\begin{details}
		\paragraph{Proof of Slutsky's theorem.}
		Let $S$ be an arbitrary metric space. Assume $X_{n}\to X$ in $S$ in law and $Y_{n}\to c$ in probability as $n\to\infty$, where $c\in S$ is deterministic. 
		
		It follows by the convergence of $X_{n}$ to $X$ that $\mbE[f(X_{n},c)]\to\mbE[f(X,c)]$ for every bounded and uniformly continuous function $f\from S^{2}\to\mbR$. By the Portmanteau theorem~\cite[Chpt.~3,~Thm.~3.1]{ethier_kurtz_86}, this implies that $(X_{n},c)$ converges to $(X,c)$ in law.
		
		Furthermore, $\abs{(X_{n},Y_{n})-(X_{n},c)}=\abs{Y_{n}-c}$ so that $\abs{(X_{n},Y_{n})-(X_{n},c)}\to0$ in probability. We can now apply~\cite[Chpt.~3, Cor.~3.3]{ethier_kurtz_86} to deduce $(X_{n},Y_{n})\to(X,c)$ in law. The continuous mapping theorem~\cite[p.~108]{ethier_kurtz_86} then implies $X_{n}+Y_{n}\to X+c$ in law.
	\end{details}
	
	Having shown~\ref{it:var_LDP_compact} and~\ref{it:var_LDP_conv}, we can apply~\cite[Thm.~6]{budhiraja_dupuis_maroulas_08} to deduce a large deviation principle for $(\eps^{\sfrac{1}{2}}\ti^{\delta(\eps)})_{\eps>0}$ in $C_{T}\mcH^{\gamma}(\mbT^{2})\cap L_{T}^{2}\mcH^{\gamma+1}(\mbT^{2})$ with good rate function $\rate_{\ti}$.
\end{proof}
We can now apply the contraction principle and Lemma~\ref{lem:rKS_well_posedness_regular} to deduce a large deviation principle for solutions to~\eqref{eq:rKS_intro_mild} in $(\mcH^{\gamma}(\mbT^{2}))^{\sol}_{T}$ under the relative scaling $\delta(\eps)\to0$ and $\eps^{\sfrac{1}{2}}\delta(\eps)^{-\gamma-2}\to0$ as $\eps\to0$ for some $\gamma\in(-1/2,0)$
\begin{theorem}\label{thm:LDP_regular}
	Let $\rho_{0}\in C(\mbT^{2})$ be such that $\rho_{0}>0$ and $\mean{\rho_{0}}=1$, $\rdet$ be the weak solution to~\eqref{eq:Determ_KS} with initial data $\rho_{0}$ and chemotactic sensitivity $\chem\in\mbR$, and $\Trdet$ be its maximal time of existence (see Lemma~\ref{lem:Determ_KS_continuity}). Assume $\delta(\eps)\to0$ and $\eps^{\sfrac{1}{2}}\delta(\eps)^{-\gamma-2}\to0$ as $\eps\to0$ for some $\gamma\in(-1/2,0)$. Then for all $T<\Trdet$, it follows that the sequence $(\rho^{(\eps)}_{\delta(\eps)})_{\eps>0}$ satisfies a large deviation principle in $(\mcH^{\gamma}(\mbT^{2}))^{\sol}_{T}$ with speed $\eps$ and good rate function
	\begin{equation*}
		\rate(\rho)\defeq\inf\Bigl\{\frac{1}{2}\norm{h}_{L^{2}([0,T]\times\mbT^{2};\mbR^{2})}^{2}:\rho=P\rho_{0}-\chem\vdiv\mcI[\rho\nabla\Phi_{\rho}]-\nabla\cdot\mcI[\srdet h]\Bigr\}.
	\end{equation*}
\end{theorem}
\begin{proof}
	The solution map constructed in Lemma~\ref{lem:rKS_well_posedness_regular} is locally Lipschitz continuous in the noise $\eps^{\sfrac{1}{2}}\ti^{\delta(\eps)}\in C_{T}\mcH^{\gamma}(\mbT^{2})\cap L_{T}^{2}\mcH^{\gamma+1}(\mbT^{2})$, hence we can apply Lemma~\ref{thm:LDP_regular_lolli} and the contraction principle~\cite[Thm.~4.2.1]{dembo_zeitouni_10} to deduce that $(\rho^{(\eps)}_{\delta(\eps)})_{\eps>0}$ satisfies a large deviation principle in $(\mcH^{\gamma}(\mbT^{2}))^{\sol}_{T}$ with speed $\eps$ and good rate function $\rate$.
\end{proof}
\subsection{Controlling Negative Values}\label{subsec:negative_values}
The solution $\rho^{(\eps)}_{\delta(\eps)}$ to~\eqref{eq:rKS_intro_mild} is an additive-noise approximation to the Keller--Segel--Dean--Kawasaki dynamics~\eqref{eq:KSDK_intro_torus} which models the density fluctuations of the chemotactic particle system~\eqref{eq:IPS_KS_intro_torus}. However, while densities are assumed to be non-negative, $\rho^{(\eps)}_{\delta(\eps)}$ can produce negative values due to the additive nature of the noise. The goal of this subsection is to control the probability of the event that $\rho^{(\eps)}_{\delta(\eps)}$ produces negative values.

Let $\Trdet$ be the maximal time of existence for $\rdet$ (see Lemma~\ref{lem:Determ_KS_continuity}), $T<\Trdet$, $L>0$ and denote by
\begin{equation}\label{eq:exceedance_time}
	S_{L}\defeq T\wedge L\wedge\inf\Bigl\{t\in[0,T]:\norm{\rho^{(\eps)}_{\delta(\eps)}(t)}_{L^{2}}>L\Bigr\}
\end{equation}
the first time that the $L^{2}(\mbT^{2})$-norm of $\rho^{(\eps)}_{\delta(\eps)}$ exceeds the level $L$. Denote the negative part of $u\in L^{2}(\mbT^{2})$ by $u^{-}\defeq\min\{0,u\}$. Given a tolerance $\lambda>0$ we show that the event $\{\norm{(\rho^{(\eps)}_{\delta(\eps)})^{-}}_{C_{S_{L}}L^{2}}\geq\lambda\}$ has exponentially small probability in $\eps^{-1}(1+\delta(\eps)^{-2})^{-2}$ asymptotically as $\eps\to0$.
\begin{theorem}\label{thm:negative_values}
	Let $\rho_{0}\in C(\mbT^{2})$ be such that $\rho_{0}>0$ and $\mean{\rho_{0}}=1$, $\rdet$ be the weak solution to~\eqref{eq:Determ_KS} with initial data $\rho_{0}$ and chemotactic sensitivity $\chem\in\mbR$, and $\Trdet$ be its maximal time of existence (see Lemma~\ref{lem:Determ_KS_continuity}). Assume $\delta(\eps)\to0$ and $\eps^{\sfrac{1}{2}}\delta(\eps)^{-2}\to0$ as $\eps\to0$. Then for all $T<\Trdet$, $L>0$ and $\lambda>0$, there exists some $\mu>0$ such that
	\begin{equation*}
		\limsup_{\eps\to0}\eps(1+\delta(\eps)^{-2})^{2}\log\mbP\Bigl(\norm{(\rho^{(\eps)}_{\delta(\eps)})^{-}}_{C_{S_{L}}L^{2}}\geq\lambda\Bigr)\lesssim-\mu^{2}\norm{\srdet}_{C_{T}L^{2}\cap L_{T}^{2}\mcH^{1}}^{-2}.
	\end{equation*}
\end{theorem}
\begin{proof}
	For every $L>0$ let $S_{L}$ be given by~\eqref{eq:exceedance_time}. By combining the non-negativity of $\rdet$ (Theorem~\ref{thm:Determ_KS_Well_Posedness}) with the Lipschitz continuity of $u\mapsto u^{-}$ on $L^{2}(\mbT^{2})$
	\begin{details}
		(see~\eqref{eq:negative_part_Lipschitz})
	\end{details}
	we obtain for every $S\in[0,T]$,
	\begin{equation*}
		\norm{(\rho^{(\eps)}_{\delta(\eps)})^{-}}_{C_{S}L^{2}}=\norm{(\rho^{(\eps)}_{\delta(\eps)})^{-}-\rdet^{-}}_{C_{S}L^{2}}\leq\norm{\rho^{(\eps)}_{\delta(\eps)}-\rdet}_{C_{S}L^{2}},
	\end{equation*}
	which implies for every $\lambda>0$,
	\begin{equation}\label{eq:bound_negative_Lipschitz}
		\mbP\Bigl(\norm{(\rho^{(\eps)}_{\delta(\eps)})^{-}}_{C_{S_{L}}L^{2}}\geq\lambda\Bigr)\leq\mbP\Bigl(\norm{\rho^{(\eps)}_{\delta(\eps)}-\rdet}_{C_{S_{L}}L^{2}}\geq\lambda\Bigr).
	\end{equation}
	Further, using the explicit form of the metric $D_{T}^{L^{2}}$ on $(L^{2}(\mbT^{2}))^{\sol}_{T}$ (see~Subsection~\ref{subsec:notation} and~\cite[Sec.~1.5.1]{chandra_chevyrev_hairer_shen_22}) we deduce 
	\begin{equation}\label{eq:bound_negative_Frechet}
		\mbP\Bigl(\norm{\rho^{(\eps)}_{\delta(\eps)}-\rdet}_{C_{S_{L}}L^{2}}\geq\lambda\Bigr)\leq\mbP\Bigl(D_{T}^{L^{2}}(\rho^{(\eps)}_{\delta(\eps)},\rdet)\geq\widetilde{\lambda}\Bigr)
	\end{equation}
	where $\widetilde{\lambda}>0$ depends on $\lambda$ and $\norm{\rdet}_{C_{T}L^{2}}$.
	\begin{details}
		\paragraph{Proof of~\eqref{eq:bound_negative_Frechet}.}
		Let $L\in\mbN$ and
		\begin{equation*}
			T_{L}\defeq T\wedge L\wedge\inf\Bigl\{t\in[0,T]:\norm{\rho^{(\eps)}_{\delta(\eps)}(t)}_{L^{2}}>L~\text{or}~\norm{\rdet(t)}_{L^{2}}>L\Bigr\}.
		\end{equation*}
		Using the same notation as in~\cite[Details of Subsec.~3.2]{martini_mayorcas_25}, we obtain for each $t\leq T_{L}$ that
		\begin{equation*}
			S_{\rho^{(\eps)}_{\delta(\eps)}}(t)\leq L,\qquad S_{\rdet}(t)\leq L,
		\end{equation*}
		hence
		\begin{equation*}
			\Theta_{L}(\rho^{(\eps)}_{\delta(\eps)})(t)=(1,\rho^{(\eps)}_{\delta(\eps)}(t)),\qquad\Theta_{L}(\rdet)(t)=(1,\rdet(t))
		\end{equation*}
		and
		\begin{equation*}
			\begin{split}
				\hat{d}\Bigl(\Theta_{L}(\rho^{(\eps)}_{\delta(\eps)})(t),\Theta_{L}(\rdet)(t)\Bigr)&=\abs{1-1}+(1\wedge 1)d^{\cem}\Bigl(\rho^{(\eps)}_{\delta(\eps)}(t),\rdet(t)\Bigr)\\
				&=\norm{\rho^{(\eps)}_{\delta(\eps)}(t)-\rdet(t)}_{L^{2}}\wedge\Big(\frac{1}{1+\norm{\rho^{(\eps)}_{\delta(\eps)}(t)}_{L^{2}}}+\frac{1}{1+\norm{\rdet(t)}_{L^{2}}}\Big)\\
				&\geq\norm{\rho^{(\eps)}_{\delta(\eps)}(t)-\rdet(t)}_{L^{2}}\wedge\frac{2}{1+L}.
			\end{split}
		\end{equation*}
		By the definition of $D_{T}^{L^{2}}$,
		\begin{equation}\label{eq:blow_up_metric_explicit_frechet}
			\begin{split}
				D_{T}^{L^{2}}(\rho^{(\eps)}_{\delta(\eps)},\rdet)&=\sum_{L=1}^{\infty}2^{-L}\sup_{t\in[0,T\wedge L]}\hat{d}\Bigl(\Theta_{L}(\rho^{(\eps)}_{\delta(\eps)})(t),\Theta_{L}(\rdet)(t)\Bigr)\\
				&\geq\sum_{L=1}^{\infty}2^{-L}\sup_{t\leq T_{L}}\hat{d}\Bigl(\Theta_{L}(\rho^{(\eps)}_{\delta(\eps)})(t),\Theta_{L}(\rdet)(t)\Bigr)\\
				&\geq\sum_{L=1}^{\infty}2^{-L}\sup_{t\leq T_{L}}\Bigl(\norm{\rho^{(\eps)}_{\delta(\eps)}(t)-\rdet(t)}_{L^{2}}\wedge\frac{2}{1+L}\Bigr).
			\end{split}
		\end{equation}
		Using that $x\mapsto(x\wedge\frac{2}{1+L})$ is continuous and non-decreasing on $\mbR$, we obtain
		\begin{equation*}
			\sup_{t\leq T_{L}}\Bigl(\norm{\rho^{(\eps)}_{\delta(\eps)}(t)-\rdet(t)}_{L^{2}}\wedge\frac{2}{1+L}\Bigr)=\Bigl(\norm{\rho^{(\eps)}_{\delta(\eps)}-\rdet}_{C_{T_{L}}L^{2}}\wedge\frac{2}{1+L}\Bigr),
		\end{equation*}
		and, using that $\frac{1}{L}\leq\frac{2}{1+L}\leq1$ for all $L\in\mbN$,
		\begin{equation}\label{eq:blow_up_metric_explicit_summand}
			\begin{split}
				\Bigl(\norm{\rho^{(\eps)}_{\delta(\eps)}-\rdet}_{C_{T_{L}}L^{2}}\wedge\frac{2}{1+L}\Bigr)&\geq\Bigl(\norm{\rho^{(\eps)}_{\delta(\eps)}-\rdet}_{C_{T_{L}}L^{2}}\wedge\frac{1}{L}\Bigr)\\
				&\geq\Bigl(\frac{1}{L}\norm{\rho^{(\eps)}_{\delta(\eps)}-\rdet}_{C_{T_{L}}L^{2}}\wedge\frac{1}{L}\Bigr)\\
				&=\frac{1}{L}\Bigl(\norm{\rho^{(\eps)}_{\delta(\eps)}-\rdet}_{C_{T_{L}}L^{2}}\wedge1\Bigr).
			\end{split}
		\end{equation}
		Combining~\eqref{eq:blow_up_metric_explicit_frechet} and~\eqref{eq:blow_up_metric_explicit_summand}, we obtain
		\begin{equation*}
			D_{T}^{L^{2}}(\rho^{(\eps)}_{\delta(\eps)},\rdet)\geq\sum_{L=1}^{\infty}\frac{2^{-L}}{L}\Bigl(\norm{\rho^{(\eps)}_{\delta(\eps)}-\rdet}_{C_{T_{L}}L^{2}}\wedge1\Bigr).
		\end{equation*}
		For every $L\in\mbN$ we further obtain that
		\begin{equation*}
			\frac{2^{-L}}{L}\Bigl(\norm{\rho^{(\eps)}_{\delta(\eps)}-\rdet}_{C_{T_{L}}L^{2}}\wedge 1\Bigr)\leq\sum_{L=1}^{\infty}\frac{2^{-L}}{L}\Bigl(\norm{\rho^{(\eps)}_{\delta(\eps)}-\rdet}_{C_{T_{L}}L^{2}}\wedge1\Bigr).
		\end{equation*}
		Let $\lambda>0$ and $L\in\mbN$; we bound
		\begin{equation*}
			\begin{split}
				\mbP\Bigl(\norm{\rho^{(\eps)}_{\delta(\eps)}-\rdet}_{C_{T_{L}}L^{2}}\geq\lambda\Bigr)&\leq\mbP\Bigl(\norm{\rho^{(\eps)}_{\delta(\eps)}-\rdet}_{C_{T_{L}}L^{2}}\geq(\lambda\wedge1)\Bigr)\\
				&=\mbP\Bigl((\norm{\rho^{(\eps)}_{\delta(\eps)}-\rdet}_{C_{T_{L}}L^{2}}\wedge 1)\geq(\lambda\wedge1)\Bigr)\\
				&=\mbP\Bigl(\frac{2^{-L}}{L}(\norm{\rho^{(\eps)}_{\delta(\eps)}-\rdet}_{C_{T_{L}}L^{2}}\wedge 1)\geq\frac{2^{-L}}{L}(\lambda\wedge1)\Bigr)\\
				&\leq\mbP\Bigl(D_{T}^{L^{2}}(\rho^{(\eps)}_{\delta(\eps)},\rdet)\geq\frac{2^{-L}}{L}(\lambda\wedge1)\Bigr)
			\end{split}
		\end{equation*}
		where for the first equality we used that $x\geq(\lambda\wedge1)$ if and only if $(x\wedge1)\geq(\lambda\wedge1)$. (Indeed, if $x<(\lambda\wedge1)$, then $(x\wedge 1)=x<(\lambda\wedge1)$. If $x\geq(\lambda\wedge1)$, then either $x\in[\lambda\wedge1,1)$ or $x\geq1$. If $x\in[\lambda\wedge1,1)$, then $(x\wedge1)=x\geq(\lambda\wedge1)$. If $x\geq1$, then $(x\wedge1)=1\geq(\lambda\wedge1)$.)
		
		For $L\in(0,\infty)\setminus\mbN$, we can estimate $L\leq\ceil{L}$ and
		\begin{equation*}
			\mbP\Bigl(\norm{\rho^{(\eps)}_{\delta(\eps)}-\rdet}_{C_{T_{L}}L^{2}}\geq\lambda\Bigr)\leq\mbP\Bigl(\norm{\rho^{(\eps)}_{\delta(\eps)}-\rdet}_{C_{T_{\ceil{L}}}L^{2}}\geq\lambda\Bigr)\leq\mbP\Bigl(D_{T}^{L^{2}}(\rho^{(\eps)}_{\delta(\eps)},\rdet)\geq\frac{2^{-\ceil{L}}}{\ceil{L}}(\lambda\wedge1)\Bigr).
		\end{equation*}
		If $L\geq\norm{\rdet}_{C_{T}L^{2}}$, then $S_{L}=T_{L}$ and we obtain
		\begin{equation*}
			\mbP\Bigl(\norm{\rho^{(\eps)}_{\delta(\eps)}-\rdet}_{C_{S_{L}}L^{2}}\geq\lambda\Bigr)=\mbP\Bigl(\norm{\rho^{(\eps)}_{\delta(\eps)}-\rdet}_{C_{T_{L}}L^{2}}\geq\lambda\Bigr).
		\end{equation*}
		If $L<\norm{\rdet}_{C_{T}L^{2}}\eqdef M$, we can estimate $S_{L}\leq S_{M}=T_{M}$ and
		\begin{equation*}
			\mbP\Bigl(\norm{\rho^{(\eps)}_{\delta(\eps)}-\rdet}_{C_{S_{L}}L^{2}}\geq\lambda\Bigr)\leq\mbP\Bigl(\norm{\rho^{(\eps)}_{\delta(\eps)}-\rdet}_{C_{T_{M}}L^{2}}\geq\lambda\Bigr).
		\end{equation*}
		This yields for every $L>0$ and $\lambda>0$,
		\begin{equation*}
			\mbP\Bigl(\norm{\rho^{(\eps)}_{\delta(\eps)}-\rdet}_{C_{S_{L}}L^{2}}\geq\lambda\Bigr)\leq\mbP\Bigl(D_{T}^{L^{2}}(\rho^{(\eps)}_{\delta(\eps)},\rdet)\geq\widetilde{\lambda}\Bigr)
		\end{equation*}
		where $\widetilde{\lambda}>0$ depends on $\lambda$ and $\norm{\rdet}_{C_{T}L^{2}}$. This proves~\eqref{eq:bound_negative_Frechet}.
		
	\end{details}
	Using the (local Lipschitz) continuity of the solution $\rho^{(\eps)}_{\delta(\eps)}$ in the noise $\eps^{\sfrac{1}{2}}\ti^{\delta(\eps)}$ (Lemma~\ref{lem:rKS_well_posedness_regular}), we deduce that there exists some $\mu>0$ such that
	\begin{equation*}
		\Bigl\{D_{T}^{L^{2}}(\rho^{(\eps)}_{\delta(\eps)},\rdet)\geq\widetilde{\lambda}\Bigr\}\subset\Bigl\{\eps^{\sfrac{1}{2}}\norm{\ti^{\delta(\eps)}}_{C_{T}L^{2}\cap L_{T}^{2}\mcH^{1}}\geq\mu\Bigr\}
	\end{equation*}
	and consequently,
	\begin{equation}\label{eq:bound_negative_continuity}
		\mbP\Bigl(D_{T}^{L^{2}}(\rho^{(\eps)}_{\delta(\eps)},\rdet)\geq\widetilde{\lambda}\Bigr)\leq\mbP\Bigl(\eps^{\sfrac{1}{2}}\norm{\ti^{\delta(\eps)}}_{C_{T}L^{2}\cap L_{T}^{2}\mcH^{1}}\geq\mu\Bigr).
	\end{equation}
	Combining the relative scaling $\eps^{\sfrac{1}{2}}\delta(\eps)^{-2}\to0$ with Proposition~\ref{prop:lolli_regular_srdet} we can show that there exists some $\eps_{0}>0$ such that for every $\eps<\eps_{0}$ it holds that $\eps^{\sfrac{1}{2}}\mbE[\norm{\ti^{\delta(\eps)}}_{C_{T}L^{2}\cap L_{T}^{2}\mcH^{1}}]\leq\mu/2$. 
	Denote $E\defeq C_{T}L^{2}(\mbT^{2})\cap L_{T}^{2}\mcH^{1}(\mbT^{2})$; we can now apply Lemma~\ref{lem:lolli_regular_upper_deviation} to bound for every $\eps<\eps_{0}$,
	\begin{equation}\label{eq:bound_negative_Lolli_exponential}
		\begin{split}
			&\eps(1+\delta(\eps)^{-2})^{2}\log\mbP\Bigl(\eps^{\sfrac{1}{2}}\norm{\ti^{\delta(\eps)}}_{E}\geq\mu\Bigr)\\
			&=\eps(1+\delta(\eps)^{-2})^{2}\log\mbP\Bigl(\eps^{\sfrac{1}{2}}\norm{\ti^{\delta(\eps)}}_{E}-\eps^{\sfrac{1}{2}}\mbE[\norm{\ti^{\delta(\eps)}}_{E}]\geq\mu-\eps^{\sfrac{1}{2}}\mbE[\norm{\ti^{\delta(\eps)}}_{E}]\Bigr)\\
			&=\eps(1+\delta(\eps)^{-2})^{2}\log\mbP\Bigl(\norm{\ti^{\delta(\eps)}}_{E}-\mbE[\norm{\ti^{\delta(\eps)}}_{E}]\geq\eps^{-\sfrac{1}{2}}\mu-\mbE[\norm{\ti^{\delta(\eps)}}_{E}]\Bigr)\\
			&\lesssim-\eps\Bigl(\eps^{-\sfrac{1}{2}}\mu-\mbE[\norm{\ti^{\delta(\eps)}}_{E}]\Bigr)^{2}\norm{\srdet}_{E}^{-2}\\
			&=-\Bigl(\mu-\eps^{\sfrac{1}{2}}\mbE[\norm{\ti^{\delta(\eps)}}_{E}]\Bigr)^{2}\norm{\srdet}_{E}^{-2}\\
			&\lesssim-\mu^{2}\norm{\srdet}_{C_{T}L^{2}\cap L_{T}^{2}\mcH^{1}}^{-2},
		\end{split}
	\end{equation}
	where in the application of Lemma~\ref{lem:lolli_regular_upper_deviation} we used that $\eps^{-\sfrac{1}{2}}\mu-\mbE[\norm{\ti^{\delta(\eps)}}_{E}]>0$ for every $\eps<\eps_{0}$.
	Combining~\eqref{eq:bound_negative_Lipschitz}, \eqref{eq:bound_negative_Frechet}, \eqref{eq:bound_negative_continuity} and~\eqref{eq:bound_negative_Lolli_exponential}, we arrive at
	\begin{equation*}
		\begin{split}
			&\limsup_{\eps\to0}\eps(1+\delta(\eps)^{-2})^{2}\log\mbP\Bigl(\norm{(\rho^{(\eps)}_{\delta(\eps)})^{-}}_{C_{S_{L}}L^{2}}\geq\lambda\Bigr)\lesssim-\mu^{2}\norm{\srdet}_{C_{T}L^{2}\cap L_{T}^{2}\mcH^{1}}^{-2},
		\end{split}
	\end{equation*}
	which yields the claim.
\end{proof}
\section{Rough Theory for Large Noise Intensities}\label{sec:rough_theory}
In this section we first show the well-posedness of the solution $\rho^{(\eps)}_{\delta(\eps)}$ to~\eqref{eq:rKS_intro_mild} in the H\"{o}lder--Besov space $\mcC^{\alpha+1}(\mbT^{2})$ of regularity $\alpha+1<-1$ (Theorem~\ref{thm:rKS_well_posedness_rough}), which we then use to establish a law of large numbers (Theorem~\ref{thm:LLN_rough}) and a large deviation principle (Theorem~\ref{thm:LDP_rough}) under the relative scaling $\delta(\eps)\to0$ and $\eps\log(\delta(\eps)^{-1})\to0$ as $\eps\to0$. Furthermore, we also establish a central limit theorem under the scaling $\delta(\eps)\to0$ and $\eps^{\sfrac{1}{2}}\log(\delta(\eps)^{-1})\to0$ as $\eps\to0$ (Theorem~\ref{thm:CLT}).

Throughout this section we consider initial data $\rho_{0}\in\mcH^{2}(\mbT^{2})$ such that $\rho_{0}>0$ and $\mean{\rho_{0}}=1$, which ensures that $\sqrt{\rdet}\in C_{T}\mcH^{2}(\mbT^{2})$ (see Lemma~\ref{lem:regularity_srdet}).
\subsection{Well-Posedness}\label{subsec:well_posedness_rough}
In this subsection we use the arguments of~\cite{martini_mayorcas_25} to show that~\eqref{eq:rKS_intro_mild} is well-posed in a space of H\"{o}lder--Besov distributions $\mcC^{\alpha+1}(\mbT^{2})$ of regularity $\alpha+1<-1$ (Theorem~\ref{thm:rKS_well_posedness_rough}). Our technique relies on the paracontrolled decomposition~\cite{gubinelli_15_GIP} (cf.~\eqref{eq:rKS_paracon}), which, compared to the mild formulation~\eqref{eq:rKS_intro_mild}, introduces more structure to achieve the well-posedness of~\eqref{eq:rKS_intro_mild} in a larger space.

The paracontrolled decomposition of~\eqref{eq:rKS_intro_mild} has been carried out in~\cite[Subsec.~1.2]{martini_mayorcas_25} for the equation with generic heterogeneity $\het\in C_{T}\mcH^{2}(\mbT^{2})$ driven by space-time white noise (correlation length zero; suitably renormalised). Here, we recap the decomposition for positive correlation lengths $\delta=\delta(\eps)>0$ without renormalisation and further specialize the heterogeneity to $\srdet\in C_{T}\mcH^{2}(\mbT^{2})$ (see Lemma~\ref{lem:regularity_srdet}).

We first apply the so-called Da~Prato--Debussche trick and consider the remainder $u\defeq\rho+\eps^{\sfrac{1}{2}}\ti^{\delta}$, which then satisfies the equation
\begin{equation*}
	u=P\rho_{0}-\chem\vdiv\mcI[u\nabla\Phi_{u}]+\chem\eps^{\sfrac{1}{2}}\vdiv\mcI[u\nabla\Phi_{\ti^{\delta}}]+\chem\eps^{\sfrac{1}{2}}\vdiv\mcI[\ti^{\delta}\nabla\Phi_{u}]-\chem\eps\vdiv\mcI[\ti^{\delta}\nabla\Phi_{\ti^{\delta}}].
\end{equation*}
The regularity of $u$ is consequently determined by the second-order stochastic object given by $\ty^{\delta}_{\can}\defeq\vdiv\mcI[\ti^{\delta}\nabla\Phi_{\ti^{\delta}}]$. While this improves the situation, this is not yet enough to make sense of the equation under the assumption $\ti^{\delta}\in C_{T}\mcC^{\alpha+1}(\mbT^{2})$, so we apply the Da~Prato--Debussche trick again and define the remainder $w\defeq u+\chem\eps\ty^{\delta}_{\can}$. This is still not enough and subsequent applications of the Da~Prato--Debussche trick no longer yield any improvements.
\begin{details}
	The second Da Prato--Debussche remainder satisfies
	\begin{equation*}
		\begin{split}
			w&=P\rho_{0}-\chem\vdiv\mcI[w\nabla\Phi_{w}]+\chem^{2}\eps\vdiv\mcI[w\nabla\Phi_{\ty^{\delta}_{\can}}]+\chem^{2}\eps\vdiv\mcI[\ty^{\delta}_{\can}\nabla\Phi_{w}]-\chem^{3}\eps^{2}\vdiv\mcI[\ty^{\delta}_{\can}\nabla\Phi_{\ty^{\delta}_{\can}}]\\
			&\quad+\chem\eps^{\sfrac{1}{2}}\vdiv\mcI[w\nabla\Phi_{\ti^{\delta}}]-\chem^{2}\eps^{\sfrac{3}{2}}\vdiv\mcI[\ty^{\delta}_{\can}\nabla\Phi_{\ti^{\delta}}]+\chem\eps^{\sfrac{1}{2}}\vdiv\mcI[\ti^{\delta}\nabla\Phi_{w}]-\chem^{2}\eps^{\sfrac{3}{2}}\vdiv\mcI[\ti^{\delta}\nabla\Phi_{\ty^{\delta}_{\can}}]\\
			&=P\rho_{0}-\chem\vdiv\mcI[w\nabla\Phi_{w}]+\chem^{2}\eps\vdiv\mcI[w\nabla\Phi_{\ty^{\delta}_{\can}}]+\chem^{2}\eps\vdiv\mcI[\ty^{\delta}_{\can}\nabla\Phi_{w}]-\chem^{3}\eps^{2}\vdiv\mcI[\ty^{\delta}_{\can}\nabla\Phi_{\ty^{\delta}_{\can}}]\\
			&\quad+\chem\eps^{\sfrac{1}{2}}\vdiv\mcI[w\pa\nabla\Phi_{\ti^{\delta}}]+\chem\eps^{\sfrac{1}{2}}\vdiv\mcI[\nabla\Phi_{\ti^{\delta}}\pa w]-\chem^{2}\eps^{\sfrac{3}{2}}\vdiv\mcI[\ty^{\delta}_{\can}\pa\nabla\Phi_{\ti^{\delta}}]\\
			&\quad-\chem^{2}\eps^{\sfrac{3}{2}}\vdiv\mcI[\nabla\Phi_{\ti^{\delta}}\pa \ty^{\delta}_{\can}]-\chem^{2}\eps^{\sfrac{3}{2}}\vdiv\mcI[\ty^{\delta}_{\can}\re\nabla\Phi_{\ti^{\delta}}]+\chem\eps^{\sfrac{1}{2}}\vdiv\mcI[\ti^{\delta}\nabla\Phi_{w}]\\
			&\quad-\chem^{2}\eps^{\sfrac{3}{2}}\vdiv\mcI[\ti^{\delta}\nabla\Phi_{\ty^{\delta}_{\can}}]+\chem\eps^{\sfrac{1}{2}}\vdiv\mcI[w\re\nabla\Phi_{\ti^{\delta}}].
		\end{split}
	\end{equation*}
	We further decompose
	\begin{equation*}
		\begin{split}
			&\chem\eps^{\sfrac{1}{2}}\vdiv\mcI[\ti^{\delta}\nabla\Phi_{w}]-\chem^{2}\eps^{\sfrac{3}{2}}\vdiv\mcI[\ti^{\delta}\nabla\Phi_{\ty^{\delta}_{\can}}]\\
			&=\chem\eps^{\sfrac{1}{2}}\vdiv\mcI[(\nabla\Phi_{w}-\chem\eps\nabla\Phi_{\ty^{\delta}_{\can}})\pa\ti^{\delta}]+\chem\eps^{\sfrac{1}{2}}\vdiv\mcI[\ti^{\delta}\pa\nabla\Phi_{w}]+\chem\eps^{\sfrac{1}{2}}\vdiv\mcI[\ti^{\delta}\re\nabla\Phi_{w}]\\
			&\quad-\chem^{2}\eps^{\sfrac{3}{2}}\vdiv\mcI[\ti^{\delta}\pa\nabla\Phi_{\ty^{\delta}_{\can}}]-\chem^{2}\eps^{\sfrac{3}{2}}\vdiv\mcI[\ti^{\delta}\re\nabla\Phi_{\ty^{\delta}_{\can}}],
		\end{split}
	\end{equation*}
	which yields
	\begin{equation*}
		\begin{split}
			w&=P\rho_{0}-\chem\vdiv\mcI[w\nabla\Phi_{w}]+\chem^{2}\eps\vdiv\mcI[w\nabla\Phi_{\ty^{\delta}_{\can}}]+\chem^{2}\eps\vdiv\mcI[\ty^{\delta}_{\can}\nabla\Phi_{w}]-\chem^{3}\eps^{2}\vdiv\mcI[\ty^{\delta}_{\can}\nabla\Phi_{\ty^{\delta}_{\can}}]\\
			&\quad+\chem\eps^{\sfrac{1}{2}}\vdiv\mcI[w\pa\nabla\Phi_{\ti^{\delta}}]+\chem\eps^{\sfrac{1}{2}}\vdiv\mcI[\nabla\Phi_{\ti^{\delta}}\pa w]-\chem^{2}\eps^{\sfrac{3}{2}}\vdiv\mcI[\ty^{\delta}_{\can}\pa\nabla\Phi_{\ti^{\delta}}]\\
			&\quad-\chem^{2}\eps^{\sfrac{3}{2}}\vdiv\mcI[\nabla\Phi_{\ti^{\delta}}\pa \ty^{\delta}_{\can}]-\chem^{2}\eps^{\sfrac{3}{2}}\vdiv\mcI[\ty^{\delta}_{\can}\re\nabla\Phi_{\ti^{\delta}}]-\chem^{2}\eps^{\sfrac{3}{2}}\vdiv\mcI[\ti^{\delta}\re\nabla\Phi_{\ty^{\delta}_{\can}}]\\
			&\quad+\chem\eps^{\sfrac{1}{2}}\vdiv\mcI[(\nabla\Phi_{w}-\chem\eps\nabla\Phi_{\ty^{\delta}_{\can}})\pa\ti^{\delta}]+\chem\eps^{\sfrac{1}{2}}\vdiv\mcI[\ti^{\delta}\pa\nabla\Phi_{w}]+\chem\eps^{\sfrac{1}{2}}\vdiv\mcI[\ti^{\delta}\re\nabla\Phi_{w}]\\
			&\quad-\chem^{2}\eps^{\sfrac{3}{2}}\vdiv\mcI[\ti^{\delta}\pa\nabla\Phi_{\ty^{\delta}_{\can}}]+\chem\eps^{\sfrac{1}{2}}\vdiv\mcI[w\re\nabla\Phi_{\ti^{\delta}}].
		\end{split}
	\end{equation*}
\end{details}

To proceed further, we apply the paracontrolled Ansatz of~\cite{gubinelli_15_GIP}, which allows us to make sense of~\eqref{eq:rKS_intro_mild} in a space of distributions, if we can define a vector of stochastic objects $\mbX_{\delta}^{(\eps)}=(\eps^{\sfrac{1}{2}}\ti^{\delta},\eps\ty_{\can}^{\delta},\eps^{\sfrac{3}{2}}\tp_{\can}^{\delta},\eps\tc^{\delta})$ that satisfy the relations
\begin{equation}\label{eq:enh_can}
	\begin{split}
		&\ti^{\delta}\defeq\vdiv\mcI[\srdet\boldsymbol{\xi}^{\delta}],\quad\ty_{\can}^{\delta}\defeq\vdiv\mcI[\ti^{\delta}\nabla\Phi_{\ti^{\delta}}],\\
		&\tp_{\can}^{\delta}\defeq\ty_{\can}^{\delta}\re\nabla\Phi_{\ti^{\delta}}+\nabla\Phi_{\ty_{\can}^{\delta}}\re\ti^{\delta},\quad\tc^{\delta}\defeq\nabla\mcI[\ti^{\delta}]\re\nabla\Phi_{\ti^{\delta}}+\nabla^{2}\mcI[\Phi_{\ti^{\delta}}]\re\ti^{\delta},
	\end{split}
\end{equation}
and are suitably regular. Recall that $\re$ denotes the resonant product, see e.g.\ Subsection~\ref{subsec:notation} or~\cite{gubinelli_15_GIP} for a definition. We call this vector the (canonical) \emph{enhancement}. For the existence and regularity of $\mbX^{(\eps)}_{\delta}$, see Theorem~\ref{thm:existence_enh_can}.

The fully decomposed equation and its solution take the following form (cf.~\cite[Subsec.~1.2]{martini_mayorcas_25}): We call $\boldsymbol{w}=(w,w',w^{\#})$ a paracontrolled triple to~\eqref{eq:rKS_intro_mild}, if
\begin{equation}\label{eq:rKS_paracon}
	\begin{cases}
		\begin{aligned}
			w&\defeq\chem\eps^{\sfrac{1}{2}}\vdiv\mcI[w'\pa\ti^{\delta}]+w^{\#},\quad w'\defeq\nabla\Phi_{w}-\chem\eps\nabla\Phi_{\ty^{\delta}_{\can}},\\
			w^{\#}&\defeq P\rho_0+\chem\vdiv\mcI[\Omega^{(\eps)}_{\delta}(\boldsymbol{w})],
		\end{aligned}
	\end{cases}
\end{equation}
with
\begin{equation*}
	\begin{split}
		\Omega^{(\eps)}_{\delta}(\boldsymbol{w})&\defeq-w\nabla\Phi_{w}+\chem\eps (w\nabla\Phi_{\ty^{\delta}_{\can}})+\chem\eps(\ty^{\delta}_{\can}\nabla\Phi_{w})-\chem^{2}\eps^{2}(\ty^{\delta}_{\can}\nabla\Phi_{\ty^{\delta}_{\can}})-\chem\eps^{\sfrac{3}{2}}\tp^{\delta}_{\can}\\
		&\quad-\chem\eps^{\sfrac{3}{2}}(\nabla\Phi_{\ti^{\delta}}\pa\ty^{\delta}_{\can})-\chem\eps^{\sfrac{3}{2}}(\ty^{\delta}_{\can}\pa\nabla\Phi_{\ti^{\delta}})-\chem\eps^{\sfrac{3}{2}}(\ti^{\delta}\pa\nabla\Phi_{\ty^{\delta}_{\can}})+\eps^{\sfrac{1}{2}}(w\pa\nabla\Phi_{\ti^{\delta}})\\
		&\quad+\eps^{\sfrac{1}{2}}(\nabla\Phi_{\ti^{\delta}}\pa w)+\eps^{\sfrac{1}{2}}(\ti^{\delta}\pa\nabla\Phi_{w})+\msP(\boldsymbol{w},\mbX^{(\eps)}_{\delta})
	\end{split}
\end{equation*}
and
\begin{equation*}
	\begin{split}
		\msP(\boldsymbol{w},\mbX^{(\eps)}_{\delta})&\defeq\eps^{\sfrac{1}{2}}(w\re\nabla\Phi_{\ti^{\delta}})+\eps^{\sfrac{1}{2}}(\ti^{\delta}\re\nabla\Phi_{w})\\
		&=\chem\eps\msC(w',\nabla\mcI[\ti^{\delta}],\nabla\Phi_{\ti^{\delta}})+\chem\eps\msC(w',\nabla^{2}\mcI[\Phi_{\ti^{\delta}}],\ti^{\delta})\\
		&\quad+\eps^{\sfrac{1}{2}}\Bigl(w^{\#}+\chem\eps^{\sfrac{1}{2}}\vdiv\mcI[w'\pa\ti^{\delta}]-\chem\eps^{\sfrac{1}{2}}(w'\pa\nabla\mcI[\ti^{\delta}])\Bigr)\re\nabla\Phi_{\ti^{\delta}}\\
		&\quad+\eps^{\sfrac{1}{2}}\Bigl(\nabla\Phi_{w^{\#}}+\chem\eps^{\sfrac{1}{2}}\nabla\vdiv\Phi_{\mcI[w'\pa\ti^{\delta}]}-\chem\eps^{\sfrac{1}{2}}(w'\pa\nabla^{2}\mcI[\Phi_{\ti^{\delta}}])\Bigr)\re\ti^{\delta}+\chem\eps(w'\tc^{\delta}),
	\end{split}
\end{equation*}
where $\msC(f,g,h)$ denotes the commutator $(f\pa g)\re h-f(g\re h)$ of~\cite[Lem.~2.4]{gubinelli_15_GIP}. We frequently denote $w^{(\eps)}_{\delta}=w$ to point out the dependence of $w$ on $(\eps,\delta)$. We call $\rho^{(\eps)}_{\delta}\defeq-\eps^{\sfrac{1}{2}}\ti^{\delta}-\chem\eps\ty^{\delta}_{\can}+w^{(\eps)}_{\delta}$ the paracontrolled solution to~\eqref{eq:rKS_intro_mild}. Since the correlation length is non-zero, it follows that $\rho^{(\eps)}_{\delta}$ solves~\eqref{eq:rKS_intro_mild} in the classical sense.

The next theorem recaps the existence and regularity of the noise enhancement $\mbX^{(\eps)}_{\delta}$ in $\rksnoise{\alpha}{\kappa}_{T}$ (cf.~Subsection~\ref{subsec:notation}) along with its blow-up as $\delta\to0$.
\begin{theorem}\label{thm:existence_enh_can}
	Let $\rho_{0}\in\mcH^{2}(\mbT^{2})$ be such that $\rho_{0}>0$ and $\mean{\rho_{0}}=1$, $\rdet$ be the weak solution to~\eqref{eq:Determ_KS} with initial data $\rho_{0}$ and chemotactic sensitivity $\chem\in\mbR$, $\Trdet$ be its maximal time of existence (see Theorem~\ref{thm:Determ_KS_Well_Posedness}), $T<\Trdet$, $\boldsymbol{\xi}$ be a two-dimensional vector of space-time white noises, $(\psi_{\delta})_{\delta>0}$ be a sequence of mollifiers as in~\eqref{eq:def_mollifiers}, $\boldsymbol{\xi}^{\delta}\defeq\psi_{\delta}\ast\boldsymbol{\xi}\defeq(\psi_{\delta}\ast\xi^{1},\psi_{\delta}\ast\xi^{2})$ and $\mbX_{\delta}^{(\eps)}=(\eps^{\sfrac{1}{2}}\ti^{\delta},\eps\ty_{\can}^{\delta},\eps^{\sfrac{3}{2}}\tp_{\can}^{\delta},\eps\tc^{\delta})$ be given by~\eqref{eq:enh_can} for every $\eps>0$ and $\delta>0$. Then for all $\alpha\in(-5/2,-2)$, $\kappa\in(0,1/2)$ and $p\in[1,\infty)$, it holds that
	\begin{equation*}
		\begin{split}
			&\mbE[\norm{\ti^{\delta}}_{\msL_{T}^{\kappa}\mcC^{\alpha+1}}^{p}]^{1/p}\lesssim\norm{\srdet}_{C_{T}L^{\infty}},\quad\mbE[\norm{\ty^{\delta}_{\can}}_{\msL_{T}^{\kappa}\mcC^{2\alpha+4}}^{p}]^{1/p}\lesssim(1\vee\log(\delta^{-1}))\norm{\srdet}_{C_{T}\mcH^{2}}^{2},\\
			&\mbE[\norm{\tp^{\delta}_{\can}}_{\msL_{T}^{\kappa}\mcC^{3\alpha+6}}^{p}]^{1/p}\lesssim(1\vee\log(\delta^{-1}))\norm{\srdet}_{C_{T}\mcH^{2}}^{3},\quad\mbE[\norm{\tc^{\delta}}_{\msL_{T}^{\kappa}\mcC^{2\alpha+4}}^{p}]^{1/p}\lesssim\norm{\srdet}_{C_{T}\mcH^{2}}^{2},
		\end{split}
	\end{equation*}
	and in particular $\mbX_{\delta}^{(\eps)}\in\rksnoise{\alpha}{\kappa}_{T}$ almost surely.
\end{theorem}
\begin{proof}
	The proof follows by~\cite[Thm.~2.3, Lem.~2.7, Lem.~2.20 \& Lem.~2.25]{martini_mayorcas_25} with the particular choice of heterogeneity $\het=\srdet\in C_{T}\mcH^{2}(\mbT^{2})$ (Lemma~\ref{lem:regularity_srdet}).
\end{proof}
The next theorem establishes the well-posedness of paracontrolled solutions to~\eqref{eq:rKS_intro_mild} (cf.~\cite[Def.~3.10]{martini_mayorcas_25}).
\begin{theorem}\label{thm:rKS_well_posedness_rough}
	Let $\rho_{0}\in\mcH^{2}(\mbT^{2})$ be such that $\rho_{0}>0$  and $\mean{\rho_{0}}=1$, $\rdet$ be the weak solution to~\eqref{eq:Determ_KS} with initial data $\rho_{0}$ and chemotactic sensitivity $\chem\in\mbR$, and $\Trdet$ be its maximal time of existence (see Theorem~\ref{thm:Determ_KS_Well_Posedness}). Then for all $T<\Trdet$, $\eps>0$, $\delta>0$ and $\alpha\in(-9/4,-2)$ there exists a unique (paracontrolled) solution $\rho^{(\eps)}_{\delta}$ to~\eqref{eq:rKS_intro_mild} such that $\rho^{(\eps)}_{\delta}\in(\mcC^{\alpha+1}(\mbT^{2}))^{\sol}_{T}$ almost surely. Furthermore, the solution map is locally Lipschitz continuous in the noise enhancement $\mbX^{(\eps)}_{\delta}$.
\end{theorem}
\begin{proof}
	The proof follows by Theorem~\ref{thm:existence_enh_can} and~\cite[Lem.~3.20]{martini_mayorcas_25}, where we used that the arguments of~\cite{martini_mayorcas_25} are agnostic as to changes in $\chem\in\mbR$. For the local Lipschitz continuity in the noise enhancement, see~\cite[Lem.~3.13]{martini_mayorcas_25}.
	\begin{details}
		Let $(\alpha,\infty,1,\beta,\beta',\beta^{\#},\beta^{\#},\kappa,0)$ be a tuple of exponents satisfying~\cite[(3.1)]{martini_mayorcas_25}. It follows by Besov's embedding that $\mcH^{2}(\mbT^{2})=\mcB_{2,2}^{2}(\mbT^{2})\embed\mcB_{\infty,2}^{1}(\mbT^{2})\embed\mcB_{\infty,1}^{\beta^{\#}}(\mbT^{2})$, where we used that $\beta^{\#}<1$. The claim then follows by Theorem~\ref{thm:existence_enh_can} and~\cite[Lem.~3.20]{martini_mayorcas_25} (with $\beta_{0}=\beta^{\#}$).
	\end{details}
\end{proof}
\begin{remark}
	Following~\cite[Rem.~2.4]{martini_mayorcas_25}, it should be possible to consider more general initial data $\rho_{0}$ such that $(1-\Delta)\srdet\in C_{T}L^{1}(\mbT^{2})$ (modulo the introduction of a suitable weight at $0$ to solve~\eqref{eq:rKS_intro_mild}, cf.~\cite[Lem.~3.20]{martini_mayorcas_25}.) This would suggest taking initial data in the Bessel potential space $\mcH^{2}_{1}(\mbT^{2})$ of regularity $2$ and integrability $1$. Unfortunately, $\mcH^{2}_{1}(\mbT^{2})$ seems to be highly pathological (see~\cite[Chpt.~V,~Ex.~6.6]{stein_70} for a discussion in the full space) and it is not even clear if one can establish $\rdet\in C_{T}\mcH^{2}_{1}(\mbT^{2})$. A suitable, albeit suboptimal, replacement is given by $\rho_{0}\in\mcB_{1,1}^{2}(\mbT^{2})$, which allows us to solve $\rdet\in C_{T}\mcB_{1,1}^{2}(\mbT^{2})$. A generalization of Lemma~\ref{lem:regularity_srdet} would then imply $\srdet\in C_{T}\mcB_{1,1}^{2}(\mbT^{2})$, so that in particular $(1-\Delta)\srdet\in C_{T}\mcB_{1,1}^{0}(\mbT^{2})\embed C_{T}L^{1}(\mbT^{2})$. However, this extension doesn't seem sufficiently interesting to warrant an extensive discussion.
\end{remark}
\subsection{Law of Large Numbers}\label{subsec:LLN_rough}
In this subsection we establish a (weak) law of large numbers for solutions $\rho^{(\eps)}_{\delta(\eps)}$ to~\eqref{eq:rKS_intro_mild} in $(\mcC^{\alpha+1}(\mbT^{2}))^{\sol}_{T}$ for every $\alpha\in(-9/4,-2)$ under the relative scaling $\delta(\eps)\to0$ and $\eps\log(\delta(\eps)^{-1})\to0$ as $\eps\to0$ (Theorem~\ref{thm:LLN_rough}). See~\cite[Cor.~4.7]{gubinelli_15_GIP} and~\cite[Cor.~3.13]{catellier_chouk_18} for similar results in the literature.
\begin{theorem}\label{thm:LLN_rough}
	Let $\rho_{0}\in\mcH^{2}(\mbT^{2})$ be such that $\rho_{0}>0$ and $\mean{\rho_{0}}=1$, $\rdet$ be the weak solution to~\eqref{eq:Determ_KS} with initial data $\rho_{0}$ and chemotactic sensitivity $\chem\in\mbR$, and $\Trdet$ be its maximal time of existence (see Theorem~\ref{thm:Determ_KS_Well_Posedness}). Assume $\delta(\eps)\to0$ and $\eps\log(\delta(\eps)^{-1})\to0$ as $\eps\to0$. Then for all $T<\Trdet$ and $\alpha\in(-9/4,-2)$, it follows that $\rho^{(\eps)}_{\delta(\eps)}\to\rdet$ in $(\mcC^{\alpha+1}(\mbT^{2}))^{\sol}_{T}$ in probability as $\eps\to0$.
	%
	%
\end{theorem}
\begin{proof}
	It follows by the relative scaling $\eps\log(\delta(\eps)^{-1})\to0$, Theorem~\ref{thm:existence_enh_can} and Markov's inequality that $\mbX^{(\eps)}_{\delta(\eps)}\to0$ in $\rksnoise{\alpha}{\kappa}_{T}$ in probability for every $\alpha\in(-5/2,-2)$ and $\kappa\in(0,1/2)$. Further for every $\alpha\in(-9/4,-2)$ we can apply Theorem~\ref{thm:rKS_well_posedness_rough} to construct the paracontrolled solution $\rho^{(\eps)}_{\delta(\eps)}$ to~\eqref{eq:rKS_intro_mild} such that $\rho^{(\eps)}_{\delta(\eps)}\in(\mcC^{\alpha+1}(\mbT^{2}))^{\sol}_{T}$ almost surely. Using that $\rho^{(\eps)}_{\delta(\eps)}$ is locally Lipschitz continuous in $\mbX^{(\eps)}_{\delta(\eps)}$ (Theorem~\ref{thm:rKS_well_posedness_rough}), we can apply the continuous mapping theorem to deduce $\rho^{(\eps)}_{\delta(\eps)}\to\rdet$ in $(\mcC^{\alpha+1}(\mbT^{2}))^{\sol}_{T}$ in probability.
\end{proof}
\subsection{Central Limit Theorem}\label{subsec:CLT}
Having established a law of large numbers (Theorem~\ref{thm:LLN_rough}), we now turn to a central limit theorem (Theorem~\ref{thm:CLT}). For this we need to specialize our relative scaling to $\delta(\eps)\to0$ and $\eps^{\sfrac{1}{2}}\log(\delta(\eps)^{-1})\to0$ as $\eps\to0$ since we lose one order of $\eps^{\sfrac{1}{2}}$ on the scale of the leading-order fluctuations $\eps^{-\sfrac{1}{2}}(\rho^{(\eps)}_{\delta(\eps)}-\rdet)$.

We first show the well-posedness of the generalized Ornstein--Uhlenbeck process, which describes the limiting dynamics of the leading-order fluctuations.
\begin{lemma}\label{lem:genOU_well_posedness}
	Let $\rho_{0}\in\mcH^{2}(\mbT^{2})$ be such that $\rho_{0}>0$ and $\mean{\rho_{0}}=1$, $\rdet$ be the weak solution to~\eqref{eq:Determ_KS} with initial data $\rho_{0}$ and chemotactic sensitivity $\chem\in\mbR$, and $\Trdet$ be its maximal time of existence (see Theorem~\ref{thm:Determ_KS_Well_Posedness}). Then for all $T<\Trdet$ and $\alpha\in(-3,-2)$ there exists a unique mild solution $v\in C_{T}\mcC^{\alpha+1}(\mbT^{2})$ to
	\begin{equation}\label{eq:genOU}
		\begin{cases}
			\begin{aligned}
				(\partial_{t}-\Delta)v&=-\chem\vdiv(v\nabla\Phi_{\rdet})-\chem\vdiv(\rdet\nabla\Phi_{v})-\vdiv(\srdet\boldsymbol{\xi}),\\
				v\tzero&=0.
			\end{aligned}
		\end{cases}
	\end{equation}
	We refer to this $v$ as the generalized Ornstein--Uhlenbeck process.
\end{lemma}
\begin{proof}
	It suffices to construct a fixed point $u$ to the linear mild equation
	\begin{equation}\label{eq:genOU_Da_Prato_Debussche_remainder}
		u=-\chem\vdiv\mcI[u\nabla\Phi_{\rdet}]-\chem\vdiv\mcI[\rdet\nabla\Phi_{u}]+\chem\vdiv\mcI[\ti\nabla\Phi_{\rdet}]+\chem\vdiv\mcI[\rdet\nabla\Phi_{\ti}],
	\end{equation}
	from which we may recover a suitable solution $v$ to~\eqref{eq:genOU} via the Da~Prato--Debussche decomposition $v=-\ti+u$. In particular, the space regularity of $v$ is determined by the regularity of $\ti$ which is given by $C_{T}\mcC^{\alpha+1}(\mbT^{2})$ (cf.~\cite[Lem.~2.7]{martini_mayorcas_25}). It follows by the linearity of~\eqref{eq:genOU_Da_Prato_Debussche_remainder} that the lifetimes of $v$ and $u$ are determined by the lifetimes of their coefficients. The lower bound on $\alpha$ ensures that the products $\ti\nabla\Phi_{\rdet}$ and $\rdet\nabla\Phi_{\ti}$ appearing in~\eqref{eq:genOU_Da_Prato_Debussche_remainder} are well-posed by Bony's estimates~\cite[Lem.~2.1]{gubinelli_15_GIP}.
	\begin{details}
		\paragraph{Detailed Construction.}
		An application of Theorem~\ref{thm:Determ_KS_Well_Posedness} yields $\rdet\in C_{T}\mcH^{2}(\mbT^{2})\embed C_{T}\mcB_{\infty,2}^{1}(\mbT^{2})\embed C_{T}\mcC^{1}(\mbT^{2})$ for every $T<\Trdet$. We construct a solution to the equation
		\begin{equation}\label{eq:genOU_remainder}
			u=Pu_{0}-\chem\vdiv\mcI[u\nabla\Phi_{\rdet}]-\chem\vdiv\mcI[\rdet\nabla\Phi_{u}]+\chem\vdiv\mcI[\ti\nabla\Phi_{\rdet}]+\chem\vdiv\mcI[\rdet\nabla\Phi_{\ti}]
		\end{equation}
		in the H\"{o}lder--Besov space $\mcC^{\beta}(\mbT^{2})$ of regularity $\beta\in(-1,\alpha+2)$ with $\alpha\in(-3,-2)$. Here we consider initial data  $u_{0}\in\mcC^{\beta}(\mbT^{2})$, which will be convenient in the proof of the well-posedness of $u$ on $[0,T]$ via an iteration argument.
		
		Let $\bar{T}\in(0,1\wedge T]$ and $\kappa\in[0,1]$, we define the map $\Psi$, acting on  $\msL_{\bar{T}}^{\kappa}\mcC^{\beta}(\mbT^{2})$, by
		\begin{equation*}
			\Psi(u)\defeq Pu_{0}-\chem\vdiv\mcI[u\nabla\Phi_{\rdet}]-\chem\vdiv\mcI[\rdet\nabla\Phi_{u}]+\chem\vdiv\mcI[\ti\nabla\Phi_{\rdet}]+\chem\vdiv\mcI[\rdet\nabla\Phi_{\ti}].
		\end{equation*}
		An application of Schauder's estimate~\cite[Lem.~A.6]{martini_mayorcas_25} (using that $\alpha+\beta+2\leq\beta+1<\alpha+\beta+4$ and $\alpha+1\leq\beta+1<\alpha+3$) combined with Bony's decomposition yields
		\begin{equation}\label{eq:genOU_bound_Schauder}
			\begin{split}
				&\norm{\Psi(u)}_{\msL_{\bar{T}}^{\kappa}\mcC^{\beta}}\\
				&\lesssim\norm{u_{0}}_{\mcC^{\beta}}+\abs{\chem}(\bar{T}^{\frac{\alpha+3}{2}}\vee \bar{T}^{1-\kappa})\Big(\norm{u\nabla\Phi_{\rdet}}_{C_{\bar{T}}\mcC^{\alpha+\beta+2}}+\norm{\rdet\nabla\Phi_{u}}_{C_{\bar{T}}\mcC^{\alpha+\beta+2}}+\norm{\ti\pa\nabla\Phi_{\rdet}}_{C_{\bar{T}}\mcC^{\alpha+\beta+2}}\\
				&\multiquad[14]+\norm{\rdet\pa\nabla\Phi_{\ti}}_{C_{\bar{T}}\mcC^{\alpha+\beta+2}}+\norm{\nabla\Phi_{\ti}\pa\rdet}_{C_{\bar{T}}\mcC^{\alpha+\beta+2}}\\
				&\multiquad[14]+\norm{\rdet\re\nabla\Phi_{\ti}}_{C_{\bar{T}}\mcC^{\alpha+\beta+2}}+\norm{\ti\re\nabla\Phi_{\rdet}}_{C_{\bar{T}}\mcC^{\alpha+\beta+2}}\Big)\\
				&\quad+\abs{\chem}(\bar{T}^{1-\frac{\beta-\alpha}{2}}\vee\bar{T}^{1-\kappa})\norm{\nabla\Phi_{\rdet}\pa\ti}_{C_{\bar{T}}\mcC^{\alpha+1}}.
			\end{split}
		\end{equation}
		We further bound with Bony's estimates~\cite[Lem.~A.4]{martini_mayorcas_25}, using that $\alpha+2<0$ and $\beta\in(-2,0)$,
		\begin{equation*}
			\begin{split}
				\norm{u\nabla\Phi_{\rdet}}_{C_{\bar{T}}\mcC^{\alpha+\beta+2}}\lesssim\norm{u\nabla\Phi_{\rdet}}_{C_{\bar{T}}\mcC^{\beta}}&\leq\norm{u\pa\nabla\Phi_{\rdet}}_{C_{\bar{T}}\mcC^{\beta+2}}+\norm{\nabla\Phi_{\rdet}\pa u}_{C_{\bar{T}}\mcC^{\beta}}+\norm{u\re\nabla\Phi_{\rdet}}_{C_{\bar{T}}\mcC^{\beta+2}}\\
				&\lesssim\norm{u}_{C_{\bar{T}}\mcC^{\beta}}\norm{\rdet}_{C_{\bar{T}}\mcC^{1}};
			\end{split}
		\end{equation*}
		using that $\alpha+2<0$ and $\beta\in(-2,0)$,
		\begin{equation*}
			\begin{split}
				\norm{\rdet\nabla\Phi_{u}}_{C_{\bar{T}}\mcC^{\alpha+\beta+2}}\lesssim\norm{\rdet\nabla\Phi_{u}}_{C_{\bar{T}}\mcC^{\beta+1}}&\leq\norm{\rdet\pa\nabla\Phi_{u}}_{C_{\bar{T}}\mcC^{\beta+1}}+\norm{\nabla\Phi_{u}\pa\rdet}_{C_{\bar{T}}\mcC^{\beta+1}}+\norm{\rdet\re\nabla\Phi_{u}}_{C_{\bar{T}}\mcC^{\beta+2}}\\
				&\lesssim\norm{\rdet}_{C_{\bar{T}}\mcC^{1}}\norm{u}_{C_{\bar{T}}\mcC^{\beta}};
			\end{split}
		\end{equation*}
		using that $\alpha+2<0$ and $\beta<0$,
		\begin{equation*}
			\begin{split}
				\norm{\ti\pa\nabla\Phi_{\rdet}}_{C_{\bar{T}}\mcC^{\alpha+\beta+2}}&\lesssim\norm{\ti}_{C_{\bar{T}}\mcC^{\alpha+1}}\norm{\rdet}_{C_{\bar{T}}\mcC^{\beta}},\\
				\norm{\rdet\pa\nabla\Phi_{\ti}}_{C_{\bar{T}}\mcC^{\alpha+\beta+2}}&\lesssim\norm{\rdet}_{C_{\bar{T}}\mcC^{\beta}}\norm{\ti}_{C_{\bar{T}}\mcC^{\alpha+1}},\\
				\norm{\nabla\Phi_{\ti}\pa\rdet}_{C_{\bar{T}}\mcC^{\alpha+\beta+2}}&\lesssim\norm{\ti}_{C_{\bar{T}}\mcC^{\alpha+1}}\norm{\rdet}_{C_{\bar{T}}\mcC^{\beta}};
			\end{split}
		\end{equation*}
		using that $\alpha+3>0$,
		\begin{equation*}
			\begin{split}
				\norm{\rdet\re\nabla\Phi_{\ti}}_{C_{\bar{T}}\mcC^{\alpha+\beta+2}}&\lesssim\norm{\rdet\re\nabla\Phi_{\ti}}_{C_{\bar{T}}\mcC^{\alpha+3}}\lesssim\norm{\rdet}_{C_{\bar{T}}\mcC^{1}}\norm{\ti}_{C_{\bar{T}}\mcC^{\alpha+1}},\\
				\norm{\ti\re\nabla\Phi_{\rdet}}_{C_{\bar{T}}\mcC^{\alpha+\beta+2}}&\lesssim\norm{\ti\re\nabla\Phi_{\rdet}}_{C_{\bar{T}}\mcC^{\alpha+3}}\lesssim\norm{\ti}_{C_{\bar{T}}\mcC^{\alpha+1}}\norm{\rdet}_{C_{\bar{T}}\mcC^{1}};
			\end{split}
		\end{equation*}
		and using that $\beta+1>0$,
		\begin{equation*}
			\norm{\nabla\Phi_{\rdet}\pa\ti}_{C_{\bar{T}}\mcC^{\alpha+1}}\lesssim\norm{\rdet}_{C_{\bar{T}}\mcC^{\beta}}\norm{\ti}_{C_{\bar{T}}\mcC^{\alpha+1}}.
		\end{equation*}
		Plugging these into~\eqref{eq:genOU_bound_Schauder}, we obtain that there exists some $\theta>0$ such that
		\begin{equation}\label{eq:genOU_a_priori}
			\norm{\Psi(u)}_{\msL_{\bar{T}}^{\kappa}\mcC^{\beta}}\lesssim(\norm{u_{0}}_{\mcC^{\beta}}+\abs{\chem}+\abs{\chem}\bar{T}^{\theta}\norm{u}_{C_{\bar{T}}\mcC^{\beta}})(1+\norm{\rdet}_{C_{\bar{T}}\mcC^{1}})(1+\norm{\ti}_{C_{\bar{T}}\mcC^{\alpha+1}}).
		\end{equation}
		Let $C$ be the implicit constant of~\eqref{eq:genOU_a_priori} above and let $M_{0},M,R>0$ be constants such that
		\begin{equation*}
			\norm{u_{0}}_{\mcC^{\beta}}<M_{0},
		\end{equation*}
		\begin{equation*}
			(1+\norm{\rdet}_{C_{T}\mcC^{1}})(1+\norm{\ti}_{C_{T}\mcC^{\alpha+1}})<M
		\end{equation*}
		and
		\begin{equation*}
			C(M_{0}+2\chem)M<R.
		\end{equation*}
		Define
		\begin{equation*}
			\mfB_{R;\bar{T}}\defeq\{u\in\msL_{\bar{T}}^{\kappa}\mcC^{\beta}(\mbT^{2}):\norm{u}_{\msL_{\bar{T}}^{\kappa}\mcC^{\beta}}<R\}
		\end{equation*}
		and let $\bar{T}=\bar{T}(R,\theta)\leq1\wedge T$ be sufficiently small such that
		\begin{equation*}
			\norm{\Psi(u)}_{\msL_{\bar{T}}^{\kappa}\mcC^{\beta}}\leq C(M_{0}+\abs{\chem}+\abs{\chem}\bar{T}^{\theta}R)M<R.
		\end{equation*}
		Hence, $\Psi$ is self-mapping on $\mfB_{R;\bar{T}}$. By linearity it is also a contraction so that we can construct a unique fixed-point in $\mfB_{R;\bar{T}}$ by Banach's fixed-point theorem, which yields a solution to~\eqref{eq:genOU_remainder}.
		
		Let $u,\widetilde{u}$ be two tentative solutions to~\eqref{eq:genOU_remainder}. We obtain by the same bounds as above for each $\bar{T}\leq1\wedge T$,
		\begin{equation*}
			\norm{u-\widetilde{u}}_{\msL_{\bar{T}}^{\kappa}\mcC^{\beta}}\lesssim\abs{\chem}(\bar{T}^{\frac{\alpha+3}{2}}\vee\bar{T}^{1-\kappa})\norm{u-\widetilde{u}}_{C_{\bar{T}}\mcC^{\beta}}\norm{\rdet}_{C_{\bar{T}}\mcC^{1}}.
		\end{equation*}
		Choosing $\bar{T}$ sufficiently small depending on $\abs{\chem}$, $\alpha$, $\kappa$, $\norm{\rdet}_{C_{T}\mcC^{1}}$ (and some implicit constant), we obtain $u=\widetilde{u}$; hence the solution to~\eqref{eq:genOU_remainder} is unique in $\msL_{\bar{T}}^{\kappa}\mcC^{\beta}(\mbT^{2})$.
		
		Iterating the local existence, we can construct $u\in(\mcC^{\beta}(\mbT^{2}))^{\sol}_{T}$. Next we use the linearity of $u$ to exclude the possibility of blow-up before time $T$. Given a fixed point $u=\Psi(u)$ it follows by~\eqref{eq:genOU_a_priori} that for all $\bar{T}\in(0,1\wedge T]$,
		\begin{equation*}
			\norm{u}_{\msL_{\bar{T}}^{\kappa}\mcC^{\beta}}\lesssim(\norm{u_{0}}_{\mcC^{\beta}}+\abs{\chem}+\abs{\chem}\bar{T}^{\theta}\norm{u}_{C_{\bar{T}}\mcC^{\beta}})(1+\norm{\rdet}_{C_{T}\mcC^{1}})(1+\norm{\ti}_{C_{T}\mcC^{\alpha+1}}).
		\end{equation*}
		Choosing $\bar{T}$ sufficiently small depending on $\abs{\chem}$, $\theta$, $\norm{\rdet}_{C_{T}\mcC^{1}}$, $\norm{\ti}_{C_{T}\mcC^{\alpha+1}}$ (and some implicit constant), but independent of $\norm{u_{0}}_{\mcC^{\beta}}$, we can absorb terms to obtain the bound
		\begin{equation*}
			\norm{u}_{\msL_{\bar{T}}^{\kappa}\mcC^{\beta}}\lesssim(\norm{u_{0}}_{\mcC^{\beta}}+\abs{\chem})(1+\norm{\rdet}_{C_{T}\mcC^{1}})(1+\norm{\ti}_{C_{T}\mcC^{\alpha+1}}).
		\end{equation*}
		Assume $T^{\cem}_{\mcC^{\beta}}[u]<T$ and let $S<T^{\cem}_{\mcC^{\beta}}[u]$ be such that $S+\bar{T}>T^{\cem}_{\mcC^{\beta}}[u]$. In particular, we obtain
		\begin{equation*}
			\norm{u}_{\msL_{[S,S+\bar{T}]}^{\kappa}\mcC^{\beta}}\lesssim(\norm{u_{S}}_{\mcC^{\beta}}+\abs{\chem})(1+\norm{\rdet}_{C_{T}\mcC^{1}})(1+\norm{\ti}_{C_{T}\mcC^{\alpha+1}}),
		\end{equation*}
		which shows that $\lim_{t\nearrow T^{\cem}_{\mcC^{\beta}}[u]}\norm{u_{t}}_{\mcC^{\beta}}<\infty$, contradicting the definition of $T^{\cem}_{\mcC^{\beta}}[u]$, thereby ruling out the possibility of blow-up before time $T$. By the same argument we can also rule out the possibility of blow-up at time $T$, which implies $u\in C_{T}\mcC^{\beta}(\mbT^{2})$. Setting $v=-\ti+u$ and using that $\ti\in C_{T}\mcC^{\alpha+1}(\mbT^{2})$~\cite[Lem.~2.7]{martini_mayorcas_25}, combined with the embedding $\mcC^{\beta}(\mbT^{2})\embed\mcC^{\alpha+1}(\mbT^{2})$ (since $\alpha+1<-1<\beta$), we obtain $v\in C_{T}\mcC^{\alpha+1}(\mbT^{2})$, which yields the claim.
	\end{details}
\end{proof}
Next we show the central limit theorem for $(\rho^{(\eps)}_{\delta(\eps)})_{\eps>0}$.
\begin{theorem}\label{thm:CLT}
	Let $\rho_{0}\in\mcH^{2}(\mbT^{2})$ be such that $\rho_{0}>0$ and $\mean{\rho_{0}}=1$, $\rdet$ be the weak solution to~\eqref{eq:Determ_KS} with initial data $\rho_{0}$ and chemotactic sensitivity $\chem\in\mbR$, and $\Trdet$ be its maximal time of existence (see Theorem~\ref{thm:Determ_KS_Well_Posedness}). Assume $\delta(\eps)\to0$ and $\eps^{\sfrac{1}{2}}\log(\delta(\eps)^{-1})\to0$ as $\eps\to0$. Then for all $T<\Trdet$ and  $\alpha\in(-9/4,-2)$, it follows that $\eps^{-\sfrac{1}{2}}(\rho^{(\eps)}_{\delta(\eps)}-\rdet)\to v$ in $(\mcC^{\alpha+1}(\mbT^{2}))^{\sol}_{T}$ in probability as $\eps\to0$.
\end{theorem}
\begin{proof}
	We use the Da Prato--Debussche decompositions for $\rho^{(\eps)}_{\delta(\eps)}$ and $v$, given by $\rho^{(\eps)}_{\delta(\eps)}=-\eps^{\sfrac{1}{2}}\ti^{\delta(\eps)}-\chem\eps\ty^{\delta(\eps)}_{\can}+w^{(\eps)}_{\delta(\eps)}$ (see Subsection~\ref{subsec:well_posedness_rough}) and $v=-\ti+u$ (see the proof of Lemma~\ref{lem:genOU_well_posedness}), to deduce
	\begin{equation*}
		\eps^{-\sfrac{1}{2}}(\rho^{(\eps)}_{\delta(\eps)}-\rdet)-v=\eps^{-\sfrac{1}{2}}(w^{(\eps)}_{\delta(\eps)}-\rdet)-u-\ti^{\delta(\eps)}+\ti-\chem\eps^{\sfrac{1}{2}}\ty_{\can}^{\delta(\eps)}.
	\end{equation*}
	The convergence $\eps^{-\sfrac{1}{2}}(w^{(\eps)}_{\delta(\eps)}-\rdet)\to u$ follows by the paracontrolled decomposition, in particular the local Lipschitz continuity of the paracontrolled solution in the enhancement (Theorem~\ref{thm:rKS_well_posedness_rough}); the convergence $\ti^{\delta(\eps)}\to\ti$ follows by~\cite[Lem.~2.7]{martini_mayorcas_25}; the convergence $\eps^{\sfrac{1}{2}}\ty^{\delta(\eps)}_{\can}\to0$ follows by Theorem~\ref{thm:existence_enh_can} combined with the relative scaling $\eps^{\sfrac{1}{2}}\log(\delta(\eps)^{-1})\to0$.
	\begin{details}
		\paragraph{Detailed Proof.}
		We denote $v^{(\eps)}\defeq\eps^{-\sfrac{1}{2}}(\rho^{(\eps)}_{\delta(\eps)}-\rdet)$ for every $\eps>0$. To deduce that $(v^{(\eps)})_{\eps>0}$ convergences to $v$ in $(\mcC^{\alpha+1}(\mbT^{2}))^{\sol}_{T}$ in probability as $\eps\to0$, it suffices to show that every subsequence of $(v^{(\eps)})_{\eps>0}$ has a sub-subsequence that converges almost surely to $v$ in $(\mcC^{\alpha+1}(\mbT^{2}))^{\sol}_{T}$.
		
		It follows by Theorem~\ref{thm:existence_enh_can} and the relative scaling $\eps^{\sfrac{1}{2}}\log(\delta(\eps)^{-1})\to0$, that $\eps^{-\sfrac{1}{2}}\mbX^{(\eps)}_{\delta(\eps)}\to0$ in $\rksnoise{\alpha}{\kappa}_{T}$ in probability as $\eps\to0$. In particular, there exists a subsequence (which we do not relabel) such that almost sure convergence holds. For every $\eps>0$ let $\boldsymbol{w}=\boldsymbol{w}^{(\eps)}_{\delta(\eps)}$ be a solution to~\eqref{eq:rKS_paracon} driven by $\mbX^{(\eps)}_{\delta(\eps)}$ as constructed in Theorem~\ref{thm:rKS_well_posedness_rough}. We use the paracontrolled decomposition of $w$ and collect powers of $\eps^{\sfrac{1}{2}}$ to obtain the expansion,
		\begin{equation*}
			\begin{split}
				w&=P\rho_{0}-\chem\vdiv\mcI[w\nabla\Phi_{w}]\\
				&\quad+\chem\eps^{\sfrac{1}{2}}\Bigl(\vdiv\mcI[\nabla\Phi_{w}\pa\ti^{\delta(\eps)}]+\vdiv\mcI[w\pa\nabla\Phi_{\ti^{\delta(\eps)}}]+\vdiv\mcI[\nabla\Phi_{\ti^{\delta(\eps)}}\pa w]\\
				&\multiquad[5]+\vdiv\mcI[\ti^{\delta(\eps)}\pa\nabla\Phi_{w}]+\vdiv\mcI[w^{\#}\re\nabla\Phi_{\ti^{\delta(\eps)}}]+\vdiv\mcI[\nabla\Phi_{w^{\#}}\re\ti^{\delta(\eps)}]\Bigr)\\
				&\quad+\chem^{2}\eps^{\sfrac{1}{2}}\Bigl(\eps^{\sfrac{1}{2}}\vdiv\mcI[w\nabla\Phi_{\ty^{\delta(\eps)}_{\can}}]+\eps^{\sfrac{1}{2}}\vdiv\mcI[\ty^{\delta(\eps)}_{\can}\nabla\Phi_{w}]\Bigr)\\
				&\quad+\chem^{2}\eps\Bigl(\vdiv\mcI[\msC(w',\nabla\mcI[\ti^{\delta(\eps)}],\nabla\Phi_{\ti^{\delta(\eps)}})]+\vdiv\mcI[\msC(w',\nabla^{2}\mcI[\Phi_{\ti^{\delta(\eps)}}],\ti^{\delta(\eps)})]\\
				&\multiquad[4]+\vdiv\mcI[(\vdiv\mcI[w'\pa\ti^{\delta(\eps)}]-w'\pa\nabla\mcI[\ti^{\delta(\eps)}])\re\nabla\Phi_{\ti^{\delta(\eps)}}]\\
				&\multiquad[4]+\vdiv\mcI[(\nabla\vdiv\Phi_{\mcI[w'\pa\ti^{\delta(\eps)}]}-w'\pa\nabla^{2}\mcI[\Phi_{\ti^{\delta(\eps)}}])\re\ti^{\delta(\eps)}]+\vdiv\mcI[w'\tc^{\delta(\eps)}]\Bigr)\\
				&\quad-\chem^{2}\eps\Bigl(\eps^{\sfrac{1}{2}}\vdiv\mcI[\nabla\Phi_{\ty^{\delta(\eps)}_{\can}}\pa\ti^{\delta(\eps)}]+\eps^{\sfrac{1}{2}}\vdiv\mcI[\tp^{\delta(\eps)}_{\can}]+\eps^{\sfrac{1}{2}}\vdiv\mcI[\nabla\Phi_{\ti^{\delta(\eps)}}\pa\ty^{\delta(\eps)}_{\can}]\\
				&\multiquad[4]+\eps^{\sfrac{1}{2}}\vdiv\mcI[\ty^{\delta(\eps)}_{\can}\pa\nabla\Phi_{\ti^{\delta(\eps)}}]+\eps^{\sfrac{1}{2}}\vdiv\mcI[\ti^{\delta(\eps)}\pa\nabla\Phi_{\ty^{\delta(\eps)}_{\can}}]\Bigr)\\
				&\quad-\chem^{3}\eps\Bigl(\eps\vdiv\mcI[\ty^{\delta(\eps)}_{\can}\nabla\Phi_{\ty^{\delta(\eps)}_{\can}}]\Bigr)\\
				&\eqdef P\rho_{0}-\chem\vdiv\mcI[w\nabla\Phi_{w}]\\
				&\quad+\chem\eps^{\sfrac{1}{2}}\Bigl(\vdiv\mcI[\nabla\Phi_{w}\pa\ti^{\delta(\eps)}]+\vdiv\mcI[w\pa\nabla\Phi_{\ti^{\delta(\eps)}}]+\vdiv\mcI[\nabla\Phi_{\ti^{\delta(\eps)}}\pa w]\\
				&\multiquad[5]+\vdiv\mcI[\ti^{\delta(\eps)}\pa\nabla\Phi_{w}]+\vdiv\mcI[w^{\#}\re\nabla\Phi_{\ti^{\delta(\eps)}}]+\vdiv\mcI[\nabla\Phi_{w^{\#}}\re\ti^{\delta(\eps)}]\Bigr)\\
				&\quad+\chem^{2}\eps^{\sfrac{1}{2}}\mcR_{1}^{\eps}+\chem^{2}\eps\mcR_{2}^{\eps}-\chem^{2}\eps\mcR_{3}^{\eps}-\chem^{3}\eps\mcR_{4}^{\eps}.
			\end{split}
		\end{equation*}
		Let $u$ be a solution to the linear mild equation~\eqref{eq:genOU_remainder} with $u_{0}=0$, as constructed in the proof of Lemma~\ref{lem:genOU_well_posedness}. We expand the difference between $u$ and the leading order fluctuations of $w^{(\eps)}_{\delta(\eps)}$ to obtain
		\begin{equation*}
			\begin{split}
				&\eps^{-\sfrac{1}{2}}(w^{(\eps)}_{\delta(\eps)}-\rdet)-u\\
				&=-\chem\eps^{-\sfrac{1}{2}}\Bigl(\vdiv\mcI[w\nabla\Phi_{w}]-\vdiv\mcI[\rdet\nabla\Phi_{\rdet}]\Bigr)+\chem\Bigl(\vdiv\mcI[u\nabla\Phi_{\rdet}]+\vdiv\mcI[\rdet\nabla\Phi_{u}]\Bigr)\\
				&\quad+\chem\Bigl(\vdiv\mcI[\nabla\Phi_{w}\pa\ti^{\delta(\eps)}]+\vdiv\mcI[w\pa\nabla\Phi_{\ti^{\delta(\eps)}}]+\vdiv\mcI[\nabla\Phi_{\ti^{\delta(\eps)}}\pa w]\\
				&\multiquad[4]+\vdiv\mcI[\ti^{\delta(\eps)}\pa\nabla\Phi_{w}]+\vdiv\mcI[w^{\#}\re\nabla\Phi_{\ti^{\delta(\eps)}}]+\vdiv\mcI[\nabla\Phi_{w^{\#}}\re\ti^{\delta(\eps)}]\Bigr)\\
				&\quad-\chem\Bigl(\vdiv\mcI[\nabla\Phi_{\rdet}\pa\ti]+\vdiv\mcI[\rdet\pa\nabla\Phi_{\ti}]+\vdiv\mcI[\nabla\Phi_{\ti}\pa \rdet]\\
				&\multiquad[4]+\vdiv\mcI[\ti\pa\nabla\Phi_{\rdet}]+\vdiv\mcI[\rdet\re\nabla\Phi_{\ti}]+\vdiv\mcI[\nabla\Phi_{\rdet}\re\ti]\Bigr)\\
				&\quad+\chem^{2}\mcR_{1}^{\eps}+\chem^{2}\eps^{\sfrac{1}{2}}\mcR_{2}^{\eps}-\chem^{2}\eps^{\sfrac{1}{2}}\mcR_{3}^{\eps}-\chem^{3}\eps^{\sfrac{1}{2}}\mcR_{4}^{\eps}\\
				&\eqdef-\chem\eps^{-\sfrac{1}{2}}\Bigl(\vdiv\mcI[w\nabla\Phi_{w}]-\vdiv\mcI[\rdet\nabla\Phi_{\rdet}]\Bigr)+\chem\Bigl(\vdiv\mcI[u\nabla\Phi_{\rdet}]+\vdiv\mcI[\rdet\nabla\Phi_{u}]\Bigr)\\
				&\quad+\chem^{2}\mcR_{1}^{\eps}+\chem^{2}\eps^{\sfrac{1}{2}}\mcR_{2}^{\eps}-\chem^{2}\eps^{\sfrac{1}{2}}\mcR_{3}^{\eps}-\chem^{3}\eps^{\sfrac{1}{2}}\mcR_{4}^{\eps}+\chem\mcR_{5}^{\eps}.
			\end{split}
		\end{equation*}
		Denote $\rem^{(\eps)}\defeq\eps^{-\sfrac{1}{2}}(w^{(\eps)}_{\delta(\eps)}-\rdet)-u$, for which we obtain the equation
		\begin{equation*}
			\begin{split}
				\rem^{(\eps)}&=-\chem\vdiv\mcI[\rem^{(\eps)}\nabla\Phi_{w^{(\eps)}_{\delta(\eps)}}]-\chem\vdiv\mcI[\rdet\nabla\Phi_{\rem^{(\eps)}}]-\chem\vdiv\mcI[u\nabla\Phi_{w^{(\eps)}_{\delta(\eps)}-\rdet}]\\
				&\quad+\chem^{2}\mcR_{1}^{\eps}+\chem^{2}\eps^{\sfrac{1}{2}}\mcR_{2}^{\eps}-\chem^{2}\eps^{\sfrac{1}{2}}\mcR_{3}^{\eps}-\chem^{3}\eps^{\sfrac{1}{2}}\mcR_{4}^{\eps}+\chem\mcR_{5}^{\eps}\\
				&\eqdef-\chem\vdiv\mcI[\rem^{(\eps)}\nabla\Phi_{w^{(\eps)}_{\delta(\eps)}}]-\chem\vdiv\mcI[\rdet\nabla\Phi_{\rem^{(\eps)}}]+\mcR^{\eps}.
			\end{split}
		\end{equation*}
		Let $(\alpha,\infty,1,\beta,\beta',\beta^{\#},\beta^{\#},\kappa,0)$ be a tuple of exponents satisfying~\cite[(3.1)]{martini_mayorcas_25}. Proceeding as in the proof of~\cite[Lem.~3.7~\&~3.8]{martini_mayorcas_25} we can show that almost surely
		\begin{equation*}
			\norm{\mcR^{\eps}}_{\msL_{\bar{T}}^{\kappa}\mcC^{\beta}}\to0\quad\text{as}~\eps\to0.
		\end{equation*}
		To establish the almost sure convergence of $\rem^{(\eps)}\to0$ in $(\mcC^{\alpha+1}(\mbT^{2}))^{\sol}_{T}$, it suffices to argue on small time intervals of size $\bar{T}\in(0,1\wedge T]$. An application of Schauder's estimate~\cite[Lem.~A.6]{martini_mayorcas_25} yields
		\begin{equation*}
			\begin{split}
				\norm{\rem^{(\eps)}}_{\msL_{\bar{T}}^{\kappa}\mcC^{\beta}}&\lesssim(\bar{T}^{\frac{\alpha+3}{2}}\vee \bar{T}^{1-\kappa})\Bigl(\norm{\rem^{(\eps)}\nabla\Phi_{w^{(\eps)}_{\delta(\eps)}}}_{C_{\bar{T}}\mcC^{\alpha+\beta+2}}+\norm{\rdet\nabla\Phi_{\rem^{(\eps)}}}_{C_{\bar{T}}\mcC^{\alpha+\beta+2}}\Bigr)+\norm{\mcR^{\eps}}_{\msL_{\bar{T}}^{\kappa}\mcC^{\beta}}\\
				&\lesssim(\bar{T}^{\frac{\alpha+3}{2}}\vee \bar{T}^{1-\kappa})\Bigl(\norm{\rem^{(\eps)}}_{C_{\bar{T}}\mcC^{\beta}}\norm{w^{(\eps)}_{\delta(\eps)}}_{C_{\bar{T}}\mcC^{\beta}}+\norm{\rdet}_{C_{\bar{T}}\mcC^{\beta}}\norm{\rem^{(\eps)}}_{C_{\bar{T}}\mcC^{\beta}}\Bigr)+\norm{\mcR^{\eps}}_{\msL_{\bar{T}}^{\kappa}\mcC^{\beta}},
			\end{split}
		\end{equation*}
		which allows us to choose $\bar{T}=\bar{T}(\alpha,\kappa,\norm{w^{(\eps)}_{\delta(\eps)}}_{C_{T}\mcC^{\beta}},\norm{\rdet}_{C_{T}\mcC^{\beta}})$ sufficiently small to absorb terms and control
		\begin{equation*}
			\norm{\rem^{(\eps)}}_{\msL_{\bar{T}}^{\kappa}\mcC^{\beta}}\lesssim\norm{\mcR^{\eps}}_{\msL_{\bar{T}}^{\kappa}\mcC^{\beta}}\to0\quad\text{as}~\eps\to0.
		\end{equation*}
		Note here that $\norm{w^{(\eps)}_{\delta(\eps)}}_{C_{T}\mcC^{\beta}}$ depends only on $\norm{\rho_{0}}_{\mcB_{\infty,1}^{\beta^{\#}}}$ and $\norm{\mbX^{(\eps)}_{\delta(\eps)}}_{\rksnoise{\alpha}{\kappa}_{T}}$. In particular, it follows by Lemma~\ref{lem:blow_up_lsc} and the convergence $w^{(\eps)}_{\delta(\eps)}\to\rdet$ in $(C^{\alpha+1}(\mbT^{2}))^{\sol}_{T}$ as $\eps\to0$, that $T^{\cem}_{\mcC^{\beta}}[\rdet]\leq\liminf_{\eps\to0}T^{\cem}_{\mcC^{\beta}}[w^{(\eps)}_{\delta(\eps)}]$. Furthermore by Theorem~\ref{thm:Determ_KS_maximal} and the embedding $\mcH^{2}(\mbT^{2})\embed\mcC^{\beta}(\mbT^{2})$, it follows that $T^{\cem}_{L^{2}}[\rdet]=T^{\cem}_{\mcH^{2}}[\rdet]\leq T^{\cem}_{\mcC^{\beta}}[\rdet]$. Therefore, for every $T<\Trdet=T^{\cem}_{L^{2}}[\rdet]$ there exists a $\eps_{0}>0$ such that $T<T^{\cem}_{\mcC^{\beta}}[w^{(\eps)}_{\delta(\eps)}]$ for every $\eps>\eps_{0}$.
		
		From this we can deduce the convergence of $v^{(\eps)}$ to $v$ in $(\mcC^{\alpha+1}(\mbT^{2}))^{\sol}_{T}$ almost surely as $\eps\to0$ along a subsequence. In particular for every subsequence of $(v^{(\eps)})_{\eps>0}$ there exists a further subsequence such that $v^{(\eps)}\to v$ almost surely as $\eps\to0$, which implies convergence in probability. This yields the claim.
	\end{details}
\end{proof}
\subsection{Large Deviation Principle}\label{subsec:LDP_rough}
In this subsection we apply a generalization of Freidlin and Wentzell's theory (see Appendix~\ref{app:freidlin_wentzell} as well as~\cite{hairer_weber_15}) to establish a large deviation principle for solutions $\rho^{(\eps)}_{\delta(\eps)}$ to~\eqref{eq:rKS_intro_mild} in $(\mcC^{\alpha+1}(\mbT^{2}))^{\sol}_{T}$ for every $\alpha\in(-9/4,-2)$ under the relative scaling $\delta(\eps)\to0$ and $\eps\log(\delta(\eps)^{-1})\to0$ as $\eps\to0$ (Theorem~\ref{thm:LDP_rough}). As in Subsection~\ref{subsec:LLN_rough} we first consider the noise enhancement $(\mbX^{(\eps)}_{\delta(\eps)})_{\eps>0}$ driving $(\rho^{(\eps)}_{\delta(\eps)})_{\eps>0}$ (cf.\ Theorem~\ref{thm:rKS_well_posedness_rough}) for which we establish a large deviation principle (Theorem~\ref{thm:LDP_enhancement}) that we can then transfer to $(\rho^{(\eps)}_{\delta(\eps)})_{\eps>0}$ via the contraction principle (Theorem~\ref{thm:LDP_rough}).

Recall that $\mcC^{\alpha}_{\mfs}([0,T]\times\mbT^{2};\mbR^{2})$ denotes the space of vector-valued space-time H\"{o}lder distributions of regularity $\alpha$ equipped with the parabolic scaling $\mfs=(2,1,1)$ (see Definition~\ref{def:Cs_alpha_torus}).
\begin{theorem}\label{thm:LDP_enhancement}
	Let $\rho_{0}\in\mcH^{2}(\mbT^{2})$ be such that $\rho_{0}>0$ and $\mean{\rho_{0}}=1$, $\rdet$ be the weak solution to~\eqref{eq:Determ_KS} with initial data $\rho_{0}$ and chemotactic sensitivity $\chem\in\mbR$, and $\Trdet$ be its maximal time of existence (see Theorem~\ref{thm:Determ_KS_Well_Posedness}). Assume $\delta(\eps)\to0$ and $\eps\log(\delta(\eps)^{-1})\to0$ as $\eps\to0$. Then for all $T<\Trdet$, $\alpha\in(-5/2,-2)$ and $\kappa\in(0,1/2)$, it follows that the sequence $(\mbX^{(\eps)}_{\delta(\eps)})_{\eps>0}$ satisfies a large deviation principle in $\rksnoise{\alpha}{\kappa}_{T}$ (cf.~Subsection~\ref{subsec:notation}) with speed $\eps$ and good rate function
	\begin{equation*}
		\rate_{\Enh}(\boldsymbol{s})\defeq\inf\Bigl\{\frac{1}{2}\norm{h}^{2}_{L^{2}([0,T]\times\mbT^{2};\mbR^{2})}:\mbX^{h}=\boldsymbol{s}\Bigr\},
	\end{equation*}
	where $\mbX^{h}$ is the enhancement driven by the Cameron--Martin element $h\in L^{2}([0,T]\times\mbT^{2};\mbR^{2})$ defined as in~\eqref{eq:enhancement_h_def} (see also Lemma~\ref{lem:existence_enhancement_h}.)
\end{theorem}
\begin{proof}
	 It follows by the Wiener chaos decomposition~\cite[Subsec.~2.2]{martini_mayorcas_25}, Theorem~\ref{thm:abstract_Wiener_space} and Lemma~\ref{lem:iterated_Ito_in_hom_Wiener_chaos} that every enhancement $\mbX_{\delta(\eps)}^{(\eps)}=(\eps^{\sfrac{1}{2}}\ti^{\delta(\eps)},\eps\ty_{\can}^{\delta(\eps)},\eps^{\sfrac{3}{2}}\tp_{\can}^{\delta(\eps)},\eps\tc^{\delta(\eps)})$ is of the form~\eqref{eq:rv_direct_sum_rescaled_corr_length} with abstract Wiener space
	 \begin{equation*}
	 	\nban\defeq\mcC^{\alpha}_{\mfs}([0,T]\times\mbT^{2};\mbR^{2}),\qquad\gm\defeq\Law(\boldsymbol{\xi}),\qquad\cm\defeq L^{2}([0,T]\times\mbT^{2};\mbR^{2}),
	 \end{equation*}
	index set
	\begin{equation*}
		\mcW\defeq\{\ti,\ty,\tp,\tc\},\qquad(K_{\ti},K_{\ty},K_{\tp},K_{\tc})\defeq(1,2,3,2)
	\end{equation*}
	and target space
	\begin{equation*}
		\bban\defeq\bigoplus_{\tau\in\mcW}\ban_{\tau}\defeq\msL^{\kappa}_{T}\mcC^{\alpha+1}(\mbT^{2};\mbR)\times\msL^{\kappa}_{T}\mcC^{2\alpha+4}(\mbT^{2};\mbR)\times\msL^{\kappa}_{T}\mcC^{3\alpha+6}(\mbT^{2};\mbR^{2})\times\msL^{\kappa}_{T}\mcC^{2\alpha+4}(\mbT^{2};\mbR^{2\times2}),
	\end{equation*}
	which is separable by an application of Lemma~\ref{lem:separability_interpolation} combined with the fact that finite Cartesian products of separable spaces are again separable.
	
	\begin{details}
		\paragraph{Proof that $\mbX_{\delta(\eps)}^{(\eps)}$ is of the form~\eqref{eq:rv_direct_sum_rescaled_corr_length}.}
		As an example, let us consider $\ti^{\delta}$. Let $\multi\in\mbN_{0}^{\mbN}$ be a multi-index with finitely-many non-zero entries, it suffices to show that for all $\abs{\nu}\neq1$,
		\begin{equation*}
			\mbE[\ti^{\delta}H_{\multi}((\inner{\boldsymbol{\xi}}{e_{i}})_{i\in\mbN})]=0\qquad\text{in}\quad\msL_{T}^{\kappa}\mcC^{\alpha+1}(\mbT^{2};\mbR).
		\end{equation*}
		The $C_{T}\mcC^{\alpha+1}(\mbT^{2};\mbR)$-norm of the expectation is given by
		\begin{equation*}
			\norm{\mbE[\ti^{\delta}H_{\nu}((\inner{\boldsymbol{\xi}}{e_{i}})_{i\in\mbN})]}_{C_{T}\mcC^{\alpha+1}}=\sup_{t\in[0,T]}\sup_{q\in\mbN_{-1}}2^{q(\alpha+1)}\sup_{x\in\mbT^{2}}\abs{\mbE[\Delta_{q}\ti^{\delta}(t,x)H_{\multi}((\inner{\boldsymbol{\xi}}{e_{i}})_{i\in\mbN})]}
		\end{equation*}
		and a similar decomposition applies to the $C_{T}^{\kappa}\mcC^{\alpha+1-2\kappa}(\mbT^{2};\mbR)$-seminorm. Hence, it suffices to show for each $s,t\in[0,T]$, with $s<t$, and $x\in\mbT^{2}$, that
		\begin{equation*}
			\mbE[\Delta_{q}(\ti^{\delta}(t)-\ti^{\delta}(s))(x)H_{\multi}((\inner{\boldsymbol{\xi}}{e_{i}})_{i\in\mbN})]=0.
		\end{equation*}
		The Littlewood--Paley block $\Delta_{q}(\ti^{\delta}(t)-\ti^{\delta}(s))(x)$ is given by an iterated It\^{o} integral, see~\cite[Sec.~2]{martini_mayorcas_25}, and the same holds true for the generalized Hermite polynomial $H_{\nu}$,
		\begin{equation*}
			\sqrt{\multi!}H_{\multi}((\inner{\boldsymbol{\xi}}{e_{i}})_{i\in\mbN})=\boldsymbol{\xi}^{\otimes \abs{\nu}}\Bigl(\bigotimes_{i=1}^{\infty}e_{i}^{\otimes\multi_{i}}\Bigr),
		\end{equation*}
		see e.g.~\cite[below Thm.~1.1.2]{nualart_06} or the proof of~\cite[Prop.~1.42]{klose_msc_17}. By Lemma~\ref{lem:iterated_Ito_in_hom_Wiener_chaos} and the orthogonality of the iterated It\^{o} integrals~\cite[Sec.~1.1.2]{nualart_06} or~\cite[Def.~1.32]{klose_msc_17}, it follows that
		\begin{equation*}
			\mbE[\Delta_{q}(\ti^{\delta}(t)-\ti^{\delta}(s))(x)H_{\multi}((\inner{\boldsymbol{\xi}}{e_{i}})_{i\in\mbN})]=0
		\end{equation*}
		for every $s,t\in[0,T]$ and $x\in\mbT^{2}$. 
		
		To deduce the claim for the full enhancement $\mbX^{(\eps)}_{\delta}$, it suffices to decompose each entry of $\mbX_{\delta}^{(\eps)}$ as in~\cite[Sec.~2.2]{martini_mayorcas_25} and argue as above.
	\end{details}
	
	Using the relative scaling $\delta(\eps)\to0$ and $\eps\log(\delta(\eps)^{-1})\to0$ combined with the bounds~\cite[Lem.~2.7]{martini_mayorcas_25} (for $\eps^{\sfrac{1}{2}}\ti^{\delta(\eps)}$), \cite[Lem.~2.16~\&~Lem.~2.26]{martini_mayorcas_25} (for $\eps\ty^{\delta(\eps)}_{\can}$), \cite[Lem.~2.19, Lem.~2.24 and Lem.~2.26]{martini_mayorcas_25} (for $\eps^{\sfrac{3}{2}}\tp^{\delta(\eps)}_{\can}$) and~\cite[Lem.~2.20~\&~Lem.~2.25]{martini_mayorcas_25} (for $\eps\tc^{\delta(\eps)}$), one can show that $\mbX^{(\eps)}_{\delta(\eps)}$ converges as $\eps\to0$ in the sense of~\eqref{eq:rv_convergence} to a non-trivial limit $\mbX=\bigoplus_{\tau\in\mcW}\mbX_{\tau,K_{\tau}}$ which has no contributions from lower-order Wiener chaoses, \eqref{eq:rv_limit}.
	
	\begin{details}
		In the notation of~\cite[Subsec.~2.2]{martini_mayorcas_25},
		\begin{equation*}
			\mbX\defeq(\ti,\ty,\PreThree{10}+\PreThree{20},\Checkmark{10}+\Checkmark{20}).
		\end{equation*}
	\end{details}
	An application of Theorem~\ref{thm:LDP_wiener_chaos} then implies that $(\mbX_{\delta(\eps)}^{(\eps)})_{\eps>0}$ satisfies a large deviation principle in $\bban$ with speed $\eps$ and good rate function
	\begin{equation*}
		\rate_{\Enh}(\boldsymbol{s})\defeq\inf\Bigl\{\frac{1}{2}\norm{h}^{2}_{L^{2}([0,T]\times\mbT^{2};\mbR^{2})}:\mbX_{\hom}(h)=\boldsymbol{s}\Bigr\},
	\end{equation*}
	where (cf.~Definition~\ref{def:hom_part}) the homogeneous part $\mbX_{\hom}(h)$ is given by
	\begin{equation*}
		\mbX_{\hom}(h)=\bigoplus_{\tau\in\mcW}(\mbX_{\tau,K_{\tau}})_{\hom}=\bigoplus_{\tau\in\mcW}\int_{\nban}\mbX_{\tau,K_{\tau}}(\xi+h)\gm(\dd\xi).
	\end{equation*}
	We can further apply Lemma~\ref{lem:hom_part_fourier} to identify the homogeneous part $\mbX_{\hom}(h)$ with the enhancement $\mbX^{h}$ driven by the Cameron--Martin element $h$ (cf.~Appendix~\ref{app:enh_driven_by_h}), which yields the rate function of the claim.
	
	The large deviation principle on $\rksnoise{\alpha}{\kappa}_{T}$ then follows by~\cite[Lem.~4.1.5~(b)]{dembo_zeitouni_10}, using that $\rksnoise{\alpha}{\kappa}_{T}\subset\bban$ is closed (by definition; cf.\ Subsection~\ref{subsec:notation}) and that $\mbX^{(\eps)}_{\delta(\eps)}\in\rksnoise{\alpha}{\kappa}_{T}$ almost surely for every $\eps>0$ (Theorem~\ref{thm:existence_enh_can}).
\end{proof}
We can now apply the contraction principle and Theorem~\ref{thm:LDP_enhancement} to deduce a large deviation principle for solutions to~\eqref{eq:rKS_intro_mild} in $(\mcC^{\alpha+1}(\mbT^{2}))^{\sol}_{T}$ for every $\alpha\in(-9/4,-2)$ under the relative scaling $\delta(\eps)\to0$ and $\eps\log(\delta(\eps)^{-1})\to0$ as $\eps\to0$.
\begin{theorem}\label{thm:LDP_rough}
	Let $\rho_{0}\in\mcH^{2}(\mbT^{2})$ be such that $\rho_{0}>0$ and $\mean{\rho_{0}}=1$, $\rdet$ be the weak solution to~\eqref{eq:Determ_KS} with initial data $\rho_{0}$ and chemotactic sensitivity $\chem\in\mbR$, and $\Trdet$ be its maximal time of existence (see Theorem~\ref{thm:Determ_KS_Well_Posedness}). Assume $\delta(\eps)\to0$ and $\eps\log(\delta(\eps)^{-1})\to0$ as $\eps\to0$. Then for all $T<\Trdet$ and $\alpha\in(-9/4,-2)$, it follows that $(\rho_{\delta(\eps)}^{(\eps)})_{\eps>0}$ satisfies a large deviation principle in $(\mcC^{\alpha+1}(\mbT^{2}))^{\sol}_{T}$ with speed $\eps$ and good rate function
	\begin{equation*}
		\rate(\rho)\defeq\inf\Bigl\{\frac{1}{2}\norm{h}_{L^{2}([0,T]\times\mbT^{2};\mbR^{2})}^{2}:\rho=P\rho_{0}-\chem\nabla\cdot\mcI[\rho\nabla\Phi_{\rho}]-\nabla\cdot\mcI[\srdet h]\Bigr\}.
	\end{equation*}
\end{theorem}
\begin{proof}
	An application of Theorem~\ref{thm:rKS_well_posedness_rough} shows that $\rho_{\delta(\eps)}^{(\eps)}\in(\mcC^{\alpha+1}(\mbT^{2}))^{\sol}_{T}$ is locally Lipschitz continuous in $\mbX^{(\eps)}_{\delta(\eps)}\in\rksnoise{\alpha}{\kappa}_{T}$, which we can combine with the contraction principle~\cite[Thm.~4.2.1]{dembo_zeitouni_10} and Theorem~\ref{thm:LDP_enhancement} to deduce the claim.
\end{proof}
As the following proposition shows, it is possible to remove the relative scaling in Theorem~\ref{thm:LDP_rough}, if the initial data $\rho_{0}$ is uniform, i.e.\ $\rho_{0}\equiv1$.
\begin{proposition}\label{prop:LDP_hom_IC}
	Let $\rho_{0}\equiv1$ and $\rdet$ be the weak solution to~\eqref{eq:Determ_KS} with initial data $\rho_{0}$ and chemotactic sensitivity $\chem\in\mbR$. Assume $\delta(\eps)\to0$ as $\eps\to0$. Then for all $T>0$ and $\alpha\in(-9/4,-2)$, it follows that $(\rho_{\delta(\eps)}^{(\eps)})_{\eps>0}$ satisfies a large deviation principle in $(\mcC^{\alpha+1}(\mbT^{2}))^{\sol}_{T}$ with speed $\eps$ and good rate function $\rate$.
\end{proposition}
\begin{proof}
	The deterministic Keller--Segel equation~\eqref{eq:Determ_KS} preserves uniform initial data, that is, $\rho_{0}\equiv1$ yields $\rdet\equiv1$ globally.
	\begin{details}
		By uniqueness, it suffices to show that $\rdet\equiv1$ solves~\eqref{eq:Determ_KS}. Indeed, it holds that
		\begin{equation*}
			0=\partial_{t}\rdet=\Delta\rdet-\chem\nabla\rdet\cdot\nabla\Phi_{\rdet}-\chem\rdet\Delta\Phi_{\rdet}=0,
		\end{equation*}
		where we used $\Delta\rdet\equiv0$, $\nabla\rdet\equiv0$ and $-\Delta\Phi_{\rdet}=\rdet-\inner{\rdet}{1}_{L^{2}}=0$ for $\rdet\equiv1$.
	\end{details}
	In particular, $\srdet\equiv1$ which by~\cite[Rem.~2.28]{martini_mayorcas_25} implies that for every $\delta\geq0$ and $p\in[1,\infty)$,
	\begin{equation}\label{eq:enhancement_bounds_homogeneous}
		\mbE[\norm{\ty^{\delta}_{\can}}_{\msL_{T}^{\kappa}\mcC^{2\alpha+4}}^{p}]^{1/p}\lesssim1,\qquad\mbE[\norm{\tp^{\delta}_{\can}}_{\msL_{T}^{\kappa}\mcC^{3\alpha+6}}^{p}]^{1/p}\lesssim1.
	\end{equation}
	The relative scaling of Theorem~\ref{thm:LDP_enhancement} (and, consequently, Theorem~\ref{thm:LDP_rough}) was chosen such that $\mbX^{(\eps)}_{\delta(\eps)}$ converges as $\eps\to0$ in the sense of~\eqref{eq:rv_convergence} to a limit that has no contributions from lower-order Wiener chaoses, \eqref{eq:rv_limit}. As demonstrated by~\eqref{eq:enhancement_bounds_homogeneous}, in the uniform case there is no such divergence as $\delta\to0$, hence we can choose $\delta(\eps)\to0$ arbitrarily.
\end{proof}
\begin{remark}
	In contrast to Proposition~\ref{prop:LDP_hom_IC}, which concerns the large deviations in the rough setting, one cannot improve the relative scaling in the regular setting (Theorem~\ref{thm:LDP_regular}) by choosing uniform initial data.
\end{remark}
\subsection{Controlling Blow-Ups}\label{subsec:blow_up}
In this subsection we apply the large deviation principle (Theorem~\ref{thm:LDP_rough}) to control the probability that $\rho^{(\eps)}_{\delta(\eps)}$ blows up before time $T<\Trdet$ (Theorem~\ref{thm:blow_up_time_estimate}). In particular, we show that the probability of observing a blow-up is asymptotically exponentially small in $\eps^{-1}$, a result that was inspired by~\cite[Rem.~4.6]{hairer_weber_15}.

We combine Lemma~\ref{lem:blow_up_lsc} and Theorem~\ref{thm:LDP_rough} to deduce a provisional large deviations upper bound for $(T^{\cem}_{\mcC^{\alpha+1}}[\rho^{(\eps)}_{\delta(\eps)}])_{\eps>0}$ for $\alpha\in(-9/4,-2)$, which will then be applied to prove the exponential unlikelihood of observing a blow-up in Theorem~\ref{thm:blow_up_time_estimate} below.
\begin{lemma}\label{lem:LDP_blow_up}
	Let $\rho_{0}\in\mcH^{2}(\mbT^{2})$ be such that $\rho_{0}>0$ and $\mean{\rho_{0}}=1$, $\rdet$ be the weak solution to~\eqref{eq:Determ_KS} with initial data $\rho_{0}$ and chemotactic sensitivity $\chem\in\mbR$, and $\Trdet$ be its maximal time of existence (see Theorem~\ref{thm:Determ_KS_Well_Posedness}). Assume $\delta(\eps)\to0$ and $\eps\log(\delta(\eps)^{-1})\to0$ as $\eps\to0$. Then for all $T<\Trdet$ and $\alpha\in(-9/4,-2)$, it follows that $(T^{\cem}_{\mcC^{\alpha+1}}[\rho^{(\eps)}_{\delta(\eps)}])_{\eps>0}$ satisfies for every $S\in[0,T]$,
	\begin{equation*}
		\limsup_{\eps\to0}\eps\log\mbP\Bigl(T^{\cem}_{\mcC^{\alpha+1}}[\rho^{(\eps)}_{\delta(\eps)}]\leq S\Bigr)\leq-\inf\Bigl\{\frac{1}{2}\norm{h}^{2}_{L^{2}([0,T]\times\mbT^{2};\mbR^{2})}:T^{\cem}_{\mcC^{\alpha+1}}[\rho^{h}]\leq S\Bigr\},
	\end{equation*}
	where $\rho^{h}$ is the solution to the mild equation $\rho^{h}\defeq P\rho_{0}-\chem\vdiv\mcI[\rho^{h}\nabla\Phi_{\rho^{h}}]-\nabla\cdot\mcI[\srdet h]$.
\end{lemma}
\begin{proof}
	It follows by the lower semicontinuity of $T^{\cem}_{\mcC^{\alpha+1}}$ (Lemma~\ref{lem:blow_up_lsc}) that the sublevel sets $\{\rho\in(\mcC^{\alpha+1}(\mbT^{2}))^{\sol}_{T}:T^{\cem}_{\mcC^{\alpha+1}}[\rho]\leq S\}$ are closed in $(\mcC^{\alpha+1}(\mbT^{2}))^{\sol}_{T}$, hence we can deduce from the large deviation principle of $(\rho^{(\eps)}_{\delta(\eps)})_{\eps>0}$ (Theorem~\ref{thm:LDP_rough}) the upper bound
	\begin{equation*}
		\limsup_{\eps\to0}\eps\log\mbP\Bigl(T^{\cem}_{\mcC^{\alpha+1}}[\rho^{(\eps)}_{\delta(\eps)}]\leq S\Bigr)\leq-\inf_{\substack{\rho\in(\mcC^{\alpha+1}(\mbT^{2}))^{\sol}_{T}\\T^{\cem}_{\mcC^{\alpha+1}}[\rho]\leq S}}\rate(\rho).
	\end{equation*}
	We can simplify the right-hand side 
	\begin{equation*}
		\begin{split}
			&\inf_{\substack{\rho\in(\mcC^{\alpha+1}(\mbT^{2}))^{\sol}_{T}\\T^{\cem}_{\mcC^{\alpha+1}}[\rho]\leq S}}\rate(\rho)=\inf_{\substack{\rho\in(\mcC^{\alpha+1}(\mbT^{2}))^{\sol}_{T}\\T^{\cem}_{\mcC^{\alpha+1}}[\rho]\leq S}}\inf\Bigl\{\frac{1}{2}\norm{h}^{2}_{L^{2}([0,T]\times\mbT^{2};\mbR^{2})}:\rho=P\rho_{0}-\chem\vdiv\mcI[\rho\nabla\Phi_{\rho}]-\vdiv\mcI[\srdet h]\Bigr\}\\
			&\qquad=\inf\Bigl\{\frac{1}{2}\norm{h}^{2}_{L^{2}([0,T]\times\mbT^{2};\mbR^{2})}:T^{\cem}_{\mcC^{\alpha+1}}[\rho^{h}]\leq S,~\rho^{h}= P\rho_{0}-\chem\vdiv\mcI[\rho^{h}\nabla\Phi_{\rho^{h}}]-\vdiv\mcI[\srdet h]\Bigr\},
		\end{split}
	\end{equation*}
	which yields the claim.
\end{proof}
We can now use Lemma~\ref{lem:LDP_blow_up} to deduce that solutions $\rho^{(\eps)}_{\delta(\eps)}$ to~\eqref{eq:rKS_intro_mild} blow up before time $T<\Trdet$ with exponentially small probability in $\eps$.
\begin{theorem}\label{thm:blow_up_time_estimate}
	Let $\rho_{0}\in\mcH^{2}(\mbT^{2})$ be such that $\rho_{0}>0$ and $\mean{\rho_{0}}=1$, $\rdet$ be the weak solution to~\eqref{eq:Determ_KS} with initial data $\rho_{0}$ and chemotactic sensitivity $\chem\in\mbR$, and $\Trdet$ be its maximal time of existence (see Theorem~\ref{thm:Determ_KS_Well_Posedness}). Assume $\delta(\eps)\to0$ and $\eps\log(\delta(\eps)^{-1})\to0$ as $\eps\to0$. Then for all $S\leq T<\Trdet$ and $\alpha\in(-9/4,-2)$, there exists a constant $c>0$ such that
	\begin{equation*}
		\limsup_{\eps\to0}\eps\log\mbP\Bigl(T^{\cem}_{\mcC^{\alpha+1}}[\rho^{(\eps)}_{\delta(\eps)}]\leq S\Bigr)\leq-c.
	\end{equation*}
\end{theorem}
\begin{proof}
	Assume we are given constant $c>0$ with the following property: All $h\in L^{2}([0,T]\times\mbT^{2};\mbR^{2})$ such that $T^{\cem}_{\mcC^{\alpha+1}}[\rho^{h}]\leq S$ satisfy $\norm{h}^{2}_{L^{2}([0,T]\times\mbT^{2};\mbR^{2})}\geq2c$. We can then apply Lemma~\ref{lem:LDP_blow_up} to deduce
	\begin{equation*}
		\limsup_{\eps\to0}\eps\log\mbP\Bigl(T^{\cem}_{\mcC^{\alpha+1}}[\rho^{(\eps)}_{\delta(\eps)}]\leq S\Bigr)\leq-\inf\Bigl\{\frac{1}{2}\norm{h}^{2}_{L^{2}([0,T]\times\mbT^{2};\mbR^{2})}:T^{\cem}_{\mcC^{\alpha+1}}[\rho^{h}]\leq S\Bigr\}\leq-c,
	\end{equation*}
	which yields the claim. Hence it suffices to find such a $c$.
	
	Using the lower semicontinuity of $T^{\cem}_{\mcC^{\alpha+1}}$ (Lemma~\ref{lem:blow_up_lsc}), it follows that for all $S\leq T$ there exists a constant $\nu>0$ such that $D_{T}^{\mcC^{\alpha+1}}(\rho,\rdet)<\nu$ implies $S<T^{\cem}_{\mcC^{\alpha+1}}[\rho]$. 
	\begin{details}
		The Besov embedding $L^{2}(\mbT^{2})=\mcB_{2,2}^{0}(\mbT^{2})\embed\mcC^{-1}(\mbT^{2})\embed\mcC^{\alpha+1}(\mbT^{2})$ yields $S\leq T<\Trdet\leq T^{\cem}_{\mcC^{\alpha+1}}[\rdet]$.
	\end{details}
	Using the local Lipschitz continuity of the solution map in the noise enhancement (Theorem~\ref{thm:rKS_well_posedness_rough}), we deduce that for every $\nu>0$ there exists an $R>0$ such that $\norm{\mbX^{h}}_{\rksnoise{\alpha}{\kappa}_{T}}<R$ implies $D_{T}^{\mcC^{\alpha+1}}(\rho^{h},\rdet)<\nu$, where $\mbX^{h}$ denotes the enhancement driven by the Cameron--Martin element $h$ (cf.\ Lemma~\ref{lem:existence_enhancement_h}). Hence, combining both we deduce that for every $S\leq T$ there exists an $R>0$ such that $\norm{\mbX^{h}}_{\rksnoise{\alpha}{\kappa}_{T}}<R$ implies $S<T^{\cem}_{\mcC^{\alpha+1}}[\rho^{h}]$. 
	
	By an application of Lemma~\ref{lem:existence_enhancement_h}, we can find for every $R>0$ a constant $c>0$ such that $\norm{h}^{2}_{L^{2}([0,T]\times\mbT^{2};\mbR^{2})}<2c$ implies $\norm{\mbX^{h}}_{\rksnoise{\alpha}{\kappa}_{T}}<R$.
	\begin{details}
		By Lemma~\ref{lem:existence_enhancement_h}, we can bound for all $h\in L^{2}([0,T]\times\mbT^{2};\mbR^{2})$,
		\begin{equation*}
			\norm{\mbX^{h}}_{\rksnoise{\alpha}{\kappa}_{T}}\lesssim\norm{\srdet}_{C_{T}L^{\infty}}\norm{h}_{L^{2}([0,T]\times\mbT^{2};\mbR^{2})}(1\vee\norm{\srdet}^{2}_{C_{T}L^{\infty}}\norm{h}^{2}_{L^{2}([0,T]\times\mbT^{2};\mbR^{2})}).
		\end{equation*}
		Hence for all $R>0$, we can find some $c>0$ such that $\norm{h}^{2}_{L^{2}([0,T]\times\mbT^{2};\mbR^{2})}<2c$ implies $\norm{\mbX^{h}}_{\rksnoise{\alpha}{\kappa}_{T}}<R$.
	\end{details}
	Consequently, for every $S\leq T$ we can find a constant $c>0$ such that $\norm{h}^{2}_{L^{2}([0,T]\times\mbT^{2};\mbR^{2})}<2c$ implies $S<T^{\cem}_{\mcC^{\alpha+1}}[\rho^{h}]$, which yields the claim.
\end{proof}
\appendix
\section{Analytic Toolbox}\label{app:toolbox}
In this appendix we collect general, useful results which are applied throughout. In particular, the results of this appendix are agnostic as to the dimension $d\in\mbN$, which we keep arbitrary but fixed throughout. The definitions introduced in Subsection~\ref{subsec:notation} generalize to $d$ dimensions \emph{mutatis mutandis}.
\subsection{Bounds and Embeddings}\label{subsec:bounds}
We apply the following Sobolev and Gagliardo--Nirenberg--Sobolev inequalities in the two-dimensional case to prove Theorem~\ref{thm:Determ_KS_Well_Posedness}.
\begin{lemma}[Sobolev's inequality]\label{lem:Sobolev_inequality}
	The following hold:
	\begin{enumerate}
		\item Let $s>0$ and $2<q<\infty$ be such that $\frac{s}{d}\geq\frac{1}{2}-\frac{1}{q}$; then $\mcH^{s}(\mbT^{d})\embed L^{q}(\mbT^{d})$ and
		\begin{equation}\label{eq:Sobolev_embedding_I}
			\norm{u}_{L^{q}(\mbT^{d})}\lesssim\norm{u}_{\mcH^{s}(\mbT^{d})}.
		\end{equation}
		\item Let $s>d/2$; then $\mcH^{s}(\mbT^{d})\embed C(\mbT^{d})$ and
		\begin{equation}\label{eq:Sobolev_embedding_II}
			\norm{u}_{L^{\infty}(\mbT^{d})}\lesssim\norm{u}_{\mcH^{s}(\mbT^{d})}.
		\end{equation}
	\end{enumerate}
\end{lemma}
\begin{proof}
	See~\cite[Cor.~1.2~\&~Subsec.~2.1]{benyi_oh_13}.
\end{proof}
A combination of Sobolev’s inequality (Lemma~\ref{lem:Sobolev_inequality}) with the Gagliardo--Nirenberg interpolation estimate~\cite{brezis_mironescu_18} yields the following Gagliardo--Nirenberg--Sobolev inequality, see~\cite{brezis_mironescu_19}. 
\begin{lemma}\label{lem:GNS}
	Let $q,r\in[1,\infty]$; then given $u\in\mcH^{1}(\mbT^{d})\cap L^{r}(\mbT^{d})$ such that $\mean{u}=0$ and $\theta\in(0,1)$ satisfying
	\begin{equation*}
		\theta\Bigl(\frac{1}{2}-\frac{1}{d}\Bigl)+\frac{1-\theta}{r}=\frac{1}{q},
	\end{equation*}
	there exists a constant $C\defeq C(d,q,r)>0$ such that
	\begin{equation}\label{eq:GNS_Per}
		\norm{u}_{L^{q}(\mbT^{d})}\leq C\norm{\nabla u}^{\theta}_{L^{2}(\mbT^{d})}\norm{u}^{1-\theta}_{L^{r}(\mbT^{d})}.
	\end{equation}
	We define $C_{\gnsper}(d,q,r)$ to be the optimal constant such that~\eqref{eq:GNS_Per} holds.
\end{lemma}
\begin{proof}
	The result follows by an application of~\cite[Cor.~1]{brezis_mironescu_19} upon passing between $(0,1)^2$ and $\mbT^{2}$. 
\end{proof}
\begin{details}
	\begin{proof}
		Let $u\in\mcH^{1}(\mbT^{d})\cap L^{r}(\mbT^{d})$. Then upon passing to the full space, $u$ is a periodic element of $W^{1,2}_{\loc}(\mbR^{d})\cap L^{r}_{\loc}(\mbR^{d})$, where the subspace of periodic functions in $W^{1,2}_{\loc}(\mbR^{d})$ (resp.\ $L^{r}_{\loc}(\mbR^{d})$) is equipped with the norm $\norm{\place}_{W^{1,2}((0,1)^{d})}$ (resp.\ $\norm{\place}_{L^{r}((0,1)^{d})}$).
		(Periodic in the sense of Definition~\ref{def:periodic_distribution}, see~\cite[Sec.~2.1]{temam_95} and~\cite[Lem.~30.9]{vanzuijlen_22}.)
		
		Let $q,r\in[1,\infty]$ and $\theta\in(0,1)$ satisfy~\eqref{eq:GNS_def_theta}. We apply the Gagliardo--Nirenberg--Sobolev inequality on $(0,1)^{d}$ (see e.g.~\cite[Cor.~1]{brezis_mironescu_19}) to deduce
		\begin{equation*}
			\norm{u}_{L^{q}((0,1)^{d})}\lesssim\norm{u}_{W^{1,2}((0,1)^{d})}^{\theta}\norm{u}_{L^{r}((0,1)^{d})}^{1-\theta}.
		\end{equation*}
		To apply~\cite[Cor.~1]{brezis_mironescu_19}, we need to exclude the exceptions~\cite[Cor.~1, 1.]{brezis_mironescu_19} and~\cite[Cor.~1, 2.]{brezis_mironescu_19}. Exception 1 is already excluded since $p_{2}=2\neq1$ (in the notation of~\cite{brezis_mironescu_19}.) Excluding exception 2 is equivalent to assuming $d\neq2$ or $q<\infty$ or $r<\infty$. Hence assume $d=2$, then the assumption $\theta=1-\frac{r}{q}\in(0,1)$ already implies $r<q<\infty$, which is stronger than exception 2.
		
		Passing back to the torus, we obtain
		\begin{equation*}
			\norm{u}_{L^{q}(\mbT^{d})}\lesssim\norm{u}_{\mcH^{1}(\mbT^{d})}^{\theta}\norm{u}_{L^{r}(\mbT^{d})}^{1-\theta}.
		\end{equation*}
		Using that $u$ is mean-free, the claim follows from the Poincar\'{e} inequality on $\mbT^{d}$.
		\begin{equation*}
			\norm{u}_{\mcH^{1}(\mbT^{d})}^{2}=\sum_{\om\in\mbZ^{d}\setminus\{0\}}(1+\abs{2\uppi\om}^{2})\abs{\hat{u}(\om)}^{2}\leq2\sum_{\om\in\mbZ^{d}\setminus\{0\}}\abs{2\uppi\om}^{2}\abs{\hat{u}(\om)}^{2}=2\norm{\nabla u}_{L^{2}(\mbT^{d})}^{2}
		\end{equation*}
	\end{proof}
\end{details}
\begin{details}
	The following product estimates are applied throughout Section~\ref{sec:det_PDE} and~\ref{sec:regular_theory}.
	\begin{lemma}\label{lem:product_estimates}
		It holds that:
		\begin{enumerate}
			\item Let $\gamma\geq0$, $f\in\mcH^{\gamma+1}(\mbT^{2})$ and $g\in\mcH^{\gamma+1}(\mbT^{2})$, then
			\begin{equation}\label{eq:product_estimate_Sobolev_I}
				\norm{fg}_{\mcH^{\gamma}}\lesssim\norm{f}_{\mcH^{\gamma+1}}\norm{g}_{\mcH^{\gamma+1}}.
			\end{equation}
			\item Let $\gamma>-1$, $f\in\mcH^{\gamma}(\mbT^{2})$ and $g\in\mcH^{\gamma+2}(\mbT^{2})$, then
			\begin{equation}\label{eq:product_estimate_Sobolev_II}
				\norm{fg}_{\mcH^{\gamma}}\lesssim\norm{f}_{\mcH^{\gamma}}\norm{g}_{\mcH^{\gamma+2}}.
			\end{equation}
			\item Let $\gamma>0$, $f\in\mcH^{\gamma}(\mbT^{2})$ and $g\in\mcH^{\gamma+1}(\mbT^{2})$, then
			\begin{equation}\label{eq:product_estimate_Sobolev_III}
				\norm{fg}_{\mcH^{\gamma}}\lesssim\norm{f}_{\mcH^{\gamma}}\norm{g}_{\mcH^{\gamma+1}}.
			\end{equation}
			\item Let $\vartheta>0$, $f\in L^{2}(\mbT^{2})$ and $g\in\mcH^{1}(\mbT^{2})$, then
			\begin{equation}\label{eq:product_estimate_Sobolev_IV}
				\norm{fg}_{\mcH^{-\vartheta}}\lesssim\norm{f}_{L^{2}}\norm{g}_{\mcH^{1}}.
			\end{equation}
			\item Let $\gamma\in(-1/2,0)$, $f\in\mcH^{\gamma}(\mbT^{2})$ and $g\in\mcH^{\gamma+1}(\mbT^{2})$, then
			\begin{equation}\label{eq:product_estimate_Sobolev_V}
				\norm{fg}_{\mcH^{2\gamma}}\lesssim\norm{f}_{\mcH^{\gamma}}\norm{g}_{\mcH^{\gamma+1}}.
			\end{equation}
			\item Let $\vartheta>0$, $f\in L^{\infty}(\mbT^{2})$ and $g\in\mcH^{1}(\mbT^{2})$, then
			\begin{equation}\label{eq:product_estimate_VI}
				\norm{fg}_{\mcB_{\infty,1}^{-\vartheta}}\lesssim\norm{f}_{L^{\infty}}\norm{g}_{\mcH^{1}}
			\end{equation}
			\item Let $\gamma\in(-1,0]$, $\gamma'\in(-1,\gamma)$, $f\in\mcH^{\gamma}(\mbT^{2})$ and $g\in\mcH^{\gamma'+2}(\mbT^{2})$, then
			\begin{equation}\label{eq:product_estimate_Sobolev_VII}
				\norm{fg}_{\mcH^{\gamma}}\lesssim\norm{f}_{\mcH^{\gamma}}\norm{g}_{\mcH^{\gamma'+2}}.
			\end{equation}
		\end{enumerate}
	\end{lemma}
	\begin{proof}
		Let $\gamma\geq0$, the estimate~\eqref{eq:product_estimate_Sobolev_I} follows from~\cite[Thm.~7.4]{behzadan_holst_21} using the embedding $\mcH^{\gamma}(\mbT^{2})\embed W^{\gamma,2}((0,1)^{2})$ (see~\cite[Sec.~2.1]{temam_95} and~\cite[Prop.~1.4]{benyi_oh_13}); or alternatively from the paraproduct estimates~\cite[Thm.~27.5~\&~27.10]{vanzuijlen_22},
		\begin{equation*}
			\begin{split}
				&\norm{f\pa g}_{\mcB_{2,2}^{\gamma}}\lesssim\norm{f}_{\mcB_{\infty,\infty}^{-1}}\norm{g}_{\mcB_{2,2}^{\gamma+1}}\lesssim\norm{f}_{\mcB_{\infty,2}^{\gamma}}\norm{g}_{\mcB_{2,2}^{\gamma+1}}\lesssim\norm{f}_{\mcB_{2,2}^{\gamma+1}}\norm{g}_{\mcB_{2,2}^{\gamma+1}},\\
				&\norm{f\re g}_{\mcB_{2,2}^{\gamma}}\lesssim\norm{f\re g}_{\mcB_{2,2}^{2\gamma+1}}\lesssim\norm{f\re g}_{\mcB_{1,1}^{2\gamma+2}}\lesssim\norm{f}_{\mcB_{2,2}^{\gamma+1}}\norm{g}_{\mcB_{2,2}^{\gamma+1}}.
			\end{split}
		\end{equation*}
		Let $\gamma>-1$, the estimate~\eqref{eq:product_estimate_Sobolev_II} follows from the paraproduct estimates~\cite[Thm.~27.5~\&~27.10]{vanzuijlen_22},
		\begin{equation*}
			\begin{split}
				&\norm{f\pa g}_{\mcB_{2,2}^{\gamma}}\lesssim\norm{f}_{\mcB_{\infty,\infty}^{-2}}\norm{g}_{\mcB_{2,2}^{\gamma+2}}\lesssim\norm{f}_{\mcB_{\infty,2}^{\gamma-1}}\norm{g}_{\mcB_{2,2}^{\gamma+2}}\lesssim\norm{f}_{\mcB_{2,2}^{\gamma}}\norm{g}_{\mcB_{2,2}^{\gamma+2}},\\
				&\norm{g\pa f}_{\mcB_{2,2}^{\gamma}}\lesssim\norm{g}_{L^{\infty}}\norm{f}_{\mcB_{2,2}^{\gamma}}\lesssim\norm{g}_{\mcB_{\infty,1}^{0}}\norm{f}_{\mcB_{2,2}^{\gamma}}\lesssim\norm{g}_{\mcB_{\infty,2}^{\gamma+1}}\norm{f}_{\mcB_{2,2}^{\gamma}}\lesssim\norm{g}_{\mcB_{2,2}^{\gamma+2}}\norm{f}_{\mcB_{2,2}^{\gamma}},\\
				&\norm{f\re g}_{\mcB_{2,2}^{\gamma}}\lesssim\norm{f\re g}_{\mcB_{2,2}^{2\gamma+1}}\lesssim\norm{f\re g}_{\mcB_{1,1}^{2\gamma+2}}\lesssim\norm{f}_{\mcB_{2,2}^{\gamma}}\norm{g}_{\mcB_{2,2}^{\gamma+2}}.
			\end{split}
		\end{equation*}
		Let $\gamma>0$, the estimate~\eqref{eq:product_estimate_Sobolev_III} follows from~\cite[Thm.~7.4]{behzadan_holst_21}; or alternatively from the paraproduct estimates~\cite[Thm.~27.5~\&~27.10]{vanzuijlen_22},
		\begin{equation*}
			\begin{split}
				&\norm{f\pa g}_{\mcB_{2,2}^{\gamma}}\lesssim\norm{f}_{L^{2}}\norm{g}_{\mcB_{\infty,2}^{\gamma}}\lesssim\norm{f}_{\mcB_{2,2}^{\gamma}}\norm{g}_{\mcB_{2,2}^{\gamma+1}},\\
				&\norm{g\pa f}_{\mcB_{2,2}^{\gamma}}\lesssim\norm{g}_{L^{\infty}}\norm{f}_{\mcB_{2,2}^{\gamma}}\lesssim\norm{g}_{\mcB_{\infty,1}^{0}}\norm{f}_{\mcB_{2,2}^{\gamma}}\lesssim\norm{g}_{\mcB_{\infty,2}^{\gamma}}\norm{f}_{\mcB_{2,2}^{\gamma}}\lesssim\norm{g}_{\mcB_{2,2}^{\gamma+1}}\norm{f}_{\mcB_{2,2}^{\gamma}},\\
				&\norm{f\re g}_{\mcB_{2,2}^{\gamma}}\lesssim\norm{f\re g}_{\mcB_{2,1}^{2\gamma}}\lesssim\norm{f}_{\mcB_{2,2}^{\gamma}}\norm{g}_{\mcB_{\infty,2}^{\gamma}}\lesssim\norm{f}_{\mcB_{2,2}^{\gamma}}\norm{g}_{\mcB_{2,2}^{\gamma+1}}.
			\end{split}
		\end{equation*}
		Let $\vartheta>0$, the estimate~\eqref{eq:product_estimate_Sobolev_IV} follows from the paraproduct estimates~\cite[Thm.~27.5~\&~27.10]{vanzuijlen_22},
		\begin{equation*}
			\begin{split}
				&\norm{f\pa g}_{\mcB_{2,2}^{-\vartheta}}\lesssim\norm{f\pa g}_{\mcB_{2,2}^{0}}\lesssim\norm{f\pa g}_{\mcB_{1,2}^{1}}\lesssim\norm{f}_{L^{2}}\norm{g}_{\mcB_{2,2}^{1}},\\
				&\norm{g\pa f}_{\mcB_{2,2}^{-\vartheta}}\lesssim\norm{g}_{\mcB_{\infty,\infty}^{-\vartheta}}\norm{f}_{\mcB_{2,2}^{0}}\lesssim\norm{g}_{\mcB_{\infty,2}^{0}}\norm{f}_{\mcB_{2,2}^{0}}\lesssim\norm{g}_{\mcB_{2,2}^{1}}\norm{f}_{\mcB_{2,2}^{0}},\\
				&\norm{f\re g}_{\mcB_{2,2}^{0}}\lesssim\norm{f\re g}_{\mcB_{1,1}^{1}}\lesssim\norm{f}_{\mcB_{2,2}^{0}}\norm{g}_{\mcB_{2,2}^{1}}.
			\end{split}
		\end{equation*}
		Let $\gamma\in(-1/2,0)$, the estimate~\eqref{eq:product_estimate_Sobolev_V} follows by the paraproduct estimates~\cite[Thm.~27.5~\&~27.10]{vanzuijlen_22},
		\begin{equation*}
			\begin{split}
				&\norm{f\pa g}_{\mcB_{2,2}^{2\gamma}}\lesssim\norm{f\pa g}_{\mcB_{1,1}^{2\gamma+1}}\lesssim\norm{f}_{\mcB_{2,2}^{\gamma}}\norm{g}_{\mcB_{2,2}^{\gamma+1}},\\
				&\norm{g\pa f}_{\mcB_{2,2}^{2\gamma}}\lesssim\norm{g\pa f}_{\mcB_{2,1}^{2\gamma}}\lesssim\norm{g}_{\mcB_{\infty,2}^{\gamma}}\norm{f}_{\mcB_{2,2}^{\gamma}}\lesssim\norm{g}_{\mcB_{2,2}^{\gamma+1}}\norm{f}_{\mcB_{2,2}^{\gamma}},\\
				&\norm{f\re g}_{\mcB_{2,2}^{2\gamma}}\lesssim\norm{f\re g}_{\mcB_{1,1}^{2\gamma+1}}\lesssim\norm{f}_{\mcB_{2,2}^{\gamma}}\norm{g}_{\mcB_{2,2}^{\gamma+1}}.
			\end{split}
		\end{equation*}
		Let $\vartheta>0$, the estimate~\eqref{eq:product_estimate_VI} follows from the paraproduct estimates~\cite[Thm.~27.5~\&~27.10]{vanzuijlen_22},
		\begin{equation*}
			\begin{split}
				&\norm{f\pa g}_{\mcB_{\infty,1}^{-\vartheta}}\lesssim\norm{f\pa g}_{\mcB_{\infty,2}^{0}}\lesssim\norm{f\pa f}_{\mcB_{2,2}^{1}}\lesssim\norm{f}_{L^{\infty}}\norm{g}_{\mcB_{2,2}^{1}},\\
				&\norm{g\pa f}_{\mcB_{\infty,1}^{-\vartheta}}\lesssim\norm{g}_{\mcB_{\infty,1}^{-\vartheta}}\norm{f}_{\mcB_{\infty,\infty}^{0}}\lesssim\norm{g}_{\mcB_{\infty,2}^{0}}\norm{f}_{\mcB_{\infty,\infty}^{0}}\lesssim\norm{g}_{\mcB_{2,2}^{1}}\norm{f}_{L^{\infty}},\\
				&\norm{f\re g}_{\mcB_{\infty,1}^{-\vartheta}}\lesssim\norm{f\re g}_{\mcB_{\infty,2}^{0}}\lesssim\norm{f\re g}_{\mcB_{2,2}^{1}}\lesssim\norm{f}_{\mcB_{\infty,\infty}^{0}}\norm{g}_{\mcB_{2,2}^{1}}\lesssim\norm{f}_{L^{\infty}}\norm{g}_{\mcB_{2,2}^{1}}.
			\end{split}
		\end{equation*}
		Let $\gamma\in(-1,0]$ and $\gamma'\in(-1,\gamma)$, the estimate~\eqref{eq:product_estimate_Sobolev_VII} follows from the paraproduct estimates~\cite[Thm.~27.5~\&~27.10]{vanzuijlen_22},
		\begin{equation*}
			\begin{split}
				&\norm{f\pa g}_{\mcB_{2,2}^{\gamma}}\lesssim\norm{f\pa g}_{\mcB_{2,2}^{\gamma+\gamma'+1}}\lesssim\norm{f\pa g}_{\mcB_{1,2}^{\gamma+\gamma'+2}}\lesssim\norm{f}_{\mcB_{2,2}^{\gamma}}\norm{g}_{\mcB_{2,2}^{\gamma'+2}},\\
				&\norm{g\pa f}_{\mcB_{2,2}^{\gamma}}\lesssim\norm{g}_{L^{\infty}}\norm{f}_{\mcB_{2,2}^{\gamma}}\lesssim\norm{g}_{\mcB_{\infty,1}^{0}}\norm{f}_{\mcB_{2,2}^{\gamma}}\lesssim\norm{g}_{\mcB_{\infty,2}^{\gamma'+1}}\norm{f}_{\mcB_{2,2}^{\gamma}}\lesssim\norm{g}_{\mcB_{2,2}^{\gamma'+2}}\norm{f}_{\mcB_{2,2}^{\gamma}},\\
				&\norm{f\re g}_{\mcB_{2,2}^{\gamma}}\lesssim\norm{f\re g}_{\mcB_{2,1}^{\gamma+\gamma'+1}}\lesssim\norm{f\re g}_{\mcB_{1,1}^{\gamma+\gamma'+2}}\lesssim\norm{f}_{\mcB_{2,2}^{\gamma}}\norm{g}_{\mcB_{2,2}^{\gamma'+2}}.
			\end{split}
		\end{equation*}
		This yields the claim.
	\end{proof}
\end{details}
Recall from Subsection~\ref{subsec:notation} the notation for the heat semigroup,
\begin{equation*}
		P_{t} f=\msF^{-1}(\euler^{-t\abs{2\uppi\place}^{2}}\hat{f}(\place)),\qquad \text{for every}~t\geq0,\,f\in\mcS'(\mbT^{2})
\end{equation*}
and the resolution of the heat equation,
\begin{equation*}
	\mcI[f]_{t}=\int_{0}^{t}P_{t-s}f_s\dd s,\qquad \text{for every}~T>0,\,t\in[0,T],\,f\from[0,T]\to\mcS'(\mbT^{2}).
\end{equation*}
The following Schauder estimates are applied in the proof of Lemma~\ref{lem:rKS_well_posedness_regular}.
\begin{lemma}\label{lem:Schauder_Sobolev}
	Let $T>0$ and $\gamma\in\mbR$, then the following hold:
	\begin{enumerate}
		\item\label{it:Schauder_Sobolev_I}
		The map $P\from\mcH^{\gamma}(\mbT^{d})\to C_{T}\mcH^{\gamma}(\mbT^{d})\cap L_{T}^{2}\mcH^{\gamma+1}(\mbT^{d})$ is continuous and
		\begin{equation*}
			\norm{Pf}_{C_{T}\mcH^{\gamma}\cap L_{T}^{2}\mcH^{\gamma+1}}\lesssim_{T}\norm{f}_{\mcH^{\gamma}}.
		\end{equation*}
		\item\label{it:Schauder_Sobolev_II}
		If $\gamma'\leq\gamma<\gamma'+2$, then the map $\mcI\from C_{T}\mcH^{\gamma'}(\mbT^{d})\to C_{T}\mcH^{\gamma}(\mbT^{d})$ is continuous and
		\begin{equation*}
			\norm{\mcI[f]}_{C_{T}\mcH^{\gamma}}\lesssim(T+T^{1-\frac{\gamma-\gamma'}{2}})\norm{f}_{C_{T}\mcH^{\gamma'}}.
		\end{equation*}
		\item\label{it:Schauder_Sobolev_III}
		If $\gamma'\leq\gamma\leq\gamma'+2$, then the map $\mcI\from L_{T}^{2}\mcH^{\gamma'}(\mbT^{d})\to L_{T}^{2}\mcH^{\gamma}(\mbT^{d})$ is continuous and
		\begin{equation*}
			\norm{\mcI[f]}_{L_{T}^{2}\mcH^{\gamma}}\lesssim(T+T^{1-\frac{\gamma-\gamma'}{2}})\norm{f}_{L_{T}^{2}\mcH^{\gamma'}}.
		\end{equation*}
	\end{enumerate}
\end{lemma}
\begin{proof}
	To show that $P$ and $\mcI$ are continuous maps into $C_{T}\mcH^{\gamma}(\mbT^{d})$, it suffices to argue as in the proof of~\cite[Lem.~A.6]{martini_mayorcas_25}. To establish their regularities in $L_{T}^{2}\mcH^{\gamma+1}(\mbT^{d})$ and $L_{T}^{2}\mcH^{\gamma}(\mbT^{d})$ respectively, one can use the Fourier transform (in space), Young's convolution inequality and interpolation.
	\begin{details}
		\paragraph{Proof of Claim~\ref{it:Schauder_Sobolev_I}.}
		Let $f\in\mcH^{\gamma}(\mbT^{d})$, we obtain by an application of~\cite[Lem.~A.5]{martini_mayorcas_25},
		\begin{equation*}
			\norm{Pf}_{C_{T}\mcH^{\gamma}}\lesssim\norm{f}_{\mcH^{\gamma}}.
		\end{equation*}
		Hence to show that $P\from\mcH^{\gamma}(\mbT^{d})\to C_{T}\mcH^{\gamma}(\mbT^{d})$ is continuous, it suffices to establish that $Pf\in C_{T}\mcH^{\gamma}(\mbT^{d})$. Using that the smooth functions are dense in $\mcH^{\gamma}(\mbT^{d})$, there exists a sequence $(f_{n})_{n\in\mbN}$ of approximations such that $f_{n}\in C^{\infty}(\mbT^{d})$ for every $n\in\mbN$ and $f_{n}\to f\in\mcH^{\gamma}(\mbT^{d})$ as $n\to\infty$. Let $\kappa\in[0,1]$, it follows by~\cite[Lem.~A.5]{martini_mayorcas_25} for every $s<t\in[0,T]$, 
		\begin{equation*}
			\norm{P_{s}f_{n}-P_{t}f_{n}}_{\mcH^{\gamma}}=\norm{P_{s}(1-P_{t-s})f_{n}}_{\mcH^{\gamma}}\lesssim\norm{(1-P_{t-s})f_{n}}_{\mcH^{\gamma}}\lesssim\abs{t-s}^{\kappa}\norm{f_{n}}_{\mcH^{\gamma+2\kappa}},
		\end{equation*}
		which yields
		\begin{equation*}
			\norm{Pf_{n}}_{C_{T}^{\kappa}\mcH^{\gamma}}\lesssim\norm{f_{n}}_{\mcH^{\gamma+2\kappa}}.
		\end{equation*}
		Therefore, $Pf_{n}\in C_{T}\mcH^{\gamma}(\mbT^{d})$ and $Pf_{n}\to Pf\in C_{T}\mcH^{\gamma}(\mbT^{d})$, from which we obtain the continuity of $Pf$ by the completeness of $C_{T}\mcH^{\gamma}(\mbT^{d})$.
		
		To establish $Pf\in L^{2}_{T}\mcH^{\gamma+1}(\mbT^{d})$, we estimate using the Fourier transform,
		\begin{equation*}
			\begin{split}
				\norm{Pf}_{L^{2}_{T}\mcH^{\gamma+1}}^{2}&=\int_{0}^{T}\sum_{\om\in\mbZ^{d}}(1+\abs{2\uppi\om}^{2})^{\gamma+1}\euler^{-2t\abs{2\uppi\om}^{2}}\abs{\hat{f}(\om)}^{2}\dd t\\
				&=T\abs{\hat{f}(0)}^{2}+\sum_{\om\in\mbZ^{d}\setminus\{0\}}(1+\abs{2\uppi\om}^{2})^{\gamma+1}\abs{\hat{f}(\om)}^{2}\int_{0}^{T}\euler^{-2t\abs{2\uppi\om}^{2}}\dd t\\
				&\lesssim T\abs{\hat{f}(0)}^{2}+\sum_{\om\in\mbZ^{d}\setminus\{0\}}(1+\abs{2\uppi\om}^{2})^{\gamma}\abs{\hat{f}(\om)}^{2}\\
				&\lesssim(1+T)\norm{f}_{\mcH^{\gamma}}^{2}.
			\end{split}
		\end{equation*}
		\paragraph{Proof of Claim~\ref{it:Schauder_Sobolev_II}.}
		Let $\gamma'\leq\gamma<\gamma'+2$, we estimate for $t\in[0,T]$ by~\cite[Lem.~A.5]{martini_mayorcas_25},
		\begin{equation*}
			\begin{split}
				\norm{\mcI[f]_{t}}_{\mcH^{\gamma}}=\int_{0}^{t}\norm{P_{t-s}f(s)}_{\mcH^{\gamma}}\dd s&\lesssim\int_{0}^{t}(1\vee \abs{t-s}^{-\frac{\gamma-\gamma'}{2}})\norm{f(s)}_{\mcH^{\gamma'}}\dd s\\
				&\lesssim\norm{f}_{C_{T}\mcH^{\gamma'}}\int_{0}^{t}(1\vee \abs{t-s}^{-\frac{\gamma-\gamma'}{2}})\dd s.
			\end{split}
		\end{equation*}
		The singularity is integrable since $\gamma<\gamma'+2$, i.e.\ if $t\in[0,1]$,
		\begin{equation*}
			\int_{0}^{t}(1\vee \abs{t-s}^{-\frac{\gamma-\gamma'}{2}})\dd s=\int_{0}^{t}\abs{t-s}^{-\frac{\gamma-\gamma'}{2}}\dd s=t^{1-\frac{\gamma-\gamma'}{2}}\int_{0}^{1}\abs{1-s}^{-\frac{\gamma-\gamma'}{2}}\dd s\lesssim t^{1-\frac{\gamma-\gamma'}{2}}
		\end{equation*}
		and if $t\geq1$,
		\begin{equation*}
			\int_{0}^{t}(1\vee \abs{t-s}^{-\frac{\gamma-\gamma'}{2}})\dd s=\int_{0}^{1}\abs{t-s}^{-\frac{\gamma-\gamma'}{2}}\dd s+\int_{1}^{t}\dd s\lesssim t^{1-\frac{\gamma-\gamma'}{2}}+t.
		\end{equation*}
		Therefore, we obtain
		\begin{equation*}
			\norm{\mcI[f]}_{C_{T}\mcH^{\gamma}}\lesssim(T+T^{1-\frac{\gamma-\gamma'}{2}})\norm{f}_{C_{T}\mcH^{\gamma'}}.
		\end{equation*}
		Hence to show that $\mcI\from C_{T}\mcH^{\gamma'}(\mbT^{d})\to C_{T}\mcH^{\gamma}(\mbT^{d})$ is continuous, it suffices to establish that $\mcI[f]\in C_{T}\mcH^{\gamma}(\mbT^{d})$. Using the density of $C^{\infty}(\mbT^{d})\subset\mcH^{\gamma'}(\mbT^{d})$ and a partition of unity in time, we can find a sequence $(f_{n})_{n\in\mbN}$ such that $f_{n}\in C_{T}C^{\infty}(\mbT^{d})$ for every $n\in\mbN$ and $f_{n}\to f\in C_{T}\mcH^{\gamma'}(\mbT^{d})$ as $n\to\infty$. By the completeness of $C_{T}\mcH^{\gamma}(\mbT^{d})$, it suffices to show that $\mcI[f_{n}]\in C_{T}\mcH^{\gamma}(\mbT^{d})$ for each $n\in\mbN$.
		
		We decompose $\mcI[f_{n}]_{t}-\mcI[f_{n}]_{s}=(P_{t-s}-1)\mcI[f_{n}]_{s}+\int_{s}^{t}P_{t-r}f_{n}(r)\dd r$ and estimate both terms separately. Let $\kappa\in[0,1]$, using that $f_{n}\in C_{T}C^{\infty}(\mbT^{d})\subset C_{T}\mcH^{\gamma'+2\kappa}(\mbT^{d})$, we obtain by~\cite[Lem.~A.5]{martini_mayorcas_25},
		\begin{equation*}
			\norm{(P_{t-s}-1)\mcI[f_{n}]_{s}}_{\mcH^{\gamma}}\lesssim\abs{t-s}^{\kappa}\norm{\mcI[f_{n}]_{s}}_{\mcH^{\gamma+2\kappa}}\lesssim_{T}\abs{t-s}^{\kappa}\norm{f_{n}}_{C_{T}\mcH^{\gamma'+2\kappa}}
		\end{equation*}
		and if we further assume $2\kappa\leq\gamma-\gamma'$, then
		\begin{equation*}
			\Bigl\lVert\int_{s}^{t}P_{t-r}f_{n}(r)\dd r\Bigr\rVert_{\mcH^{\gamma}}\leq\int_{s}^{t}\norm{P_{t-r}f_{n}(r)}_{\mcH^{\gamma}}\dd r\lesssim\norm{f_{n}}_{C_{T}\mcH^{\gamma'+2\kappa}}\int_{s}^{t}(1\vee\abs{t-r}^{-\frac{\gamma-(\gamma'+2\kappa)}{2}})\dd r.
		\end{equation*}
		Assume $\abs{t-s}<1$, we can integrate the singularity by
		\begin{equation*}
			\int_{s}^{t}(1\vee\abs{t-r}^{-\frac{\gamma-(\gamma'+2\kappa)}{2}})\dd r=\int_{s}^{t}\abs{t-r}^{-\frac{\gamma-(\gamma'+2\kappa)}{2}}\dd r\lesssim\abs{t-s}^{\kappa}\int_{s}^{t}\abs{t-r}^{-\frac{\gamma-\gamma'}{2}}\dd r\lesssim T^{1-\frac{\gamma-\gamma'}{2}}\abs{t-s}^{\kappa},
		\end{equation*}
		which implies
		\begin{equation*}
			\norm{\mcI[f_{n}]_{t}-\mcI[f_{n}]_{s}}_{\mcH^{\gamma}}\lesssim_{T}\abs{t-s}^{\kappa}\norm{f_{n}}_{C_{T}\mcH^{\gamma'+2\kappa}}.
		\end{equation*}
		If $\abs{t-s}\geq1$, we can instead estimate directly by the triangle inequality,
		\begin{equation*}
			\norm{\mcI[f_{n}]_{t}-\mcI[f_{n}]_{s}}_{\mcH^{\gamma}}\leq2\norm{\mcI[f_{n}]}_{C_{T}\mcH^{\gamma+2\kappa}}\lesssim_{T}\norm{f_{n}}_{C_{T}\mcH^{\gamma'+2\kappa}}\leq\abs{t-s}^{\kappa}\norm{f_{n}}_{C_{T}\mcH^{\gamma'+2\kappa}}.
		\end{equation*}
		All in all, we obtain for each $n\in\mbN$, $\kappa\in[0,1\wedge(\gamma-\gamma')/2]$ and $s<t$,
		\begin{equation*}
			\norm{\mcI[f_{n}]_{t}-\mcI[f_{n}]_{s}}_{\mcH^{\gamma}}\lesssim_{T}\abs{t-s}^{\kappa}\norm{f_{n}}_{C_{T}\mcH^{\gamma'+2\kappa}},
		\end{equation*}
		which implies $\mcI[f_{n}]\in C_{T}\mcH^{\gamma}(\mbT^{d})$. This yields Claim~\ref{it:Schauder_Sobolev_II}.
		\paragraph{Proof of Claim~\ref{it:Schauder_Sobolev_III}.}
		Let $\gamma'\leq\gamma\leq\gamma'+2$, we estimate using the Fourier transform, Young's convolution estimate and interpolation,
		\begin{equation*}
			\begin{split}
				\norm{\mcI[f]}_{L_{T}^{2}\mcH^{\gamma}}^{2}&=\sum_{\om\in\mbZ^{d}}(1+\abs{2\uppi\om}^{2})^{\gamma}\int_{0}^{T}\Bigl\lvert\int_{0}^{t}\euler^{-\abs{t-s}\abs{2\uppi\om}^{2}}\hat{f}(s,\om)\dd s\Bigr\rvert^{2}\dd t\\
				&\leq\sum_{\om\in\mbZ^{d}}(1+\abs{2\uppi\om}^{2})^{\gamma}\Bigl\lvert\int_{0}^{T}\euler^{-t\abs{2\uppi\om}^{2}}\dd t\Bigr\rvert^{2}\int_{0}^{T}\abs{\hat{f}(t,\om)}^{2}\dd t\\
				&=T^{2}\int_{0}^{T}\abs{\hat{f}(t,0)}^{2}\dd t+\sum_{\om\in\mbZ^{d}\setminus\{0\}}(1+\abs{2\uppi\om}^{2})^{\gamma}\Bigl\lvert\int_{0}^{T}\euler^{-t\abs{2\uppi\om}^{2}}\dd t\Bigr\rvert^{2}\int_{0}^{T}\abs{\hat{f}(t,\om)}^{2}\dd t\\
				&=T^{2}\int_{0}^{T}\abs{\hat{f}(t,0)}^{2}\dd t+\sum_{\om\in\mbZ^{d}\setminus\{0\}}(1+\abs{2\uppi\om}^{2})^{\gamma}\abs{2\uppi\om}^{-4}(1-\euler^{-T\abs{2\uppi\om}^{2}})^{2}\int_{0}^{T}\abs{\hat{f}(t,\om)}^{2}\dd t\\
				&\leq T^{2}\int_{0}^{T}\abs{\hat{f}(t,0)}^{2}\dd t+T^{2+\gamma'-\gamma}\sum_{\om\in\mbZ^{d}\setminus\{0\}}(1+\abs{2\uppi\om}^{2})^{\gamma'}\int_{0}^{T}\abs{\hat{f}(t,\om)}^{2}\dd t\\
				&\lesssim(T^{2}+T^{2+\gamma'-\gamma})\norm{f}_{L_{T}^{2}\mcH^{\gamma'}}^{2}.
			\end{split}
		\end{equation*}
		This yields Claim~\ref{it:Schauder_Sobolev_III}.
	\end{details}
\end{proof}
Let $u\in C^{2}_{1}((0,T]\times\mbT^{d})\cap C([0,T]\times\mbT^{d})$ be the classical solution (if it exists) to the general diffusion-advection equation
\begin{equation}\label{eq:Gen_Advec_Diffusion}
	\begin{cases}
		\begin{aligned}
			(\partial_{t}-\Delta)u&=\vdiv(bu),\quad&&\text{in}~(0,T]\times\mbT^{d},\\
			u\tzero&=u_{0},\quad&&\text{on}~\mbT^{d},
		\end{aligned}
	\end{cases}
\end{equation}
with positive initial data $u_{0}\in C(\mbT^{d})$ and advection $b\in C([0,T]\times\mbT^{d};\mbR^{d})$ satisfying $\vdiv b\in C([0,T]\times\mbT^{d};\mbR)$. In the next proposition we show that $u$ stays bounded away from zero, which is used to prove Lemma~\ref{lem:Determ_KS_positivity} (where existence of a solution is guaranteed.)
\begin{proposition}\label{prop:diff_adv_positivity}
	Let $u_{0}\in C(\mbT^{d})$ be such that $u_{0}>0$ and $b\in C([0,T]\times\mbT^{d};\mbR^{d})$ be such that $\vdiv b\in C([0,T]\times\mbT^{d};\mbR)$. Let $u\in C^{2}_{1}((0,T]\times\mbT^{d})\cap C([0,T]\times\mbT^{d})$ be any classical solution to~\eqref{eq:Gen_Advec_Diffusion}. Then there exists a constant $c>0$ such that $u(t,x)\geq c$ for all $(t,x)\in[0,T]\times\mbT^{d}$.
\end{proposition}
\begin{proof}
	Using the product rule and passing between the torus and the full space,
	\begin{details}
		\cite[Lem.~30.2]{vanzuijlen_22}
	\end{details}
	we find that $u\in C^{2}_{1}((0,T]\times\mbR^{d})\cap C([0,T]\times\mbR^{d})$ is a periodic solution to
	\begin{equation*}
		\begin{cases}
			\begin{aligned}
				\partial_{t}u&=\Delta u-uV+b\cdot\nabla u,\quad&&\text{in}~(0,T]\times\mbR^{d},\\
				u\tzero&=u_{0},\quad&&\text{on}~\mbR^{d},
			\end{aligned}
		\end{cases}
	\end{equation*}
	where we denote $V\defeq-\vdiv b$ and abuse notation to identify functions on $\mbT^{d}$ with periodic functions on $\mbR^{d}$. Using the assumption $\vdiv b\in C([0,T]\times\mbT^{d})$, it follows that $V\in C([0,T]\times\mbR^{d})$ and $\norm{V}_{C([0,T]\times\mbR^{d})}<\infty$.
	
	Let $t\in[0,T]$ and for every $v\from[0,t]\to\mbR$ define the time reversal $\back{v}(s)\defeq v(t-s)$. It follows by~\cite[Chpt.~V,~Thm.~23.5~\&~Thm.~20.1]{rogers_williams_00_vol2} that for every $x\in\mbR^{d}$ there exists a weak solution $((Y_{s})_{s\geq0}, (\mcF_{s})_{s\geq0})$ to the SDE,
	\begin{equation*}
		\begin{cases}
			\begin{aligned}
				\dd Y_{s}&=\back{b}(s,Y_{s})\mathds{1}_{s\leq t}\dd s+\sqrt{2}\dd B_{s},\\
				Y_{0}&=x,
			\end{aligned}
		\end{cases}
	\end{equation*}
	where we used $b\in C([0,T]\times\mbT^{d};\mbR^{d})$ to ensure that $(s,y)\mapsto\back{b}(s,y)\mathds{1}_{s\leq t}$ is a previsible path functional in the sense of~\cite[Chpt.~V,~Def.~8.3]{rogers_williams_00_vol2} and bounded in a neighbourhood of time $t$. By an application of~\cite[Chpt.~8,~Thm.~1.7~\&~Chpt.~4,~Cor.~4.3]{ethier_kurtz_86}, this solution is unique in law.
	
	We denote $A_{s}\defeq\int_{0}^{s}\back{V}(r,Y_{r})\dd r$ and $X_{s}\defeq\euler^{-A_{s}}\back{u}(s,Y_{s})$, where $\back{u}(s)\defeq u(t-s)$ solves
	\begin{equation*}
		\begin{cases}
			\begin{aligned}
				-\partial_{s}\back{u}&=\Delta\back{u}-\back{u}\back{V}+\back{b}\cdot\nabla\back{u},\quad&&\text{in}~[0,t)\times\mbR^{d},\\
				\back{u}|_{s=t}&=u_{0},\quad&&\text{on}~\mbR^{d}.
			\end{aligned}
		\end{cases}
	\end{equation*}
	It then follows by It\^{o}'s formula (cf.~\cite[Chpt.~3,~Thm.~3.6]{karatzas_shreve_91}) that for every $s<t$,
	\begin{equation*}
		\begin{split}
			\dd X_{s}&=-\euler^{-A_{s}}\back{u}(s,Y_{s})\back{V}(s,Y_{s})\dd s+\euler^{-A_{s}}\partial_{s}\back{u}(s,Y_{s})\dd s+\euler^{-A_{s}}\nabla\back{u}(s,Y_{s})\cdot\back{b}(s,y_{s})\dd s\\
			&\quad+\sqrt{2}\euler^{-A_{s}}\nabla\back{u}(s,Y_{s})\cdot\dd B_{s}+\euler^{-A_{s}}\Delta\back{u}(s,Y_{s})\dd s\\
			&=\sqrt{2}\euler^{-A_{s}}\nabla\back{u}(s,Y_{s})\cdot\dd B_{s},
		\end{split}
	\end{equation*}
	which implies that $(X_{s})_{s\in[0,t)}$ is a continuous local martingale up to time $t$.
	
	Let $(T_{k})_{k\in\mbN}$ be a localizing sequence, i.e.\ each $T_{k}$ is a stopping time, $T_{k}<T_{k+1}$, $t=\sup_{k\in\mbN}T_{k}$ and $(X_{s\wedge T_{k}})_{s\geq0}$ is a martingale for every $k\in\mbN$. It follows that $X_{s\wedge T_{k}}\to X_{s\wedge t}$ almost surely as $k\to\infty$. Furthermore,
	\begin{equation*}
		\mbE\Bigl[\sup_{s\in[0,t]}\abs{X_{s}}\Bigr]\leq\norm{u}_{C([0,t]\times\mbT^{d})}\euler^{t\norm{V}_{C([0,t]\times\mbT^{d})}}<\infty
	\end{equation*}
	and each $\abs{X_{s\wedge T_{k}}}$ is dominated by $\sup_{s\in[0,t]}\abs{X_{s}}$, which is integrable by the above. Therefore we obtain by the dominated convergence theorem for all $0\leq r<s<\infty$,
	\begin{equation*}
		\mbE[X_{s\wedge t}\mid\mathcal{F}_{r}]\leftarrow\mbE[X_{s\wedge T_{k}}\mid\mathcal{F}_{r}]=X_{r\wedge T_{k}}\to X_{r\wedge t},
	\end{equation*}
	so that $(X_{s\wedge t})_{s\geq0}$ is a martingale.
	
	We can then apply the martingale property to deduce
	\begin{equation*}
		u(t,x)=\back{u}(0,x)=\mbE[X_{0}]=\mbE[X_{t}]=\mbE[\euler^{-A_{t}}u(0,Y_{t})]=\mbE\Bigl[\euler^{-\int_{0}^{t}\back{V}(r,Y_{r})\dd r}u_{0}(Y_{t})\Bigr].
	\end{equation*}
	Using that $\mbT^{d}$ is compact, it follows by the extreme value theorem that there exists a constant $c_{0}>0$ such that $u_{0}>c_{0}$. Therefore for all $(t,x)\in[0,T]\times\mbR^{d}$,
	\begin{equation*}
		u(t,x)=\mbE\Bigl[\euler^{-\int_{0}^{t}\back{V}(r,Y_{r})\dd r}u_{0}(Y_{t})\Bigr]>c_{0}\euler^{-t\norm{V}_{C([0,t]\times\mbT^{d})}}\geq c_{0}\euler^{-T\norm{V}_{C([0,T]\times\mbT^{d})}}>0,
	\end{equation*}
	which yields the claim.
\end{proof}
We use the following chain rule in Bessel potential spaces in the proof of Lemma~\ref{lem:regularity_srdet}.
\begin{lemma}\label{lem:chain_rule_fractional}
	Let $F\in C^{\infty}(\mbR\setminus\{0\})$, $\gamma\geq0$, $c>0$ and $u\in\mcH^{\gamma}(\mbT^{d})\cap L^{\infty}(\mbT^{d})$ be such that $u\geq c$. Then there exists a constant $C(\gamma,F,\norm{u}_{L^{\infty}}, c^{-1})\geq0$ satisfying
	\begin{equation*}
		\norm{F(u)}_{\mcH^{\gamma}}\leq C(\gamma,F,\norm{u}_{L^{\infty}}, c^{-1})\norm{u}_{\mcH^{\gamma}}.
	\end{equation*}
	For all $\gamma$ and $F$, the function $(x,y)\mapsto C(\gamma,F,x, y)$ can be chosen to be non-decreasing.
\end{lemma}
\begin{proof}
	Since $u$ is bounded away from $0$, we may use a smooth cut-off to assume without limitation of generality that $F$ vanishes at $0$ and that $F'$ is locally bounded. 
	\begin{details}
		\paragraph{Construction of the smooth cut-off.}
		Let $\phi\in C_{c}^{\infty}(\mbR)$ be even and such that $\supp(\phi)\subset[-1,1]$. Define $\phi_{\delta}(x)\defeq\delta^{-1}\phi(\delta^{-1}x)$ for every $\delta>0$. We define the smooth cut-off
		\begin{equation*}
			\theta_{\delta}(x)\defeq\int_{-\infty}^{\infty}\phi_{\delta}(x-y)\mathds{1}_{[3c/4,\infty)}(y)\dd y=\int_{3c/4}^{\infty}\phi_{\delta}(x-y)\dd y,
		\end{equation*}
		whose support is contained in $[3c/4-\delta,\infty)$. If we choose $\delta\leq c/4$, then $\supp(\theta_{\delta})\subset[c/2,\infty)$ and furthermore $\theta_{\delta}(x)=1$ for every $x\in[c,\infty)$, since $\supp(\phi_{\delta}(\place-x))\subset[x-\delta,x+\delta]\subset[3c/4,\infty)$. 
		
		Let $F_{c}(x)\defeq F(x)\theta_{c/4}(x)$, whose support is contained in $[c/2,\infty)$. Using that $u\geq c$ on $\mbT^{d}$ and $\theta_{c/4}\equiv1$ on $[c,\infty)$, we obtain $F(u(x))=F_{c}(u(x))$ for every $x\in\mbT^{d}$. Furthermore $F_{c}'=F'\theta_{c/4}+F\phi_{c/4}(\place-3c/4)$ is locally bounded.
		
	\end{details}
	If $\gamma=0$, this yields the square integrability of $F(u)$.
	\begin{details}
		\paragraph{Square integrability of $F(u)$.}
		Using that $F_{c}(0)=0$ and the local boundedness of  the derivative $F_{c}'$, we may estimate by the mean-value theorem,
		\begin{equation*}
			\norm{F(u)}_{L^{2}}^{2}=\int_{\mbT^{d}}\abs{F_{c}(u(x))}^{2}\dd x=\int_{\mbT^{d}}\abs{F_{c}(u(x))-F_{c}(0)}^{2}\dd x\leq\sup_{x\in(0,\norm{u}_{L^{\infty}})}\abs{F_{c}'(x)}^{2}\int_{\mbT^{d}}\abs{u(x)}^{2}\dd x.
		\end{equation*}
		We denote $C(0,F,\norm{u}_{L^{\infty}},c^{-1})\defeq\sup_{x\in(0,\norm{u}_{L^{\infty}})}\abs{F_{c}'(x)}$, which yields the claim for $\gamma=0$.
		
	\end{details}
	If $\gamma>0$ we may apply~\cite[Thm.~2.87]{bahouri_chemin_danchin_11} to deduce the claim, see also~\cite[Lem.~B.2]{constantin_etal_20}.
	\begin{details}
		\paragraph{Proof for $\gamma>0$.}
		By an application of~\cite[Thm.~2.87]{bahouri_chemin_danchin_11} there exists a constant $C(\gamma,F_{c}',\norm{u}_{L^{\infty}})\geq0$ such that
		\begin{equation*}
			\norm{F(u)}_{\mcH^{\gamma}}\leq C(\gamma,F_{c}',\norm{u}_{L^{\infty}})\norm{u}_{\mcH^{\gamma}}.
		\end{equation*}
		We denote $C(\gamma,F,\norm{u}_{L^{\infty}},c^{-1})\defeq C(\gamma,F_{c}',\norm{u}_{L^{\infty}})$, which yields the claim for $\gamma>0$.
	\end{details}
\end{proof}
Recall that $T^{\cem}_{\target}[f]$ denotes the blow-up time of $f\in\target^{\sol}_{T}$ (see Subsection~\ref{subsec:notation}). We show that the map $T^{\cem}_{\target}\from\target^{\sol}_{T}\to[0,T]$ is lower semicontinuous.
\begin{lemma}\label{lem:blow_up_lsc}
	The map $T^{\cem}_{\target}\from\target^{\sol}_{T}\to[0,T]$ is lower semicontinuous.
\end{lemma}
\begin{proof}
	We define for every $L>0$ and $f\in\target^{\sol}_{T}$ the exceedance time $T^{L}_{\target}[f]\defeq T\wedge L\wedge\inf\{t\in[0,T]:\norm{f(t)}_{\target}\geq L\}$. Let $S\in[0,T]$, we first show that the sublevel set
	\begin{equation*}
		(T^{L}_{\target})^{-1}([0,S])\defeq\{f\in\target^{\sol}_{T}:T^{L}_{\target}[f]\leq S\}
	\end{equation*}
	is closed in $\target^{\sol}_{T}$, which is equivalent to the lower semicontinuity of $T^{L}_{\target}$. 
	
	The case $S\geq T\wedge L$ is trivial, since $(T^{L}_{\target})^{-1}([0,T\wedge L])=\target^{\sol}_{T}$, hence assume $S<T\wedge L$. Let $(f_{n})_{n\in\mbN}$ be a converging sequence in $(T^{L}_{\target})^{-1}([0,S])$ with limit $f\in\target^{\sol}_{T}$. Our aim is to show that $f\in(T^{L}_{\target})^{-1}([0,S])$, so let's assume for a contradiction that $T^{L}_{\target}[f]>S$. The sequence of exceedance times $(T^{L}_{\target}[f_{n}])_{n\in\mbN}$ is a sequence of real numbers in $[0,S]$, hence by the compactness of $[0,S]$ we can extract a subsequence that converges (modulo re-labelling) to a limit $\lim_{n\to\infty}T^{L}_{\target}[f_{n}]=S^{*}\in[0,S]$. Using the continuity of the norm, we obtain
	\begin{equation*}
		\norm{f(S^{*})}_{\target}=\lim_{n\to\infty}\norm{f(T^{L}_{\target}[f_{n}])}_{\target}\geq\liminf_{n\to\infty}\norm{f_{n}(T^{L}_{\target}[f_{n}])}_{\target}-\limsup_{n\to\infty}\norm{f(T^{L}_{\target}[f_{n}])-f_{n}(T^{L}_{\target}[f_{n}])}_{\target}=L,
	\end{equation*}
	where we used $f_{n}\in(T^{L}_{\target})^{-1}([0,S])$ and $S<T\wedge L$  to deduce $\norm{f_{n}(T^{L}_{\target}[f_{n}])}_{\target}=L$ and the definition of the metric $D_{T}^{\target}$ on $\target^{\sol}_{T}$ (cf.~\eqref{eq:FsolT_metric}) to ensure that the second term vanishes.
	\begin{details}
		In particular, we used that 
		\begin{equation*}
			\begin{split}
				T_{L}&\defeq T\wedge L\wedge\inf\{t\in[0,T]:\norm{f_{n}(t)}_{\target}>L~\text{or}~\norm{f(t)}_{\target}>L\}\\
				&\geq T\wedge L\wedge\inf\{t\in[0,T]:\norm{f_{n}(t)}_{\target}\geq L~\text{or}~\norm{f(t)}_{\target}\geq L\}=T^{L}_{\target}[f_{n}]\wedge T^{L}_{\target}[f]=T^{L}_{\target}[f_{n}],
			\end{split}
		\end{equation*}
		which follows by $f_{n}\in(T^{L}_{\target})^{-1}([0,S])$ and the assumption $T^{L}_{\target}[f]>S$.
	\end{details}
	
	Therefore, $\norm{f(S^{*})}_{\target}\geq L$ and by the definition of $T^{L}_{\target}[f]$, we obtain $T^{L}_{\target}[f]\leq S^{*}\leq S$. It follows that the limit $f\in\target^{\sol}_{T}$ satisfies $T^{L}_{\target}[f]\leq S$, hence $f\in(T^{L}_{\target})^{-1}([0,S])$, which proves the lower semicontinuity of $T^{L}_{\target}\from\target^{\sol}_{T}\to[0,T]$.
	
	To deduce the lower semicontinuity of the blow-up time $T^{\cem}_{\target}\from\target^{\sol}_{T}\to[0,T]$, we use that $T^{\cem}_{\target}[f]=\sup_{L>0}T^{L}_{\target}[f]$ and the general fact that lower semicontinuity is preserved under taking suprema.
\end{proof}
\subsection{Analysis of Function Spaces}\label{subsec:function_spaces}
In this appendix we show that each interpolation space $\msL_{T}^{\kappa}\mcB_{p,q}^{\alpha}(\mbT^{d})$ and each blow-up space $\target^{\sol}_{T}$ (see Subsection~\ref{subsec:notation}) is separable (Lemma~\ref{lem:separability_interpolation} and Lemma~\ref{lem:FsolT_separability}). We then define a space of space-time H\"{o}lder distributions $\mcC_{\mfs}^{\alpha}([0,T]\times\mbT^{d})$ (Definition~\ref{def:Cs_alpha_torus}) and also establish its separability (Lemma~\ref{lem:Cs_alpha_torus_separable}).

Let $p,q\in[1,\infty]$ and $\alpha\in\mbR$. Recall (see Subsection~\ref{subsec:notation}) that the Besov space $\mcB_{p,q}^{\alpha}(\mbT^{d})$ is defined as the closure of $C^{\infty}(\mbT^{d})$ under the Besov norm $u\mapsto\norm{u}_{\mcB_{p,q}^{\alpha}}\defeq\norm{(2^{k\alpha}\norm{\Delta_{k}u}_{L^{p}})_{k\in\mbN_{-1}}}_{\ell^{q}}$. We first show the separability of each $\mcB_{p,q}^{\alpha}(\mbT^{d})$. 
\begin{lemma}\label{lem:Besov_separable}
	Let $p,q\in[1,\infty]$ and $\alpha\in\mbR$. It holds that 
	\begin{equation}\label{eq:Besov_characterization}
		\mcB_{p,q}^{\alpha}(\mbT^{d})=\Bigl\{u\in\mcS'(\mbT^{d}):\norm{u}_{\mcB_{p,q}^{\alpha}}<\infty,\lim_{j\to\infty}2^{j\alpha}\norm{\Delta_{j}u}_{L^{p}}=0\Bigr\}
	\end{equation}
	and each $\mcB_{p,q}^{\alpha}(\mbT^{d})$ is separable. 
\end{lemma}
\begin{proof}
	We start by proving the characterization~\eqref{eq:Besov_characterization}.
	
	Assume $u\in\mcB_{p,q}^{\alpha}(\mbT^{d})$, we need to show that $\lim_{j\to\infty}2^{j\alpha}\norm{\Delta_{j}u}_{L^{p}}=0$. By definition, there exist a sequence $(u_{n})_{n\in\mbN}$ such that $u_{n}\in C^{\infty}(\mbT^{d})$ for every $n\in\mbN$ and $u_{n}\to u\in\mcB_{p,q}^{\alpha}(\mbT^{d})$ as $n\to\infty$. It follows that for each $j\in\mbN_{-1}$,
	\begin{equation*}
		2^{j\alpha}\norm{\Delta_j u}_{L^p}\leq2^{j\alpha}\norm{\Delta_j u-\Delta_ju_n}_{L^p}+2^{j\alpha}\norm{\Delta_j u_n}_{L^p}\leq \norm{u-u_n}_{\mcB_{p,q}^{\alpha}}+2^{j\alpha}\norm{\Delta_j u_n}_{L^p}.
	\end{equation*}
	Letting $j\to\infty$, we obtain $\limsup_{j\to\infty}2^{j\alpha}\norm{\Delta_j u}_{L^p}\leq\norm{u-u_n}_{\mcB_{p,q}^{\alpha}}$, subsequently letting $n\to\infty$ yields $\lim_{j\to\infty}2^{j\alpha}\norm{\Delta_j u}_{L^p}=0$.
	
	Conversely assume that $u\in\mcS'(\mbT^{d})$ satisfies $\norm{u}_{\mcB_{p,q}^{\alpha}}<\infty$ and $\lim_{j\to\infty}2^{j\alpha}\norm{\Delta_{j}u}_{L^{p}}=0$. We need to show that there exists a sequence $(u_{n})_{n\in\mbN}$ such that $u_{n}\in C^{\infty}(\mbT^{d})$ for every $n\in\mbN$ and $\norm{u-u_{n}}_{\mcB_{p,q}^{\alpha}}\to0$ as $n\to\infty$. We define $u_{n}\defeq\sum_{j=-1}^{n-1}\Delta_{j}u\in C^{\infty}(\mbT^{d})$ and apply~\cite[Lem.~A.3]{gubinelli_15_GIP}:
	\begin{details}
		(rather~\cite[Thm.~21.18]{vanzuijlen_22})
	\end{details}
	For $q=\infty$ we deduce $\norm{u-u_{n}}_{\mcB_{p,\infty}^{\alpha}}\lesssim\sup_{j\geq n}(2^{j\alpha}\norm{\Delta_{j}u}_{L^{p}})$, which vanishes as $n\to0$, since $\lim_{j\to\infty}2^{j\alpha}\norm{\Delta_{j}u}_{L^p}=0$.
	\begin{details}
		$\lim_{n\to\infty}\sup_{j\geq n}(2^{j\alpha}\norm{\Delta_{j}u}_{L^{p}})=\limsup_{n\to\infty}2^{n\alpha}\norm{\Delta_{n}u}_{L^{p}}=\lim_{j\to\infty}2^{j\alpha}\norm{\Delta_{j}u}_{L^{p}}=0$.
	\end{details}
	For $q<\infty$ we deduce $\norm{u-u_{n}}_{\mcB_{p,q}^{\alpha}}^{q}\lesssim\sum_{j=n}^{\infty}2^{jq\alpha}\norm{\Delta_{j}u}_{L^{p}}^{q}$, which vanishes as $n\to0$, since $\norm{u}_{\mcB_{p,q}^{\alpha}}<\infty$. This yields the characterization~\eqref{eq:Besov_characterization}.
	
	Next we show the separability of each $\mcB_{p,q}^{\alpha}(\mbT^{d})$. If $p<\infty$, let $D$ be a countable, dense subset of $L^{p}(\mbT^{d})$; and if $p=\infty$, let $D$ be a countable, dense subset of $C(\mbT^{d})$. We define the countable set 
	\begin{equation*}
		H\defeq\Bigl\{\sum_{j=-1}^{J}\Delta_jh_j:J\in\mbN_{-1},~h_j\in D,~j=1,\ldots,J\Bigr\}
	\end{equation*}
	and show that $H$ is dense in $C^{\infty}(\mbT^{d})\subset\mcB_{p,q}^{\alpha}(\mbT^{d})$. 
	
	Assume first $q=\infty$. Let $u\in C^{\infty}(\mbT^{d})$, then for each $\eps>0$ there exists a $J\in\mbN_{-1}$ such that $2^{j\alpha}\norm{\Delta_{j}u}_{L^{p}}<\eps$ for every $j>J$. For $j=-1,\ldots,J$, let $h_{j}\in D$ be such that $\norm{u-h_{j}}_{L^{p}}<\eps2^{-j\alpha}$, which combined with Poisson's summation formula and Young's convolution inequality yields
	\begin{equation*}
		\norm{\Delta_ju-\Delta_jh_{j}}_{L^p}\leq(\norm{\mathscr{F}^{-1}_{\mbR^d}\varrho_{0}}_{L^1}\vee\norm{\mathscr{F}^{-1}_{\mbR^d}\varrho_{-1}}_{L^1})\norm{u-h_{j}}_{L^p}\lesssim\eps2^{-j\alpha}.
	\end{equation*}
	It then follows that
	\begin{equation*}
		\Bigl\lVert u-\sum_{j=-1}^{J}\Delta_j h_{j}\Bigr\rVert_{\mcB_{p,\infty}^{\alpha}}\lesssim\sup_{j=-1,\ldots,J}\Bigl(2^{j\alpha}\norm{\Delta_{j}u-\Delta_{j}h_{j}}_{L^{p}}\Bigr)\vee\sup_{j=J+1,\ldots,\infty}\Bigl(2^{j\alpha}\norm{\Delta_{j}u}_{L^{p}}\Bigr)\lesssim\eps,
	\end{equation*}
	which shows that $H$ is dense in $\mcB_{p,\infty}^{\alpha}(\mbT^{d})$.
	
	Assume next $q<\infty$. Let $u\in C^{\infty}(\mbT^{d})$, then for each $\eps>0$ there exists a $J\in\mbN_{-1}$ such that $2^{j\alpha}\norm{\Delta_ju}_{L^p}<\eps2^{-j}$ for every $j>J$. For $j=-1,\ldots,J$, let $h_{j}\in D$ be such that $\norm{u-h_{j}}_{L^p}<\eps2^{-j(\alpha+1)}$, which combined with Poisson's summation formula and Young's convolution inequality yields
	\begin{equation*}
		\norm{\Delta_ju-\Delta_jh_j}_{L^p}\leq(\norm{\mathscr{F}^{-1}_{\mbR^{d}}\varrho_{0}}_{L^1}\vee\norm{\mathscr{F}^{-1}_{\mbR^{d}}\varrho_{-1}}_{L^1})\norm{u-h_j}_{L^p}\lesssim\eps2^{-j(\alpha+1)}.
	\end{equation*}
	It then follows that
	\begin{equation*}
		\Bigl\lVert u-\sum_{j=-1}^{J}\Delta_jh_j\Bigr\rVert_{\mcB_{p,q}^{\alpha}}^{q}\lesssim\sum_{j=-1}^{J}2^{jq\alpha}\norm{\Delta_{j}u-\Delta_{j}h_{j}}_{L^{p}}^{q}+\sum_{j=J+1}^{\infty}2^{jq\alpha}\norm{\Delta_{j}u}_{L^{p}}^{q}\lesssim\eps^{q}\sum_{j=-1}^{\infty}2^{-jq}\lesssim\eps^{q},
	\end{equation*}
	which shows that $H$ is dense in $\mcB_{p,q}^{\alpha}(\mbT^{d})$. This yields the claim.
\end{proof}
The next lemma presents a natural criterion that allows us to find elements of $\mcB_{p,q}^{\alpha}(\mbT^{d})$.
\begin{lemma}\label{lem:Besov_criterion}
	Let $p,q\in[1,\infty]$ and $\alpha<\alpha'\in\mbR$. 
	\begin{itemize}
		\item If $q\in[1,\infty)$, then each $u\in\mcS'(\mbT^{d})$ such that $\norm{u}_{\mcB_{p,q}^{\alpha}}<\infty$ is an element of $\mcB_{p,q}^{\alpha}(\mbT^{d})$.
		\item If $q=\infty$, then each $u\in\mcS'(\mbT^{d})$ such that $\norm{u}_{\mcB_{p,\infty}^{\alpha'}}<\infty$ is an element of $\mcB_{p,\infty}^{\alpha}(\mbT^{d})$.
	\end{itemize}
\end{lemma}
\begin{proof}
	Let $p,q\in[1,\infty]$ and $\alpha\in\mbR$. To prove the first claim, let $q<\infty$ and $u\in\mcS'(\mbT^{d})$ be such that $\norm{u}_{\mcB_{p,q}^{\alpha}}<\infty$. It follows immediately that $\lim_{j\to\infty}2^{j\alpha}\norm{\Delta_{j}u}_{L^{p}}=0$, hence $u\in\mcB_{p,q}^{\alpha}(\mbT^{d})$ by Lemma~\ref{lem:Besov_separable}. To prove the second claim, let $q=\infty$, $\alpha<\alpha'\in\mbR$ and $u\in\mcS'(\mbT^{d})$ be such that $\norm{u}_{\mcB_{p,\infty}^{\alpha'}}<\infty$. It follows that $\norm{u}_{\mcB_{p,\infty}^{\alpha}}<\infty$ and
	\begin{equation*}
		2^{j\alpha}\norm{\Delta_{j}u}_{L^p}\leq2^{j(\alpha-\alpha')}\norm{u}_{\mcB_{p,\infty}^{\alpha'}}\to0\quad\text{as}~j\to\infty,
	\end{equation*}
	hence $u\in\mcB_{p,\infty}^{\alpha}(\mbT^{d})$ by Lemma~\ref{lem:Besov_separable}. This yields the claim.
\end{proof}
The next lemma shows that Besov spaces are compactly embedded upon passing to a larger space.
\begin{lemma}\label{lem:Besov_compact}
	Let $p,q\in[1,\infty]$ and $\alpha'<\alpha\in\mbR$. It follows that the embedding $\mcB_{p,q}^{\alpha}(\mbT^{d})\embed\mcB_{p,q}^{\alpha'}(\mbT^{d})$ is compact.
\end{lemma}
\begin{proof}
	The proof follows by similar arguments as in~\cite[Thm.~2.94]{bahouri_chemin_danchin_11}.
\end{proof}
\begin{details}
	\begin{proof}
		Let $(u_{n})_{n\in\mbN}$ be a bounded sequence in $\mcB_{p,q}^{\alpha}(\mbT^{d})$, to prove the claim it suffices to show that it has a subsequence which converges in $\mcB_{p,q}^{\alpha'}(\mbT^{d})$.
		
		We follow the proof of~\cite[Thm.~2.94]{bahouri_chemin_danchin_11}. By the Fatou property of Besov spaces~\cite[Thm.~2.72]{bahouri_chemin_danchin_11}, there exists a subsequence (which we do not re-label) such that $u_{n}\to u$ in $\mcS'(\mbT^{d})$ and
		\begin{equation*}
			\norm{u}_{\mcB_{p,q}^{\alpha}}\leq\liminf_{n\to\infty}\norm{u_{n}}_{\mcB_{p,q}^{\alpha}}.
		\end{equation*}
		Denote $v_{n}\defeq u_{n}-u$, so that it suffices to show $\lim_{n\to\infty}\norm{v_{n}}_{\mcB_{p,q}^{\alpha'}}=0$. First let us assume $q<\infty$, we estimate for every $K>0$,
		\begin{equation}\label{eq:compact_embedding_cut_off}
			\norm{v_{n}}_{\mcB_{p,q}^{\alpha'}}^{q}=\sum_{j\in\mbN_{-1}}2^{jq\alpha'}\norm{\Delta_{j}v_{n}}_{L^{p}}^{q}\leq\sum_{\substack{j\in\mbN_{-1}\\2^{j}\leq K}}2^{jq\alpha'}\norm{\Delta_{j}v_{n}}_{L^{p}}^{q}+K^{q(\alpha'-\alpha)}\sum_{\substack{j\in\mbN_{-1}\\K<2^{j}}}2^{jq\alpha}\norm{\Delta_{j}v_{n}}_{L^{p}}^{q}.
		\end{equation}
		To control the second summand in~\eqref{eq:compact_embedding_cut_off}, we can use that $(\norm{v_{n}}_{\mcB_{p,q}^{\alpha}})_{n\in\mbN}$ is bounded (by the boundedness of $(\norm{u_{n}}_{\mcB_{p,q}^{\alpha}})_{n\in\mbN}$), to choose for every $\eps>0$ the constant $K$ sufficiently large such that uniformly in $n\in\mbN$,
		\begin{equation*}
			K^{q(\alpha'-\alpha)}\sum_{\substack{j\in\mbN_{-1}\\K<2^{j}}}2^{jq\alpha}\norm{\Delta_{j}v_{n}}_{L^{p}}^{q}\leq K^{q(\alpha'-\alpha)}\norm{v_{n}}_{\mcB_{p,q}^{\alpha}}^{q}<\eps^{q}.
		\end{equation*}
		Given a $K$ independent of $n$, we can then let $n\to\infty$ in the first summand of~\eqref{eq:compact_embedding_cut_off} to deduce
		\begin{equation*}
			\lim_{n\to\infty}\sum_{\substack{j\in\mbN_{-1}\\2^{j}\leq K}}2^{jq\alpha'}\norm{\Delta_{j}v_{n}}_{L^{p}}^{q}=\sum_{\substack{j\in\mbN_{-1}\\2^{j}\leq K}}2^{jq\alpha'}\lim_{n\to\infty}\norm{\Delta_{j}v_{n}}_{L^{p}}^{q}.
		\end{equation*}
		Next we show $\lim_{n\to\infty}\norm{\Delta_{j}v_{n}}_{L^{p}}=0$ for every $j\in\mbN_{-1}$. Let us first consider the case $j\in\mbN$. Let $\msF$ and $\msF^{-1}$ be the Fourier (resp.\ inverse Fourier) transform on $\mbT^{d}$. It follows by~\cite[Thm.~34.4]{vanzuijlen_22} that $\msF(\msF^{-1}\rho_{j}\ast v_{n})(\om)=\varrho_{j}(\om)\hat{v_{n}}(\om)$, where the convolution is given by $(\msF^{-1}\varrho_{j}\ast v_{n})(x)=v_{n}(\msF^{-1}\varrho_{j}(x-\place))$. Hence, we obtain $\Delta_{j}v_{n}(x)=(\msF^{-1}\varrho_{j}\ast v_{n})(x)=v_{n}(\msF^{-1}\varrho_{j}(x-\place))$ and for every $x\in\mbT^{d}$, the convergence $v_{n}\to0$ in $\mcS'(\mbT^{d})$ implies $\lim_{n\to\infty}\Delta_{j}v_{n}(x)=0$. By an application of Bernstein's inequality~\cite[Thm~35.13]{vanzuijlen_22} (where we use that $j\in\mbN$), we can control the $L^{p}(\mbT^{d})$-norm of $\Delta_{j}v_{n}$ by the $L^{1}(\mbT^{d})$-norm (especially for $p=\infty$),
		\begin{equation*}
			\norm{\Delta_{j}v_{n}}_{L^{p}}\lesssim\norm{\Delta_{j}v_{n}}_{L^{1}}
		\end{equation*}
		and it suffices to consider the case $p=1$. It follows by~\cite[Thm.~35.14]{vanzuijlen_22} that there exists some $m\in\mbN_{0}$ such that
		\begin{equation*}
			\abs{\Delta_{j}v_{n}(x)}=\abs{v_{n}(\msF^{-1}\varrho_{j}(x-\place))}\lesssim\norm{v_{n}}_{\mcB_{p,q}^{\alpha}}\norm{\msF^{-1}\varrho_{j}(x-\place)}_{C^{m}(\mbT^{d})}.
		\end{equation*}
		Using that the sequence $(\norm{v_{n}}_{\mcB_{p,q}^{\alpha}})_{n\in\mbN}$ is bounded and that the $C^{m}(\mbT^{d})$-norm is translation invariant, we can bound uniformly in $n\in\mbN$ and $x\in\mbT^{d}$,
		\begin{equation*}
			\abs{\Delta_{j}v_{n}(x)}\lesssim1.
		\end{equation*}
		Using that $\mbT^{d}$ is compact, this yields a dominant of $\Delta_{j}v_{n}$, hence an application of the dominated convergence theorem implies $\lim_{n\to\infty}\norm{\Delta_{j}v_{n}}_{L^{1}}=0$.
			
		For the case $j=-1$, note that $\Delta_{-1}v_{n}(x)=\inner{v_{n}}{1}_{L^{2}}$, which follows by our definition of $\varrho_{-1}$ and by the spectrum being discrete. Consequently, $\norm{\Delta_{-1}v_{n}}_{L^{p}}=\inner{v_{n}}{1}_{L^{2}}\to0$, where we used that $\Delta_{-1}v_{n}(x)$ is independent of $x\in\mbT^{d}$ followed by the convergence $v_{n}\to0$ in $\mcS'(\mbT^{d})$.
		
		Hence, passing to the limit in~\eqref{eq:compact_embedding_cut_off} we obtain $\lim_{n\to\infty}\norm{v_{n}}_{\mcB_{p,q}^{\alpha'}}<\eps$ for every $\eps>0$, which yields the claim for $q<\infty$.
		
		Next assume $q=\infty$.
		Letting $K>0$, we decompose
		\begin{equation*}
			\norm{v_{n}}_{B_{p,\infty}^{\alpha'}}=\sup_{\substack{j\in\mbN_{-1}\\2^{j}\leq K}}(2^{j\alpha'}\norm{\Delta_{j}v_{n}}_{L^{p}})\vee K^{\alpha'-\alpha}\sup_{\substack{j\in\mbN_{-1}\\K<2^{j}}}(2^{j\alpha}\norm{\Delta_{j}v_{n}}_{L^{p}}).
		\end{equation*}
		For every $\eps>0$ we can find $K>0$ sufficiently large such that uniformly in $n\in\mbN$,
		\begin{equation*}
			K^{\alpha'-\alpha}\sup_{\substack{j\in\mbN_{-1}\\K<2^{j}}}(2^{j\alpha}\norm{\Delta_{j}v_{n}}_{L^{p}})\leq K^{\alpha'-\alpha}\norm{v_{n}}_{\mcB_{p,\infty}^{\alpha}}<\eps.
		\end{equation*}
		Given a $K$ independent of $n$, we can let $n\to\infty$ in the first term to deduce
		\begin{equation*}
			\lim_{n\to\infty}\sup_{\substack{j\in\mbN_{-1}\\2^{j}\leq K}}(2^{j\alpha'}\norm{\Delta_{j}v_{n}}_{L^{p}})=\sup_{\substack{j\in\mbN_{-1}\\2^{j}\leq K}}(2^{j\alpha'}\lim_{n\to\infty}\norm{\Delta_{j}v_{n}}_{L^{p}})=0,
		\end{equation*}
		where we used that limits and maxima over finite sets commute. Hence, passing to the limit, we obtain $\lim_{n\to\infty}\norm{v_{n}}_{\mcB_{p,\infty}^{\alpha'}}<\eps$ for every $\eps>0$, which yields the claim for $q=\infty$.
	\end{proof}
\end{details}
Recall that the space $C_{T}^{\kappa}\mcB_{p,q}^{\alpha}(\mbT^{d})$ is defined as the completion of $C^{\infty}_{T}\mcB_{p,q}^{\alpha}(\mbT^{d})$ under the usual H\"{o}lder norm for $\kappa\in(0,1)$ and that $C_{T}^{0}\mcB_{p,q}^{\alpha}(\mbT^{d})\defeq C_{T}\mcB_{p,q}^{\alpha}(\mbT^{d})$ (cf.\ Subsection~\ref{subsec:notation}). Next we show that each $C_{T}^{\kappa}\mcB_{p,q}^{\alpha}(\mbT^{d})$ is separable.
\begin{lemma}\label{lem:little_Hoelder_separable}
	Let $T>0$, $p,q\in[1,\infty]$, $\alpha\in\mbR$ and $\kappa\in[0,1)$. It holds that $C_{T}^{\kappa}\mcB_{p,q}^{\alpha}(\mbT^{d})$ consists of those $f\from[0,T]\to\mcB_{p,q}^{\alpha}(\mbT^{d})$ such that
	\begin{equation}\label{eq:little_Hoelder_characterisation}
		\norm{f}_{C_{T}^{\kappa}\mcB_{p,q}^{\alpha}}<\infty\quad\text{and}\quad\lim_{r\to0}\sup\biggl\{\frac{\norm{f(t)-f(s)}_{\mcB_{p,q}^{\alpha}}}{\abs{t-s}^{\kappa}}:s\neq t\in [0,T],~\abs{t-s}<r\biggr\}=0
	\end{equation} 
	and each $C_{T}^{\kappa}\mcB_{p,q}^{\alpha}(\mbT^{d})$ is separable.
\end{lemma}
\begin{proof}
	We first prove that $C_{T}^{\kappa}\mcB_{p,q}^{\alpha}(\mbT^{d})$ is characterized by~\eqref{eq:little_Hoelder_characterisation}, that is, each $f\in C_{T}^{\kappa}\mcB_{p,q}^{\alpha}(\mbT^{d})$ satisfies~\eqref{eq:little_Hoelder_characterisation} and each $f$ satisfying~\eqref{eq:little_Hoelder_characterisation} lies in $C_{T}^{\kappa}\mcB_{p,q}^{\alpha}(\mbT^{d})$.
	
	Assume $f\in C_{T}^{\kappa}\mcB_{p,q}^{\alpha}(\mbT^{d})$, we need to show that~\eqref{eq:little_Hoelder_characterisation} holds. It is clear that $\norm{f}_{C_{T}^{\kappa}\mcB_{p,q}^{\alpha}}<\infty$ and it suffices to prove the second condition of~\eqref{eq:little_Hoelder_characterisation}. By definition, there exists a sequence $(f_{n})_{n\in\mbN}$ such that $f_{n}\in C^{\infty}_{T}\mcB_{p,q}^{\alpha}(\mbT^{d})$ for every $n\in\mbN$ and $\lim_{n\to\infty}\norm{f_{n}-f}_{C_{T}^{\kappa}\mcB_{p,q}^{\alpha}}=0$. Let $r>0$ and $s\neq t\in[0,T]$ such that $\abs{t-s}<r$, then
	\begin{equation*}
		\frac{\norm{f(t)-f(s)}_{\mcB_{p,q}^{\alpha}}}{\abs{t-s}^{\kappa}}\leq\norm{f-f_{n}}_{C_{T}^{\kappa}\mcB_{p,q}^{\alpha}}+\frac{\norm{f_{n}(t)-f_{n}(s)}_{\mcB_{p,q}^{\alpha}}}{\abs{t-s}^{\kappa}},
	\end{equation*}
	which yields 
	\begin{equation*}
		\lim_{r\to0}\sup\biggl\{\frac{\norm{f(t)-f(s)}_{\mcB_{p,q}^{\alpha}}}{\abs{t-s}^{\kappa}}:s\neq t\in [0,T],~\abs{t-s}<r\biggr\}=\norm{f-f_{n}}_{C_{T}^{\kappa}\mcB_{p,q}^{\alpha}}.
	\end{equation*} 
	Since $n\in\mbN$ is arbitrary, we can pass to the limit to show that each $f\in C_{T}^{\kappa}\mcB_{p,q}^{\alpha}(\mbT^{d})$ satisfies~\eqref{eq:little_Hoelder_characterisation}.
	
	Next we need to show that each $f\from[0,T]\to\mcB_{p,q}^{\alpha}(\mbT^{d})$ satisfying~\eqref{eq:little_Hoelder_characterisation} lies in $C_{T}^{\kappa}\mcB_{p,q}^{\alpha}(\mbT^{d})$, i.e.\ that it can be approximated by a sequence $(f_{n})_{n\in\mbN}$ such that $f_{n}\in C_{T}^{\infty}\mcB_{p,q}^{\alpha}(\mbT^{d})$ for every $n\in\mbN$. We extend $f$ from $[0,T]$ to $\mbR$ by setting $f(t)=f(0)$ for $t<0$ and $f(t)=f(T)$ for $t>T$. We define
	\begin{equation*}
		\phi_{n}(r)\defeq n\biggl(1-\frac{\abs{r}^{2}}{n^{2}}\biggr)^{n^{4}}\uppi^{-1/2}
	\end{equation*}
	and construct the polynomials
	\begin{equation}\label{eq:polynomial_approx}
		f_n(t)\defeq\int_{\mbR}f(r)\phi_n(t-r)\dd r=\int_{\mbR}f(t-r)\phi_n(r)\dd r,
	\end{equation}
	which converge to $f$ in $C_{T}\mcB_{p,q}^{\alpha}(\mbT^{d})$ as $n\to\infty$, see~\cite[Appendix,~Prop.~7.1]{ethier_kurtz_86} for the case of real-valued functions. The extension to Banach space-valued functions follows by the discussion in~\cite[p.~177]{amann_escher_09}.
	
	\begin{details}
		To show that $f_{n}$ has $\mcB_{p,q}^{\alpha}(\mbT^{d})$-valued coefficients, we distinguish the cases $q<\infty$ and $q=\infty$. Let us first consider the $0$th-order coefficient. If $q<\infty$, then it suffices to show that $\norm{f_{n}(0)}_{\mcB_{p,q}^{\alpha}}<\infty$, which  follows by Minkowski's inequality,
		\begin{equation*}
			\norm{f_{n}(0)}_{\mcB_{p,q}^{\alpha}}\leq\int_{\mbR}\norm{f(r)}_{\mcB_{p,q}^{\alpha}}\abs{\phi_{n}(-r)}\dd r\lesssim\norm{f}_{C_{T}\mcB_{p,q}^{\alpha}}<\infty.
		\end{equation*}
		If $q=\infty$, then we need to show that $\lim_{j\to\infty}2^{j\alpha}\norm{\Delta_{j}f_{n}(0)}_{L^{p}}=0$, which follows by the dominated convergence theorem,
		\begin{equation*}
			\lim_{j\to\infty}2^{j\alpha}\norm{\Delta_{j}f_{n}(0)}_{L^{p}}\leq\int_{[0,T]}\lim_{j\to\infty}2^{j\alpha}\norm{\Delta_{j}f(r)}_{L^{p}}\abs{\phi_{n}(-r)}\dd r=0.
		\end{equation*}
		This proves that the $0$th-order coefficient lies in $\mcB_{p,q}^{\alpha}(\mbT^{d})$. For higher-order coefficients, it suffices to consider derivates at $t=0$.
	\end{details}
	
	Next we show that $f_{n}\to f$ in $C_{T}^{\kappa}\mcB_{p,q}^{\alpha}(\mbT^{d})$, where we argue as in~\cite{hajlasz_19}. Let $\eps>0$ be arbitrary, by the definition of $C_{T}^{\kappa}\mcB_{p,q}^{\alpha}(\mbT^{d})$ we can find some $\tau>0$ such that if $s\neq t\in [0,T]$ and $\abs{t-s}<\tau$, then 
	\begin{equation}\label{eq:little_Hoelder_characterisation_application}
		\norm{f(t)-f(s)}_{\mcB_{p,q}^{\alpha}}<\frac{1}{2}\eps\abs{t-s}^{\kappa},
	\end{equation}
	which also extends from $[0,T]$ to $\mbR$ by our definition of $f$.
	\begin{details}
		Assume $s<T<t$, then
		\begin{equation*}
			\norm{f(t)-f(s)}_{\mcB_{p,q}^{\alpha}}=\norm{f(T)-f(s)}_{\mcC^{\alpha}}<\frac{1}{2}\eps\abs{T-s}^{\kappa}<\frac{1}{2}\eps\abs{t-s}^{\kappa}.
		\end{equation*}
		Assume $s<0<t$, then
		\begin{equation*}
			\norm{f(t)-f(s)}_{\mcB_{p,q}^{\alpha}}=\norm{f(t)-f(0)}_{\mcC^{\alpha}}<\frac{1}{2}\eps t^{\kappa}<\frac{1}{2}\eps\abs{t-s}^{\kappa}.
		\end{equation*}
	\end{details}
	For every $n\geq T$ it follows by~\cite[Appendix, (7.3)]{ethier_kurtz_86} that $\int_{[-T,T]}\abs{\phi_{n}(r)}\dd r\leq1$,
	\begin{details}
		By definition, it is clear that $\phi_{n}(r)\geq0$ for $r\in[-n,n]$. Further, by~\cite[Appendix, (7.3)]{ethier_kurtz_86},
		\begin{equation*}
			\begin{split}
				\int_{[-n,n]}\abs{\phi_{n}(r)}\dd r&=\int_{[-n,n]}\phi_{n}(r)\euler^{\abs{rn}^{2}}\euler^{-\abs{rn}^{2}}\dd r=\int_{[-n^{2},n^{2}]}\Bigl(1-\frac{\abs{u}^{2}}{n^{4}}\Bigr)^{n^{4}}\euler^{\abs{u}^{2}}\uppi^{-1/2}\euler^{-\abs{u}^{2}}\dd u\\
				&\leq\int_{[-n^{2},n^{2}]}\uppi^{-1/2}\euler^{-\abs{u}^{2}}\dd u\leq1.
			\end{split}
		\end{equation*}
		Hence for $n\geq T$,
		\begin{equation*}
			\int_{[-T,T]}\abs{\phi_{n}(r)}\dd r\leq\int_{[-n,n]}\abs{\phi_{n}(r)}\dd r\leq1.
		\end{equation*}
	\end{details}
	which combined with~\eqref{eq:little_Hoelder_characterisation_application} and the definition of $f_{n}$ yields
	\begin{align*}
		\norm{f_n(t)-f_n(s)}_{\mcB_{p,q}^{\alpha}}\leq\int_{[-T,T]}\norm{f(t-r)-f(s-r)}_{\mcB_{p,q}^{\alpha}}\abs{\phi_n(r)}\dd r\leq\frac{1}{2}\eps\abs{t-s}^{\kappa}.
	\end{align*}
	\begin{details}
		We can restrict the domain of integration to $[-T,T]$, by the following argument: If $r>T$, then $t-r<t-T\leq0$ and $s-r<s-T\leq0$ so that $f(t-r)-f(s-r)=f(0)-f(0)=0$. If $r<-T$, then $t-r>t+T\geq T$ and $s-r>s+T\geq T$ so that $f(T)-f(T)=0$.
		
	\end{details}
	Consequently we obtain for all $\abs{t-s}<\tau$, that
	\begin{equation*}
		\norm{f(t)-f(s)-f_n(t)+f_n(s)}_{\mcB_{p,q}^{\alpha}}<\eps\abs{t-s}^{\kappa}.
	\end{equation*}
	Let $n$ be sufficiently large such that $\norm{f-f_n}_{C_T\mcB_{p,q}^{\alpha}}<\eps\tau^{\kappa}/2$. Assume $\abs{t-s}>\tau$, then
	\begin{equation*}
		\norm{f(t)-f(s)-f_n(t)+f_n(s)}_{\mcB_{p,q}^{\alpha}}\leq 2\norm{f-f_n}_{C_T\mcB_{p,q}^{\alpha}}<\eps\tau^{\kappa}<\eps\abs{t-s}^{\kappa}.
	\end{equation*}
	This implies that for all $s\neq t\in[0,T]$ and $n$ sufficiently large,
	\begin{equation*}
		\norm{(f-f_{n})(t)-(f-f_{n})(s)}_{\mcB_{p,q}^{\alpha}}<\eps\abs{t-s}^{\kappa}.
	\end{equation*}
	Since $\eps>0$ was arbitrary, we obtain $\lim_{n\to\infty}\norm{f-f_{n}}_{C_{T}^{\kappa}\mcB_{p,q}^{\alpha}}=0$, which shows that each $f\from[0,T]\to\mcB_{p,q}^{\alpha}(\mbT^{d})$ satisfying~\eqref{eq:little_Hoelder_characterisation} can be approximated by a polynomial. In particular, $f\in C_{T}^{\kappa}\mcB_{p,q}^{\alpha}(\mbT^{d})$, which proves that~\eqref{eq:little_Hoelder_characterisation} characterises the space $C_{T}^{\kappa}\mcB_{p,q}^{\alpha}(\mbT^{d})$.
	
	To show that $C_{T}^{\kappa}\mcB_{p,q}^{\alpha}(\mbT^{d})$ is separable, it suffices to approximate every $f\in C_{T}^{\kappa}\mcB_{p,q}^{\alpha}(\mbT^{d})$ by a polynomial with coefficients in a countable, dense subset of $\mcB_{p,q}^{\alpha}(\mbT^{d})$, which exists by~\eqref{eq:polynomial_approx} and the separability of $\mcB_{p,q}^{\alpha}(\mbT^{d})$ (Lemma~\ref{lem:Besov_separable}). This yields the claim.
	\begin{details}
		\paragraph{Construction of the countable, dense subset.}
		Let $f\in C_{T}^{\kappa}\mcB_{p,q}^{\alpha}(\mbT^{d})$, by~\eqref{eq:polynomial_approx} we can find a sequence of polynomials $(f_{n})_{n\in\mbN}$ with coefficients in $\mcB_{p,q}^{\alpha}(\mbT^{d})$ such that $f_{n}\to f$ in $C_{T}^{\kappa}\mcB_{p,q}^{\alpha}(\mbT^{d})$ as $n\to\infty$. Hence, it suffices to approximate every polynomial with coefficients in $\mcB_{p,q}^{\alpha}(\mbT^{d})$ by a polynomial with coefficients in a countable, dense subset $D\subset\mcB_{p,q}^{\alpha}(\mbT^{d})$, which exists by the separability of $\mcB_{p,q}^{\alpha}(\mbT^{d})$ (Lemma~\ref{lem:Besov_separable}). By the triangle inequality, it further suffices to approximate the monomial $t^{k}u$ for each $k\in\mbN_{0}$ and $u\in\mcB_{p,q}^{\alpha}(\mbT^{d})$. Let $(u_{n})_{n\in\mbN}$ be a sequence such that $u_{n}\in D$ for every $n\in\mbN$ and $u_{n}\to u\in\mcB_{p,q}^{\alpha}(\mbT^{d})$ as $n\to\infty$. We obtain
		\begin{equation*}
			\sup_{t\in[0,T]}\norm{t^{k}u-t^{k}u_{n}}_{\mcB_{p,q}^{\alpha}}\leq T^{k}\norm{u-u_{n}}_{\mcB_{p,q}^{\alpha}}\to0\quad\text{as}~n\to\infty
		\end{equation*}
		and for $k\neq0$,
		\begin{equation*}
			\begin{split}
				\sup_{s\neq t\in[0,T]}\frac{\norm{t^{k}u-t^{k}u_{n}-(s^{k}u-s^{k}u_{n})}_{\mcB_{p,q}^{\alpha}}}{\abs{t-s}^{\kappa}}&=\sup_{s\neq t\in[0,T]}\frac{\norm{(t^{k}-s^{k})u-(t^{k}-s^{k})u_{n}}_{\mcB_{p,q}^{\alpha}}}{\abs{t-s}^{\kappa}}\\
				&=\sup_{s\neq t\in[0,T]}\frac{\abs{t^{k}-s^{k}}}{\abs{t-s}^{\kappa}}\norm{u-u_{n}}_{\mcB_{p,q}^{\alpha}}\\
				&\leq k T^{k-\kappa}\norm{u-u_{n}}_{\mcB_{p,q}^{\alpha}}\to0\qquad\text{as}~n\to\infty.
			\end{split}
		\end{equation*}
		Hence, we can approximate every monomial with coefficients in $\mcB_{p,q}^{\alpha}(\mbT^{d})$ by a monomial with coefficients in $D$, which by the argument above yields the separability of $C_{T}^{\kappa}\mcB_{p,q}^{\alpha}(\mbT^{d})$.
	\end{details}
\end{proof}
The next lemma presents a natural criterion that allows us to find elements of $C_{T}^{\kappa}\mcB_{p,q}^{\alpha}(\mbT^{d})$.
\begin{lemma}\label{lem:little_Hoelder_criterion}
	Let $T>0$, $p,q\in[1,\infty]$, $\alpha<\alpha'\in\mbR$ and $0\leq\kappa<\kappa'<1$.
	\begin{itemize}
		\item If $q\in[1,\infty)$, then each $f\from[0,T]\to\mcS'(\mbT^{d})$ such that $\norm{f}_{C_{T}^{\kappa'}\mcB_{p,q}^{\alpha}}<\infty$ is an element of $C_{T}^{\kappa}\mcB_{p,q}^{\alpha}(\mbT^{d})$.
		\item If $q=\infty$, then each $f\from[0,T]\to\mcS'(\mbT^{d})$ such that $\norm{f}_{C_{T}^{\kappa'}\mcB_{p,\infty}^{\alpha'}}<\infty$ is an element of $C_{T}^{\kappa}\mcB_{p,q}^{\alpha}(\mbT^{d})$.
	\end{itemize}
\end{lemma}
\begin{proof}
	Let $T>0$, $p,q\in[1,\infty]$, $\alpha\in\mbR$ and $0\leq\kappa<\kappa'<1$. To prove the first claim, let $q<\infty$ and $f\from[0,T]\to\mcS'(\mbT^{d})$ be such that $\norm{f}_{C_{T}^{\kappa'}\mcB_{p,q}^{\alpha}}<\infty$. By Lemma~\ref{lem:Besov_criterion} it follows that $f\from[0,T]\to\mcB_{p,q}^{\alpha}(\mbT^{d})$. Let $r>0$ and $s\neq t\in[0,T]$ with $\abs{t-s}<r$, we obtain
	\begin{equation*}
		\frac{\norm{f(t)-f(s)}_{\mcB_{p,q}^{\alpha}}}{\abs{t-s}^{\kappa}}\lesssim\abs{t-s}^{\kappa'-\kappa}\frac{\norm{f(t)-f(s)}_{\mcB_{p,q}^{\alpha}}}{\abs{t-s}^{\kappa'}}\leq r^{\kappa'-\kappa}\norm{f}_{C_{T}^{\kappa'}\mcB_{p,q}^{\alpha}}\to0\quad\text{as}~r\to0,
	\end{equation*}
	hence $f$ satisfies~\eqref{eq:little_Hoelder_characterisation}, which implies $f\in C_{T}^{\kappa}\mcB_{p,q}^{\alpha}(\mbT^{d})$ by Lemma~\ref{lem:little_Hoelder_criterion}. The second claim follows by the same argument, using that $\norm{f}_{C_{T}^{\kappa'}\mcB_{p,\infty}^{\alpha'}}<\infty$ implies $f\from[0,T]\to\mcB_{p,\infty}^{\alpha}(\mbT^{d})$ by Lemma~\ref{lem:Besov_criterion}.
\end{proof}
Recall the definition $\msL_{T}^{\kappa}\mcB_{p,q}^{\alpha}(\mbT^{d})\defeq C_{T}^{\kappa}\mcB_{p,q}^{\alpha-2\kappa}(\mbT^{d})\cap C_{T}\mcB_{p,q}^{\alpha}(\mbT^{d})$ (see Subsection~\ref{subsec:notation}). Next we show that each $\msL_{T}^{\kappa}\mcB_{p,q}^{\alpha}(\mbT^{d})$ is separable.
\begin{lemma}\label{lem:separability_interpolation}
	Let $T>0$, $\alpha\in\mbR$ and $\kappa\in(0,1)$, then $\msL_{T}^{\kappa}\mcB_{p,q}^{\alpha}(\mbT^{d})=C_{T}^{\kappa}\mcB_{p,q}^{\alpha-2\kappa}(\mbT^{d})\cap C_{T}\mcB_{p,q}^{\alpha}(\mbT^{d})$ is separable.
\end{lemma}
\begin{proof}
	The proof is classical: There exists an isometric isomorphism that maps $\msL_{T}^{\kappa}\mcB_{p,q}^{\alpha}(\mbT^{d})=C_{T}^{\kappa}\mcB_{p,q}^{\alpha-2\kappa}(\mbT^{d})\cap C_{T}\mcB_{p,q}^{\alpha}(\mbT^{d})$ to the diagonal $\{(f,f):f\in\msL_{T}^{\kappa}\mcB_{p,q}^{\alpha}(\mbT^{d})\}\subset C_{T}^{\kappa}\mcB_{p,q}^{\alpha-2\kappa}(\mbT^{d})\times C_{T}\mcB_{p,q}^{\alpha}(\mbT^{d})$. The product space $C_{T}^{\kappa}\mcB_{p,q}^{\alpha-2\kappa}(\mbT^{d})\times C_{T}\mcB_{p,q}^{\alpha}(\mbT^{d})$ is a separable metric space by Lemma~\ref{lem:little_Hoelder_separable}, which implies that the diagonal is a separable subspace. Using that isometric isomorphisms preserve separability, it follows that each $\msL_{T}^{\kappa}\mcB_{p,q}^{\alpha}(\mbT^{d})$ is separable.
\end{proof}
Let $\target$ be a separable, normed vector space and define $\target^{\sol}_{T}$ as in Subsection~\ref{subsec:notation}. In the next lemma we show that $\target^{\sol}_{T}$ is separable.
%
%
\begin{lemma}\label{lem:FsolT_separability}
	Let $T>0$ and $\target$ be a separable, normed vector space, then $\target^{\sol}_{T}$ is separable.
\end{lemma}
\begin{proof}
	The space $\target^{\sol}_{T}$ can be identified with a subspace of the metric space $\target^{\sol}$ defined in~\cite[Sec.~1.5.1]{chandra_chevyrev_hairer_shen_22}. As discussed in~\cite[Sec.~1.5.1]{chandra_chevyrev_hairer_shen_22}, $\target^{\sol}$ is a separable metric space, hence $\target^{\sol}_{T}$ is a separable subspace.
\end{proof}
Next we define the space of periodic space-time H\"{o}lder distributions (Definition~\ref{def:Cs_alpha_periodic}; see also Definition~\ref{def:Cs_alpha_torus} for the analogous notion on the torus), which we use in Subsection~\ref{subsec:existence_and_regularity} to quantify the regularity of space-time white noise (Theorem~\ref{thm:existence_stWN}).

Denote by $\mfs=(\mfs_{i})_{i=0}^{d}=(2,1,\ldots,1)$ the parabolic scaling of $\mbR\times\mbR^{d}$ with norm $\abs{\mfs}=\sum_{i=0}^{d}\mfs_{i}=2+d$.
\begin{definition}[{\cite[(2.10)~\&~(2.13)]{hairer_14_RegStruct}}]
	Let $z=(t,x)\in\mbR\times\mbR^{d}$ and $z'=(t',x')\in\mbR\times\mbR^{d}$, we define the parabolic distance
	\begin{equation*}
		d_{\mfs}(z,z')\defeq\abs{t-t'}^{1/\mfs_{0}}+\sum_{i=1}^{d}\abs{x_{i}-x'_{i}}^{1/\mfs_{i}}
	\end{equation*}
	and denote the parabolic unit ball by $B_{\mfs}(0,1)\defeq\{z\in\mbR\times\mbR^{d}:\abs{z}_{\mfs}<1\}$, where
	\begin{equation*}
		\abs{z}_{\mfs}\defeq\abs{t}^{1/\mfs_{0}}+\sum_{i=1}^{d}\abs{x_{i}}^{1/\mfs_{i}}
	\end{equation*}
	(by a slight abuse of notation, since $\abs{\place}_{\mfs}$ is not homogeneous, hence fails to be a norm.) We define the scaling of a function by $\lambda\in(0,1]$ around the point $z$ via
	\begin{equation*}
		\mcS_{\mfs,z}^{\lambda}f(s,y)\defeq\lambda^{-\abs{\mfs}}f(\lambda^{-\mfs_{0}}(s-t),\lambda^{-\mfs_{1}}(y_{1}-x_{1}),\ldots,\lambda^{-\mfs_{d}}(y_{d}-x_{d})),\qquad f\from\mbR\times\mbR^{d}\to\mbR.
	\end{equation*}
\end{definition}
We can now define the space $\mcC_{\mfs,\per}^{\alpha}([0,T]\times\mbR^{d})$ of periodic distributions of H\"{o}lder regularity $\alpha<0$.
\begin{definition}\label{def:Cs_alpha_periodic}
	Let $\alpha<0$ and $r_{\alpha}\defeq\ceil{-\alpha}$, we define for each compact $\mfK\subset\mbR\times\mbR^{d}$ the seminorm
	\begin{equation*}
		\norm{u}_{\mcC_{\mfs}^{\alpha}(\mfK)}\defeq\sup_{\lambda\in(0,1]}\sup_{\eta\in \mcB_{\mfs}^{r_{\alpha}}}\sup_{z\in\mfK}\frac{\abs{\inner{u}{\mcS_{\mfs,z}^{\lambda}\eta}}}{\lambda^{\alpha}},\qquad u\in\mcS'(\mbR\times\mbR^{d}).
	\end{equation*}
	where
	\begin{equation*}
		\mcB^{r_{\alpha}}_{\mfs}\defeq\{\eta\in C^{r_{\alpha}}(\mbR\times\mbR^{d}):\supp(\eta)\subset B_{\mfs}(0,1),~\norm{\eta}_{C^{r_{\alpha}}(\mbR\times\mbR^{d})}\leq1\},
	\end{equation*}
	and
	\begin{equation*}
		\norm{\eta}_{C^{r_{\alpha}}(\mbR\times\mbR^{d})}\defeq\sup_{\substack{\multi\in\mbN_{0}^{1+d}\\\abs{\multi}\leq r_{\alpha}}}\sup_{z\in\mbR\times\mbR^{d}}\abs{\partial^{\multi}\eta(z)}.
	\end{equation*}
	We define $\mcC_{\mfs,\per}^{\alpha}([0,T]\times\mbR^{d})\subset\mcS'_{\per}(\mbR\times\mbR^{d})$ as the completion of $C^{\infty}([0,T]\times[0,1]^{d})$ under the norm $\norm{\place}_{\mcC^{\alpha}_{\mfs}([0,T]\times[0,1]^{d})}$, where we extend each element of $C^{\infty}([0,T]\times[0,1]^{d})$ to $\mbR\times\mbR^{d}$ periodically in space and by $0$ in time.
\end{definition}
\begin{details}
	From now on we identify each $u\in C^{\infty}([0,T]\times[0,1]^{d})$ with its extension to $\mbR\times\mbR^{d}$.
	
	To show that the space $\mcC_{\mfs,\per}^{\alpha}([0,T]\times\mbR^{d})$ is well-defined, we first prove that $\norm{\place}_{\mcC^{\alpha}_{\mfs}([0,T]\times[0,1]^{d})}$ is a norm on (extended) elements of $C^{\infty}([0,T]\times[0,1]^{d})$.
	\begin{lemma}
		Let $u\in C^{\infty}([0,T]\times[0,1]^{d})$ be such that $\norm{u}_{\mcC^{\alpha}_{\mfs}([0,T]\times[0,1]^{d})}=0$, then $u\equiv0$ in $\mbR\times\mbR^{d}$.
	\end{lemma}
	\begin{proof}
		Let $z\in[0,T]\times[0,1]^{d}$, $\lambda\leq1$ and $\eta\in\mcB_{\mfs}^{r_{\alpha}}$ be arbitrary. We obtain by assumption $\abs{\inner{u}{\mcS^{\lambda}_{\mfs,z}\eta}}=0$, which together with the continuity of $u$ implies that $u\equiv0$ in $[0,T]\times[0,1]^{d}$. By the periodicity of $u$, we can deduce $u\equiv0$ in $[0,T]\times\mbR^{d}$, hence $u\equiv0$.
	\end{proof}
	The completion of $C^{\infty}([0,T]\times[0,1]^{d})$ under $\norm{\place}_{\mcC^{\alpha}_{\mfs}([0,T]\times[0,1]^{d})}$ is a Banach space that consists of equivalence classes of Cauchy sequences. To identify each equivalence class with an element of $\mcS_{\per}'(\mbR\times\mbR^{d})$, we show that convergence under $\norm{\place}_{\mcC^{\alpha}_{\mfs}([0,T]\times[0,1]^{d})}$ implies weak*-convergence in $\mcS'(\mbR\times\mbR^{d})$.
	\begin{lemma}\label{lem:duality}
		Let $u\in C^{\infty}([0,T]\times[0,1]^{d})$ and $\test\in\mcS(\mbR\times\mbR^{d})$, then for every $\alpha<0$,
		\begin{equation*}
			\inner{u}{\test}_{L^{2}(\mbR\times\mbR^{d})}\lesssim\norm{u}_{\mcC^{\alpha}_{\mfs}([0,T]\times[0,1]^{d})}\Bigl\lVert\sum_{z\in\mbZ^{d}}\mcT_{z}\test\Bigr\rVert_{C^{r_{\alpha}}((-1,T+1)\times(-1,1)^{d})},
		\end{equation*}
		where $\norm{\sum_{z\in\mbZ^{d}}\mcT_{z}\test}_{C^{r_{\alpha}}((-1,T+1)\times(-1,1)^{d})}<\infty$ by~\cite[(29.2)]{vanzuijlen_22}.
	\end{lemma}
	\begin{proof}
		We denote for each $y\in\mbR\times\mbR^{d}$ and $\lambda\in(0,1]$,
		\begin{equation*}
			\mcB_{\mfs,y}^{r_{\alpha},\lambda}\defeq\{\mcS_{\mfs,y}^{\lambda}\eta:\eta\in\mcB_{\mfs}^{r_{\alpha}}\}\quad\text{and}\quad B_{\mfs}(y,\lambda)\defeq\{z\in\mbR\times\mbR^{d}:\abs{z-y}_{\mfs}<\lambda\}.
		\end{equation*}
		Let $u\in C^{\infty}([0,T]\times[0,1]^{d})$ and $\test\in\mcS(\mbR\times\mbR^{d})$, combining $\supp(u)\subset[0,T]\times\mbR^{d}$ with~\cite[Lem.~29.3~\&~Thm.~29.4]{vanzuijlen_22} there exists a smooth cut-off $\chi\in C_{\comp}^{\infty}(\mbR\times\mbR^{d};[0,1])$ with $\supp(\chi)\subset[-1,T+1]\times[-1,1]^{d}$ such that
		\begin{equation}\label{eq:periodicity_cutoff}
			\inner{u}{\test}_{L^{2}(\mbR\times\mbR^{d})}=\Bigl\langle u,\chi\sum_{z\in\mbZ^{d}}\mcT_{z}\test\Bigr\rangle_{L^{2}(\mbR\times\mbR^{d})},
		\end{equation}
		where $\mcT_{z}$ denotes the translation $\mcT_{z}\test(x)=\test(x-z)$. Hence without limitation of generality we may assume that $\supp(\test)\subset[-1,T+1]\times[-1,1]^{d}$.
		
		Next we decompose $\test\from[-1,T+1]\times[-1,1]^{d}\to\mbR$ into a sum of (rescaled) test functions in $\mcB_{\mfs}^{r_{\alpha}}$. By the proof of~\cite[Lem.~14.13]{friz_hairer_20} there exists a lattice $\Lambda\subset\mbR\times\mbR^{d}$ and a partition of unity $(\chi_{y})_{y\in\Lambda}$ such that each $\chi_{y}$ is an element of $\mcB_{\mfs,y}^{r_{\alpha},1}$. We obtain
		\begin{equation}\label{eq:partition_of_unity}
			\inner{u}{\test}_{L^{2}(\mbR\times\mbR^{d})}=\sum_{\substack{y\in\Lambda\\ B_{\mfs}(y,1)\cap[0,T]\times[-1,1]^{d}\neq\emptyset}}\inner{u}{\chi_{y}\test}_{L^{2}(\mbR\times\mbR^{d})},
		\end{equation}
		where we used that $\supp(u)\cap\supp(\test)\subset[0,T]\times[-1,1]^{d}$ to reduce the sum. Define $\mfK=[0,T]\times[-1,1]^{d}$ and denote by $\overline{\mfK}_{1}$ the $1$-fattening of $\mfK$, i.e.\ the closure of $\mfK+B_{\mfs}(0,1)$. The restriction $B_{\mfs}(y,1)\cap\mfK\neq\emptyset$ in~\eqref{eq:partition_of_unity} implies that $y\in\overline{\mfK}_{1}$. Consequently, the sum over $y$ is finite and every $y$ is an element of $\Lambda\cap\overline{\mfK}_{1}$. We obtain by Leibniz' rule uniformly for every $y\in\Lambda\cap\overline{\mfK}_{1}$,
		\begin{equation}\label{eq:cut_off_estimate}
			\norm{\chi_{y}\test}_{C^{r_{\alpha}}(\mbR\times\mbR^{d})}\lesssim\norm{\chi_{y}}_{C^{r_{\alpha}}(\mbR\times\mbR^{d})}\norm{\test}_{C^{r_{\alpha}}((-1,T+1)\times(-1,1)^{d})}\leq\norm{\test}_{C^{r_{\alpha}}((-1,T+1)\times(-1,1)^{d})}
		\end{equation}
		so that, upon rescaling $\test$ by a multiple of $\norm{\test}_{C^{r_{\alpha}}((-1,T+1)\times(-1,1)^{d})}^{-1}$ if necessary, we may assume $\norm{\chi_{y}\test}_{C^{r_{\alpha}}(\mbR\times\mbR^{d})}\leq1$, which implies $\chi_{y}\test\in\mcB_{\mfs,y}^{r_{\alpha},1}$.
		
		So far, the right hand side of~\eqref{eq:partition_of_unity} runs over $y\in\Lambda$ such that $y\in\overline{\mfK}_{1}$. In the next step, we control~\eqref{eq:partition_of_unity} by a sum over $y\in\Lambda$ such that $y\in[0,T]\times[0,1]^{d}$.
		
		Assume that $y=(t,x)\in\Lambda\cap\overline{\mfK}_{1}$. Using the periodicity of $u$,
		\begin{equation*}
			\inner{u}{\chi_{y}\test}_{L^{2}(\mbR\times\mbR^{d})}=\inner{u}{\mcT_{-\floor{x}}\chi_{y}\test}_{L^{2}(\mbR\times\mbR^{d})}
		\end{equation*}
		where $x-\floor{x}\in[0,1]^{d}$. Hence, we can undo the fattening of the space component.
		
		Assume $y=(t,x)\in\Lambda\cap\overline{\mfK}_{1}$ is such that $t<0$ and define $y_{0}=(0,x)$ (resp.\ $y_{T}=(T,x)$ if $t>T$). Then there exists a smooth cut-off function $\varphi\from\mbR\times\mfK\to[0,1]$ such that $\varphi\equiv1$ on $[0,T]\times\mfK$ and $\supp(\varphi\chi_{y}\test)\subset B_{\mfs}(y_{0},1)$ for all such $y$ (resp.\ $\supp(\varphi\chi_{y}\test)\subset B_{\mfs}(y_{T},1)$ if $t>T$), see Figure~\ref{fig:cut_off_construction}. By Leibniz' rule
		\begin{equation*}
			\norm{\varphi\chi_{y}\test}_{C^{r_{\alpha}}(\mbR\times\mbR^{d})}\lesssim\norm{\varphi}_{C^{r_{\alpha}}(\mbR\times\mfK)}\norm{\chi_{y}\test}_{C^{r_{\alpha}}(\mbR\times\mbR^{d})}
		\end{equation*}
		and upon rescaling $\test$ further, if necessary, we may assume $\norm{\varphi\chi_{y}\test}_{C^{r_{\alpha}}(\mbR\times\mbR^{d})}\leq1$ and $\varphi\chi_{y}\test\in\mcB_{\mfs,y_{0}}^{r_{\alpha},1}$ for all $y=(t,x)\in\Lambda\cap\overline{\mfK}_{1}$ such that $t<0$ (resp.\ $\varphi\chi_{y}\test\in\mcB_{\mfs,y_{T}}^{r_{\alpha},1}$ for all $y=(t,x)\in\Lambda\cap\overline{\mfK}_{1}$ such that $t>T$.) Using that $u$ is $0$ on $[0,T]^{c}$ and $\varphi$ is $1$ on $[0,T]$, we obtain for such $y$,
		\begin{equation*}
			\inner{u}{\chi_{y}\test}_{L^{2}(\mbR\times\mbR^{d})}=\inner{u}{\varphi\chi_{y}\test}_{L^{2}(\mbR\times\mbR^{d})}
		\end{equation*}
		and we may assume that all $y=(t,x)$ in~\eqref{eq:partition_of_unity} satisfy $t\in[0,T]$.
		
		\begin{figure}[tbph!]
			\centering
			\begin{tikzpicture}
				\node at (0,0) [] (KBL) {$0$}; 
				\node at (0,4) [] (KTL) {$T$}; 
				\node at (4,0) [] (KBR) {$1$}; 
				\node at (4,4) [] (KTR) {}; 
				\node at (2,2) [] (K) {$\mfK$};
				\node at (0,-1) [] (K1BL) {}; 
				\node at (0,5) [] (K1TL) {}; 
				\node at (4,-1) [] (K1BR) {}; 
				\node at (4,5) [] (K1TR) {}; 
				\node at (-1,0) [] (K2BL) {}; 
				\node at (-1,4) [] (K2TL) {}; 
				\node at (5,0) [] (K2BR) {}; 
				\node at (5,4) [] (K2TR) {};
				\node at (4.5,2) [] (K1) {$\overline{\mfK}_{1}$};
				\node at (0,0) [] (S1L) {}; 
				\node at (1,1) [] (S1T) {}; 
				\node at (2,0) [] (S1R) {}; 
				\node at (1,-1) [] (S1B) {};
				\node at (1,0)[](){$y_{0}^{*}$};
				\node at (0,-0.5) [] (S2L) {}; 
				\node at (1,0.5) [] (S2T) {}; 
				\node at (2,-0.5) [] (S2R) {}; 
				\node at (1,-1.5) [] (S2B) {};
				\node at (1,-0.5) [] () {$y^{*}$};
				\node at (2,0) [] (S3L) {}; 
				\node at (3,1) [] (S3T) {}; 
				\node at (4,0) [] (S3R) {}; 
				\node at (3,-1) [] (S3B) {};
				\node at (3,0)[](){$y_{0}$};
				\node at (2,-1) [] (S4L) {}; 
				\node at (3,0) [] (S4T) {}; 
				\node at (4,-1) [] (S4R) {}; 
				\node at (3,-2) [] (S4B) {};
				\node at (3,-1) [] () {$y$};
				\node at (0,-0.25) [] (cut_off_BL) {};
				\node at (4,-0.25) [] (cut_off_BR) {};
				\node at (0,4.5) [] (cut_off_TL) {};
				\node at (4,4.5) [] (cut_off_TR) {};
				\node at (2,4.25) [] () {$\supp(\varphi)$};
				\draw (KBL) to (KTL); 
				\draw (KBL) to (KBR); 
				\draw (KTL) to (KTR.center); 
				\draw (KBR.center) to (KTR.center); 
				\draw (K1BL.center) to (K1BR.center); 
				\draw (K1TL.center) to (K1TR.center);
				\draw (K2BL.center) to (K2TL.center); 
				\draw (K2BR.center) to (K2TR.center); 
				\draw[bend left=45] (K1TL.center) to (K2TL.center);
				\draw[bend right=45] (K1BL.center) to (K2BL.center);
				\draw[bend right=45] (K1TR.center) to (K2TR.center);
				\draw[bend left=45] (K1BR.center) to (K2BR.center);
				\draw[bend left=45] (S1T.center) to (S1L.center);
				\draw[bend right=45] (S1T.center) to (S1R.center);
				\draw[bend right=45] (S1B.center) to (S1L.center);
				\draw[bend left=45] (S1B.center) to (S1R.center);
				\draw[bend left=45] (S2T.center) to (S2L.center);
				\draw[bend right=45] (S2T.center) to (S2R.center);
				\draw[bend right=45] (S2B.center) to (S2L.center);
				\draw[bend left=45] (S2B.center) to (S2R.center);
				\draw[bend left=45] (S3T.center) to (S3L.center);
				\draw[bend right=45] (S3T.center) to (S3R.center);
				\draw[bend right=45] (S3B.center) to (S3L.center);
				\draw[bend left=45] (S3B.center) to (S3R.center);
				\draw[bend left=45] (S4T.center) to (S4L.center);
				\draw[bend right=45] (S4T.center) to (S4R.center);
				\draw[bend right=45] (S4B.center) to (S4L.center);
				\draw[bend left=45] (S4B.center) to (S4R.center);
				\draw[dashed] (cut_off_BL.center) to (cut_off_BR.center); 
				\draw[dashed] (cut_off_TL.center) to (cut_off_TR.center); 
				\draw[dashed] (cut_off_BL.center) to (cut_off_TL.center); 
				\draw[dashed] (cut_off_BR.center) to (cut_off_TR.center); 
			\end{tikzpicture}
			\caption{Construction of the cut-off $\varphi$ in the proof of Lemma~\ref{lem:duality}. Let $y^{*}=(x^{*},t^{*})\in\Lambda\cap\overline{\mfK}_{1}$ be such that $t^{*}<0$ and $t\leq t^{*}$ for every $y=(x,t)\in\Lambda\cap\overline{\mfK}_{1}\cap(-\infty,0)\times\mbR^{d}$. The dashed region then shows the support of a suitable $\varphi$ satisfying $\supp(\varphi)\subset[t^{*}/2,\infty)\times\mfK$. An analogous argument further yields an upper bound (in time) on the support.}
			\label{fig:cut_off_construction}
		\end{figure}
		
		By the argument above, we may control for all $y\in\Lambda\cap\overline{\mfK}_{1}$,
		\begin{equation*}
			\inner{u}{\chi_{y}\test}_{L^{2}(\mbR\times\mbR^{d})}\leq\sup_{\lambda\in(0,1]}\sup_{\eta\in\mcB_{\mfs}^{r_{\alpha}}}\sup_{y\in[0,T]\times[0,1]^{d}}\frac{\abs{\inner{u}{\mcS_{\mfs,y}^{\lambda}\eta}}}{\lambda^{\alpha}}=\norm{u}_{\mcC^{\alpha}_{\mfs}([0,T]\times[0,1]^{d})}.
		\end{equation*}
		\emph{To summarize:} Let $\test\in\mcS(\mbR\times\mbR^{d})$ be arbitrary. Let $\chi$ be a smooth cut-off such that $\supp(\chi)\subset[-1,T+1]\times[-1,1]^{d}$ as in~\cite[Lem.~29.3]{vanzuijlen_22} and let $(\chi_{y})_{y\in\Lambda}$ be as in the proof of~\cite[Lem.~14.13]{friz_hairer_20}. Using~\eqref{eq:periodicity_cutoff},
		\begin{equation*}
			\inner{u}{\test}_{L^{2}(\mbR\times\mbR^{d})}=C\Bigl\lVert\chi\sum_{z\in\mbZ^{d}}\mcT_{z}\test\Bigr\rVert_{C^{r_{\alpha}}((-1,T+1)\times(-1,1)^{d})}\Bigl\langle u,\frac{\chi\sum_{z\in\mbZ^{d}}\mcT_{z}\test}{C\norm{\chi\sum_{z\in\mbZ^{2}}\mcT_{z}\test}_{C^{r_{\alpha}}((-1,T+1)\times(-1,1)^{d})}}\Bigr\rangle_{L^{2}(\mbR\times\mbR^{d})},
		\end{equation*}
		where $C>0$ is sufficiently large (cf.~\eqref{eq:cut_off_estimate}) such that for all $y\in\Lambda\cap\overline{\mfK}_{1}$,
		\begin{equation*}
			\frac{\norm{\chi_{y}\chi\sum_{z\in\mbZ^{d}}\mcT_{z}\test}_{C^{r_{\alpha}}(\mbR\times\mbR^{d})}}{C\norm{\chi\sum_{z\in\mbZ^{d}}\mcT_{z}\test}_{C^{r_{\alpha}}((-1,T+1)\times(-1,1)^{d})}}\leq 1.
		\end{equation*}
		We may decompose as in the proof of~\cite[Lem.~14.13]{friz_hairer_20},
		\begin{equation*}
			\begin{split}
				&\inner{u}{\test}_{L^{2}(\mbR\times\mbR^{d})}\\
				&=C\Bigl\lVert\chi\sum_{z\in\mbZ^{d}}\mcT_{z}\test\Bigr\rVert_{C^{r_{\alpha}}((-1,T+1)\times(-1,1)^{d})}\sum_{y\in\Lambda\cap\overline{\mfK}_{1}}\Bigl\langle u,\chi_{y}\frac{\chi\sum_{z\in\mbZ^{d}}\mcT_{z}\test}{C\norm{\chi\sum_{z\in\mbZ^{d}}\mcT_{z}\test}_{C^{r_{\alpha}}((-1,T+1)\times(-1,1)^{d})}}\Bigr\rangle_{L^{2}(\mbR\times\mbR^{d})}.
			\end{split}
		\end{equation*}
		We use the periodicity of $u$ and smooth cut-offs in time to estimate
		\begin{equation*}
			\sum_{y\in\Lambda\cap\overline{\mfK}_{1}}\Bigl\langle u,\chi_{y}\frac{\chi\sum_{z\in\mbZ^{d}}\mcT_{z}\test}{C\norm{\chi\sum_{z\in\mbZ^{d}}\mcT_{z}\test}_{C^{r_{\alpha}}((-1,T+1)\times(-1,1)^{d})}}\Bigr\rangle_{L^{2}(\mbR\times\mbR^{d})}\lesssim\norm{u}_{\mcC^{\alpha}_{\mfs}([0,T]\times[0,1]^{d})}.
		\end{equation*}
		It follows that
		\begin{equation*}
			\begin{split}
				\inner{u}{\test}_{L^{2}(\mbR\times\mbR^{d})}&\lesssim\Bigl\lVert\chi\sum_{z\in\mbZ^{d}}\mcT_{z}\test\Bigr\rVert_{C^{r_{\alpha}}((-1,T+1)\times(-1,1)^{d})}\norm{u}_{\mcC^{\alpha}_{\mfs}([0,T]\times[0,1]^{d})}\\
				&\lesssim\norm{\chi}_{C^{r_{\alpha}}((-1,1)^{d})}\Bigl\lVert\sum_{z\in\mbZ^{d}}\mcT_{z}\test\Bigr\rVert_{C^{r_{\alpha}}((-1,T+1)\times(-1,1)^{d})}\norm{u}_{\mcC^{\alpha}_{\mfs}([0,T]\times[0,1]^{d})},
			\end{split}
		\end{equation*}
		which yields the claim.
	\end{proof}
	\begin{corollary}
		Let $\alpha<0$ and $(u_{n})_{n\in\mbN}$ be a Cauchy sequence under $\norm{\place}_{\mcC_{\mfs}^{\alpha}([0,T]\times[0,1]^{d})}$ of elements $u_{n}$ in $C^{\infty}([0,T]\times[0,1]^{d})$. It then follows that there exists a tempered distribution $u$ such that $u_{n}\to u$ in $\mcS'(\mbR\times\mbR^{d})$ under the weak*-topology as $n\to\infty$.
	\end{corollary}
	\begin{proof}
		Let $(u_{n})_{n\in\mbN}$ be a sequence that is Cauchy in $\norm{\place}_{\mcC_{\mfs}^{\alpha}([0,T]\times[0,1]^{d})}$ and such that $u_{n}\in C^{\infty}([0,T]\times[0,1]^{d})$ for every $n\in\mbN$. As a consequence of~\cite[Lem.~2.9~\&~Thm.~29.4]{vanzuijlen_22}, every element $u_{n}$, extended to $\mbR\times\mbR^{d}$ periodically in space and by $0$ in time, of the approximating sequence is an element of $\mcS'(\mbR\times\mbR^{d})$.
		
		Let $\test\in\mcS(\mbR\times\mbR^{d})$, it then follows by Lemma~\ref{lem:duality} that $(\inner{u_{n}}{\test}_{L^{2}(\mbR\times\mbR^{d})})_{n\in\mbN}$ induces a Cauchy sequence in $\mbR$. Using that $\mcS'(\mbR\times\mbR^{d})$ is weak*-sequentially complete~\cite[Thm.~15.10]{vanzuijlen_22}, we can conclude that there exists an $u\in\mcS'(\mbR\times\mbR^{d})$ such that $\inner{u}{\test}=\lim_{n\to\infty}\inner{u_{n}}{\test}_{L^{2}(\mbR\times\mbR^{d})}$.
	\end{proof}
	Next we show that each element of $\mcC_{\mfs,\per}^{\alpha}([0,T]\times\mbR^{d})$ is periodic in space and compactly supported in time (Lemma~\ref{lem:Cs_alpha_periodic_and_compact_support_in_time}). Let us first recap the definitions of a periodic distribution (Definition~\ref{def:periodic_distribution}) and of a distribution's support (Definition~\ref{def:supp_distr}).
	\begin{definition}[{\cite[Sec.~29--31]{vanzuijlen_22}}]\label{def:periodic_distribution}
		Define for each $z\in\mbZ^{d}$ the shift operator $\mcT_{z}f(x)\defeq f(x-z)$ on functions $f\from\mbR^{d}\to\mbR$. We call a function $f$ periodic if $\mcT_{z}f=f$ for each $z\in\mbZ^{d}$. Define $\mcT_{z}$ on $u\in\mcD'(\mbR^{d})$ by $\mcT_{z}u(\test)=u(\mcT_{-z}\test)$ for all $\test\in\mcD(\mbR^{d})$. We call a distribution $u\in\mcD'(\mbR^{d})$ periodic, if $\mcT_{z}u=u$ for each $z\in\mbZ^{d}$.
	\end{definition}
	\begin{definition}[{\cite[Def.~5.1~\&~Def.~15.3]{vanzuijlen_22}}]\label{def:supp_distr}
		We call a closed subset $A\subset\mbR\times\mbR^{d}$ a \emph{carrier} of $u\in\mcS'(\mbR\times\mbR^{d})$, if $u(\test)=0$ for every $\test\in\mcD(\mbR\times\mbR^{d})$ such that $\supp(\test)\cap A=\emptyset$. We define the support of $u$, denoted by $\supp(u)$, to be the smallest carrier of $u$ with respect to set inclusion.
	\end{definition}
	\begin{lemma}\label{lem:Cs_alpha_periodic_and_compact_support_in_time}
		Let $\alpha<0$ and $u\in\mcC_{\mfs,\per}^{\alpha}([0,T]\times\mbR^{d})$, then $\supp(u)\subset[0,T]\times\mbR^{d}$ and $u$ is periodic in space.
	\end{lemma}
	\begin{proof}
		By Definition~\ref{def:supp_distr}, $\supp(u)$ is the smallest carrier of $u$. Hence for the first claim it suffices to show that $\inner{u}{\test}=0$ for every $\test\in\mcD(\mbR\times\mbR^{d})$ such that $\supp(\test)\cap[0,T]\times\mbR^{d}=\emptyset$. 
		
		Let $(u_{n})_{n\in\mbN}$ be a sequence such that $u_{n}\in C^{\infty}([0,T]\times[0,1]^{d})$ for every $n\in\mbN$ and $\norm{u-u_{n}}_{\mcC^{\alpha}_{\mfs}([0,T]\times[0,1]^{d})}\to0$ as $n\to\infty$ and let $\test\in\mcD(\mbR\times\mbR^{d})$ be such that $\supp(\test)\cap[0,T]\times\mbR^{d}=\emptyset$. We extend each $u_{n}$ to $\mbR\times\mbR^{d}$ periodically in space and by $0$ in time. We obtain by Lemma~\ref{lem:duality} and the weak*-sequential completeness of $\mcS'(\mbR\times\mbR^{d})$,
		\begin{equation*}
			\inner{u}{\test}=\lim_{n\to\infty}\inner{u_{n}}{\test}_{L^{2}(\mbR\times\mbR^{d})}=0.
		\end{equation*}
		It follows that $\inner{u}{\test}=0$ for every $\test\in\mcD(\mbR\times\mbR^{d})$ such that $\supp(\test)\cap[0,T]\times\mbR^{d}=\emptyset$, which implies $\supp(u)\subset[0,T]\times\mbR^{d}$.
		
		For the second claim, let $\test\in\mcD(\mbR\times\mbR^{d})$ and $z\in\mbZ^{d}$. We obtain
		\begin{equation*}
			\inner{u}{\mcT_{z}\test}=\lim_{n\to\infty}\inner{u_{n}}{\mcT_{z}\test}_{L^{2}(\mbR\times\mbR^{d})}=\lim_{n\to\infty}\inner{u_{n}}{\test}_{L^{2}(\mbR\times\mbR^{d})}=\inner{u}{\test},
		\end{equation*}
		which proves that $u$ is periodic in space.
	\end{proof}
\end{details}
\begin{lemma}\label{lem:Cs_alpha_periodic_separable}
	For every $\alpha<0$ it follows that the normed vector space $\mcC_{\mfs,\per}^{\alpha}([0,T]\times\mbR^{d})$ is complete and separable.
\end{lemma}
\begin{proof}
	Completeness is clear since $\mcC_{\mfs,\per}^{\alpha}([0,T]\times\mbR^{d})$ is defined as a completion. To prove separability it suffices to show the existence of a countable subset of $\mcC_{\mfs,\per}^{\alpha}([0,T]\times\mbR^{d})$ which can approximate every $u\in C^{\infty}([0,T]\times[0,1]^{d})$ (extended to $\mbR\times\mbR^{d}$ periodically in space and by $0$ in time) under the norm $\norm{\place}_{\mcC^{\alpha}_{\mfs}([0,T]\times[0,1]^{d})}$.
	
	An application of the Stone--Weierstrass theorem shows that for each $\eps>0$ there exists a polynomial $p\from[0,T]\times[0,1]^{d}\to\mbR$ such that $\norm{u-p}_{L^{\infty}([0,T]\times[0,1]^{d})}\leq\eps$. We extend this polynomial to $\mbR\times\mbR^{d}$ periodically in space and by $0$ in time; in particular it still holds that $\norm{u-p}_{L^{\infty}(\overline{\mfK}_{1})}\leq\eps$, where $\mfK=[0,T]\times[0,1]^{d}$ and $\overline{\mfK}_{1}$ denotes the $1$-fattening of $\mfK$, that is, the closure of $\mfK+B_{\mfs}(0,1)$. We can now control the $\mcC^{\alpha}_{\mfs}(\mfK)$-norm of $u-p$ by its $L^{\infty}(\overline{\mfK}_{1})$-norm, since for every $\lambda\in(0,1]$, $\eta\in\mcB^{r_{\alpha}}_{\mfs}$ and $z\in\mfK$,
	\begin{equation*}
		\abs{\inner{u-p}{\mcS_{\mfs,z}^{\lambda}\eta}}\leq\norm{u-p}_{L^{\infty}(\overline{\mfK}_{1})}\norm{\eta}_{L^{1}(\mbR\times\mbR^{d};\mbR)}\lesssim\norm{u-p}_{L^{\infty}(\overline{\mfK}_{1})}\leq\eps,
	\end{equation*}
	which implies
	\begin{equation*}
		\norm{u-p}_{\mcC^{\alpha}_{\mfs}(\mfK)}=\sup_{\lambda\in(0,1]}\sup_{\eta\in\mcB_{\mfs}^{r_{\alpha}}}\sup_{z\in\mfK}\frac{\abs{\inner{u-p}{\mcS_{\mfs,z}^{\lambda}\eta}}}{\lambda^{\alpha}}\lesssim\eps,
	\end{equation*}
	showing that the (suitably extended) polynomials are dense in $C^{\infty}([0,T]\times[0,1]^{d})$ under the norm $\norm{\place}_{\mcC^{\alpha}_{\mfs}(\mfK)}$. A similar argument then implies that the (suitably extended) polynomials with rational coefficients are dense as well, which yields the separability of $\mcC_{\mfs,\per}^{\alpha}([0,T]\times\mbR^{d})$.
	\begin{details}
		Using the density of the rationals in the reals, we can find a polynomial $q$ with rational coefficients such that $\norm{p-q}_{L^{\infty}([0,T]\times[0,1]^{d})}\leq\eps$. As above, this implies $\norm{p-q}_{\mcC^{\alpha}_{\mfs}(\mfK)}\lesssim\eps$.
	\end{details}
\end{proof}
We define the space $\mcC^{\alpha}_{\mfs}([0,T]\times\mbT^{d})$ as those distributions in $\mcS'(\mbR\times\mbT^{d})$, whose periodization lies in $\mcC^{\alpha}_{\mfs,\per}([0,T]\times\mbR^{d})$.
\begin{definition}\label{def:Cs_alpha_torus}
	Let $\hompd\from\mcS'(\mbR\times\mbT^{d})\to\mcS'_{\per}(\mbR\times\mbR^{d})$ be a linear homeomorphism (which exists by Lemma~\ref{lem:homeom_per} below.) We define for every $\alpha<0$,
	\begin{equation*}
		\mcC_{\mfs}^{\alpha}([0,T]\times\mbT^{d})\defeq\{u\in\mcS'(\mbR\times\mbT^{d}):\hompd(u)\in\mcC_{\mfs,\per}^{\alpha}([0,T]\times\mbR^{d})\}
	\end{equation*}
	and equip $\mcC_{\mfs}^{\alpha}([0,T]\times\mbT^{d})$ with the natural norm that turns $\hompd$ into a linear isometric isomorphism.
\end{definition}
\begin{lemma}\label{lem:Cs_alpha_torus_separable}
	For every $\alpha<0$ it follows that $\mcC_{\mfs}^{\alpha}([0,T]\times\mbT^{d})$ is a separable Banach space.
\end{lemma}
\begin{proof}
	Completeness and separability are preserved under isometric isomorphisms; hence the claim follows by Lemma~\ref{lem:Cs_alpha_periodic_separable}.
\end{proof}
The following lemma is classical, here we follow the presentation contained in the unpublished notes~\cite{vanzuijlen_22}.
\begin{lemma}[{\cite[Lem.~31.3]{vanzuijlen_22}}]\label{lem:homeom_per}
	Let $\Psi\from\mcS(\mbR^{d})\to C^{\infty}_{\per}(\mbR^{d})$ be given by
	\begin{equation*}
		\Psi(\test)\defeq\sum_{z\in\mbZ^{d}}\mcT_{z}\test,\qquad\mcT_{z}\test(x)\defeq\test(x-z)
	\end{equation*}
	and $\gamma\from C^{\infty}_{\per}(\mbR^{d})\to C^{\infty}(\mbT^{d})$ be the linear homeomorphism
	\begin{equation*}
		\gamma(\test)([x])\defeq\test(x),\qquad[x]\defeq x+\mbZ^{d}\in\mbT^{d}.
	\end{equation*}
	Then $\hompd\from\mcS'(\mbT^{d})\to\mcS'_{\per}(\mbR^{d})$ given by
	\begin{equation*}
		\inner{\hompd(u)}{\test}=\inner{u}{\gamma(\Psi(\test))}
	\end{equation*}
	is a linear homeomorphism with inverse
	\begin{equation*}
		\hompd^{-1}\from\mcS'_{\per}(\mbR^{d})\to\mcS'(\mbT^{d}),\qquad\inner{\hompd^{-1}(u)}{\test}=\inner{u}{\chi\gamma^{-1}(\test)},
	\end{equation*}
	where $\chi\in\mcS(\mbR^{d})$ is such that $(\mcT_{z}\chi)_{z\in\mbZ^{d}}$ forms a partition of unity (which exists by~\cite[Lem.~29.3]{vanzuijlen_22}.) The above extends to $\mcS'_{\per}(\mbR\times\mbR^{d})$ and $\mcS'(\mbR\times\mbT^{d})$ \emph{mutatis mutandis}.
\end{lemma}
\section{Space-Time White Noise}\label{app:periodic_wn}
In this appendix we show the existence and regularity of space-time white noise (Theorem~\ref{thm:existence_stWN}), construct an abstract Wiener space (Theorem~\ref{thm:abstract_Wiener_space}) and show that iterated It\^{o} integrals as in~\eqref{eq:def_iterated_Ito_Fourier} are elements of inhomogeneous Wiener chaoses (Lemma~\ref{lem:iterated_Ito_in_hom_Wiener_chaos}).
\subsection{Existence and Regularity}\label{subsec:existence_and_regularity}
We first define space-time white noise as a family of Gaussian random variables.
\begin{definition}[{\cite[Sec.~10.1]{hairer_14_RegStruct}}]\label{def:space_time_WN_torus}
	We define a space-time white noise on $[0,T]\times\mbT^{2}$ to be a family of real-valued, centred Gaussian random variables $\{\xi(\phi):\phi\in L^{2}([0,T]\times\mbT^{2};\mbR)\}$ on a probability space $(\Omega,\mcF,\mbP)$, such that $\mbE[\xi(\phi)\xi(\psi)]=\inner{\phi}{\psi}_{L^2([0,T]\times\mbT^{2};\mbR)}$ for every $\phi,\psi\in L^2([0,T]\times\mbT^{2};\mbR)$.
\end{definition}
By an application of Kolmogorov's extension theorem, there exists a probability space $(\Omega,\mcF,\mbP)$ and a family as in Definition~\ref{def:space_time_WN_torus}. Indeed, following~\cite[Ex.~1.27]{janson_97}, we may take $\Omega$ to be the algebraic dual to $L^{2}([0,T]\times\mbT^{2})$, define for all $h\in L^{2}([0,T]\times\mbT^{2})$ the functions $\xi_{h}\from\Omega\to\mbR$, $\om\mapsto\xi_{h}(\om)\defeq\om(h)$, and set  $\mcF=\sigma(\xi_{h}:h\in L^{2}([0,T]\times\mbT^{2};\mbR))$. Then there exists a unique probability measure $\mbP$ on $(\Omega,\mcF)$ such that each $\xi_{h}$ is centred Gaussian with variance $\norm{h}_{L^{2}([0,T]\times\mbT^{2};\mbR)}^{2}$.

Next we show that there exists a modification of this family that is a random distribution in $\mcS'(\mbR\times\mbT^{2};\mbR)$ and that this modification lies in the space $\mcC^{\alpha}_{\mfs}([0,T]\times\mbT^{2};\mbR)$ (see Definition~\ref{def:Cs_alpha_torus}).
\begin{theorem}\label{thm:existence_stWN}
	A space-time white noise on $[0,T]\times\mbT^{2}$ as defined in Definition~\ref{def:space_time_WN_torus} exists as a random distribution in $\mcS'(\mbR\times\mbT^{2};\mbR)$ and is almost surely an element of $\mcC^{\alpha}_{\mfs}([0,T]\times\mbT^{2};\mbR)$ for each $\alpha<-2$.
\end{theorem}
\begin{proof}
	Following~\cite[below Rem.~10.1]{hairer_14_RegStruct}, let us consider the periodically-extended noise
	\begin{equation*}
		\mcS(\mbR\times\mbR^{2};\mbR)\ni\test\mapsto\hompd(\xi\mathds{1}_{[0,T]})(\test),
	\end{equation*} 
	where $\hompd$ denotes the homeomorphism constructed in Lemma~\ref{lem:homeom_per}. For each $\alpha<-2$ it follows by~\cite[Lem.~10.2]{hairer_14_RegStruct} that $\hompd(\xi\mathds{1}_{[0,T]})$ is almost surely a random variable in $\mcC^{\alpha}_{\mfs,\per}([0,T]\times\mbR^{2};\mbR)$. In particular, by Definition~\ref{def:Cs_alpha_torus} we obtain almost surely $\xi\mathds{1}_{[0,T]}\in\mcC^{\alpha}_{\mfs}([0,T]\times\mbT^{2};\mbR)$.
\end{proof}
Given two independent space-time white noises $\xi^{1}$ and $\xi^{2}$ as in Definition~\ref{def:space_time_WN_torus} on the same probability space, we can define a vector-valued space-time white noise $\boldsymbol{\xi}=(\xi^{1},\xi^{2})$ by
\begin{equation*}
	\boldsymbol{\xi}(\test)\defeq(\xi^{1}(\test),\xi^{2}(\test)),\qquad\test\in L^{2}([0,T]\times\mbT^{2};\mbR).
\end{equation*}
With this definition it is clear how to pass from scalar-valued to vector-valued noise. In the remainder of this appendix, for the sake of notation, we only consider scalar-valued noise and omit the target space from our definitions.
\subsection{Abstract Wiener Space}\label{subsec:abstract_Wiener_space}
In the following we do not distinguish between $L^{2}([0,T]\times\mbT^{2})$ and the functions in $L^{2}(\mbR\times\mbT^{2})$ that are zero on the complement of $[0,T]$.
\begin{theorem}\label{thm:abstract_Wiener_space}
	For every $\alpha<-2$ it follows that the law of the space-time white noise on $[0,T]\times\mbT^{2}$ can be realized on the abstract Wiener space $(\nban,\cm,\mu)$ with separable Banach space $\nban\defeq\mcC_{\mfs}^{\alpha}([0,T]\times\mbT^{2})$, Cameron--Martin space $\cm=L^{2}([0,T]\times\mbT^{2})$ and Gaussian probability measure $\gm=\Law(\xi)$.
\end{theorem}
\begin{proof}
	It follows by Theorem~\ref{thm:existence_stWN} that $\gm$ is a Borel probability measure on $\nban$. Hence to prove that $\gm$ is a Gaussian probability measure in the sense of~\cite[Sec.~4]{ledoux_94}, we need to show that the law of each continuous linear functional on $\nban$ is Gaussian.
	
	It follows by the Cauchy--Schwarz inequality and the assumption $\alpha<-2$, that $L^{2}([0,T]\times\mbT^{2})\subset\mcC_{\mfs}^{\alpha}([0,T]\times\mbT^{2})=\nban$, 
	\begin{details}
		\paragraph{Proof that $L^{2}([0,T]\times\mbT^{2})\subset\mcC_{\mfs}^{\alpha}([0,T]\times\mbT^{2})$.}
		It suffices to show that each $h\in L^{2}_{\per}([0,T]\times\mbT^{2})$ is an element of $\mcC_{\mfs,\per}^{\alpha}([0,T]\times\mbR^{2})$. We estimate by the Cauchy--Schwarz inequality,
		\begin{equation*}
			\begin{split}
				\norm{h}_{\mcC_{\mfs}^{\alpha}([0,T]\times[0,1]^{2})}&=\sup_{\lambda\in(0,1]}\sup_{\eta\in \mcB_{\mfs}^{r_{\alpha}}}\sup_{z\in[0,T]\times[0,1]^{2}}\frac{\abs{\inner{u}{\mcS_{\mfs,z}^{\lambda}\eta}}}{\lambda^{\alpha}}\\
				&\leq\norm{u}_{L^{2}([0,T]\times\mbR^{2})}\sup_{\lambda\in(0,1]}\sup_{\eta\in \mcB_{\mfs}^{r_{\alpha}}}\sup_{z\in[0,T]\times[0,1]^{2}}\frac{\norm{\mcS_{\mfs,z}^{\lambda}\eta}_{L^{2}(\mbR\times\mbR^{2})}}{\lambda^{\alpha}}\\
				&\lesssim\norm{u}_{L^{2}([0,T]\times\mbR^{2})}\sup_{\lambda\in(0,1]}\lambda^{-2-\alpha}<\infty.
			\end{split}
		\end{equation*}
		The claim $h\in\mcC_{\mfs,\per}^{\alpha}([0,T]\times\mbR^{2})$ then follows by considering a smooth approximation sequence.
		
	\end{details}
	which implies by standard arguments $\nban'\subset L^{2}([0,T]\times\mbT^{2})$. 
	\begin{details}
		\paragraph{Proof that $\nban'\subset L^{2}([0,T]\times\mbT^{2})$.}
		Let $l\in\nban'$, then for every $\test\in L^{2}([0,T]\times\mbT^{2})$,
		\begin{equation*}
			\abs{l(\test)}\leq\norm{l}_{\nban'}\norm{\test}_{\mcC^{\alpha}_{\mfs}([0,T]\times\mbT^{2})}\lesssim\norm{l}_{\nban'}\norm{\test}_{L^{2}([0,T]\times\mbT^{2})},
		\end{equation*}
		hence $l$ is a bounded linear functional on $L^{2}([0,T]\times\mbT^{2})$. By Riesz's representation theorem, we can identify $l\in L^{2}([0,T]\times\mbT^{2})$.
		
	\end{details}
	Consequently we obtain by Definition~\ref{def:space_time_WN_torus} that for each $l\in\nban'$, $l_{\#}\mu$ is a centred Gaussian measure on $\mbR$, which implies that $\gm$ is a Borel probability measure on $\nban$.
	
	It remains to identify the Cameron--Martin space $\cm$. To do so, we first consider the reproducing kernel Hilbert space $\rep$, i.e.\ the closure of $\nban'$ in $L^{2}(\nban,\mu)$, which is isomorphic to $\cm$. Using a duality argument, we obtain $C_{\comp}^{\infty}((0,T)\times\mbT^{2})\subset\nban'$ (where we extend each element of $C_{\comp}^{\infty}((0,T)\times\mbT^{2})$ by $0$ to $\mbR\times\mbT^{2}$).
	\begin{details}
		\paragraph{Proof that $C_{\comp}^{\infty}((0,T)\times\mbT^{2})\subset\nban'$.}
		Let $u\in\mcC^{\alpha}_{\mfs}([0,T]\times\mbT^{2})$ and $\test\in C_{\comp}^{\infty}((0,T)\times\mbT^{2})$. It follows by Lemma~\ref{lem:homeom_per},
		\begin{equation*}
			\inner{u}{\test}=\inner{\hompd^{-1}\hompd(u)}{\test}=\inner{\hompd(u)}{\chi\gamma^{-1}(\test)},
		\end{equation*}
		where $\chi\in\mcS(\mbR^{2})$ is as in~\cite[Lem.~29.3]{vanzuijlen_22}. In particular $\chi\gamma^{-1}(\test)\in\mcS(\mbR\times\mbR^{2})$ and by the definition of $\mcC^{\alpha}_{\mfs}([0,T]\times\mbT^{2})$, $\hompd(u)\in\mcC^{\alpha}_{\mfs,\per}([0,T]\times\mbR^{2})$. By the definition of $\mcC^{\alpha}_{\mfs,\per}([0,T]\times\mbR^{2})$, there exists an approximating sequence $(\hompd(u)_{n})_{n\in\mbN}$ such that $\hompd(u)_{n}\in C^{\infty}([0,T]\times[0,1]^{2})$ for every $n\in\mbN$ (extended periodically to $[0,T]\times\mbR^{2}$ and by $0$ to $\mbR\times\mbR^{2}$) and $\norm{\hompd(u)-\hompd(u)_{n}}_{\mcC^{\alpha}_{\mfs}([0,T]\times[0,1]^{2})}\to0$.
					
		It follows by Lemma~\ref{lem:duality}, the weak*-sequential completeness of $\mcS'(\mbR\times\mbR^{2})$ and the continuity of the norm $\norm{\place}_{\mcC^{\alpha}_{\mfs}([0,T]\times[0,1]^{2})}$, that
		\begin{equation*}
			\begin{split}
				\inner{\hompd(u)}{\chi\gamma^{-1}(\test)}&=\lim_{n\to\infty}\inner{\hompd(u)_{n}}{\chi\gamma^{-1}(\test)}_{L^{2}(\mbR\times\mbR^{2})}\\
				&\lesssim\lim_{n\to\infty}\norm{\hompd(u)_{n}}_{\mcC^{\alpha}_{\mfs}([0,T]\times[0,1]^{2})}\Bigl\lVert\sum_{z\in\mbZ^{2}}\mcT_{z}(\chi\gamma^{-1}(\test))\Bigr\rVert_{C^{r_{\alpha}}(\mbR\times(-1,1)^{2})}\\
				&=\norm{\hompd(u)}_{\mcC^{\alpha}_{\mfs}([0,T]\times[0,1]^{2})}\Bigl\lVert\sum_{z\in\mbZ^{2}}\mcT_{z}(\chi\gamma^{-1}(\test))\Bigr\rVert_{C^{r_{\alpha}}(\mbR\times(-1,1)^{2})}.
			\end{split}
		\end{equation*}
		Hence every $\test\in C_{\comp}^{\infty}((0,T)\times\mbT^{2})$ induces a continuous linear functional on $\mcC^{\alpha}_{\mfs}([0,T]\times\mbT^{2})$.
		
	\end{details}
	The closure of $\nban'$ in $L^{2}(\nban,\gm)$ coincides with the closure of $\nban'$ in $L^{2}([0,T]\times\mbT^{2})$ since for all $l\in\nban'\subset L^{2}([0,T]\times\mbT^{2})$ it holds that $\norm{l}_{L^{2}(\nban,\gm)}=\norm{l}_{L^{2}([0,T]\times\mbT^{2})}$. Hence, using that $C_{\comp}^{\infty}((0,T)\times\mbT^{2})\subset\nban'$, we obtain $\rep=L^{2}([0,T]\times\mbT^{2})$. 
	
	By~\cite[Sec.~4]{ledoux_94} there exists a linear isometric isomorphism
	\begin{equation*}
		\iota\from\rep\to\mcH_{\mu},\qquad\iota(l)\defeq\int_{\nban}xl(x)\gm(\dd x),
	\end{equation*}
	so that it suffices to identify each element $\iota(l)\in\cm$ with an element of $L^{2}([0,T]\times\mbT^{2})$. Let $\test\in\mcS(\mbR\times\mbT^{2})$, we obtain
	\begin{equation*}
		\inner{\iota(l)}{\test}=\int_{\nban}\test(x)l(x)\gm(\dd x)=\inner{l}{\test}_{L^{2}([0,T]\times\mbT^{2})},
	\end{equation*}
	which shows that $\iota(l)$ coincides with the tempered distribution 
	\begin{equation*}
		u_{l}\in\mcS'(\mbR\times\mbT^{2}),\qquad\mcS(\mbR\times\mbT^{2})\ni\test\mapsto u_{l}(\test)=\int_{[0,T]\times\mbT^{2}}l(z)\test(z)\dd z.
	\end{equation*}
	Using that the map $L^{2}(\mbR\times\mbT^{2})\ni l\mapsto u_{l}\in\mcS'(\mbR\times\mbT^{2})$ is an embedding,
	\begin{details}
		\cite[Thm.~15.6]{vanzuijlen_22}
	\end{details}
	we can identify $\iota(l)$ and $l$, which yields $\mcH_{\mu}=L^{2}([0,T]\times\mbT^{2})$.
\end{proof}
\subsection{Iterated It\^{o} Integrals}\label{subsec:iterated_Ito_integrals}
In this subsection we relate the iterated It\^{o} integrals defined in~\eqref{eq:def_iterated_Ito_Fourier} to our space-time white noise (Lemma~\ref{lem:space_time_white_noise_yields_complex_BM}), show that each integral yields an element of a homogenous Wiener chaos (Lemma~\ref{lem:iterated_Ito_in_hom_Wiener_chaos}) and characterize its homogeneous part (Lemma~\ref{lem:hom_part_fourier}).
\begin{lemma}\label{lem:space_time_white_noise_yields_complex_BM}
	Let $\boldsymbol{\xi}$ be a vector-valued space-time white noise on $[0,T]\times\mbT^{2}$ as in Subsection~\ref{subsec:existence_and_regularity}. Then there exists a family of complex-valued Brownian motions $(W^{j}(\place,\om))_{\om\in\mbZ^{2},j=1,2}$ (on the same underlying probability space) that satisfy $\overline{W^{j}(\place,\om)}=W^{j}(\place,-\om)$ but are otherwise independent. 
\end{lemma}
\begin{proof}
	Let $\boldsymbol{\xi}$ be a vector-valued space-time white noise on $[0,T]\times\mbT^{2}$ as in Subsection~\ref{subsec:existence_and_regularity}. Let $t\in[0,T]$, $\om\in\mbZ^{2}$, $j=1,2$ and define
	\begin{equation*}
		B^{j}(t,\om)\defeq\boldsymbol{\xi}(\mathds{1}_{[0,t]}(\place)\otimes\cos(2\uppi\inner{\om}{\place})\otimes\delta_{j,\place}),\quad \tilde{B}^{j}(t,\om)\defeq\boldsymbol{\xi}(\mathds{1}_{[0,t]}(\place)\otimes-\sin(2\uppi\inner{\om}{\place})\otimes\delta_{j,\place}),
	\end{equation*}
	which yields a family of Brownian motions such that $\mbE[(B^{j}(t,\om))^{2}]=\mbE[(\tilde{B}^{j}(t,\om))^{2}]=t/2$, $B^{j}(t,\om)=B^{j}(t,-\om)$ and $\tilde{B}^{j}(t,\om)=-\tilde{B}^{j}(t,-\om)$ and which are otherwise independent. Using this family, we can define the complex Brownian motion $W^{j}(t,\om)\defeq B^{j}(t,\om)+\upi\tilde{B}^{j}(t,\om)$.
	\begin{details}
		
		We show that $B$ and $\tilde{B}$ are Brownian motions. Let $\om,\om'\in\mbZ^{2}$, $t,t'\in[0,T]$ and $j,j'\in\{1,2\}$. It follows by Definition~\ref{def:space_time_WN_torus} that
		\begin{equation*}
			\begin{split}
				&\mbE[B^{j}(t,\om)B^{j'}(t',\om')]=\frac{1}{2}\delta_{j,j'}(t\wedge t')(\delta_{\om,\om'}+\delta_{\om,-\om'}),\\
				&\mbE[\tilde{B}^{j}(t,\om)\tilde{B}^{j'}(t',\om')]=\frac{1}{2}\delta_{j,j'}(t\wedge t')(\delta_{\om,\om'}-\delta_{\om,-\om'}),\\
				&\mbE[B^{j}(t,\om)\tilde{B}^{j'}(t',\om')]=0,
			\end{split}
		\end{equation*}
		since
		\begin{equation}\label{eq:trigo_integrals}
			\begin{split}
				&\int_{[0,1]^{2}}\cos(2\uppi\inner{\om}{x})\cos(2\uppi\inner{\om'}{x})\dd x=\frac{1}{2}(\delta_{\om,\om'}+\delta_{\om,-\om'}),\\
				&\int_{[0,1]^{2}}\sin(2\uppi\inner{\om}{x})\sin(2\uppi\inner{\om'}{x})\dd x=\frac{1}{2}(\delta_{\om,\om'}-\delta_{\om,-\om'}),\\
				&\int_{[0,1]^{2}}\sin(2\uppi\inner{\om}{x})\cos(2\uppi\inner{\om'}{x})\dd x=0.
			\end{split}
		\end{equation}
		Hence, $(B,\tilde{B})$ is a family of Brownian motions.
		
		To show~\eqref{eq:trigo_integrals}, we can use the representation of $\cos$ and $\sin$ as complex exponentials 
		\begin{equation*}
			\cos(2\uppi\inner{\om}{x})=\frac{\euler^{2\uppi\upi\inner{\om}{x}}+\euler^{-2\uppi\upi\inner{\om}{x}}}{2},\qquad\sin(2\uppi\inner{\om}{x})=\frac{\euler^{2\uppi\upi\inner{\om}{x}}-\euler^{-2\uppi\upi\inner{\om}{x}}}{2\upi}
		\end{equation*}	
		and the orthogonality of $x\mapsto\euler^{2\uppi\upi\inner{\om}{x}}$ in $L^2(\mbT^{2};\mbC)$,
		\begin{equation*}
			\begin{split}
				\int_{[0,1]^{2}}\cos(2\uppi\inner{\om}{x})\cos(2\uppi\inner{\om'}{x})\dd x&=\frac{1}{4}(\delta_{\om,-\om'}+\delta_{\om,-\om'}+\delta_{\om,\om'}+\delta_{\om,\om'})=\frac{1}{2}(\delta_{\om,-\om'}+\delta_{\om,\om'}),\\
				\int_{[0,1]^{2}}\sin(2\uppi\inner{\om}{x})\sin(2\uppi\inner{\om'}{x})\dd x&=-\frac{1}{4}(\delta_{\om,-\om'}+\delta_{\om,-\om'}-\delta_{\om,\om'}-\delta_{\om,\om'})=\frac{1}{2}(\delta_{\om,\om'}-\delta_{\om,-\om'}),\\
				\int_{[0,1]^{2}}\cos(2\uppi\inner{\om}{x})\sin(2\uppi\inner{\om'}{x})\dd x&=\frac{1}{4\upi}(\delta_{\om,-\om'}-\delta_{\om,-\om'}+\delta_{\om,\om'}-\delta_{\om,\om'})=0.
			\end{split}
		\end{equation*}
	\end{details}
\end{proof}
Let $\Sigma_{n}(1,\ldots,n)$ denote the permutation group over $\{1,\ldots,n\}$. We call $h\in L^2([0,T]^{n}\times\mbT^{2n};\mbR^{2n})$ symmetric, if 
\begin{equation*}
	h(z_{1},\ldots, z_{n})=\Sym(h)(z_{1},\ldots, z_{n})=\frac{1}{n!}\sum_{\varsigma\in\Sigma_{n}(1,\ldots,n)}h(u_{\varsigma(1)},x_{\varsigma(1)},j_{\varsigma(1)},\ldots,u_{\varsigma(n)},x_{\varsigma(n)},j_{\varsigma(n)}),
\end{equation*}
where for all $k\in\{1,\ldots,n\}$ we denote $u_{k}\in[0,T]$, $x_{k}\in\mbT^{2}$, $j_{k}\in\{1,2\}$ and $z_{k}=(u_{k},x_{k},j_{k})$. We denote the space of symmetric, square-integrable functions by $f\in L^{2}_{\sym}([0,T]^{n}\times\mbT^{2n};\mbR^{2n})$. 
Next we show that each iterated It\^{o} integrals of a symmetric integrand is an element of a homogeneous Wiener chaos as defined in~\eqref{eq:hom_Wiener_chaos}.
\begin{lemma}\label{lem:iterated_Ito_in_hom_Wiener_chaos}
	For each $n\in\mbN$ and $f\in L^{2}_{\sym}([0,T]^{n}\times\mbT^{2n};\mbR^{2n})$, it holds that $I^{n}(f)\in\chaos^{(n)}(\Omega,\mbP;\mbR)$, where the homogeneous Wiener chaos $\chaos^{(n)}(\Omega,\mbP;\mbR)$ is defined as in~\eqref{eq:hom_Wiener_chaos}.
\end{lemma}
\begin{proof}
	Let $f\in L^{2}_{\sym}([0,T]^{n}\times\mbT^{2n};\mbR^{2n})$, it is now a consequence of~\cite[below Thm.~1.1.2]{nualart_06} that $I^{n}(f)\in\mcH^{(n)}(\Omega,\mbP;\mbR)$.
	\begin{details}
		see also~\cite[Prop.~1.42]{klose_msc_17}.
	\end{details}
	\begin{details}
		We need to show that our notion of iterated It\^{o} integral~\eqref{eq:def_iterated_Ito_Fourier} coincides with iterated It\^{o} integral used in~\cite[Sec.~1.1.2]{nualart_06} resp.~\cite[Sec.~1.2.2]{klose_msc_17}.
		\begin{definition}\label{def:iterated_Ito_real}
			Let $\boldsymbol{\xi}$ be a vector-valued space-time white noise on $[0,T]\times\mbT^{2}$ as in~\eqref{eq:vector_space_time_WN_torus}. Let $m,n\in\mbN$ and $0\leq s_{1}<\ldots<s_{m}<s_{m+1}\leq T$, $A_{k}\in\mcB(\mbT^{2})$ and $l_{k}\in\{1,2\}$ for all $k\in\{1,\ldots,m\}$. Let $a\from\{1,\ldots,m\}^{n}\to\mbR$ be such that $a_{\textbf{i}}=0$ if $i_{k}=i_{k'}$ for some $k\neq k'\in\{1,\ldots,n\}$, and such that for each permutation $\varsigma$ of $\{1,\ldots,n\}$, it holds that $a_{\textbf{i}}=a_{\varsigma(\textbf{i})}$, where $\varsigma(\textbf{i})\defeq(i_{\varsigma(1)},\ldots,i_{\varsigma(n)})$. We then define the symmetric simple function $f_{n}$ by
			\begin{equation}\label{eq:simple}
				f_{n}(z_{1},\ldots,z_{n})=\sum_{\textbf{i}\in\{1,\ldots,m\}^{n}}a_{\textbf{i}}\prod_{k=1}^{n}\mathds{1}_{[s_{i_k},s_{i_k+1})}(u_k)\mathds{1}_{A_{i_k}}(x_{k})\delta_{l_{i_{k}},j_{k}},\qquad\text{where}\quad z_{k}=(u_{k},x_{k},j_{k}).
			\end{equation}
			The iterated stochastic integral $\boldsymbol{\xi}^{\otimes n}$ is then defined on such $f_{n}$ by 
			\begin{equation*}
				\boldsymbol{\xi}^{\otimes n}(f_{n})=\sum_{\textbf{i}\in\{1,\ldots,m\}^{n}}a_{\textbf{i}}\prod_{k=1}^{n}\boldsymbol{\xi}(\mathds{1}_{[s_{i_k},s_{i_k+1})}\otimes\mathds{1}_{A_{i_k}}\otimes\delta_{l_{i_{k}}})\in L^{2}(\mbP)
			\end{equation*}
			and extended to $f\in L^{2}_{\sym}([0,T]^{n}\times\mbT^{2n};\mbR^{2n})$ by continuity.
		\end{definition}
		In contrast to the definition above, Nualart~\cite[Sec.~1.1.2]{nualart_06} considers  simple functions over $\mcB([0,T]\times\mbT^{2}\times\{1,2\})$ that are not assumed to be symmetric. Here we discuss why one may assume $f_{n}$ as in~\eqref{eq:simple}.
		
		Assume $f\in L^{2}([0,T]^{n}\times\mbT^{2n};\mbR^{2n})$, is symmetric, i.e.\ $\Sym(f)=f$, and that there exits a sequence of (not necessarily symmetric) simple functions $(g_{n})_{n\in\mbN}$ that converge towards $f$ in $L^{2}$. We define a sequence of symmetric simple functions $f_{n}=\Sym(g_{n})$ and establish with Minkowski's inequality,
		\begin{equation*}
			\norm{f_{n}-f}_{L^{2}}=\norm{\Sym(g_{n})-\Sym(f)}_{L^{2}}\leq\norm{g_{n}-f}_{L^{2}}\to0.
		\end{equation*} 
		Hence, we may always assume our simple functions to be symmetric.
		
		Next, we argue that we may take simple functions over sets of the form $[s_{i_{k}},s_{i_{k+1}})\times A_{i_{k}}\times\{l_{i_{k}}\}$. Using that 
		\begin{equation*}
			\mcB([0,T]\times\mbT^{2}\times\{1,2\})=\mcB([0,T])\otimes\mcB(\mbT^{2})\otimes\mcB(\{1,2\}),
		\end{equation*}
		it follows that $\mcB([0,T]\times\mbT^{2}\times\{1,2\})$ is generated by sets of the form $[s_{i_{k}},s_{i_{k+1}})\times A_{i_{k}}\times\{l_{i_{k}}\}$. An application of the monotone class theorem shows that any function in $L^{2}([0,T]\times\mbT^{2}\times\{1,2\})$ can be approximated by simple functions as in~\eqref{eq:simple}.
		
		Next we show that both notions of iterated It\^{o} integrals, \eqref{eq:def_iterated_Ito_Fourier} and Definition~\ref{def:iterated_Ito_real}, coincide. It suffices to show that they agree on symmetric simple functions~\eqref{eq:simple} and are continuous as maps from $L^{2}_{\sym}([0,T]^{n}\times\mbT^{2n};\mbR^{2n})$ to $L^{2}(\mbP)$. Let $f_{n}$ be as in~\eqref{eq:simple}, we compute
		\begin{equation*}
			\begin{split}
				&n!I^{n}(f_{n})\\
				&=n!\sum_{\om_1,\ldots,\om_{n}\in\mbZ^{2}}\sum_{j_1,\ldots,j_n=1}^{2}\int_{0}^{T}\dd W^{j_1}(u_1,\om_1)\ldots\int_{0}^{u_{n-1}}\dd W^{j_n}(u_n,\om_n)\hat{f_{n}}(u_1,-\om_1,j_1,\ldots,u_n,-\om_n,j_n)\\
				&=\sum_{\om_1,\ldots,\om_{n}\in\mbZ^{2}}\sum_{j_1,\ldots,j_n=1}^{2}\sum_{\textbf{i}\in\{1,\ldots,m\}^{n}}a_{\textbf{i}}\prod_{k=1}^{n}\boldsymbol{\xi}(\mathds{1}_{[s_{i_k},s_{i_k+1})}(\place)\otimes\euler^{-2\uppi\upi\inner{\om_{k}}{\place}}\otimes\delta_{j_{k,\place}})\delta_{l_{i_k},j_k}\hat{\mathds{1}_{A_{i_{k}}}}(-\om_{k})\\
				&=\sum_{\textbf{i}\in\{1,\ldots,m\}^{n}}a_{\textbf{i}}\prod_{k=1}^{n}\boldsymbol{\xi}(\mathds{1}_{[s_{i_k},s_{i_k+1})}\otimes\mathds{1}_{A_{i_{k}}}\otimes\delta_{l_{i_{k}}})=\boldsymbol{\xi}^{\otimes n}(f_n),
			\end{split}
		\end{equation*}
		where in the second equality we used the symmetry of $f_{n}$  to remove the restriction $u_{n}<u_{n-1}<\ldots<u_{1}$. Therefore, both notions coincide on simple functions. To show that $I^{n}$ is a continuous map from $L^{2}_{\sym}([0,T]^{n}\times\mbT^{2n};\mbR^{2n})$ to $L^{2}(\mbP)$, we apply It\^{o}'s isometry and Parseval's theorem,
		\begin{equation*}
			\begin{split}
				\mbE[\abs{I^{n}(f_{n})}^{2}]&=\sum_{\om_1,\ldots,\om_{n}\in\mbZ^{2}}\sum_{j_1,\ldots,j_n=1}^{2}\int_{0}^{T}\dd u_{1}\ldots\int_{0}^{u_{n-1}}\dd u_{n} \abs{\hat{f_{n}}(u_1,-\om_1,j_1,\ldots,u_n,-\om_n,j_n)}^{2}\\
				&=\int_{\mbT^{2}}\dd x_{1}\ldots\int_{\mbT^{2}}\dd x_{n}\sum_{j_1,\ldots,j_n=1}^{2}\int_{0}^{T}\dd u_{1}\ldots\int_{0}^{u_{n-1}}\dd u_{n} \abs{f_{n}(u_1,x_{1},j_1,\ldots,u_n,x_{n},j_n)}^{2}\\
				&=\frac{1}{n!}\int_{\mbT^{2}}\dd x_{1}\ldots\int_{\mbT^{2}}\dd x_{n}\sum_{j_1,\ldots,j_n=1}^{2}\int_{0}^{T}\dd u_{1}\ldots\int_{0}^{T}\dd u_{n} \abs{f_{n}(u_1,x_{1},j_1,\ldots,u_n,x_{n},j_n)}^{2},
			\end{split}
		\end{equation*}
		where in the last line we again used the symmetry of $f_{n}$.
		
		Therefore $I^{n}$ extends as a (multiple of) an isometry to $L^{2}_{\sym}([0,T]^{n}\times\mbT^{2n},\mbR^{2n})$ and both notions of iterated It\^{o} integrals, \eqref{eq:def_iterated_Ito_Fourier} and Definition~\ref{def:iterated_Ito_real}, coincide.
	\end{details}
\end{proof}
Using that each iterated It\^{o} integral has zero mean, we obtain the following result (cf.~\cite[Sec.~4]{hairer_weber_15}).
\begin{lemma}\label{lem:hom_part_fourier}
	Let $n\in\mbN$, $f\in L^{2}_{\sym}([0,T]^{n}\times\mbT^{2n};\mbR^{2n})$ and $h\in L^{2}([0,T]\times\mbT^{2};\mbR^{2})$. Then it follows that $n!\mbE[I^{n}(f)(\place+h)]=\inner{f}{h^{\otimes n}}_{L^{2}([0,T]^{n}\times\mbT^{2n};\mbR^{2n})}$.
\end{lemma}
\begin{proof}[Sketch of Proof.]
	For every $h\in L^{2}([0,T]\times\mbT^{2};\mbR^{2})$ we can apply Lemma~\ref{lem:space_time_white_noise_yields_complex_BM} to identify the shifted vector-valued space-time white noise $\boldsymbol{\xi}+h$ with a complex Brownian motion plus drift, that is, for every $t\in[0,T]$, $\om\in\mbZ^{2}$ and $j=1,2$,
	\begin{equation*}
		(\boldsymbol{\xi}+h)(\mathds{1}_{[0,t](\place)}\otimes\euler^{-2\uppi\upi\inner{\om}{\place}}\otimes\delta_{j,\place})=W^{j}(t,\om)+\int_{0}^{t}\hat{h}(s,\om,j)\dd s.
	\end{equation*}
	In the case of iterated It\^{o} integrals of order $1$ it follows that for every integrand $f\in L^{2}([0,T]\times\mbT^{2};\mbR^{2})$,
	\begin{equation*}
		\begin{split}
			\mbE[I^{1}(f)(\place+h)]&=\mbE\Bigl[\sum_{\om_{1}\in\mbZ^{2}}\sum_{j_{1}=1}^{2}\int_{0}^{\infty}\hat{f}(u_{1},-\om_{1},j_{1})(\dd W^{j_{1}}(u_{1},\om_{1})+\hat{h}(u_{1},\om_{1},j_{1})\dd u_{1})\Bigr]\\
			&=\sum_{\om_{1}\in\mbZ^{2}}\sum_{j_{1}=1}^{2}\int_{0}^{\infty}\hat{f}(u_{1},-\om_{1},j_{1})\hat{h}(u_{1},\om_{1},j_{1})\dd u_{1}\\
			&=\inner{f}{h}_{L^{2}([0,T]\times\mbT^{2};\mbR^{2})},
		\end{split}
	\end{equation*}
	where in the second equality we used that It\^{o} integrals have zero mean and in the third equality applied Parseval's theorem.
	
	The case of general $n\in\mbN$ follows by linearity, where we use the symmetry of $f$ to recover the integral over $[0,T]^{n}\times\mbT^{2n}\times\{1,2\}^{n}$ modulo a factor of $n!$. For the full proof, see~\cite[Prop.~1.33]{klose_msc_17}.
\end{proof}
\section{Generalized Freidlin--Wentzell Theory}\label{app:freidlin_wentzell}
In this appendix we generalize Freidlin and Wentzell's theory to (direct sums of) random variables taking values in Banach space-valued inhomogeneous Wiener chaoses (Theorem~\ref{thm:LDP_wiener_chaos}).

A similar program has already been carried out in~\cite[Sec.~3]{hairer_weber_15} and the main observation of the present appendix is that the results of~\cite[Sec.~3]{hairer_weber_15} carry over to situations in which the notion of convergence~\cite[(3.9)]{hairer_weber_15} is replaced by~\eqref{eq:rv_convergence} (where one retains powers of the noise intensity for lower-order Wiener chaoses), if one assumes that the limit has no contributions from lower-order Wiener chaoses~\eqref{eq:rv_limit}. It is this change that allows us to apply Theorem~\ref{thm:LDP_wiener_chaos} directly to the (canonical) enhancement constructed in Theorem~\ref{thm:existence_enh_can}, rather than going through a perturbation argument as in~\cite[Thm.~4.7]{hairer_weber_15}, which does not generically apply for us as we wish to treat cases where the white-noise induced renormalisation is strictly a distribution rather than a smooth function (or as in~\cite{hairer_weber_15} a constant.)

\begin{details}
	A perturbative argument would require us to consider the products
	\begin{equation}\label{eq:divorced_diagrams}
		\tl^{\delta}\re\nabla\Phi_{\ti},\qquad\nabla\Phi_{\tl^{\delta}}\re\ti,
	\end{equation}
	where $\tl^{\delta}$ is the \emph{specific} renormalization of~\cite[Thm.~2.3]{martini_mayorcas_25} and $\ti$ is the first coordinate of a \emph{generic} enhancement. Under the relative scaling of the law of large numbers (Theorem~\ref{thm:LLN_rough}), we obtain at most the regularities $\tl^{\delta}(t)\in\mcC^{0-}(\mbT^{2})$ and $\ti(t)\in\mcC^{-1-}(\mbT^{2})$ which are not sufficient to form those products.
	
	Compare this to
	\begin{equation*}
		\tl^{\delta}\re\nabla\Phi_{\ti^{\delta}},\qquad\nabla\Phi_{\tl^{\delta}}\re\ti^{\delta},
	\end{equation*}
	which are well-posed by~\cite[Lem.~2.26]{martini_mayorcas_25}. This suggests that the fundamental problem in~\eqref{eq:divorced_diagrams} can be explained by the fact that the diagrams are `divorced' or, in other words, that~\eqref{eq:divorced_diagrams} involves `off-diagonal entries'.
	
	This is why in Theorem~\ref{thm:LDP_enhancement} we do not use the perturbative argument of~\cite{hairer_weber_15} and instead
	establish the LDP for the canonical enhancement directly.
\end{details}

Let $(\nban,\cm,\gm)$ be an abstract Wiener space where $\nban$ denotes a real, separable Banach space and $\gm$ a centred Gaussian probability measure on $\nban$ with Cameron--Martin space $(\cm,\norm{\place}_{\cm})$. Let $\rep$ be the associated reproducing kernel Hilbert space, i.e.\ the completion of the dual space $\nban'$ in $L^{2}(\nban,\gm)$. Let $(e_{k})_{k\in\mbN}$ be an orthonormal basis in $\rep$ such that $e_{k}\in\nban'$ for every $k\in\mbN$, which exists by an application of the Gram--Schmidt procedure.
\begin{details}
	\paragraph{Existence of $(e_{k})_{k\in\mbN}$.}
	The reproducing kernel Hilbert space $\rep$ is separable as a subspace of the separable metric space $L^{2}(\nban,\gm)$ (where we used that $\nban$ is separable.) By the same argument, $\nban'\subset\rep$ is also separable; hence there exists a countable dense subset in $\nban'$ with respect to the topology induced by $\rep$. Using the density of $\nban'$ in $\rep$, it is clear that this countable subset is also dense in $\rep$. We can then apply the Gram--Schmidt procedure as in~\cite[Thm.~5.11]{brezis_11}.
\end{details}

We define the Hermite polynomials $H_{n}(x)$, for every $n\in\mbN_{0}$ and $x\in\mbR$, through the series expansion
\begin{equation*}
	\lambda\mapsto\exp\Bigl(\lambda x-\frac{1}{2}\lambda^{2}\Bigr)=\sum_{n=0}^{\infty}\frac{\lambda^{n}}{\sqrt{n!}}H_{n}(x)
\end{equation*}
and the generalized Hermite polynomials $H_{\multi}(\boldsymbol{x})$, for every multi-index $\multi\in\mbN_{0}^{\mbN}$ with finitely-many non-zero entries and $\boldsymbol{x}\in\mbR^{\mbN}$, by
\begin{equation*}
	H_{\multi}(\boldsymbol{x})\defeq\prod_{i=1}^{\infty}H_{\multi_{i}}(x_{i}),
\end{equation*}
which we extend to Banach space-valued arguments $\xi\in\nban$ via
\begin{equation*}
	H_{\multi}(\xi)\defeq H_{\multi}((\inner{\xi}{e_{i}})_{i\in\mbN}).
\end{equation*}
\begin{details}
	Our first three Hermite polynomials are given by $H_{1}(x)=x$, $H_{2}(x)=\frac{1}{\sqrt{2}}(x^2-1)$ and $H_{3}(x)=\frac{1}{\sqrt{6}}(x^{3}-3x)$.
	\paragraph{Conventions.}
	Note that the convention used in this paper differs from the ones employed in~\cite{mourrat_weber_xu_16}, \cite[Sec.~5]{ledoux_94} and~\cite{nualart_06}. To pass from our Hermite polynomials ($H_{n}$) to the ones used in~\cite{mourrat_weber_xu_16}  ($H^{\textnormal{MWX}}_{n}$, say), we need to scale $H^{\textnormal{MWX}}_{n}(x,1)=\sqrt{n!}H_{n}(x)$. To pass from our Hermite polynomials to the ones used in~\cite[Sec.~5]{ledoux_94} ($H^{\textnormal{Led}}_{n}$, say), we need to scale $H^{\textnormal{Led}}_{n}(x)=1/\sqrt{n!}H_{n}(x)$. To pass from our Hermite polynomials to the ones used~\cite{nualart_06} ($H^{\textnormal{Nua}}_{n}$, say), we need to scale $\sqrt{n}H^{\textnormal{Nua}}_{n}(x)=H_{n}(x)$.
	
\end{details}
Let $\ban$ be a real, separable Banach space, we denote by $L^{2}(\nban,\mu;\ban)$ the space of functions $\Psi\from\nban\to\ban$ that are Bochner measurable and square integrable under $\mu$. We can now define the $\ban$-valued homogeneous Wiener chaos of degree $k\in\mbN$, denoted by $\chaos^{(k)}(\nban,\mu;\ban)$, as in~\cite[(3.3)]{hairer_weber_15}:
\begin{equation}\label{eq:hom_Wiener_chaos}
	\chaos^{(k)}(\nban,\gm;\ban)\defeq\Bigl\{\Psi\in L^2(\nban,\gm;\ban):\int_{\nban}\Psi(\xi)H_{\multi}(\xi)\gm(\dd\xi)=0~\text{for all}~\abs{\multi}\neq k\Bigr\},
\end{equation}
where the absolute value of the multi-index $\nu$ is given by $\abs{\multi}\defeq\sum_{i=1}^{\infty}\multi_{i}$. We further set $\chaos^{(0)}(\nban,\gm;\ban)\defeq\ban$.
\begin{definition}\label{def:direct_sum_Banach_inh_Wiener_chaos}
	Let $\mcW$ be a finite index set. For each $\tau\in\mcW$, let $K_{\tau}$ be a natural number, $\ban_{\tau}$ be a real, separable Banach space and $\Psi_{\tau}$ be an element of the $\ban_{\tau}$-valued inhomogeneous Wiener chaos of order $K_{\tau}$, i.e.\
	\begin{equation*}
		\Psi_{\tau}\defeq\sum_{k=0}^{K_{\tau}}\Psi_{\tau,k},\qquad \Psi_{\tau,k}\in\chaos^{(k)}(\nban,\gm;\ban_{\tau}).
	\end{equation*}
	Denote $\bban=\bigoplus_{\tau\in\mcW}\ban_{\tau}$ and let $\boldsymbol{\Psi}$ be given by
	\begin{equation}\label{eq:rv_direct_sum}
		\boldsymbol{\Psi}\defeq\bigoplus_{\tau\in\mcW}\Psi_{\tau}=\bigoplus_{\tau\in\mcW}\sum_{k=0}^{K_{\tau}}\Psi_{\tau,k},\qquad\Psi_{\tau,k}\in\chaos^{(k)}(\nban,\gm;\ban_{\tau}).
	\end{equation}
	Given $\eps>0$ and a random variable $\boldsymbol{\Psi}$ of the form~\eqref{eq:rv_direct_sum}, we define the rescaling 
	\begin{equation*}
		\boldsymbol{\Psi}^{(\eps)}\defeq\bigoplus_{\tau\in\mcW}\eps^{\frac{K_{\tau}}{2}}\Psi_{\tau}.
	\end{equation*}
\end{definition}
\begin{definition}
	Let $\boldsymbol{\Psi}$ be a random variable of the form~\eqref{eq:rv_direct_sum}, $\lim_{\eps\to0}\delta(\eps)=0$ and $(\boldsymbol{\Psi}^{(\eps)}_{\delta(\eps)})_{\eps>0}$ be a sequence such that
	\begin{equation}\label{eq:rv_direct_sum_rescaled_corr_length}
		\boldsymbol{\Psi}^{(\eps)}_{\delta(\eps)}=\bigoplus_{\tau\in\mcW}\eps^{\frac{K_{\tau}}{2}}\sum_{k=0}^{K_{\tau}}\Psi_{\delta(\eps);\tau,k},\qquad\Psi_{\delta(\eps);\tau,k}\in\chaos^{(k)}(\nban,\gm;\ban_{\tau}).
	\end{equation}
	We say that $\boldsymbol{\Psi}^{(\eps)}_{\delta(\eps)}$ converges to $\boldsymbol{\Psi}$ as $\eps\to0$ if for each $\tau\in\mcW$ and $k\leq K_{\tau}$,
	\begin{equation}\label{eq:rv_convergence}
		\lim_{\eps\to0}\eps^{K_{\tau}-k}\int_{\nban}\norm{\Psi_{\delta(\eps);\tau,k}(\xi)-\Psi_{\tau,k}(\xi)}_{\ban_{\tau}}^{2}\mu(\dd\xi)=0.
	\end{equation}
\end{definition}
\begin{definition}\label{def:hom_part}
	Let $\boldsymbol{\Psi}$ be a random variable of the form~\eqref{eq:rv_direct_sum}. We define the homogeneous part $\boldsymbol{\Psi}_{\hom}\from\cm\to\bban$ of $\boldsymbol{\Psi}$ by
	\begin{equation*}
		\boldsymbol{\Psi}_{\hom}\defeq\bigoplus_{\tau\in\mcW}(\Psi_{\tau,K_{\tau}})_{\hom},
	\end{equation*}
	where for each $\tau\in\mcW$ and $h\in\mcH_{\mu}$,
	\begin{equation*}
		(\Psi_{\tau,K_{\tau}})_{\hom}(h)\defeq\int_{\nban} \Psi_{\tau,K_{\tau}}(\xi+h)\mu(\dd\xi).
	\end{equation*} 
\end{definition}
The homogeneous part $(\Psi_{\tau,K_{\tau}})_{\hom}(h)$ is well-defined by the Cameron--Martin theorem~\cite[(4.11)]{ledoux_94} and~\cite[(3.8)]{hairer_weber_15}.
\begin{details}
	See~\cite[Lem.~1.53]{klose_msc_17} for more details.
\end{details}

We can now state a large deviation principle for random variables of the form~\eqref{eq:rv_direct_sum_rescaled_corr_length} that converge in the sense of~\eqref{eq:rv_convergence} to a limit that does not have any contributions from lower-order Wiener chaoses.
\begin{theorem}\label{thm:LDP_wiener_chaos}
	Let $(\boldsymbol{\Psi}^{(\eps)}_{\delta(\eps)})_{\eps>0}$ be a family of random variables of the form~\eqref{eq:rv_direct_sum_rescaled_corr_length} converging in the sense of~\eqref{eq:rv_convergence} to a random variable $\boldsymbol{\Psi}$ such that 
	\begin{equation}\label{eq:rv_limit}
		\boldsymbol{\Psi}=\bigoplus_{\tau\in\mcW}\Psi_{\tau,K_{\tau}},\qquad\Psi_{\tau,K_{\tau}}\in\chaos^{(K_{\tau})}(\nban,\gm;\ban_{\tau}).
	\end{equation}
	It then follows that $(\boldsymbol{\Psi}_{\delta(\eps)}^{(\eps)})_{\eps>0}$ satisfies a large deviation principle in $\bban$ with speed $\eps$ and good rate function $\rate_{\boldsymbol{\Psi}}$ given by
	\begin{equation*}
		\rate_{\boldsymbol{\Psi}}(\boldsymbol{s})\defeq\inf\Bigl\{\frac{1}{2}\norm{h}_{\cm}^{2}:h\in\cm~\text{with}~\boldsymbol{\Psi}_{\hom}(h)=\boldsymbol{s}\Bigr\}.
	\end{equation*}
\end{theorem}
\begin{proof}
	The key observation is that the proof of~\cite[Lem.~3.7]{hairer_weber_15} carries over if we assume~\eqref{eq:rv_convergence} and~\eqref{eq:rv_limit} instead of~\cite[(3.9)]{hairer_weber_15}. One can then adapt the strategy of~\cite[Thm.~3.5]{hairer_weber_15} to our setting.
	\begin{details}
		Let us discuss our changes to the proof of~\cite[Thm.~3.5]{hairer_weber_15}.
		
		As in~\cite[Sec.~3]{hairer_weber_15}, we decompose $\boldsymbol{\Psi}^{(\eps)}_{\delta(\eps)}$ as
		\begin{align*}
			&\boldsymbol{\Psi}^{(\eps)}_{\delta(\eps)}(\xi)=\bigoplus_{\tau\in\mcW}\eps^{K_{\tau}/2}\Psi_{\delta(\eps);\tau}(\xi),\qquad\Psi_{\delta(\eps);\tau}(\xi)=\sum_{k=0}^{K_{\tau}}\Psi_{\delta(\eps);\tau,k}(\xi),\\
			&\Psi_{\delta(\eps);\tau,k}(\xi)=\sum_{\abs{\multi}=k}y_{\multi;\delta(\eps);\tau,k}H_{\multi}(\xi),\qquad y_{\multi;\delta(\eps);\tau,k}\in\ban_{\tau}
		\end{align*}
		and consider for $N\in\mbN$ the truncated expansion
		\begin{align*}
			&\boldsymbol{\Psi}^{(\eps)}_{N;\delta(\eps)}(\xi)=\bigoplus_{\tau\in\mcW}\eps^{K_{\tau}/2}\Psi_{N;\delta(\eps);\tau}(\xi),\qquad\Psi_{N;\delta(\eps);\tau}(\xi)=\sum_{k=0}^{K_{\tau}}\Psi_{N;\delta(\eps);\tau,k}(\xi),\\
			&\Psi_{N;\delta(\eps);\tau,k}(\xi)=\sum_{\substack{\abs{\multi}=k\\\multi_{i}=0,~i>N}}y_{\multi;\delta(\eps);\tau,k}H_{\multi}(\xi).
		\end{align*}
		We further decompose the limit as
		\begin{equation*}
			\boldsymbol{\Psi}(\xi)=\bigoplus_{\tau\in\mcW}\Psi_{\tau,K_{\tau}}(\xi),\qquad\Psi_{\tau,K_{\tau}}(\xi)=\sum_{\abs{\multi}=K_{\tau}}y_{\multi;\tau,K_{\tau}}H_{\multi}(\xi),\qquad y_{\multi;\tau,K_{\tau}}\in\ban_{\tau}
		\end{equation*}
		We can now adapt~\cite[Lem.~3.7]{hairer_weber_15} to our mode of convergence~\eqref{eq:rv_convergence} towards a limit that satisfies~\eqref{eq:rv_limit}.
		\begin{lemma}\label{lem:LDP_approx}
			Let $N\in\mbN$, it follows that the random variables $(\boldsymbol{\Psi}^{(\eps)}_{N;\delta(\eps)})_{\eps>0}$ satisfy a large deviation principle in $\bban$ with speed $\eps$ and good rate function
			\begin{equation*}
				\rate_{N}(\boldsymbol{s})\defeq\inf\Bigl\{\frac{1}{2}\norm{h}_{\cm}^{2}:h\in\cm~\text{with}~\boldsymbol{\Psi}_{\hom}(h)=\boldsymbol{s}\Bigr\}.
			\end{equation*}
		\end{lemma} 
		\begin{proof}
			Let us rewrite $\boldsymbol{\Psi}^{(\eps)}_{N;\delta(\eps)}(\xi)$ so as to explicit the dependence on $\eps^{\sfrac{1}{2}}\xi$,
			\begin{align*}
				&\boldsymbol{\Psi}^{(\eps)}_{N;\delta(\eps)}(\xi)=\bigoplus_{\tau\in\mcW}\sum_{k=0}^{K_{\tau}}\Phi_{N;\tau,k}^{(\eps)}(\eps^{\sfrac{1}{2}}\xi),\qquad\Phi_{N;\tau,k}^{(\eps)}(\xi)\defeq\eps^{\frac{K_{\tau}-k}{2}}\sum_{\substack{\abs{\multi}=k\\\multi_{i}=0,\,i>N}}y_{\multi;\delta(\eps);\tau,k}\Phi_{\multi}^{(\eps)}(\xi)\eqdef\boldsymbol{\Phi}_{N}^{(\eps)}(\eps^{\sfrac{1}{2}}\xi),\\
				&\Phi_{\multi}^{(\eps)}(\xi)\defeq\prod_{i=1}^{\infty}\eps^{\frac{\multi_{i}}{2}}H_{\multi_{i}}(\eps^{-\frac{1}{2}}\inner{\xi}{e_{i}}).
			\end{align*}
			Using that each $H_{\multi_{i}}$ is a polynomial of degree $\multi_{i}$ with leading-order coefficient $1/\sqrt{\multi_{i}!}$, we obtain
			\begin{equation*}
				\Phi_{\multi}^{(\eps)}(\xi)\to\frac{1}{\sqrt{\multi!}}\prod_{i=1}^{\infty}\inner{\xi}{e_{i}}^{\multi_{i}}\eqdef\frac{1}{\sqrt{\multi!}}\xi^{\multi}\quad\text{as}~\eps\to0.
			\end{equation*}
			We now use~\eqref{eq:rv_convergence} and~\cite[(3.17)]{hairer_weber_15} to deduce for every $\tau\in\mcW$, $k\leq K_{\tau}$ and $\abs{\nu}=k$, that
			\begin{equation*}
				\eps^{\frac{K_{\tau}-k}{2}}\norm{y_{\multi;\delta(\eps);\tau;k}-y_{\multi;\tau,k}}_{E_{\tau}}\leq\eps^{\frac{K_{\tau}-k}{2}}\Bigl(\int_{\nban}\norm{\Psi_{\delta(\eps);\tau,k}(\xi)-\Psi_{\tau,k}(\xi)}_{\ban_{\tau}}^{2}\gm(\dd\xi)\Bigr)^{1/2}\to0\quad\text{as}~\eps\to0.
			\end{equation*}
			Further for all $k<K_{\tau}$, assumption~\eqref{eq:rv_limit} implies $y_{\multi;\tau,k}=0$, which combined with the above yields
			\begin{equation*}
				\eps^{\frac{K_{\tau}-k}{2}}\norm{y_{\multi;\delta(\eps);\tau;k}}_{E_{\tau}}\to0.
			\end{equation*}
			Therefore,
			\begin{equation*}
				\lim_{\eps\to0}\Phi_{N;\tau,k}^{(\eps)}(\xi)
				=
				\begin{cases}
					\begin{aligned}
						&\qquad0\qquad&&\text{for}~k<K_{\tau},\\
						&\sum_{\substack{\abs{\multi}=K_{\tau}\\\multi_{i}=0,\,i>N}}y_{\multi;\tau,K_{\tau}}\frac{1}{\sqrt{\multi!}}\xi^{\multi}\qquad&&\text{for}~k=K_{\tau}
					\end{aligned}
				\end{cases}
			\end{equation*}
			uniformly on bounded subset of $\nban$. It follows by~\cite[Lem.~1.53]{klose_msc_17}, that for every $h\in\cm$, $\lim_{\eps\to0}\Phi_{N;\tau,K_{\tau}}^{(\eps)}(\xi)=(\Psi_{N;\tau,K_{\tau}})_{\hom}(h)$, which yields
			\begin{equation*}
				\lim_{\eps\to0}\boldsymbol{\Phi}_{N}^{(\eps)}(h)=\bigoplus_{\tau\in\mcW}(\Psi_{N;\tau,K_{\tau}})_{\hom}(h)=(\boldsymbol{\Psi}_{N})_{\hom}(h).
			\end{equation*}
			The claim then follows by a combination of Schilder's theorem with the generalized contraction principle~\cite[Lem.~3.3]{hairer_weber_15}, see~\cite[Lem.~3.7]{hairer_weber_15} or~\cite[Prop.~3.15]{klose_msc_17} for more details.
		\end{proof}
		Next, let us adapt the proof of~\cite[Lem.~3.9]{hairer_weber_15} to our mode of convergence~\eqref{eq:rv_convergence} towards a limit that satisfies~\eqref{eq:rv_limit}.
		\begin{lemma}\label{lem:exp_good_approx}
			For every $\lambda>0$ it holds that
			\begin{equation*}
				\limsup_{N\to\infty}\limsup_{\eps\to0}\eps\log\gm\Bigl(\xi:\norm{\boldsymbol{\Psi}^{(\eps)}_{\delta(\eps)}(\xi)-\boldsymbol{\Psi}^{(\eps)}_{N;\delta(\eps)}(\xi)}_{\bban}\geq\lambda\Bigr)=-\infty.
			\end{equation*}
		\end{lemma}
		\begin{proof}
			Let us denote 
			\begin{equation*}
				\begin{split}
					&\Psi^{(\eps)}_{\delta(\eps);\tau}(\xi)\defeq\eps^{K_{\tau}/2}\Psi_{\delta(\eps);\tau}(\xi),\qquad\Psi_{\delta(\eps);\tau,k}^{(\eps)}(\xi)\defeq\eps^{K_{\tau}/2}\Psi_{\delta(\eps);\tau,k}(\xi),\\
					&\Psi^{(\eps)}_{N;\delta(\eps);\tau}(\xi)\defeq\eps^{K_{\tau}/2}\Psi_{N;\delta(\eps);\tau}(\xi),\qquad\Psi_{N;\delta(\eps);\tau,k}^{(\eps)}(\xi)\defeq\eps^{K_{\tau}/2}\Psi_{N;\delta(\eps);\tau,k}(\xi).
				\end{split}
			\end{equation*}
			Since $\norm{\place}_{\bban}=\sum_{\tau\in\mcW}\norm{\place}_{\ban_{\tau}}$, we may estimate
			\begin{equation*}
				\begin{split}
					\gm\Bigl(\xi:\norm{\boldsymbol{\Psi}^{(\eps)}_{\delta(\eps)}(\xi)-\boldsymbol{\Psi}^{(\eps)}_{N;\delta(\eps)}(\xi)}_{\bban}\geq\lambda\Bigr)&\leq\gm\Bigl(\xi:\norm{\Psi^{(\eps)}_{\delta(\eps);\tau}(\xi)-\Psi^{(\eps)}_{N;\delta(\eps);\tau}(\xi)}_{\ban_{\tau}}\geq\lambda/\abs{\mcW}~\text{for some}~\tau\in\mcW\Bigr)\\
					&\leq\sum_{\tau\in\mcW}\gm\Bigl(\xi:\norm{\Psi^{(\eps)}_{\delta(\eps);\tau}(\xi)-\Psi^{(\eps)}_{N;\delta(\eps);\tau}(\xi)}_{\ban_{\tau}}\geq\lambda/\abs{\mcW}\Bigr).
				\end{split}
			\end{equation*}
			Similarly we can decompose each $\Psi^{(\eps)}_{\delta(\eps);\tau}$ into its homogeneous Wiener chaoses,
			\begin{equation*}
				\begin{split}
					&\gm\Bigl(\xi:\norm{\Psi^{(\eps)}_{\delta(\eps);\tau}(\xi)-\Psi^{(\eps)}_{N;\delta(\eps);\tau}(\xi)}_{\ban_{\tau}}\geq\lambda/\abs{\mcW}\Bigr)\\
					&\leq\gm\Bigl(\xi:\sum_{k=0}^{K_{\tau}}\norm{\Psi^{(\eps)}_{\delta(\eps);\tau,k}(\xi)-\Psi^{(\eps)}_{N;\delta(\eps);\tau,k}(\xi)}_{\ban_{\tau}}\geq\lambda/\abs{\mcW}\Bigr)\\
					&\leq\sum_{k=0}^{K_{\tau}}\gm\Bigl(\xi:\norm{\Psi^{(\eps)}_{\delta(\eps);\tau,k}(\xi)-\Psi^{(\eps)}_{N;\delta(\eps);\tau,k}(\xi)}_{\ban_{\tau}}\geq\frac{\lambda}{\abs{\mcW}(K_{\tau}+1)}\Bigr).
				\end{split}
			\end{equation*}
			To summarize, we can estimate the direct sum of inhomogeneous Wiener chaoses by a sum of its components
			\begin{equation}\label{eq:dsum_inhom_chaos_upper_bound}
				\gm\Bigl(\xi:\norm{\boldsymbol{\Psi}^{(\eps)}_{\delta(\eps)}(\xi)-\boldsymbol{\Psi}^{(\eps)}_{N;\delta(\eps)}(\xi)}_{\bban}\geq\lambda\Bigr)\leq\sum_{\tau\in\mcW}\sum_{k=0}^{K_{\tau}}\gm\Bigl(\xi:\norm{\Psi^{(\eps)}_{\delta(\eps);\tau,k}(\xi)-\Psi^{(\eps)}_{N;\delta(\eps);\tau,k}(\xi)}_{\ban_{\tau}}\geq\frac{\lambda}{\abs{\mcW}(K_{\tau}+1)}\Bigr).
			\end{equation}
			We combine~\eqref{eq:dsum_inhom_chaos_upper_bound} with~\cite[Lem.~1.2.15]{dembo_zeitouni_10}, to control
			\begin{equation*}
				\begin{split}
					&\limsup_{\eps\to0}\eps\log\gm\Bigl(\xi:\norm{\boldsymbol{\Psi}^{(\eps)}_{\delta(\eps)}(\xi)-\boldsymbol{\Psi}^{(\eps)}_{N;\delta(\eps)}(\xi)}_{\bban}\geq\lambda\Bigr)\\
					&\leq\max_{\tau\in\mcW}\max_{k=0}^{K_{\tau}}\limsup_{\eps\to0}\eps\log\gm\Bigl(\xi:\norm{\Psi^{(\eps)}_{\delta(\eps);\tau,k}(\xi)-\Psi^{(\eps)}_{N;\delta(\eps);\tau,k}(\xi)}_{\ban_{\tau}}\geq\frac{\lambda}{\abs{\mcW}(K_{\tau}+1)}\Bigr),
				\end{split}
			\end{equation*}
			Using that limit superior and maxima commute, it suffices to show
			\begin{equation}\label{eq:exponentially_good_approx_separate}
				\limsup_{N\to\infty}\limsup_{\eps\to0}\eps\log\gm\Bigl(\xi:\norm{\Psi^{(\eps)}_{\delta(\eps);\tau,k}(\xi)-\Psi^{(\eps)}_{N;\delta(\eps);\tau,k}(\xi)}_{\ban_{\tau}}\geq\lambda\Bigl)=-\infty
			\end{equation}
			for each $\tau\in\mcW$ and $k\leq K_{\tau}$ separately.
			
			\emph{Case $k=K_{\tau}$.}
			It follows by the proof of~\cite[Lem.~3.9]{hairer_weber_15} that there exists some $\beta_{N;\tau,K_{\tau}}\to\infty$ as $N\to\infty$ and $C>0$, such that for every $N\in\mbN$ and $\eps<\eps_{N}$,
			\begin{equation*}
				\gm\Bigl(\xi:\norm{\Psi^{(\eps)}_{\delta(\eps);\tau,K_{\tau}}(\xi)-\Psi^{(\eps)}_{N;\delta(\eps);\tau,K_{\tau}}(\xi)}_{\ban_{\tau}}\geq\lambda\Bigr)\leq C\exp\Bigl(-\beta_{N;\tau,K_{\tau}}\lambda^{2/K_{\tau}}\eps^{-1}\Bigl),
			\end{equation*}
			where we used that~\eqref{eq:rv_convergence} coincides with~\cite[(3.9)]{hairer_weber_15} for $k=K_{\tau}$. Consequently,
			\begin{equation*}
				\limsup_{N\to\infty}\limsup_{\eps\to0}\eps\log\gm\Bigl(\xi:\norm{\Psi^{(\eps)}_{\delta(\eps);\tau,K_{\tau}}(\xi)-\Psi^{(\eps)}_{N;\delta(\eps);\tau,K_{\tau}}(\xi)}_{\ban_{\tau}}\geq\lambda\Bigr)=-\infty,
			\end{equation*}
			which yields~\eqref{eq:exponentially_good_approx_separate} for $k=K_{\tau}$.
			
			\emph{Case $k<K_{\tau}$.}
			Let $(\beta_{N;\tau,k})_{N\in\mbN}$ be a sequence such that $\beta_{N;\tau,k}\to\infty$ as $N\to\infty$. We estimate by Markov's inequality,
			\begin{equation}\label{eq:Wiener_chaos_bound_Markov}
				\begin{split}
					&\gm\Bigl(\xi:\norm{\Psi^{(\eps)}_{\delta(\eps);\tau,k}(\xi)-\Psi^{(\eps)}_{N;\delta(\eps);\tau,k}(\xi)}_{\ban_{\tau}}\geq\lambda\Bigr)\\
					&=\gm\Bigl(\xi:\exp\Bigl(\beta_{N;\tau,k}\eps^{\frac{K_{\tau}-k}{k}}\norm{\Psi_{\delta(\eps);\tau,k}(\xi)-\Psi_{N;\delta(\eps);\tau,k}(\xi)}_{\ban_{\tau}}^{2/k}\Bigr)\geq\exp\Bigl(\beta_{N;\tau,k}\eps^{\frac{K_{\tau}-k}{k}}\eps^{-\frac{K_{\tau}}{k}}\lambda^{2/k}\Bigr)\Bigr)\\
					&\leq\int_{\nban}\exp\Bigl(\beta_{N;\tau,k}\eps^{\frac{K_{\tau}-k}{k}}\norm{\Psi_{\delta(\eps);\tau,k}(\xi)-\Psi_{N;\delta(\eps);\tau,k}(\xi)}_{\ban_{\tau}}^{2/k}\Bigr)\gm(\dd\xi)\exp(-\beta_{N;\tau,k}\eps^{-1}\lambda^{2/k}).
				\end{split}
			\end{equation}
			Assume that we can find some $C>0$ and a sequence $(\eps_{N})_{N\in\mbN}$ of positive real numbers such that for every $N\in\mbN$ and $\eps<\eps_{N}$,
			\begin{equation}\label{eq:exp_moment}
				\int_{\nban}\exp\Bigl(\beta_{N;\tau,k}\eps^{\frac{K_{\tau}-k}{k}}\norm{\Psi_{\delta(\eps);\tau,k}(\xi)-\Psi_{N;\delta(\eps);\tau,k}(\xi)}_{\ban_{\tau}}^{2/k}\Bigr)\gm(\dd\xi)\leq C.
			\end{equation}
			Combining~\eqref{eq:exp_moment} and~\eqref{eq:Wiener_chaos_bound_Markov} we obtain
			\begin{equation*}
				\gm\Bigl(\xi:\norm{\Psi^{(\eps)}_{\delta(\eps);\tau,k}(\xi)-\Psi^{(\eps)}_{N;\delta(\eps);\tau,k}(\xi)}_{\ban_{\tau}}\geq\lambda\Bigr)\leq C\exp(-\beta_{N;\tau,k}\eps^{-1}\lambda^{2/k}),
			\end{equation*}
			which yields
			\begin{equation*}
				\limsup_{N\to\infty}\limsup_{\eps\to0}\eps\log\gm\Bigl(\xi:\norm{\Psi^{(\eps)}_{\delta(\eps);\tau,k}(\xi)-\Psi^{(\eps)}_{N;\delta(\eps);\tau,k}(\xi)}_{\ban_{\tau}}\geq\lambda\Bigr)=-\infty.
			\end{equation*}
			Hence to establish~\eqref{eq:exponentially_good_approx_separate} for $k<K_{\tau}$, it suffices to show~\eqref{eq:exp_moment}.
			
			To establish~\eqref{eq:exp_moment}, we use a power-series expansion to represent
			\begin{equation*}
				\begin{split}
					&\int_{\nban}\exp\Bigl(\beta_{N;\tau,k}\eps^{\frac{K_{\tau}-k}{k}}\norm{\Psi_{\delta(\eps);\tau,k}(\xi)-\Psi_{N;\delta(\eps);\tau,k}(\xi)}_{\ban_{\tau}}^{2/k}\Bigr)\mu(\dd\xi)\\
					&=\sum_{l=0}^{\infty}\frac{1}{l!}\beta_{N;\tau,k}^{l}\eps^{l\frac{K_{\tau}-k}{k}}\int_{\nban}\norm{\Psi_{\delta(\eps);\tau,k}(\xi)-\Psi_{N;\delta(\eps);\tau,k}(\xi)}_{\ban_{\tau}}^{2l/k}\mu(\dd\xi).
				\end{split}
			\end{equation*}
			For $l\leq k$, we estimate by H\"{o}lder's inequality
			\begin{equation*}
				\int_{\nban}\norm{\Psi_{\delta(\eps);\tau,k}(\xi)-\Psi_{N;\delta(\eps);\tau,k}(\xi)}_{\ban_{\tau}}^{2l/k}\mu(\dd\xi)\leq\Bigl(\int_{\nban}\norm{\Psi_{\delta(\eps);\tau,k}(\xi)-\Psi_{N;\delta(\eps);\tau,k}(\xi)}_{\ban_{\tau}}^{2}\mu(\dd\xi)\Bigr)^{l/k}
			\end{equation*}
			and for $l>k$, we bound by Nelson's estimate~\cite[Lem.~2]{friz_victoir_07},
			\begin{equation*}
				\int_{\nban}\norm{\Psi_{\delta(\eps);\tau,k}(\xi)-\Psi_{N;\delta(\eps);\tau,k}(\xi)}_{\ban_{\tau}}^{2l/k}\mu(\dd\xi)\leq\Bigl(\frac{2l}{k}-1\Bigr)^{l}\Bigl(\int_{\nban}\norm{\Psi_{\delta(\eps);\tau,k}(\xi)-\Psi_{N;\delta(\eps);\tau,k}(\xi)}_{\ban_{\tau}}^{2}\mu(\dd\xi)\Bigr)^{l/k},
			\end{equation*}
			which yields
			\begin{equation}\label{eq:power_series_exp}
				\begin{split}
					&\int_{\nban}\exp\Bigl(\beta_{N;\tau,k}\eps^{\frac{K_{\tau}-k}{k}}\norm{\Psi_{\delta(\eps);\tau,k}(\xi)-\Psi_{N;\delta(\eps);\tau,k}(\xi)}_{\ban_{\tau}}^{2/k}\Bigr)\mu(\dd\xi)\\
					&\leq\sum_{l=0}^{k}\frac{1}{l!}\beta_{N;\tau,k}^{l}\Bigl(\eps^{K_{\tau}-k}\int_{\nban}\norm{\Psi_{\delta(\eps);\tau,k}(\xi)-\Psi_{N;\delta(\eps);\tau,k}(\xi)}_{\ban_{\tau}}^{2}\gm(\dd\xi)\Bigr)^{l/k}\\
					&\quad+\sum_{l=k+1}^{\infty}\frac{1}{l!}\beta_{N;\tau,k}^{l}\Bigl(\frac{2l}{k}-1\Bigr)^{l}\Bigl(\eps^{K_{\tau}-k}\int_{\nban}\norm{\Psi_{\delta(\eps);\tau,k}(\xi)-\Psi_{N;\delta(\eps);\tau,k}(\xi)}_{\ban_{\tau}}^{2}\gm(\dd\xi)\Bigr)^{l/k}.
				\end{split}
			\end{equation}
			A combination of Jensen's inequality, the triangle inequality~\cite[Sec.~2.1]{ledoux_talagrand_11} and the tower property implies
			\begin{equation}\label{eq:exp_bound_lot_1}
				\begin{split}
					&\int_{B}\norm{\Psi_{\delta(\eps);\tau,k}(\xi)-\Psi_{N;\delta(\eps);\tau,k}(\xi)}_{\ban_{\tau}}^{2}\gm(\dd\xi)\\
					&\leq2\int_{B}\norm{\Psi_{\delta(\eps);\tau,k}(\xi)}^{2}\gm(\dd\xi)+2\int_{B}\norm{\Psi_{N;\delta(\eps);\tau,k}(\xi)}_{\ban_{\tau}}^{2}\gm(\dd\xi)\\
					&\leq4\int_{B}\norm{\Psi_{\delta(\eps);\tau,k}(\xi)}^{2}\gm(\dd\xi),
				\end{split}
			\end{equation}
			where we used that $\Psi_{N;\delta(\eps);\tau,k}$ is a conditional expectation of $\Psi_{\delta(\eps);\tau,k}$ (see~\cite[(3.4)]{hairer_weber_15}). Combining~\eqref{eq:exp_bound_lot_1} and~\eqref{eq:power_series_exp}, we obtain
			\begin{equation*}
				\begin{split}
					&\int_{\nban}\exp\Bigl(\beta_{N;\tau,k}\eps^{\frac{K_{\tau}-k}{k}}\norm{\Psi_{\delta(\eps);\tau,k}(\xi)-\Psi_{N;\delta(\eps);\tau,k}(\xi)}_{\ban_{\tau}}^{2/k}\Bigr)\gm(\dd\xi)\\
					&\leq\sum_{l=0}^{k}\frac{1}{l!}\beta_{N;\tau,k}^{l}4^{l/k}\Bigl(\eps^{K_{\tau}-k}\int_{\nban}\norm{\Psi_{\delta(\eps);\tau,k}(\xi)}_{\ban_{\tau}}^{2}\gm(\dd\xi)\Bigr)^{l/k}\\
					&\quad+\sum_{l=k+1}^{\infty}\frac{1}{l!}\beta_{N;\tau,k}^{l}\Bigl(\frac{2l}{k}-1\Bigr)^{l}4^{l/k}\Bigl(\eps^{K_{\tau}-k}\int_{\nban}\norm{\Psi_{\delta(\eps);\tau,k}(\xi)}_{\ban_{\tau}}^{2}\gm(\dd\xi)\Bigr)^{l/k}.
				\end{split}
			\end{equation*}
			Using our mode of convergence~\eqref{eq:rv_convergence} towards a limit without lower-order terms~\eqref{eq:rv_limit}, it follows that for all $N\in\mbN$ there exists some $\eps_{N}>0$ such that for all $\eps<\eps_{N}$,
			\begin{equation*}
				y\defeq\beta_{N;\tau,k}\frac{2}{k}4^{1/k}\Bigl(\eps^{K_{\tau}-k}\int_{\nban}\norm{\Psi_{\delta(\eps);\tau,k}(\xi)}_{\ban_{\tau}}^{2}\gm(\dd\xi)\Bigr)^{1/k}<\euler^{-1}.
			\end{equation*}
			The power series
			\begin{equation*}
				y\mapsto\sum_{l=1}^{\infty}\frac{1}{l!}l^{l}y^{l}
			\end{equation*}
			converges absolutely for $\abs{y}<\euler^{-1}$, which can be seen by an application of the Cauchy ratio test:
			\begin{equation*}
				\lim_{l\to\infty}\frac{l!(l+1)^{l+1}\abs{y}^{l+1}}{(l+1)!l^{l}\abs{y}^{l}}=\abs{y}\lim_{l\to\infty}\Bigl(1+\frac{1}{l}\Bigr)^{l}=\abs{y}\euler^{1}<1.
			\end{equation*}
			Therefore it follows by the dominated convergence theorem, that there exists some $C>0$ and a sequence $(\eps_{N})_{N\in\mbN}$ such that for every $N\in\mbN$ and $\eps<\eps_{N}$, 
			\begin{equation*}
				\begin{split}
					&\sum_{l=0}^{k}\frac{1}{l!}\beta_{N;\tau,k}^{l}4^{l/k}\Bigl(\eps^{K_{\tau}-k}\int_{\nban}\norm{\Psi_{\delta(\eps);\tau,k}(\xi)}_{\ban_{\tau}}^{2}\gm(\dd\xi)\Bigr)^{l/k}\\
					&\quad+\sum_{l=k+1}^{\infty}\frac{1}{l!}\beta_{N;\tau,k}^{l}\Bigl(\frac{2l}{k}-1\Bigr)^{l}4^{l/k}\Bigl(\eps^{K_{\tau}-k}\int_{\nban}\norm{\Psi_{\delta(\eps);\tau,k}(\xi)}_{\ban_{\tau}}^{2}\gm(\dd\xi)\Bigr)^{l/k}\\
					&\leq C,
				\end{split}
			\end{equation*}
			which yields~\eqref{eq:exp_moment} and hence~\eqref{eq:exponentially_good_approx_separate} for $k=K_{\tau}$.
			
			All in all, we showed~\eqref{eq:exponentially_good_approx_separate} for all $\tau\in\mcW$ and $k\leq K_{\tau}$, which yields the claim.
		\end{proof}
		A combination  of Lemma~\ref{lem:LDP_approx}, Lemma~\ref{lem:exp_good_approx} and~\cite[Thm.~4.2.16~(a)]{dembo_zeitouni_10} implies that $(\boldsymbol{\Psi}^{(\eps)}_{\delta(\eps)})_{\eps>0}$ satisfies a weak large deviation principle in $\bban$ with speed $\eps$ and rate function
		\begin{equation*}
			\msJ(\boldsymbol{s})\defeq\sup_{\lambda>0}\liminf_{N\to\infty}\inf_{\boldsymbol{r}:\norm{\boldsymbol{s}-\boldsymbol{r}}_{\bban}<\lambda}\rate_{N}(\boldsymbol{r}).
		\end{equation*}
		To conclude the proof of Theorem~\ref{thm:LDP_wiener_chaos}, we can apply~\cite[Lem.~3.8]{hairer_weber_15} (see also~\cite[Lem.~3.17]{klose_msc_17} for more details) to show $\msJ=\rate_{\boldsymbol{\Psi}}$, that $\rate_{\boldsymbol{\Psi}}$ is a good rate function and that for every closed set $\boldsymbol{C}\subseteq\bban$,
		\begin{equation*}
			\inf_{\boldsymbol{s}\in\boldsymbol{C}}\rate_{\boldsymbol{\Psi}}(s)\leq\limsup_{N\to\infty}\inf_{\boldsymbol{s}\in\boldsymbol{C}}\rate_{N}(\boldsymbol{s});
		\end{equation*}
		indeed an application of~\cite[Thm.~4.2.16~(b)]{dembo_zeitouni_10} then implies the large deviation principle for $(\boldsymbol{\Psi}^{(\eps)}_{\delta(\eps)})_{\eps>0}$ in $\bban$ with speed $\eps$ and good rate function $\rate_{\boldsymbol{\Psi}}$.
	\end{details}
\end{proof}
\section{The Enhancement Driven by a Cameron--Martin Element}\label{app:enh_driven_by_h}
In this appendix we analyse the regularity of the mild solution $\ti^{h}=\vdiv\mcI[\srdet h]$ driven by a Cameron--Martin element $h\in L^{2}([0,T]\times\mbT^{2};\mbR^{2})$ (Lemmas~\ref{lem:lolli_h}~\&~\ref{lem:lolli_h_L2H1}), which we then use to construct the enhancement $\mbX^{h}=(\ti^{h},\ty^{h},\tp^{h},\tc^{h})$ (cf.~\eqref{eq:enhancement_h_def}) appearing in Theorem~\ref{thm:LDP_enhancement} (Lemma~\ref{lem:existence_enhancement_h}).

For every $t\in[0,T]$ and $\om=(\om^{1},\om^{2})\in\mbZ^{2}$, we can represent the Fourier transform of $\ti^{h}$ by
\begin{equation}\label{eq:lolli_h_Fourier}
	\hat{\ti^{h}}(t,\om)=\sum_{j=1}^{2}\sum_{m\in\mbZ^{2}}\int_{0}^{t}H_{t-u}^{j}(\om)\hat{\srdet}(u,\om-m)\hat{h}(u,m,j)\dd u,
\end{equation}
where $H^{j}_{t-u}(\om)=2\uppi\upi\om^{j}\exp(-\abs{t-u}\abs{2\uppi\om}^{2})$ (see Subsection~\ref{subsec:notation}).

The next lemma establishes the regularity of $\ti^{h}$ as a (H\"{o}lder-)continuous function in time.
\begin{lemma}\label{lem:lolli_h}
	Let $\rho_{0}\in C(\mbT^{2})$ be such that $\rho_{0}\geq0$ and $\mean{\rho_{0}}=1$, $\rdet$ be the weak solution to~\eqref{eq:Determ_KS} with initial data $\rho_{0}$ and chemotactic sensitivity $\chem\in\mbR$, and $\Trdet$ be its maximal time of existence (see Lemma~\ref{lem:Determ_KS_continuity}). Then for all $T<\Trdet$, $h\in L^{2}([0,T]\times\mbT^{2};\mbR^{2})$ and $\kappa\in[0,1/2]$ it holds that
	\begin{equation*}
		\norm{\ti^{h}}_{C_{T}^{\kappa}\mcH^{-2\kappa}}\lesssim\norm{\srdet}_{C([0,T]\times\mbT^{2})}\norm{h}_{L^{2}([0,T]\times\mbT^{2};\mbR^{2})}
	\end{equation*}
	and in particular $\ti^{h}\in C_{T}L^{2}(\mbT^{2})$.
\end{lemma}
\begin{proof}
	An application of Lemma~\ref{lem:regularity_srdet} yields $\srdet\in C([0,T]\times\mbT^{2})$; in particular $\srdet h\in L^{2}([0,T]\times\mbT^{2};\mbR^{2})$ and upon redefining $h$, we may assume that $\ti^{h}=\vdiv\mcI[h]$.
	
	Let $\om\in\mbZ^{2}$ and $s,t\in[0,T]$ with $s<t$. We bound by Jensen's inequality and the Cauchy--Schwarz inequality,
	\begin{equation*}
		\begin{split}
			\abs{\hat{\ti^{h}}(t,\om)-\hat{\ti^{h}}(s,\om)}^{2}&\lesssim\sum_{j=1}^{2}\Bigl\lvert\int_{0}^{s}(H_{t-u}^{j}(\om)-H_{s-u}^{j}(\om))\hat{h}(u,\om,j)\dd u\Bigr\rvert^{2}+\Bigl\lvert\int_{s}^{t}H_{t-u}^{j}(\om)\hat{h}(u,\om,j)\dd u\Bigr\rvert^{2}\\
			&\leq\sum_{j=1}^{2}\Bigl(\int_{0}^{s}\abs{H_{t-u}^{j}(\om)-H_{s-u}^{j}(\om)}^{2}\dd u\Bigr)\Bigl(\int_{0}^{s}\abs{\hat{h}(u,\om,j)}^{2}\dd u\Bigr)\\
			&\multiquad[3]+\Bigl(\int_{s}^{t}\abs{H_{t-u}^{j}(\om)}^{2}\dd u\Bigr)\Bigl(\int_{s}^{t}\abs{\hat{h}(u,\om,j)}^{2}\dd u\Bigr).
		\end{split}
	\end{equation*}
	It follows by interpolation that for all $\kappa\in[0,1/2]$,
	\begin{equation}\label{eq:lolli_h_increment_bound}
		\abs{\hat{\ti^{h}}(t,\om)-\hat{\ti^{h}}(s,\om)}^{2}\lesssim\abs{t-s}^{2\kappa}\abs{\om}^{4\kappa}\norm{\hat{h}(\om)}^{2}_{L^{2}([0,T];\mbR^{2})},
	\end{equation}
	which allows us to estimate by~\eqref{eq:Bessel_norm},
	\begin{equation*}
		\norm{\ti^{h}_{t}-\ti^{h}_{s}}_{\mcH^{-2\kappa}}^{2}=\sum_{\om\in\mbZ^{2}}(1+\abs{2\uppi\om}^{2})^{-2\kappa}\abs{\hat{\ti^{h}}(t,\om)-\hat{\ti^{h}}(s,\om)}^{2}\lesssim\abs{t-s}^{2\kappa}\norm{h}_{L^{2}([0,T]\times\mbT^{2};\mbR^{2})}^{2}.
	\end{equation*}
	By the definition of the $C_{T}^{\kappa}\mcH^{-2\kappa}(\mbT^{2})$-norm, we obtain the bound
	\begin{equation*}
		\norm{\ti^{h}}_{C_{T}^{\kappa}\mcH^{-2\kappa}}\lesssim\norm{h}_{L^{2}([0,T]\times\mbT^{2};\mbR^{2})}.
	\end{equation*}
	
	To establish $\ti^{h}\in C_{T}L^{2}(\mbT^{2})$, we approximate and use the completeness of $C_{T}L^{2}(\mbT^{2})$ under the supremum-norm. Let $(h_{\delta})_{\delta>0}$ be a sequence such that $h_{\delta}\in L^{2}([0,T];\mcH^{\gamma}(\mbT^{2};\mbR^{2}))$ for every $\delta>0$ and $\gamma\geq0$, and such that $h_{\delta}\to h\in L^{2}([0,T]\times\mbT^{2};\mbR^{2})$ as $\delta\to0$.
	\begin{details}
		To construct such a sequence, define for all $\delta>0$,
		\begin{equation*}
			h_{\delta}(t,x)\defeq\sum_{\substack{\om\in\mbZ^{2}\\\abs{\om}\leq\delta^{-1}}}\euler^{2\uppi\upi\inner{\om}{x}}\hat{h}(t,\om).
		\end{equation*}
		The approximation $h_{\delta}$ has compact support in Fourier space, hence by Parseval's theorem,
		\begin{equation*}
			\norm{h_{\delta}(t)}_{\mcH^{\gamma}}^{2}=\sum_{\om\in\mbZ^{2}}(1+\abs{2\uppi\om}^{2})^{\gamma}\abs{\hat{h_{\delta}}(t,\om)}^{2}\lesssim(1+\delta^{-2})^{\gamma}\sum_{\om\in\mbZ^{2}}\abs{\hat{h}(t,\om)}^{2}=(1+\delta^{-2})^{\gamma}\norm{h(t)}_{L^{2}(\mbT^{2};\mbR^{2})}^{2},
		\end{equation*}
		which proves $h_{\delta}\in L^{2}([0,T];\mcH^{\gamma}(\mbT^{2};\mbR^{2}))$. To show $h_{\delta}\to h\in L^{2}([0,T]\times\mbT^{2};\mbR^{2})$, we estimate
		\begin{equation*}
			\norm{h(t)-h_{\delta}(t)}_{L^{2}(\mbT^{2};\mbR^{2})}^{2}=\sum_{\substack{\om\in\mbZ^{2}\\\abs{\om}>\delta^{-1}}}\abs{\hat{h}(t,\om)}^{2}\leq\norm{h(t)}_{L^{2}(\mbT^{2};\mbR^{2})}^{2}
		\end{equation*}
		and use the dominated convergence theorem.
	\end{details}
	An application of~\eqref{eq:lolli_h_increment_bound} yields for all $s,t\in[0,T]$ with $s<t$,
	\begin{equation*}
		\norm{\ti^{h_{\delta}}_{t}-\ti^{h_{\delta}}_{s}}_{\mcH^{\gamma-2\kappa}}^{2}\lesssim\abs{t-s}^{2\kappa}\sum_{\om\in\mbZ^{2}}(1+\abs{2\uppi\om}^{2})^{\gamma}\norm{\hat{h_{\delta}}(\om)}^{2}_{L^{2}([0,T];\mbR^{2})}=\abs{t-s}^{2\kappa}\norm{h_{\delta}}^{2}_{L^{2}([0,T];\mcH^{\gamma}(\mbT^{2};\mbR^{2}))},
	\end{equation*}
	hence $\norm{\ti^{h_{\delta}}}_{C_{T}^{\kappa}\mcH^{\gamma-2\kappa}}<\infty$, which, upon choosing $\gamma\geq2\kappa$, implies $\ti^{h_{\delta}}\in C_{T}L^{2}(\mbT^{2})$. Another application of~\eqref{eq:lolli_h_increment_bound} allows us to deduce 
	\begin{equation*}
		\norm{\ti^{h}-\ti^{h_{\delta}}}_{C_{T}L^{2}}^{2}\lesssim\norm{h-h_{\delta}}_{L^{2}([0,T]\times\mbT^{2};\mbR^{2})}^{2}\to0\quad\text{as}~\delta\to0
	\end{equation*}
	and using that $C_{T}L^{2}(\mbT^{2})$ is complete we obtain $\ti^{h}\in C_{T}L^{2}(\mbT^{2})$, which yields the claim.
\end{proof}
The next lemma establishes the regularity of $\ti^{h}$ as a square-integrable function in time.
\begin{lemma}\label{lem:lolli_h_L2H1}
	Let $\rho_{0}\in C(\mbT^{2})$ be such that $\rho_{0}\geq0$ and $\mean{\rho_{0}}=1$, $\rdet$ be the weak solution to~\eqref{eq:Determ_KS} with initial data $\rho_{0}$ and chemotactic sensitivity $\chem\in\mbR$, and $\Trdet$ be its maximal time of existence (see Lemma~\ref{lem:Determ_KS_continuity}). Then for all $T<\Trdet$ and $h\in L^{2}([0,T]\times\mbT^{2};\mbR^{2})$ it holds that
	\begin{equation*}
		\norm{\ti^{h}}_{L^{2}_{T}\mcH^{1}}\lesssim\norm{\srdet}_{C([0,T]\times\mbT^{2})}\norm{h}_{L^{2}([0,T]\times\mbT^{2};\mbR^{2})}.
	\end{equation*}
	and in particular $\ti^{h}\in L^{2}_{T}\mcH^{1}(\mbT^{2})$.
\end{lemma}
\begin{proof}
	As in the proof of Lemma~\ref{lem:lolli_h}, we may assume that $\ti^{h}=\vdiv\mcI[h]$. Using~\eqref{eq:lolli_h_Fourier}, we decompose
	\begin{equation*}
		\norm{\ti^{h}}_{L^{2}_{T}\mcH^{1}}^{2}=\int_{0}^{T}\sum_{\om\in\mbZ^{2}\setminus\{0\}}(1+\abs{2\uppi\om}^{2})\Bigl\lvert\sum_{j=1}^{2}2\uppi\upi\om^{j}\int_{0}^{t}\euler^{-\abs{t-u}\abs{2\uppi\om}^{2}}\hat{h}(u,\om,j)\dd u\Bigr\rvert^{2}\dd t.
	\end{equation*}
	We then apply the Cauchy--Schwarz inequality followed by Young's convolution inequality to estimate
	\begin{equation*}
		\begin{split}
			\norm{\ti^{h}}_{L^{2}_{T}\mcH^{1}}^{2}&\leq\int_{0}^{T}\sum_{\om\in\mbZ^{2}\setminus\{0\}}(1+\abs{2\uppi\om}^{2})\abs{2\uppi\om}^{2}\sum_{j=1}^{2}\Bigl\lvert\int_{0}^{t}\euler^{-\abs{t-u}\abs{2\uppi\om}^{2}}\hat{h}(u,\om,j)\dd u\Bigr\rvert^{2}\dd t\\
			&\leq\sum_{j=1}^{2}\sum_{\om\in\mbZ^{2}\setminus\{0\}}(1+\abs{2\uppi\om}^{2})\abs{2\uppi\om}^{2}\Bigl(\int_{0}^{T}\euler^{-t\abs{2\uppi\om}^{2}}\dd t\Bigr)^{2}\Bigl(\int_{0}^{T}\abs{\hat{h}(t,\om,j)}^{2}\dd t\Bigr)\\
			&\lesssim\sum_{j=1}^{2}\sum_{\om\in\mbZ^{2}}\int_{0}^{T}\abs{\hat{h}(t,\om,j)}^{2}\dd t=\norm{h}_{L^{2}([0,T]\times\mbT^{2};\mbR^{2})}^{2},
		\end{split}
	\end{equation*}
	which yields the claim.
\end{proof}
Let $\mbX^{h}\defeq(\ti^{h},\ty^{h},\tp^{h},\tc^{h})$ be the enhancement driven by the Cameron--Martin element $h\in L^{2}([0,T]\times\mbT^{2};\mbR^{2})$, i.e.\
\begin{equation}\label{eq:enhancement_h_def}
	\begin{split}
		&\ti^{h}\defeq\vdiv\mcI[\srdet h],\quad\ty^{h}\defeq\vdiv\mcI[\ti^{h}\nabla\Phi_{\ti^{h}}],\\
		&\tp^{h}\defeq\ty^{h}\re\nabla\Phi_{\ti^{h}}+\nabla\Phi_{\ty^{h}}\re\ti^{h},\quad\tc^{h}\defeq\nabla\mcI[\ti^{h}]\re\nabla\Phi_{\ti^{h}}+\nabla^{2}\mcI[\Phi_{\ti^{h}}]\re\ti^{h}.
	\end{split}
\end{equation}
We can construct $\mbX^{h}$ by combining Lemma~\ref{lem:lolli_h}, Schauder's estimate~\cite[Lem.~A.6]{martini_mayorcas_25} and Bony's estimate.
\begin{lemma}\label{lem:existence_enhancement_h}
	Let $\rho_{0}\in C(\mbT^{2})$ be such that $\rho_{0}\geq0$ and $\mean{\rho_{0}}=1$, $\rdet$ be the weak solution to~\eqref{eq:Determ_KS} with initial data $\rho_{0}$ and chemotactic sensitivity $\chem\in\mbR$, and $\Trdet$ be its maximal time of existence (see Lemma~\ref{lem:Determ_KS_continuity}). Then for all $T<\Trdet$, $h\in L^{2}([0,T]\times\mbT^{2};\mbR^{2})$, $\alpha<-2$ and $\kappa\in(0,1/2)$, it holds that $\mbX^{h}=(\ti^{h},\ty^{h},\tp^{h},\tc^{h})\in\rksnoise{\alpha}{\kappa}_{T}$ is well-defined as a distribution and
	\begin{equation*}
		\norm{\mbX^{h}}_{\rksnoise{\alpha}{\kappa}_{T}}\lesssim\norm{\srdet}_{C_{T}L^{\infty}}\norm{h}_{L^{2}([0,T]\times\mbT^{2};\mbR^{2})}(1\vee\norm{\srdet}^{2}_{C_{T}L^{\infty}}\norm{h}^{2}_{L^{2}([0,T]\times\mbT^{2};\mbR^{2})}).
	\end{equation*}
\end{lemma}
\begin{proof}
	Let $\vartheta\in(0,1/2)$ and $\kappa\in(0,1/2)$ be such that $3\vartheta<1-2\kappa$. We bound the $\msL_{T}^{\kappa}\mcC^{-1-\vartheta}(\mbT^{2})$-norm of $\ti^{h}$ by the Besov embedding~\cite[Lem.~A.2]{gubinelli_15_GIP} and Lemma~\ref{lem:lolli_h},
	\begin{equation*}
		\norm{\ti^{h}}_{\msL_{T}^{\kappa}\mcC^{-1-\vartheta}}\lesssim\norm{\ti^{h}}_{\msL_{T}^{\kappa}\mcB_{2,2}^{-\vartheta}}\lesssim\norm{\ti^{h}}_{\msL_{T}^{\kappa}\mcB_{2,2}^{0}}\lesssim\norm{\srdet}_{C([0,T]\times\mbT^{2})}\norm{h}_{L^{2}([0,T]\times\mbT^{2};\mbR^{2})}.
	\end{equation*}
	
	We bound the $\msL^{\kappa}_{T}\mcC^{-2\vartheta}(\mbT^{2})$-norm of $\ty^{h}$ by Schauder's estimate~\cite[Lem.~A.6]{martini_mayorcas_25}, the Besov embedding and Bony's estimate,
	\begin{details}
		(Lemma~\ref{lem:product_estimates}, \eqref{eq:product_estimate_Sobolev_V})
	\end{details}
	\begin{equation*}
		\begin{split}
			\norm{\ty^{h}}_{\msL^{\kappa}_{T}\mcC^{-2\vartheta}}&\lesssim_{T}\norm{\ti^{h}\nabla\Phi_{\ti^{h}}}_{C_{T}\mcC^{-1-2\vartheta}}\lesssim\norm{\ti^{h}\nabla\Phi_{\ti^{h}}}_{C_{T}\mcB_{2,2}^{-2\vartheta}}\lesssim\norm{\ti^{h}}_{C_{T}\mcB_{2,2}^{-\vartheta}}\norm{\nabla\Phi_{\ti^{h}}}_{C_{T}\mcB_{2,2}^{1-\vartheta}}.
		\end{split}
	\end{equation*}
	
	We bound the $\msL^{\kappa}_{T}\mcC^{-3\vartheta}(\mbT^{2})$-norm of $\tp^{h}$ by the Besov embedding and Bony's resonant product estimate,
	\begin{details}
		(cf.~\cite[Thm.~27.10]{vanzuijlen_22})
	\end{details}
	\begin{equation*}
		\norm{\tp^{h}}_{\msL^{\kappa}_{T}\mcC^{-3\vartheta}}\lesssim\norm{\tp^{h}}_{\msL^{\kappa}_{T}\mcB_{2,2}^{1-3\vartheta}}\lesssim\norm{\ty^{h}}_{\msL^{\kappa}_{T}\mcC^{-2\vartheta}}\norm{\nabla\Phi_{\ti^{h}}}_{\msL^{\kappa}_{T}\mcB_{2,2}^{1-\vartheta}}+\norm{\nabla\Phi_{\ty^{h}}}_{\msL^{\kappa}_{T}\mcC^{1-2\vartheta}}\norm{\ti^{h}}_{\msL^{\kappa}_{T}\mcB_{2,2}^{-\vartheta}}.
	\end{equation*}
	
	We bound the $\msL^{\kappa}_{T}\mcC^{-2\vartheta}(\mbT^{2})$-norm of $\tc^{h}$ by the Besov embedding, Bony's resonant product estimate and Schauder's estimate, 
	\begin{equation*}
		\begin{split}
			\norm{\tc^{h}}_{\msL^{\kappa}_{T}\mcC^{-2\vartheta}}\lesssim\norm{\tc^{h}}_{\msL^{\kappa}_{T}\mcB_{2,2}^{1-2\vartheta}}&\lesssim\norm{\nabla\mcI[\ti^{h}]}_{\msL^{\kappa}_{T}\mcC^{-\vartheta}}\norm{\nabla\Phi_{\ti^{h}}}_{\msL^{\kappa}_{T}\mcB_{2,2}^{1-\vartheta}}+\norm{\nabla^{2}\mcI[\Phi_{\ti^{h}}]}_{\msL^{\kappa}_{T}\mcC^{1-\vartheta}}\norm{\ti^{h}}_{\msL^{\kappa}_{T}\mcB_{2,2}^{-\vartheta}}\\
			&\lesssim_{T}\norm{\ti^{h}}_{C_{T}\mcC^{-1-\vartheta}}\norm{\ti^{h}}_{\msL^{\kappa}_{T}\mcB_{2,2}^{-\vartheta}}.
		\end{split}
	\end{equation*}
	
	We can now choose $\alpha=-2-\vartheta$ to obtain
	\begin{equation*}
		\norm{\mbX^{h}}_{\rksnoise{\alpha}{\kappa}_{T}}\lesssim\norm{\srdet}_{C_{T}L^{\infty}}\norm{h}_{L^{2}([0,T]\times\mbT^{2};\mbR^{2})}(1\vee\norm{\srdet}^{2}_{C_{T}L^{\infty}}\norm{h}^{2}_{L^{2}([0,T]\times\mbT^{2};\mbR^{2})}).
	\end{equation*}
	\begin{details}
		The restriction $3\vartheta<1-2\kappa$ induces the lower bound $-7+2\kappa<3\alpha$, which can be removed by using that $\norm{\mbX^{h}}_{\rksnoise{\alpha'}{\kappa}_{T}}\leq\norm{\mbX^{h}}_{\rksnoise{\alpha}{\kappa}_{T}}$ for all $\alpha'<\alpha$.
		
	\end{details}
	To show that $\mbX^{h}\in\rksnoise{\alpha}{\kappa}_{T}$ it suffices to construct a sequence of smooth approximations to $\ti^{h}$, see e.g.\ the proof of Lemma~\ref{lem:lolli_h}.
	\begin{details}
		\paragraph{Proof that $\mbX^{h}\in\rksnoise{\alpha}{\kappa}_{T}$.}
		It follows by Lemma~\ref{lem:regularity_srdet} that upon redefining $h$, we may assume that $\ti^{h}=\vdiv\mcI[h]$. Let $(\ti^{h_{\delta}})_{\delta>0}$ be as in the proof of Lemma~\ref{lem:lolli_h}, where we have shown that $\norm{\ti^{h_{\delta}}}_{\msL_{T}^{\kappa}\mcH^{\gamma}}<\infty$ for every $\delta>0$ and that $\norm{\ti^{h}-\ti^{h_{\delta}}}_{\msL_{T}^{\kappa}\mcB_{2,2}^{0}}$ as $\delta\to0$, for every $\kappa\in[0,1/2)$ and $\gamma\geq0$.
		
		Let $0<\kappa<\kappa'<1/2$ and $\gamma=1+2(\kappa'-\kappa)$, we obtain $\norm{\ti^{h_{\delta}}}_{\msL_{T}^{\kappa'}\mcH^{1+2(\kappa'-\kappa)}}<\infty$ and an application of Lemma~\ref{lem:little_Hoelder_criterion} combined with Besov's embedding yields $\ti^{h_{\delta}}\in\msL_{T}^{\kappa}\mcH^{1}(\mbT^{2})\embed\msL_{T}^{\kappa}\mcC^{0}(\mbT^{2})$.
		
		Let $\mbX^{h_{\delta}}$ be the associated enhancement. Using that $\alpha<-2$ and $\kappa<1/2$, we can find some $\alpha<\alpha'<-2$ and $0<\kappa<\kappa'<1/2$ such that $\alpha-2\kappa<\alpha'-2\kappa'$. The same derivation as above then implies $\norm{\mbX^{h_{\delta}}}_{\rksnoise{\alpha'}{\kappa'}_{T}}<\infty$, which combined with several applications of Lemma~\ref{lem:little_Hoelder_criterion} yields
		\begin{equation*}
			\mbX^{h_{\delta}}\in\msL^{\kappa}_{T}\mcC^{\alpha+1}(\mbT^{2};\mbR)\times\msL^{\kappa}_{T}\mcC^{2\alpha+4}(\mbT^{2};\mbR)\times\msL^{\kappa}_{T}\mcC^{3\alpha+6}(\mbT^{2};\mbR^{2})\times\msL^{\kappa}_{T}\mcC^{2\alpha+4}(\mbT^{2};\mbR^{4}).
		\end{equation*}
		Furthermore, the bounds above are continuous in $\norm{\ti^{h_{\delta}}}_{\msL_{T}^{\kappa}\mcB_{2,2}^{0}}$, which implies $\mbX^{h_{\delta}}\to\mbX^{h}$ as $\delta\to0$. This shows that $\mbX^{h}\in\rksnoise{\alpha}{\kappa}_{T}$.
	\end{details}
\end{proof}
\section{Regular Solutions to the Stochastic Heat Equation}\label{app:lolli_regular}
In this appendix we analyse the regularity of the mild solution $\ti^{\delta}=\vdiv\mcI[\srdet\boldsymbol{\xi}^{\delta}]$, where $\delta>0$ is the correlation length and $\boldsymbol{\xi}^{\delta}=\psi_{\delta}\ast\boldsymbol{\xi}$ denotes the mollification (in space) of a vector-valued space-time white noise $\boldsymbol{\xi}$ against a mollifier $\psi_{\delta}$ given by~\eqref{eq:def_mollifiers}. Our goal is to establish that $\ti^{\delta}\in C_{T}L^{2}(\mbT^{2})\cap L_{T}^{2}\mcH^{1}(\mbT^{2})$, if the initial data of $\rdet$ is assumed to be continuous, positive and of unit mass (Proposition~\ref{prop:lolli_regular_srdet}).

To show Proposition~\ref{prop:lolli_regular_srdet}, we first need to prove generic bounds on $\ti^{\delta}$ that depend on $\srdet$ only through its existence and regularity (cf.\ Lemmas~\ref{lem:lolli_regular}, \ref{lem:lolli_regular_L2Hgamma} and~\ref{lem:lolli_regular_unif}). Therefore, to keep the notation concise, we denote $\het=\sqrt{\rdet}$ and treat $\het$ as a generic function of a given regularity until further notice.

Let $\gamma\geq0$, $\gamma'\in[-1,\gamma)$, $\kappa\in[0,1/2)$ and $\eta\in[0,1/2)$ be such that $2\eta<\gamma-\gamma'$ and $2\eta<1-2\kappa$. Below we establish $\ti^{\delta}\in\msL^{\kappa}_{T}\mcH^{\gamma'}(\mbT^{2})$ if $\het\in C_{\eta;T}\mcH^{\gamma}(\mbT^{2})$ (Lemma~\ref{lem:lolli_regular}), $\ti^{\delta}\in L^{2}_{T}\mcH^{\gamma}(\mbT^{2})$ if $\het\in L^{2}_{T}\mcH^{\gamma}(\mbT^{2})$ (Lemma~\ref{lem:lolli_regular_L2Hgamma}) and $\ti^{\delta}\in L_{T}^{\infty}\mcH^{\gamma'}(\mbT^{2})$ if $\het\in C_{T}\mcH^{(\gamma'\vee0)}(\mbT^{2})\cap L_{T}^{2}\mcH^{\gamma'+1}(\mbT^{2})$ (Lemma~\ref{lem:lolli_regular_unif}), each almost surely. We then use all of these results to establish Proposition~\ref{prop:lolli_regular_srdet}.

The Fourier transform of a vector-valued space-time white noise is given by a family complex-valued Brownian motions $W^{j}(t,m)\defeq\boldsymbol{\xi}(\mathds{1}_{[0,t]}(\place)\otimes\euler^{-2\uppi\upi\inner{m}{\place}}\otimes\delta_{j,\place})$, where $t\geq0$, $m\in\mbZ^{2}$ and $j=1,2$ (see Lemma~\ref{lem:space_time_white_noise_yields_complex_BM}). Therefore, we can represent the Fourier transform of $\ti^{\delta}$ by
\begin{equation}\label{eq:lolli_Fourier}
	\hat{\ti^{\delta}}(t,\om)=\sum_{j=1}^{2}\sum_{m\in\mbZ^{2}}\int_{0}^{t}H^{j}_{t-u}(\om)\hat{\het}(u,\om-m)\varphi(\delta m)\dd W^{j}(u,m),\qquad t\in[0,T],\quad\om\in\mbZ^{2}
\end{equation}
see also~\cite[(2.4)]{martini_mayorcas_25}. 
\begin{details}
	We represent the integrand as
	\begin{equation*}
		H^{j}_{t-u}(\om)\hat{\het}(u,\om-m)=\int_{\mbT^{2}}\euler^{-2\uppi\upi\inner{\om}{x}}\partial_{j}P_{t-u}(\het(u,\place)\euler^{2\uppi\upi\inner{m}{\place}})(x)\dd x.
	\end{equation*}
\end{details}
Hence taking the Fourier inverse, we can deduce the real-space representation
\begin{equation}\label{eq:lolli_real}
	\ti^{\delta}(t,x)=\sum_{j=1}^{2}\sum_{m\in\mbZ^{2}}\varphi(\delta m)\int_{0}^{t}\partial_{j}P_{t-u}(\het(u,\place)\euler^{2\uppi\upi\inner{m}{\place}})(x)\dd W^{j}(u,m),\qquad x\in\mbT^{2}.
\end{equation}
Recall from~\eqref{eq:def_cut_off} that $\varphi$ is assumed to be of compact support. Hence, the number of Brownian drivers contributing to~\eqref{eq:lolli_real} is finite and tends to infinity as $\delta\to0$.

The following lemma establishes the regularity of $\ti^{\delta}$ in an interpolation space (see Subsection~\ref{subsec:notation}).
\begin{lemma}\label{lem:lolli_regular}
	Let $T>0$, $\gamma\geq0$, $\gamma'\in[-1,\gamma)$, $\kappa\in[0,1/2)$ and $\eta\in[0,1/2)$ be such that $2\eta<\gamma-\gamma'$ and $2\eta<1-2\kappa$. Then for all $\het\in C_{\eta;T}\mcH^{\gamma}(\mbT^{2})$, $\delta>0$ and $p\in[1,\infty)$ it holds that
	\begin{equation}\label{eq:lolli_regular_bound}
		\mbE[\norm{\ti^{\delta}}_{\msL_{T}^{\kappa}\mcH^{\gamma'}}^{p}]^{1/p}\lesssim_{p,T}(1+\delta^{-\gamma-1})\norm{\het}_{C_{\eta;T}\mcH^{\gamma}}
	\end{equation}
	and $\ti^{\delta}\in\msL^{\kappa}_{T}\mcH^{\gamma'}(\mbT^{2})$ almost surely.
\end{lemma}
\begin{proof}
	Using the real-space representation~\eqref{eq:lolli_real}, we first establish for $\gamma\geq0$, $\gamma'\in[-1,\gamma)$ and $\eta\in[0,1/2)$ satisfying $2\eta\leq\gamma-\gamma'$, that 
	\begin{equation}\label{eq:lolli_regular_pntw}
		\sup_{t\in[0,T]}\mbE[\norm{\ti^{\delta}_{t}}_{\mcH^{\gamma'}}^{2}]^{1/2}\lesssim_{T}(1+\delta^{-\gamma-1})\norm{\het}_{C_{\eta;T}\mcH^{\gamma}}
	\end{equation}
	and in particular $\ti^{\delta}_{t}\in\mcH^{\gamma'}(\mbT^{2})$ almost surely for each $t\in[0,T]$. 
	
	Let $(e_{k})_{k\in\mbN}$ be an orthonormal basis of $L^2(\mbT^{2})$. Using the definition of $\mcH^{\gamma'}(\mbT^{2})$ and Parseval's identity, we obtain for every $t\in[0,T]$, 
	\begin{equation*}
		\begin{split}
			\norm{\ti^{\delta}_{t}}_{\mcH^{\gamma'}}^{2}&=\norm{(1-\Delta)^{\gamma'/2}\ti^{\delta}_{t}}_{L^{2}}^{2}\\
			&=\sum_{k=1}^{\infty}\Bigl\lvert\Bigl\langle e_k,\sum_{j=1}^{2}\sum_{m\in\mbZ^{2}}\varphi(\delta m)\int_{0}^{t}(1-\Delta)^{\gamma'/2}\partial_{j}P_{t-u}(\het(u,\place)\euler^{2\uppi\upi\inner{m}{\place}})\dd W^{j}(u,m)\Bigr\rangle_{L^{2}}\Bigr\rvert^{2}.
		\end{split}
	\end{equation*}
	An application of the stochastic Fubini theorem~\cite{veraar_12} combined with It\^{o}'s isometry yields
	\begin{equation}\label{eq:lolli_regular_Ito_isometry}
		\begin{split}
			\mbE[\norm{\ti^{\delta}_{t}}_{\mcH^{\gamma'}}^{2}]&=\sum_{k=1}^{\infty}\sum_{j=1}^{2}\sum_{m\in\mbZ^{2}}\varphi(\delta m)^{2}\int_{0}^{t}\abs{\inner{e_{k}}{(1-\Delta)^{\gamma'/2}\partial_{j}P_{t-u}(\het(u,\place)\euler^{2\uppi\upi\inner{m}{\place}})}_{L^{2}}}^{2}\dd u\\
			&=\sum_{j=1}^{2}\sum_{m\in\mbZ^{2}}\varphi(\delta m)^{2}\int_{0}^{t}\norm{(1-\Delta)^{\gamma'/2}\partial_{j}P_{t-u}(\het(u,\place)\euler^{2\uppi\upi\inner{m}{\place}})}_{L^{2}}^{2}\dd u.
		\end{split}
	\end{equation}
	\begin{details}
		\paragraph{Application of the stochastic Fubini theorem.}
		Our goal is to apply~\cite[(1.2)]{veraar_12} to show
		\begin{equation*}
			\begin{split}
				&\int_{\mbT^{2}}e_{k}(x)\int_{0}^{t}(1-\Delta)^{\gamma'/2}\partial_{j}P_{t-u}(\het(u,\place)\euler^{2\uppi\upi\inner{m}{\place}})(x)\dd W^{j}(u,m)\dd x\\
				&=\int_{0}^{t}\int_{\mbT^{2}}e_{k}(x)(1-\Delta)^{\gamma'/2}\partial_{j}P_{t-u}(\het(u,\place)\euler^{2\uppi\upi\inner{m}{\place}})(x)\dd x\dd W^{j}(u,m).
			\end{split}
		\end{equation*}
		The function $(u,x)\mapsto e_{k}(x)(1-\Delta)^{\gamma'/2}\partial_{j}P_{t-u}(\het(u,\place)\euler^{2\uppi\upi\inner{m}{\place}})(x)$ is jointly measurable as a product and integral over measurable functions. Hence, it suffices to verify condition~\cite[(1.5)]{veraar_12}, i.e.\
		\begin{equation*}
			\int_{\mbT^{2}}\Bigl(\int_{0}^{t}\abs{e_{k}(x)(1-\Delta)^{\gamma'/2}\partial_{j}P_{t-u}(\het(u,\place)\euler^{2\uppi\upi\inner{m}{\place}})(x)}^{2}\dd u\Bigr)^{1/2}\dd x<\infty.
		\end{equation*}
		Indeed, we can apply the Cauchy--Schwarz inequality and Tonelli's theorem to show
		\begin{equation*}
			\begin{split}
				&\int_{\mbT^{2}}\Bigl(\int_{0}^{t}\abs{e_{k}(x)(1-\Delta)^{\gamma'/2}\partial_{j}P_{t-u}(\het(u,\place)\euler^{2\uppi\upi\inner{m}{\place}})(x)}^{2}\dd u\Bigr)^{1/2}\dd x\\
				&=\int_{\mbT^{2}}\abs{e_{k}(x)}\Bigl(\int_{0}^{t}\abs{(1-\Delta)^{\gamma'/2}\partial_{j}P_{t-u}(\het(u,\place)\euler^{2\uppi\upi\inner{m}{\place}})(x)}^{2}\dd u\Bigr)^{1/2}\dd x\\
				&\leq\Bigl(\int_{\mbT^{2}}\abs{e_{k}(x)}^{2}\dd x\Bigr)^{1/2}\Bigl(\int_{\mbT^{2}}\int_{0}^{t}\abs{(1-\Delta)^{\gamma'/2}\partial_{j}P_{t-u}(\het(u,\place)\euler^{2\uppi\upi\inner{m}{\place}})(x)}^{2}\dd u\dd x\Bigr)^{1/2}\\
				&=\Bigl(\int_{0}^{t}\norm{(1-\Delta)^{\gamma'/2}\partial_{j}P_{t-u}(\het(u,\place)\euler^{2\uppi\upi\inner{m}{\place}})}_{L^{2}}^{2}\dd u\Bigr)^{1/2}
			\end{split}
		\end{equation*}
		and we will see below that this expression is finite.
	\end{details}
	
	We now distinguish the cases $\gamma'\in[-1,\gamma-1)$ and $\gamma'\in[\gamma-1,\gamma)$. First let us assume that $\gamma'\in[\gamma-1,\gamma)$. We can bound the integrand by an application of Schauder's estimate~\cite[Lem.~A.5]{martini_mayorcas_25},
	\begin{equation*}
		\begin{split}
			\norm{(1-\Delta)^{\gamma'/2}\partial_{j}P_{t-u}(\het(u,\place)\euler^{2\uppi\upi\inner{m}{\place}})}_{L^2}&\lesssim\norm{P_{t-u}(\het(u,\place)\euler^{2\uppi\upi\inner{m}{\place}})}_{\mcH^{\gamma'+1}}\\
			&\lesssim(1\vee\abs{t-u}^{-\frac{1+\gamma'-\gamma}{2}})\norm{\sigma(u,\place)\euler^{2\uppi\upi\inner{m}{\place}}}_{\mcH^{\gamma}},
		\end{split}
	\end{equation*}
	where we used that $\gamma\leq\gamma'+1$. We bound by a change of variables and Jensen's inequality uniformly in $u\in[0,T]$,
	\begin{equation}\label{eq:het_fourier_product_bound}
		\begin{split}
			\norm{\sigma(u,\place)\euler^{2\uppi\upi\inner{m}{\place}}}_{\mcH^{\gamma}}^{2}&=\sum_{\om\in\mbZ^{2}}(1+\abs{2\uppi\om}^{2})^{\gamma}\abs{\hat{\het}(u,\om-m)}^{2}\\
			&=\sum_{\om\in\mbZ^{2}}(1+\abs{2\uppi(\om+m)}^{2})^{\gamma}\abs{\hat{\het}(u,\om)}^{2}\\
			&\lesssim\sum_{\om\in\mbZ^{2}}(1+\abs{2\uppi\om}^{2}+\abs{2\uppi m}^{2})^{\gamma}\abs{\hat{\het}(u,\om)}^{2}\\
			&\lesssim\norm{\het(u,\place)}_{\mcH^{\gamma}}^{2}+(1+\abs{2\uppi m}^{2})^{\gamma}\norm{\het(u,\place)}_{L^{2}}^{2}\\
			&\lesssim(1+\abs{2\uppi m}^{2})^{\gamma}\norm{\het(u,\place)}_{\mcH^{\gamma}}^{2},
		\end{split}
	\end{equation}
	where the second inequality follows for $\gamma\in[0,1]$ by the sub-additivity of $x\mapsto x^{\gamma}$
	\begin{details}
		(The function is sub-additive since it is concave and satisfies $0^{\gamma}=0$.)
	\end{details}
	and for $\gamma>1$ by another application of Jensen's inequality.
	
	Consequently if $\gamma'\in[\gamma-1,\gamma)$ we may estimate uniformly in $t\in[0,T]$,
	\begin{equation}\label{eq:lolli_regular_pntw_case_I}
		\begin{split}
			\mbE[\norm{\ti^{\delta}_{t}}_{\mcH^{\gamma'}}^{2}]&\lesssim\norm{\het}_{C_{\eta;T}\mcH^{\gamma}}^{2}\sum_{m\in\mbZ^{2}}\varphi(\delta m)^{2}(1+\abs{2\uppi m}^{2})^{\gamma}\int_{0}^{t}(1\vee\abs{t-u}^{-\frac{1+\gamma'-\gamma}{2}})^{2}(1\wedge u)^{-2\eta}\dd u\\
			&\lesssim_{T}(1+\delta^{-1})^{2(\gamma+1)}\norm{\het}_{C_{\eta;T}\mcH^{\gamma}}^{2},
		\end{split}
	\end{equation}
	where we used $\gamma'<\gamma$ and $2\eta<1$ to integrate the singularity; $2\eta\leq\gamma-\gamma'$ to obtain a bound that is uniform in $t\in[0,T]$; 
	\begin{details}
		\paragraph{Integrating the singularity.}
		We distinguish the cases $t\in[0,1]$, $t\in(1,2]$ and $t\in(2,\infty)$. If $t\in[0,1]$, we estimate
		\begin{equation*}
			\begin{split}
				\int_{0}^{t}(1\vee\abs{t-u}^{-(1+\gamma'-\gamma)})(1\wedge u)^{-2\eta}\dd u&=\int_{0}^{t}\abs{t-u}^{-(1+\gamma'-\gamma)}u^{-2\eta}\dd u\\
				&=t^{\gamma-\gamma'-2\eta}\int_{0}^{1}\abs{1-u}^{-(1+\gamma'-\gamma)}u^{-2\eta}\dd u\lesssim t^{\gamma-\gamma'-2\eta}.
			\end{split}
		\end{equation*}
		If $t\in(1,2]$, we estimate
		\begin{equation*}
			\begin{split}
				\int_{0}^{t}(1\vee\abs{t-u}^{-(1+\gamma'-\gamma)})(1\wedge u)^{-2\eta}\dd u&=\int_{0}^{t-1}u^{-2\eta}\dd u+\int_{t-1}^{1}\abs{t-u}^{-(1+\gamma'-\gamma)}u^{-2\eta}\dd u\\
				&\quad+\int_{1}^{t}\abs{t-u}^{-(1+\gamma'-\gamma)}\dd u\\
				&\leq t^{1-2\eta}+t^{\gamma'-\gamma-2\eta}+t^{\gamma'-\gamma}.
			\end{split}
		\end{equation*}
		If $t\in(2,\infty)$, we estimate
		\begin{equation*}
			\begin{split}
				\int_{0}^{t}(1\vee\abs{t-u}^{-(1+\gamma'-\gamma)})(1\wedge u)^{-2\eta}\dd u&=\int_{0}^{1}u^{-2\eta}\dd u+\int_{1}^{t-1}\dd u+\int_{t-1}^{t}\abs{t-u}^{-(1+\gamma'-\gamma)}\dd u\\
				&\leq t^{1-2\eta}+t+t^{\gamma'-\gamma-2\eta}.
			\end{split}
		\end{equation*}
	\end{details}
	and $\supp(\varphi)\subset B(0,1)$ to estimate the sum over $m\in\mbZ^{2}$. 
	\begin{details}
		We use~\cite[Lem.~C.2, (C.1)]{martini_mayorcas_25} to bound
		\begin{equation*}
			\sum_{m\in\mbZ^{2}}\varphi(\delta m)^{2}(1+\abs{2\uppi m}^{2})^{\gamma}\lesssim(1+\delta^{-2})^{\gamma}\sum_{m\in\mbZ^{2}}\varphi(\delta m)\lesssim(1+\delta^{-2})^{\gamma}(1\vee\delta^{-2})\lesssim(1+\delta^{-2})^{\gamma+1}\leq(1+\delta^{-1})^{2(\gamma+1)}.
		\end{equation*}
	\end{details}
	
	For $\gamma'\in[-1,\gamma-1)$, we estimate instead
	\begin{equation*}
		\mbE[\norm{\ti^{\delta}_{t}}_{\mcH^{\gamma'}}^{2}]\lesssim_{T}(1+\delta^{-1})^{2(\gamma'+2)}\norm{\het}_{C_{\eta;T}\mcH^{\gamma'+1}}^{2}\lesssim(1+\delta^{-1})^{2(\gamma+1)}\norm{\het}_{C_{\eta;T}\mcH^{\gamma}}^{2},
	\end{equation*}
	where in the first inequality we used that $\gamma'\in[\gamma',\gamma'+1)$ and~\eqref{eq:lolli_regular_pntw_case_I} with $\gamma=\gamma'+1\geq0$.
	
	To summarise, it follows for all $\gamma\geq0$, $\gamma'\in[-1,\gamma)$ and $\eta\in[0,1/2)$ satisfying $2\eta\leq\gamma-\gamma'$, that
	\begin{equation*}
		\sup_{t\in[0,T]}\mbE[\norm{\ti^{\delta}_{t}}_{C_{T}\mcH^{\gamma'}}^{2}]\lesssim_{T}(1+\delta^{-1})^{2(\gamma+1)}\norm{\het}_{C_{\eta;T}\mcH^{\gamma}}^{2}.
	\end{equation*}
	which yields~\eqref{eq:lolli_regular_pntw} and that $\ti^{\delta}_{t}\in\mcH^{\gamma'}(\mbT^{2})$ almost surely for all $t\in[0,T]$.
	
	Next by an application of Kolmogorov's continuity criterion we establish~\eqref{eq:lolli_regular_bound} and the parabolic regularity $\ti^{\delta}\in\msL_{T}^{\kappa}\mcH^{\gamma'}(\mbT^{2})$ for arbitrary $\gamma\geq0$, $\gamma'\in[-1,\gamma)$, $\kappa\in[0,1/2)$ and $\eta\in[0,1/2)$ satisfying $2\eta<\gamma-\gamma'$ and $2\eta<1-2\kappa$. 
	
	Let $\kappa'\in(\kappa,1/2]$ be such that $2\eta\leq1-2\kappa'$.
	\begin{details}
		(which is possible since $2\eta<1-2\kappa$.)
	\end{details}
	Let $s<t\in[0,T]$; we distinguish the cases $\abs{t-s}\leq1$ and $\abs{t-s}>1$ and first assume $\abs{t-s}\leq1$. Using the representation
	\begin{details}
		(this is the Markov property, see also~\cite[(5.8)]{hairer_09}),
	\end{details}
	\begin{equation*}
		\ti^{\delta}_{t}=P_{t-s}\ti^{\delta}_{s}+\sum_{j=1}^{2}\sum_{m\in\mbZ^{2}}\varphi(\delta m)\int_{s}^{t}\partial_{j}P_{t-u}(\het(u,\place)\euler^{2\uppi\upi\inner{m}{\place}})\dd W^{j}(u,m)
	\end{equation*}
	combined with the fact that increments of $(W^{j}(\place,m))_{m\in\mbZ^{2},j\in\{1,2\}}$ over $[s,t]$ are independent of $\ti^{\delta}_{s}$, we obtain
	\begin{equation}\label{eq:lolli_regular_increment_decomposition}
		\begin{split}
			\mbE[\norm{\ti^{\delta}_{t}-\ti^{\delta}_{s}}_{\mcH^{\gamma'-2\kappa'}}^{2}]&=\mbE[\norm{P_{t-s}\ti^{\delta}_{s}-\ti^{\delta}_{s}}_{\mcH^{\gamma'-2\kappa'}}^{2}]\\
			&\quad+\sum_{j=1}^{2}\sum_{m\in\mbZ^{2}}\varphi(\delta m)^{2}\int_{s}^{t}\norm{\partial_{j}P_{t-u}(\het(u,\place)\euler^{2\uppi\upi\inner{m}{\place}}) }_{\mcH^{\gamma'-2\kappa'}}^{2}\dd u.
		\end{split}
	\end{equation}
	\begin{details}
		\paragraph{Proof of~\eqref{eq:lolli_regular_increment_decomposition}.}
		We decompose the increment into
		\begin{equation*}
			\ti^{\delta}_{t}-\ti^{\delta}_{s}=(\ti^{\delta}_{t}-P_{t-s}\ti^{\delta}_{s})+(P_{t-s}-1)\ti^{\delta}_{s}.
		\end{equation*}
		The first summand corresponds to the stochastic integral over $[s,t]$,
		\begin{equation*}
			\ti^{\delta}_{t}-P_{t-s}\ti^{\delta}_{s}=\sum_{j=1}^{2}\sum_{m\in\mbZ^{2}}\varphi(\delta m)\int_{s}^{t}\partial_{j}P_{t-u}(\het(u,\place)\euler^{2\uppi\upi\inner{m}{\place}})\dd W^{j}(u,m),
		\end{equation*}
		whereas the second summand is given by a stochastic integral over $[0,s]$.
		
		By Parseval's identity and the independence of our stochastic integrals,
		\begin{equation*}
			\begin{split}
				&\mbE[\norm{\ti^{\delta}_{t}-\ti^{\delta}_{s}}^{2}_{\mcH^{\gamma'-2\kappa'}}]=\mbE\Bigl[\sum_{k\in\mbN}\abs{\inner{e_{k}}{(1-\Delta)^{(\gamma'-2\kappa')/2}(\ti^{\delta}_{t}-\ti^{\delta}_{s})}_{L^{2}}}^{2}\Bigr]\\
				&=\sum_{k\in\mbN}\mbE[\abs{\inner{e_{k}}{(1-\Delta)^{(\gamma'-2\kappa')/2}(\ti^{\delta}_{t}-P_{t-s}\ti^{\delta}_{s}+(P_{t-s}-1)\ti^{\delta}_{s})}_{L^{2}}}^{2}]\\
				&=\sum_{k\in\mbN}\mbE[\abs{\inner{e_{k}}{(1-\Delta)^{(\gamma'-2\kappa')/2}(\ti^{\delta}_{t}-P_{t-s}\ti^{\delta}_{s})}_{L^{2}}}^{2}]+\sum_{k\in\mbN}\mbE[\abs{\inner{e_{k}}{(1-\Delta)^{(\gamma'-2\kappa')/2}(P_{t-s}-1)\ti^{\delta}_{s}}_{L^{2}}}^{2}]\\
				&=\mbE[\norm{\ti^{\delta}_{t}-P_{t-s}\ti^{\delta}_{s}}^{2}_{\mcH^{\gamma'-2\kappa'}}]+\mbE[\norm{P_{t-s}\ti^{\delta}_{s}-\ti^{\delta}_{s}}^{2}_{\mcH^{\gamma'-2\kappa'}}].
			\end{split}
		\end{equation*}
		The decomposition~\eqref{eq:lolli_regular_increment_decomposition} then follows by an application of Parseval's identity, the stochastic Fubini theorem and the It\^{o} isometry as in~\eqref{eq:lolli_regular_Ito_isometry}.
	\end{details}
	
	We bound the first summand in~\eqref{eq:lolli_regular_increment_decomposition} by Schauder's estimate~\cite[Lem.~A.5]{martini_mayorcas_25},
	\begin{equation*}
		\norm{P_{t-s}\ti^{\delta}_{s}-\ti^{\delta}_{s}}_{\mcH^{\gamma'-2\kappa'}}\lesssim\abs{t-s}^{\kappa'}\norm{\ti^{\delta}_{s}}_{\mcH^{\gamma'}},
	\end{equation*}
	where we used that $\gamma'-2\kappa'\leq\gamma'\leq\gamma'-2\kappa'+2$. Therefore by~\eqref{eq:lolli_regular_pntw}, we obtain uniformly in $s<t\in[0,T]$,
	\begin{equation}\label{eq:lolli_regular_increment_bound_I}
		\mbE[\norm{P_{t-s}\ti^{\delta}_{s}-\ti^{\delta}_{s}}_{\mcH^{\gamma'-2\kappa'}}^{2}]\lesssim\abs{t-s}^{2\kappa'}\mbE[\norm{\ti^{\delta}_{s}}_{\mcH^{\gamma'}}^{2}]\lesssim_{T}\abs{t-s}^{2\kappa'}(1+\delta^{-\gamma-1})^{2}\norm{\het}_{C_{\eta;T}\mcH^{\gamma}}^{2}.
	\end{equation}
	
	To control the second summand in~\eqref{eq:lolli_regular_increment_decomposition}, we further distinguish the cases $2\kappa'\leq1+\gamma'-\gamma$ and $2\kappa'>1+\gamma'-\gamma$. Assume $2\kappa'\leq1+\gamma'-\gamma$, we estimate the integrand as in~\eqref{eq:lolli_regular_pntw} and obtain the singularity
	\begin{equation*}
		\begin{split}
			&\sum_{j=1}^{2}\sum_{m\in\mbZ^{2}}\varphi(\delta m)^{2}\int_{s}^{t}\norm{\partial_{j}P_{t-u}(\het(u,\place)\euler^{2\uppi\upi\inner{m}{\place}}) }_{\mcH^{\gamma'-2\kappa'}}^{2}\dd u\\
			&\lesssim(1+\delta^{-1})^{2(\gamma+1)}\norm{\het}_{C_{\eta;T}\mcH^{\gamma}}^{2}\int_{s}^{t}(1\vee\abs{t-u}^{-\frac{1+\gamma'-\gamma-2\kappa'}{2}})^{2}(1\wedge u)^{-2\eta}\dd u.
		\end{split}
	\end{equation*}
	Using that $\abs{t-s}\leq1$, we may bound
	\begin{equation*}
		\int_{s}^{t}(1\vee\abs{t-u}^{-\frac{1+\gamma'-\gamma-2\kappa'}{2}})^{2}(1\wedge u)^{-2\eta}\dd u=\int_{s}^{t}\abs{t-u}^{-(1+\gamma'-\gamma-2\kappa')}(1\wedge u)^{-2\eta}\dd u\lesssim_{T}\abs{t-s}^{2\kappa'},
	\end{equation*}
	where the singularity is integrable by $\gamma'<\gamma$ and $2\eta<1$; and the bound is uniform since $2\eta\leq\gamma-\gamma'$. Therefore if $2\kappa'\leq1+\gamma'-\gamma$, we can control the second summand in~\eqref{eq:lolli_regular_increment_decomposition} by
	\begin{equation*}
		\sum_{j=1}^{2}\sum_{m\in\mbZ^{2}}\varphi(\delta m)^{2}\int_{s}^{t}\norm{\partial_{j}P_{t-u}(\het(u,\place)\euler^{2\uppi\upi\inner{m}{\place}}) }_{\mcH^{\gamma'-2\kappa'}}^{2}\dd u\lesssim_{T}(1+\delta^{-1})^{2(\gamma+1)}\norm{\het}_{C_{\eta;T}\mcH^{\gamma}}^{2}\abs{t-s}^{2\kappa'}.
	\end{equation*}
	\begin{details}
		\paragraph{Integrating the singularity.}
		We distinguish the cases $s<t\leq1$, $s\leq1<t$ and $1<s<t$. If $s<t\leq1$, we estimate
		\begin{equation*}
			\begin{split}
				\int_{s}^{t}\abs{t-u}^{-(1+\gamma'-\gamma-2\kappa')}(1\wedge u)^{-2\eta}\dd u&=\int_{s}^{t}\abs{t-u}^{-(1+\gamma'-\gamma-2\kappa')}u^{-2\eta}\dd u\\
				&\leq\abs{t-s}^{2\kappa'}\int_{0}^{t}\abs{t-u}^{-(1+\gamma'-\gamma)}u^{-2\eta}\dd u\\
				&\lesssim\abs{t-s}^{2\kappa'}t^{\gamma-\gamma'-2\eta}.
			\end{split}
		\end{equation*}
		If $s\leq1<t$, we estimate
		\begin{equation*}
			\begin{split}
				\int_{s}^{t}\abs{t-u}^{-(1+\gamma'-\gamma-2\kappa')}(1\wedge u)^{-2\eta}\dd u&=\int_{s}^{1}\abs{t-u}^{-(1+\gamma'-\gamma-2\kappa')}u^{-2\eta}\dd u+\int_{1}^{t}\abs{t-u}^{-(1+\gamma'-\gamma-2\kappa')}\dd u\\
				&\leq\abs{t-s}^{2\kappa'}\int_{0}^{t}\abs{t-u}^{-(1+\gamma'-\gamma)}u^{-2\eta}\dd u+\abs{t-s}^{2\kappa'}\int_{0}^{t}\abs{t-u}^{-(1+\gamma'-\gamma)}\dd u\\
				&\lesssim\abs{t-s}^{2\kappa'}(t^{\gamma-\gamma'-2\eta}+t^{\gamma-\gamma'}).
			\end{split}
		\end{equation*}
		If $1<s<t$, we estimate
		\begin{equation*}
			\begin{split}
				\int_{s}^{t}\abs{t-u}^{-(1+\gamma'-\gamma-2\kappa')}(1\wedge u)^{-2\eta}\dd u&=\int_{s}^{t}\abs{t-u}^{-(1+\gamma'-\gamma-2\kappa')}\dd u\\
				&\leq\abs{t-s}^{2\kappa'}\int_{0}^{t}\abs{t-u}^{-(1+\gamma'-\gamma)}\dd u\\
				&\lesssim\abs{t-s}^{2\kappa'}t^{\gamma-\gamma'}.
			\end{split}
		\end{equation*}
	\end{details}
	
	Now assume $2\kappa'>1+\gamma'-\gamma$; we may skip the application of Schauder's estimate, that is,
	\begin{equation*}
		\norm{P_{t-u}(\het(u,\place)\euler^{2\uppi\upi\inner{m}{\place}})}_{\mcH^{\gamma'-2\kappa'+1}}\lesssim\norm{P_{t-u}(\het(u,\place)\euler^{2\uppi\upi\inner{m}{\place}})}_{\mcH^{\gamma}}\lesssim\norm{\sigma(u,\place)\euler^{2\uppi\upi\inner{m}{\place}}}_{\mcH^{\gamma}},
	\end{equation*}
	to control the second summand in~\eqref{eq:lolli_regular_increment_decomposition} by the singularity
	\begin{equation*}
		\sum_{j=1}^{2}\sum_{m\in\mbZ^{2}}\varphi(\delta m)^{2}\int_{s}^{t}\norm{\partial_{j}P_{t-u}(\het(u,\place)\euler^{2\uppi\upi\inner{m}{\place}}) }_{\mcH^{\gamma'-2\kappa'}}^{2}\dd u\lesssim(1+\delta^{-1})^{2(\gamma+1)}\norm{\het}_{C_{\eta;T}\mcH^{\gamma}}^{2}\int_{s}^{t}(1\wedge u)^{-2\eta}\dd u.
	\end{equation*}
	We may bound
	\begin{equation*}
		\int_{s}^{t}(1\wedge u)^{-2\eta}\dd u\lesssim_{T}\abs{t-s}^{2\kappa'},
	\end{equation*}
	where the singularity is integrable by $2\eta<1$; and the bound is uniform by $2\eta\leq1-2\kappa'$. Therefore if $2\kappa'>1+\gamma'-\gamma$, we can control the second summand in~\eqref{eq:lolli_regular_increment_decomposition} by
	\begin{equation*}
		\sum_{j=1}^{2}\sum_{m\in\mbZ^{2}}\varphi(\delta m)^{2}\int_{s}^{t}\norm{\partial_{j}P_{t-u}(\het(u,\place)\euler^{2\uppi\upi\inner{m}{\place}}) }_{\mcH^{\gamma'-2\kappa'}}^{2}\dd u\lesssim_{T}(1+\delta^{-1})^{2(\gamma+1)}\norm{\het}_{C_{\eta;T}\mcH^{\gamma}}^{2}\abs{t-s}^{2\kappa'}.
	\end{equation*}
	\begin{details}
		\paragraph{Integrating the singularity.}
		We distinguish the cases $s<t\leq1$, $s\leq1<t$ and $1<s<t$. If $s<t\leq1$, we estimate
		\begin{equation*}
			\int_{s}^{t}(1\wedge u)^{-2\eta}\dd u=\int_{s}^{t}u^{-2\eta}\dd u=\frac{1}{1-2\eta}\frac{t^{1-2\eta}-s^{1-2\eta}}{\abs{t-s}^{1-2\eta}}\frac{\abs{t-s}^{1-2\eta}}{\abs{t-s}^{2\kappa'}}\abs{t-s}^{2\kappa'}\lesssim t^{1-2\eta-2\kappa'}\abs{t-s}^{2\kappa'}.
		\end{equation*}
		If $s\leq1<t$, we estimate
		\begin{equation*}
			\int_{s}^{t}(1\wedge u)^{-2\eta}\dd u=\int_{s}^{1}u^{-2\eta}\dd u+\int_{1}^{t}\dd u\lesssim t^{1-2\eta-2\kappa'}\abs{t-s}^{2\kappa'}+\abs{t-s}\leq(t^{1-2\eta-2\kappa'}+t^{1-2\kappa'})\abs{t-s}^{2\kappa'}.
		\end{equation*}
		If $1<s<t$, we estimate
		\begin{equation*}
			\int_{s}^{t}(1\wedge u)^{-2\eta}\dd u=\int_{s}^{1}\dd u=\abs{t-s}\leq\abs{t-s}^{1-2\kappa'}\abs{t-s}^{2\kappa'}\leq t^{1-2\kappa'}\abs{t-s}^{2\kappa'}.
		\end{equation*}
	\end{details}
	
	To summarize, for all $\kappa'\in(\kappa,1/2]$ such that $2\eta\leq1-2\kappa'$, we can bound the second summand in~\eqref{eq:lolli_regular_increment_decomposition} by
	\begin{equation}\label{eq:lolli_regular_increment_bound_II}
		\sum_{j=1}^{2}\sum_{m\in\mbZ^{2}}\varphi(\delta m)^{2}\int_{s}^{t}\norm{\partial_{j}P_{t-u}(\het(u,\place)\euler^{2\uppi\upi\inner{m}{\place}}) }_{\mcH^{\gamma'-2\kappa'}}^{2}\dd u\lesssim_{T}(1+\delta^{-1})^{2(\gamma+1)}\norm{\het}_{C_{\eta;T}\mcH^{\gamma}}^{2}\abs{t-s}^{2\kappa'}.
	\end{equation}
	Hence by~\eqref{eq:lolli_regular_increment_decomposition}, \eqref{eq:lolli_regular_increment_bound_I} and~\eqref{eq:lolli_regular_increment_bound_II}, we obtain uniformly in $\abs{t-s}\leq1$ the bound
	\begin{equation*}
		\mbE[\norm{\ti^{\delta}_{t}-\ti^{\delta}_{s}}_{\mcH^{\gamma'-2\kappa'}}^{2}]\lesssim_{T}(1+\delta^{-1})^{2(\gamma+1)}\norm{\het}_{C_{\eta;T}\mcH^{\gamma}}^{2}\abs{t-s}^{2\kappa'}.
	\end{equation*}
	
	For $\abs{t-s}>1$, we instead estimate by~\eqref{eq:lolli_regular_pntw},
	\begin{equation*}
		\mbE[\norm{\ti^{\delta}_{t}-\ti^{\delta}_{s}}_{\mcH^{\gamma'-2\kappa'}}^{2}]\lesssim\mbE[\norm{\ti^{\delta}_{t}}_{\mcH^{\gamma'-2\kappa'}}^{2}]+\mbE[\norm{\ti^{\delta}_{s}}_{\mcH^{\gamma'-2\kappa'}}^{2}]\lesssim_{T}(1+\delta^{-1})^{2(\gamma+1)}\norm{\het}_{C_{\eta;T}\mcH^{\gamma}}^{2}\abs{t-s}^{2\kappa'},
	\end{equation*}
	which yields for all $s,t\in[0,T]$ the bound
	\begin{equation*}
		\mbE[\norm{\ti^{\delta}_{t}-\ti^{\delta}_{s}}_{\mcH^{\gamma'-2\kappa'}}^{2}]\lesssim_{T}(1+\delta^{-1})^{2(\gamma+1)}\norm{\het}_{C_{\eta;T}\mcH^{\gamma}}^{2}\abs{t-s}^{2\kappa'}.
	\end{equation*}
	
	Let $p\in[1,\infty)$, to control the $L^{p}(\mbP)$-norm, we use Nelson's estimate~\cite[Lem.~2]{friz_victoir_07} and H\"{o}lder's inequality to deduce
	\begin{equation*}
		\mbE[\norm{\ti^{\delta}_{t}-\ti^{\delta}_{s}}^{p}_{\mcH^{\gamma'-2\kappa'}}]^{1/p}\lesssim_{p}\mbE[\norm{\ti^{\delta}_{t}-\ti^{\delta}_{s}}^{2}_{\mcH^{\gamma'-2\kappa'}}]^{1/2},
	\end{equation*}
	which, together with the $L^{2}(\mbP)$-bound above, yields
	\begin{equation*}
		\mbE[\norm{\ti^{\delta}_{t}-\ti^{\delta}_{s}}_{\mcH^{\gamma'-2\kappa'}}^{p}]^{1/p}\lesssim_{p,T}(1+\delta^{-\gamma-1})\norm{\het}_{C_{\eta;T}\mcH^{\gamma}}\abs{t-s}^{\kappa'}.
	\end{equation*}
	An application of Kolmogorov's continuity criterion~\cite[Thm.~A.10]{friz_victoir_10} then implies that there exists a modification (which we do not relabel) such that for all $p>1/\kappa'$ and $\kappa\in[0,\kappa'-1/p)$,
	\begin{equation}\label{eq:lolli_regular_bound_prelim}
		\mbE[\norm{\ti^{\delta}}^{p}_{C^{\kappa}_{T}\mcH^{\gamma'-2\kappa'}}]^{1/p}\lesssim_{p,T}(1+\delta^{-\gamma-1})\norm{\het}_{C_{\eta;T}\mcH^{\gamma}}.
	\end{equation}
	\begin{details}
		Where we used that $\ti^{\delta}_{0}=0$ to control the uniform norm by the H\"{o}lder seminorm.
	\end{details}
	Another application of H\"{o}lder's inequality shows that this bound continuous to hold for every $\kappa\in[0,\kappa')$ and $p\in[1,\infty)$.
	
	The next step is to generalize the result to arbitrary exponents $\kappa\in[0,1/2)$ and to avoid the loss of $2(\kappa'-\kappa)$ space regularity in~\eqref{eq:lolli_regular_bound_prelim}. Let $\gamma\geq0$, $\gamma'\in[-1,\gamma)$, $\kappa\in[0,1/2)$ and $\eta\in[0,1/2)$ be such that $2\eta<\gamma-\gamma'$ and $2\eta<1-2\kappa$. We can find $\kappa'\in(\kappa,1/2]$ such that $\gamma'+2(\kappa'-\kappa)\in[-1,\gamma)$, $2\eta\leq\gamma-(\gamma'+2(\kappa'-\kappa))$ and $2\eta\leq1-2\kappa'$. An application of~\eqref{eq:lolli_regular_bound_prelim} with these inflated exponents yields
	\begin{equation*}
		\mbE[\norm{\ti^{\delta}}^{p}_{C^{\kappa}_{T}\mcH^{\gamma'-2\kappa}}]^{1/p}=\mbE[\norm{\ti^{\delta}}^{p}_{C^{\kappa}_{T}\mcH^{\gamma'+2(\kappa'-\kappa)-2\kappa'}}]^{1/p}\lesssim_{p,T}(1+\delta^{-\gamma-1})\norm{\het}_{C_{\eta;T}\mcH^{\gamma}},
	\end{equation*}
	which proves~\eqref{eq:lolli_regular_bound}. Further, an application of Lemma~\ref{lem:little_Hoelder_criterion} shows $\ti^{\delta}\in\msL_{T}^{\kappa}\mcH^{\gamma'}(\mbT^{2})$ almost surely.
	\begin{details}
		\paragraph{Proof that $\ti^{\delta}\in\msL_{T}^{\kappa}\mcH^{\gamma'}(\mbT^{2})$ almost surely.}
		Let $\gamma\geq0$, $\gamma'\in[-1,\gamma)$, $\kappa\in[0,1/2)$ and $\eta\in[0,1/2)$ be such that $2\eta<\gamma-\gamma'$ and $2\eta<1-2\kappa$. To establish $\ti^{\delta}\in C_{T}^{\kappa}\mcH^{\gamma'-2\kappa}(\mbT^{2})$, we can find some $\kappa'\in(\kappa,1/2)$ such that $\gamma'+2(\kappa'-\kappa)\in[-1,\gamma)$, $2\eta<\gamma-(\gamma'+2(\kappa'-\kappa))$ and $2\eta<1-2\kappa'$. It then follows by~\eqref{eq:lolli_regular_bound} that
		\begin{equation*}
			\norm{\ti^{\delta}}_{C^{\kappa'}_{T}\mcH^{\gamma'-2\kappa}}=\norm{\ti^{\delta}}_{C^{\kappa'}_{T}\mcH^{\gamma'+2(\kappa'-\kappa)-2\kappa'}}<\infty\quad\text{almost surely}.
		\end{equation*}
		We can then apply Lemma~\ref{lem:little_Hoelder_criterion} to deduce $\ti^{\delta}\in C_{T}^{\kappa}\mcH^{\gamma'-2\kappa}(\mbT^{2})$ almost surely.
	\end{details}
\end{proof}
The next lemma establishes the regularity of $\ti^{\delta}$ as a square-integrable function in time.
\begin{lemma}\label{lem:lolli_regular_L2Hgamma}
	Let $T>0$, $\gamma\geq0$ and $\het\in L^{2}_{T}\mcH^{\gamma}(\mbT^{2})$. Then for all $\delta>0$ and $p\in[1,\infty)$ it holds that 
	\begin{equation*}
		\mbE[\norm{\ti^{\delta}}_{L^{2}_{T}\mcH^{\gamma}}^{p}]^{1/p}\lesssim_{p}(1+\delta^{-\gamma-1})\norm{\het}_{L^{2}_{T}\mcH^{\gamma}}.
	\end{equation*}
	and in particular $\ti^{\delta}\in L^{2}_{T}\mcH^{\gamma}(\mbT^{2})$ almost surely.
\end{lemma}
\begin{proof}
	We apply It\^{o}'s isometry~\eqref{eq:lolli_regular_Ito_isometry}, Parseval's theorem and Tonelli's theorem to obtain
	\begin{equation*}
		\begin{split}
			\mbE[\norm{\ti^{\delta}}_{L^{2}_{T}\mcH^{\gamma}}^{2}]&=\sum_{j=1}^{2}\sum_{m\in\mbZ^{2}}\varphi(\delta m)^{2}\int_{0}^{T}\int_{0}^{t}\norm{(1-\Delta)^{\gamma/2}\partial_{j}P_{t-u}(\het(u,\place)\euler^{2\uppi\upi\inner{m}{\place}})}_{L^{2}}^{2}\dd u\dd t\\
			&=\sum_{m\in\mbZ^{2}}\varphi(\delta m)^{2}\sum_{\om\in\mbZ^{2}}(1+\abs{2\uppi\om}^{2})^{\gamma}\abs{2\uppi\om}^{2}\int_{0}^{T}\int_{0}^{t}\euler^{-2\abs{t-u}\abs{2\uppi\om}^{2}}\abs{\hat{\het}(u,\om-m)}^{2}\dd u\dd t\\
			&=\sum_{m\in\mbZ^{2}}\varphi(\delta m)^{2}\sum_{\om\in\mbZ^{2}}(1+\abs{2\uppi\om}^{2})^{\gamma}\abs{2\uppi\om}^{2}\int_{0}^{T}\abs{\hat{\het}(u,\om-m)}^{2}\int_{u}^{T}\euler^{-2\abs{t-u}\abs{2\uppi\om}^{2}}\dd t\dd u\\
			&=\frac{1}{2}\sum_{m\in\mbZ^{2}}\varphi(\delta m)^{2}\sum_{\om\in\mbZ^{2}}(1+\abs{2\uppi\om}^{2})^{\gamma}\int_{0}^{T}\abs{\hat{\het}(u,\om-m)}^{2}(1-\euler^{-2\abs{T-u}\abs{2\uppi\om}^{2}})\dd u,
		\end{split}
	\end{equation*}
	which we can bound by
	\begin{equation*}
		\mbE[\norm{\ti^{\delta}}_{L^{2}_{T}\mcH^{\gamma}}^{2}]\lesssim\sum_{m\in\mbZ^{2}}\varphi(\delta m)^{2}\int_{0}^{T}\sum_{\om\in\mbZ^{2}}(1+\abs{2\uppi\om}^{2})^{\gamma}\abs{\hat{\het}(u,\om-m)}^{2}\dd u.
	\end{equation*}
	Using that $\gamma\geq0$, we bound the integrand as in~\eqref{eq:het_fourier_product_bound} uniformly in $u\in[0,T]$,
	\begin{equation*}
		\sum_{\om\in\mbZ^{2}}(1+\abs{2\uppi\om}^{2})^{\gamma}\abs{\hat{\het}(u,\om-m)}^{2}\lesssim(1+\abs{2\uppi m}^{2})^{\gamma}\norm{\het(u,\place)}_{\mcH^{\gamma}}^{2},
	\end{equation*}
	which yields
	\begin{equation*}
		\mbE[\norm{\ti^{\delta}}_{L^{2}_{T}\mcH^{\gamma}}^{2}]\lesssim\sum_{m\in\mbZ^{2}}\varphi(\delta m)^{2}(1+\abs{2\uppi m}^{2})^{\gamma}\int_{0}^{T}\norm{\het(u,\place)}_{\mcH^{\gamma}}^{2}\dd u\lesssim(1+\delta^{-1})^{2(\gamma+1)}\norm{\het}_{L^{2}_{T}\mcH^{\gamma}}^{2},
	\end{equation*}
	where in the last inequality we applied~\cite[Lem.~C.2, (C.1)]{martini_mayorcas_25} to control the sum over $m\in\mbZ^{2}$. The claim for arbitrary $p\in[1,\infty)$ then follows by Nelson's estimate~\cite[Lem.~2]{friz_victoir_07} and H\"{o}lder's inequality.
\end{proof}
The next lemma establishes a uniform bound on $\ti^{\delta}$ without trading space regularity for time regularity (cf.\ Lemma~\ref{lem:lolli_regular}).
\begin{lemma}\label{lem:lolli_regular_unif}
	Let $T>0$, $\gamma\in[-1,\infty)$ and $\het\in C_{T}\mcH^{\gamma\vee0}(\mbT^{2})\cap L_{T}^{2}\mcH^{\gamma+1}(\mbT^{2})$. Then for all $\delta>0$ and $p\in[1,\infty)$ it holds that
	\begin{equation*}
		\mbE[\norm{\ti^{\delta}}_{L_{T}^{\infty}\mcH^{\gamma}}^{p}]^{1/p}\lesssim_{p}(1+\delta^{-\gamma-2})\max\{\norm{\het}_{C_{T}\mcH^{\gamma\vee0}},\norm{\het}_{L_{T}^{2}\mcH^{\gamma+1}}\}.
	\end{equation*}
\end{lemma}
\begin{proof}
	For all $\om\in\mbZ^{2}$ we obtain by~\eqref{eq:lolli_Fourier} the stochastic differential equation of the Fourier transform $(\hat{\ti^{\delta}}(t,\om))_{t\in[0,T]}$,
	\begin{equation*}
		\dd\hat{\ti^{\delta}}(t,\om)=-\abs{2\uppi\om}^{2}\hat{\ti^{\delta}}(t,\om)\dd t+\sum_{j=1}^{2}\sum_{m\in\mbZ^{2}}2\uppi\upi\om^{j}\hat{\het}(t,\om-m)\varphi(\delta m)\dd W^{j}(t,m),
	\end{equation*}
	whose (complex) martingale part we denote by
	\begin{equation*}
		\hat{M^{\delta}}(t,\om)\defeq\sum_{j=1}^{2}\sum_{m\in\mbZ^{2}}\int_{0}^{t}2\uppi\upi\om^{j}\hat{\het}(u,\om-m)\varphi(\delta m)\dd W^{j}(u,m).
	\end{equation*}
	The (sesquilinear) covariation of $\hat{M^{\delta}}(t,\om)$ is given by
	\begin{equation*}
		\dd[\hat{M^{\delta}}(\place,\om),\hat{M^{\delta}}(\place,\om')]_{t}=\inner{2\uppi\om}{2\uppi\om'}\sum_{m\in\mbZ^{2}}\hat{\het}(t,\om-m)\overline{\hat{\het}(t,\om'-m)}\varphi(\delta m)^{2}\dd t
	\end{equation*}
	and we denote the complex quadratic variation by $[\place]_{t}\defeq[\place,\place]_{t}$. It then follows by an application of the complex It\^{o} formula, that
	\begin{equation*}
		\begin{split}
			\dd\abs{\hat{\ti^{\delta}}(t,\om)}^{2}&=\hat{\ti^{\delta}}(t,-\om)\dd\hat{\ti^{\delta}}(t,\om)+\hat{\ti^{\delta}}(t,\om)\dd\hat{\ti^{\delta}}(t,-\om)+\dd[\hat{\ti^{\delta}}(\place,\om)]_{t}\\
			&=-2\abs{\hat{\ti^{\delta}}(t,\om)}^{2}\abs{2\uppi\om}^{2}\dd t+\hat{\ti^{\delta}}(t,-\om)\dd\hat{M^{\delta}}(t,\om)+\hat{\ti^{\delta}}(t,\om)\dd\hat{M^{\delta}}(t,-\om)+\dd[\hat{M^{\delta}}(\place,\om)]_{t},
		\end{split}
	\end{equation*}
	which, combined with~\eqref{eq:Bessel_norm}, yields
	\begin{equation}\label{eq:lolli_regular_norm_decomposition_SDE}
		\begin{split}
			\norm{\ti^{\delta}_{t}}_{\mcH^{\gamma}}^{2}&=-2\int_{0}^{t}\norm{\nabla\ti^{\delta}_{u}}_{\mcH^{\gamma}}^{2}\dd u+2\sum_{\om\in\mbZ^{2}}(1+\abs{2\uppi\om}^{2})^{\gamma}\int_{0}^{t}\hat{\ti^{\delta}}(u,-\om)\dd\hat{M^{\delta}}(u,\om)\\
			&\quad+\sum_{\om\in\mbZ^{2}}(1+\abs{2\uppi\om}^{2})^{\gamma}[\hat{M^{\delta}}(\place,\om)]_{t}.
		\end{split}
	\end{equation}
	Let $p\in[1,\infty)$, taking the supremum of~\eqref{eq:lolli_regular_norm_decomposition_SDE} we can deduce the uniform bound
	\begin{equation}\label{eq:lolli_regular_semimart_bound}
		\begin{split}
			\norm{\ti^{\delta}}_{L_{T}^{\infty}\mcH^{\gamma}}^{p}&\leq2^{p/2}\sup_{t\in[0,T]}\Bigl\lvert\sum_{\om\in\mbZ^{2}}(1+\abs{2\uppi\om}^{2})^{\gamma}\int_{0}^{t}\hat{\ti^{\delta}}(u,-\om)\dd\hat{M^{\delta}}(u,\om)\Bigr\rvert^{p/2}\\
			&\quad+\Bigl(\sum_{\om\in\mbZ^{2}}(1+\abs{2\uppi\om}^{2})^{\gamma}[\hat{M^{\delta}}(\place,\om)]_{T}\Bigr)^{p/2}.
		\end{split}
	\end{equation}
	\begin{details}
		We show that the infinite sum of stochastic integrals
		\begin{equation*}
			Z_{t}\defeq2\sum_{\om\in\mbZ^{2}}(1+\abs{2\uppi\om}^{2})^{\gamma}\int_{0}^{t}\hat{\ti^{\delta}}(u,-\om)\dd\hat{M^{\delta}}(u,\om)
		\end{equation*}
		is a complex martingale.
		
		To begin let us consider for $\om\in\mbZ^{2}$ the individual stochastic integral
		\begin{equation*}
			\int_{0}^{t}\hat{\ti^{\delta}}(u,-\om)\dd\hat{M^{\delta}}(u,\om),
		\end{equation*}
		which is a martingale, since
		\begin{equation*}
			\begin{split}
				&\mbE\Bigl[\int_{0}^{T}\abs{\hat{\ti^{\delta}}(u,\om)}^{2}\dd[\hat{M^{\delta}}(\place,\om)]_{u}\Bigr]\\
				&=\mbE\Bigl[\int_{0}^{T}\abs{\hat{\ti^{\delta}}(u,\om)}^{2}\abs{2\uppi\om}^{2}\sum_{m\in\mbZ^{2}}\abs{\hat{\het}(u,\om-m)}^{2}\varphi(\delta m)^{2}\dd u\Bigr]\\
				&\lesssim\sum_{\om\in\mbZ^{2}}(1+\abs{2\uppi\om}^{2})^{\gamma}\mbE\Bigl[\int_{0}^{T}\abs{\hat{\ti^{\delta}}(u,\om)}^{2}\abs{2\uppi\om}^{2}\sum_{m\in\mbZ^{2}}\abs{\hat{\het}(u,\om-m)}^{2}\varphi(\delta m)^{2}\dd u\Bigr]\\
				&\lesssim\mbE\Bigl[\int_{0}^{T}\norm{\nabla\ti^{\delta}_{u}}_{\mcH^{\gamma}}^{2}\norm{\het_{u}}_{L^{2}}^{2}\dd u\Bigr]\leq\norm{\het}_{C_{T}L^{2}}^{2}\mbE[\norm{\ti^{\delta}}_{L_{T}^{2}\mcH^{\gamma+1}}^{2}]
			\end{split}
		\end{equation*}
		is finite by Lemma~\ref{lem:lolli_regular_L2Hgamma}.
		
		Next for $\vartheta>0$ we consider the truncated sum,
		\begin{equation*}
			Z^{\vartheta}_{t}\defeq2\sum_{\substack{\om\in\mbZ^{2}\\\abs{\om}\leq\vartheta^{-1}}}(1+\abs{2\uppi\om}^{2})^{\gamma}\int_{0}^{t}\hat{\ti^{\delta}}(u,-\om)\dd\hat{M^{\delta}}(u,\om),
		\end{equation*}
		which is a martingale as a finite sum of martingales.
		
		To pass to the limit $\vartheta\to0$, we decompose
		\begin{equation}\label{eq:lolli_regular_unif_mart_dec}
			\begin{split}
				Z^{\vartheta}_{t}&=\sum_{\substack{\om\in\mbZ^{2}\\\abs{\om}\leq\vartheta^{-1}}}(1+\abs{2\uppi\om}^{2})^{\gamma}\abs{\hat{\ti^{\delta}}(t,\om)}^{2}+2\sum_{\substack{\om\in\mbZ^{2}\\\abs{\om}\leq\vartheta^{-1}}}(1+\abs{2\uppi\om}^{2})^{\gamma}\abs{2\uppi\om}^{2}\int_{0}^{t}\abs{\hat{\ti^{\delta}}(u,\om)}^{2}\dd u\\
				&\quad-\sum_{\substack{\om\in\mbZ^{2}\\\abs{\om}\leq\vartheta^{-1}}}(1+\abs{2\uppi\om}^{2})^{\gamma}[\hat{M^{\delta}}(\place,\om)]_{t}.
			\end{split}
		\end{equation}
		By the monotone convergence theorem applied to the sums over $\abs{\om}\leq\vartheta^{-1}$, each term in this decomposition converges almost surely as $\vartheta\to0$ to an element of $[0,\infty]$. Using $\het\in L_{T}^{2}\mcH^{\gamma+1}(\mbT^{2})$, it follows by Lemma~\ref{lem:lolli_regular_L2Hgamma} that $\ti^{\delta}\in L_{T}^{2}\mcH^{\gamma+1}(\mbT^{2})$ almost surely and
		\begin{equation*}
			\sum_{\substack{\om\in\mbZ^{2}\\\abs{\om}\leq\vartheta^{-1}}}(1+\abs{2\uppi\om}^{2})^{\gamma}\abs{2\uppi\om}^{2}\int_{0}^{t}\abs{\hat{\ti^{\delta}}(u,\om)}^{2}\dd u\to\norm{\nabla\ti^{\delta}}_{L_{t}^{2}\mcH^{\gamma}}^{2}\quad\text{as}\quad\vartheta\to0,
		\end{equation*}
		hence the second term in~\eqref{eq:lolli_regular_unif_mart_dec} is finite almost surely.
		
		Using~\eqref{eq:het_fourier_product_bound} with $\gamma+1\geq0$ and $\het\in L_{T}^{2}\mcH^{\gamma+1}(\mbT^{2})$ it further follows that
		\begin{equation*}
			\begin{split}
				\sum_{\substack{\om\in\mbZ^{2}\\\abs{\om}\leq\vartheta^{-1}}}(1+\abs{2\uppi\om}^{2})^{\gamma}[\hat{M^{\delta}}(\place,\om)]_{t}&=\sum_{\substack{\om\in\mbZ^{2}\\\abs{\om}\leq\vartheta^{-1}}}(1+\abs{2\uppi\om}^{2})^{\gamma}\abs{2\uppi\om}^{2}\sum_{m\in\mbZ^{2}}\int_{0}^{t}\abs{\hat{\het}(u,\om-m)}^{2}\varphi(\delta m)^{2}\dd u\\
				&=\sum_{m\in\mbZ^{2}}\varphi(\delta m)^{2}\int_{0}^{t}\sum_{\substack{\om\in\mbZ^{2}\\\abs{\om}\leq\vartheta^{-1}}}(1+\abs{2\uppi\om}^{2})^{\gamma}\abs{2\uppi\om}^{2}\abs{\hat{\het}(u,\om-m)}^{2}\dd u\\
				&\to\sum_{m\in\mbZ^{2}}\varphi(\delta m)^{2}\int_{0}^{t}\sum_{\om\in\mbZ^{2}}(1+\abs{2\uppi\om}^{2})^{\gamma}\abs{2\uppi\om}^{2}\abs{\hat{\het}(u,\om-m)}^{2}\dd u\\
				&\leq\sum_{m\in\mbZ^{2}}\varphi(\delta m)^{2}(1+\abs{2\uppi m}^{2})^{\gamma+1}\int_{0}^{t}\norm{\het_{u}}_{\mcH^{\gamma+1}}^{2}\dd u<\infty,
			\end{split}
		\end{equation*}
		hence the third term in~\eqref{eq:lolli_regular_unif_mart_dec} is finite almost surely.
		
		Consequently we obtain $\lim_{\vartheta\to0}Z^{\vartheta}_{t}=Z_{t}$ almost surely; and $Z_{t}<\infty$ if and only if $\norm{\ti^{\delta}_{t}}_{\mcH^{\gamma}}<\infty$.
		
		It follows by~\eqref{eq:lolli_regular_semimart_bound_I} and its derivation, that $(Z^{\vartheta}_{t})_{\vartheta>0}$ is $L^{p}(\mbP)$-bounded uniformly in $\vartheta>0$ for every $p\in[1,\infty]$ and $t\in[0,T]$. In particular $(Z^{\vartheta}_{t})_{\vartheta>0}$ is uniformly integrable and $\lim_{\vartheta\to0}Z^{\vartheta}_{t}=Z_{t}$ in $L^{1}(\mbP)$. In particular $(Z_{t})_{t\in[0,T]}$ is a martingale and $Z_{t}<\infty$ almost surely.
	\end{details}
	
	To control the first term in~\eqref{eq:lolli_regular_semimart_bound}, we use the (complex) Burkholder--Davis--Gundy inequality~\cite[Thm.~14.6]{friz_victoir_10}, several applications of the Cauchy--Schwarz inequality and~\eqref{eq:het_fourier_product_bound} to deduce for every $p\in[2,\infty)$,
	\begin{equation}\label{eq:lolli_regular_semimart_bound_I}
		\begin{split}
			&\mbE\Bigl[\sup_{t\in[0,T]}\Bigl\lvert\sum_{\om\in\mbZ^{2}}(1+\abs{2\uppi\om}^{2})^{\gamma}\int_{0}^{t}\hat{\ti^{\delta}}(u,-\om)\dd\hat{M^{\delta}}(u,\om)\Bigr\rvert^{p/2}\Bigr]\\
			&\lesssim_{p}\mbE\Bigl[\Bigl[\sum_{\om\in\mbZ^{2}}(1+\abs{2\uppi\om}^{2})^{\gamma}\int_{0}^{\place}\hat{\ti^{\delta}}(u,-\om)\dd\hat{M^{\delta}}(u,\om)\Bigr]_{T}^{p/4}\Bigr]\\
			&\leq\norm{\het}_{C_{T}\mcH^{\gamma\vee0}}^{p/2}\mbE[\norm{\ti^{\delta}}_{L_{T}^{2}\mcH^{\gamma+1}}^{p/2}](1+\delta^{-1})^{((\gamma\vee0)+1)p/2}.
		\end{split}
	\end{equation}
	\begin{details}
		\paragraph{Full derivation of~\eqref{eq:lolli_regular_semimart_bound_I}.}
		We bound by the (complex) Burkholder--Davis--Gundy inequality and several applications of the Cauchy--Schwarz inequality,
		\begin{equation*}
			\begin{split}
				&\mbE\Bigl[\sup_{t\in[0,T]}\Bigl\lvert\sum_{\om\in\mbZ^{2}}(1+\abs{2\uppi\om}^{2})^{\gamma}\int_{0}^{t}\hat{\ti^{\delta}}(u,-\om)\dd\hat{M^{\delta}}(u,\om)\Bigr\rvert^{p/2}\Bigr]\\
				&\lesssim_{p}\mbE\Bigl[\Bigl[\sum_{\om\in\mbZ^{2}}(1+\abs{2\uppi\om}^{2})^{\gamma}\int_{0}^{\place}\hat{\ti^{\delta}}(u,-\om)\dd\hat{M^{\delta}}(u,\om)\Bigr]_{T}^{p/4}\Bigr]\\
				&=\mbE\Bigl[\Bigl(\sum_{\om,\om'\in\mbZ^{2}}(1+\abs{2\uppi\om}^{2})^{\gamma}(1+\abs{2\uppi\om'}^{2})^{\gamma}\int_{0}^{T}\hat{\ti^{\delta}}(u,-\om)\overline{\hat{\ti^{\delta}}(u,-\om')}\dd[\hat{M^{\delta}}(\place,\om),\hat{M^{\delta}}(\place,\om')]_{u}\Bigr)^{p/4}\Bigr]\\
				&=\mbE\Bigl[\Bigl(\sum_{\om,\om'\in\mbZ^{2}}(1+\abs{2\uppi\om}^{2})^{\gamma}(1+\abs{2\uppi\om'}^{2})^{\gamma}\\
				&\multiquad[8]\times\int_{0}^{T}\hat{\ti^{\delta}}(u,-\om)\overline{\hat{\ti^{\delta}}(u,-\om')}\inner{2\uppi\om}{2\uppi\om'}\sum_{m\in\mbZ^{2}}\hat{\het}(u,\om-m)\overline{\hat{\het}(u,\om'-m)}\varphi(\delta m)^{2}\dd u\Bigr)^{p/4}\Bigr]\\
				&\leq\mbE\Bigl[\Bigl(\sum_{m\in\mbZ^{2}}\varphi(\delta m)^{2}\int_{0}^{T}\Bigl(\sum_{\om\in\mbZ^{2}}(1+\abs{2\uppi\om}^{2})^{\gamma}\abs{2\uppi\om}\abs{\hat{\ti^{\delta}}(u,\om)}\abs{\hat{\het}(u,\om-m)}\Bigr)^{2}\dd u\Bigr)^{p/4}\Bigr]\\
				&\leq\mbE\Bigl[\Bigl(\sum_{m\in\mbZ^{2}}\varphi(\delta m)^{2}\int_{0}^{T}\Bigl(\sum_{\om\in\mbZ^{2}}(1+\abs{2\uppi\om}^{2})^{\gamma}\abs{2\uppi\om}^{2}\abs{\hat{\ti^{\delta}}(u,\om)}^{2}\Bigr)\Bigl(\sum_{\om\in\mbZ^{2}}(1+\abs{2\uppi\om}^{2})^{\gamma}\abs{\hat{\het}(u,\om-m)}^{2}\Bigr)\dd u\Bigr)^{p/4}\Bigr].
			\end{split}
		\end{equation*}
		Assume $\gamma\geq0$, it then follows by~\eqref{eq:het_fourier_product_bound},
		\begin{equation*}
			\begin{split}
				&\mbE\Bigl[\Bigl(\sum_{m\in\mbZ^{2}}\varphi(\delta m)^{2}\int_{0}^{T}\Bigl(\sum_{\om\in\mbZ^{2}}(1+\abs{2\uppi\om}^{2})^{\gamma}\abs{2\uppi\om}^{2}\abs{\hat{\ti^{\delta}}(u,\om)}^{2}\Bigr)\Bigl(\sum_{\om\in\mbZ^{2}}(1+\abs{2\uppi\om}^{2})^{\gamma}\abs{\hat{\het}(u,\om-m)}^{2}\Bigr)\dd u\Bigr)^{p/4}\Bigr]\\
				&\leq\norm{\het}_{C_{T}\mcH^{\gamma}}^{p/2}\mbE\Bigl[\Bigl(\sum_{m\in\mbZ^{2}}\varphi(\delta m)^{2}(1+\abs{2\uppi m}^{2})^{\gamma}\int_{0}^{T}\norm{\nabla\ti^{\delta}_{u}}_{\mcH^{\gamma}}^{2}\dd u\Bigr)^{p/4}\Bigr]\\
				&\leq\norm{\het}_{C_{T}\mcH^{\gamma}}^{p/2}\mbE[\norm{\ti^{\delta}}_{L_{T}^{2}\mcH^{\gamma+1}}^{p/2}](1+\delta^{-1})^{(\gamma+1)p/2}.
			\end{split}
		\end{equation*}
		Assume $\gamma<0$, then
		\begin{equation*}
			\begin{split}
				&\mbE\Bigl[\Bigl(\sum_{m\in\mbZ^{2}}\varphi(\delta m)^{2}\int_{0}^{T}\Bigl(\sum_{\om\in\mbZ^{2}}(1+\abs{2\uppi\om}^{2})^{\gamma}\abs{2\uppi\om}^{2}\abs{\hat{\ti^{\delta}}(u,\om)}^{2}\Bigr)\Bigl(\sum_{\om\in\mbZ^{2}}(1+\abs{2\uppi\om}^{2})^{\gamma}\abs{\hat{\het}(u,\om-m)}^{2}\Bigr)\dd u\Bigr)^{p/4}\Bigr]\\
				&\leq\mbE\Bigl[\Bigl(\sum_{m\in\mbZ^{2}}\varphi(\delta m)^{2}\int_{0}^{T}\Bigl(\sum_{\om\in\mbZ^{2}}(1+\abs{2\uppi\om}^{2})^{\gamma}\abs{2\uppi\om}^{2}\abs{\hat{\ti^{\delta}}(u,\om)}^{2}\Bigr)\Bigl(\sum_{\om\in\mbZ^{2}}\abs{\hat{\het}(u,\om-m)}^{2}\Bigr)\dd u\Bigr)^{p/4}\Bigr]\\
				&\leq\norm{\het}_{C_{T}L^{2}}^{p/2}\mbE\Bigl[\Bigl(\sum_{m\in\mbZ^{2}}\varphi(\delta m)^{2}\int_{0}^{T}\norm{\nabla\ti^{\delta}_{u}}_{\mcH^{\gamma}}^{2}\dd u\Bigr)^{p/4}\Bigr]\\
				&\leq\norm{\het}_{C_{T}L^{2}}^{p/2}\mbE[\norm{\ti^{\delta}}_{L_{T}^{2}\mcH^{\gamma+1}}^{p/2}](1+\delta^{-1})^{p/2}.
			\end{split}
		\end{equation*}
	\end{details}
	
	To control the second term in~\eqref{eq:lolli_regular_semimart_bound}, we estimate by~\eqref{eq:het_fourier_product_bound} using that $\gamma+1\geq0$,
	\begin{equation}\label{eq:lolli_regular_semimart_bound_II}
		\begin{split}
			\sum_{\om\in\mbZ^{2}}(1+\abs{2\uppi\om}^{2})^{\gamma}[\hat{M^{\delta}}(\place,\om)]_{T}&=\sum_{\om\in\mbZ^{2}}(1+\abs{2\uppi\om}^{2})^{\gamma}\abs{2\uppi\om}^{2}\sum_{m\in\mbZ^{2}}\int_{0}^{T}\abs{\hat{\het}(u,\om-m)}^{2}\varphi(\delta m)^{2}\dd u\\
			&=\sum_{m\in\mbZ^{2}}\varphi(\delta m)^{2}\int_{0}^{T}\sum_{\om\in\mbZ^{2}}(1+\abs{2\uppi\om}^{2})^{\gamma}\abs{2\uppi\om}^{2}\abs{\hat{\het}(u,\om-m)}^{2}\dd u\\
			&\lesssim\sum_{m\in\mbZ^{2}}\varphi(\delta m)^{2}(1+\abs{2\uppi m}^{2})^{\gamma+1}\int_{0}^{T}\norm{\het_{u}}_{\mcH^{\gamma+1}}^{2}\dd u\\
			&=(1+\delta^{-1})^{2(\gamma+2)}\norm{\het}_{L^{2}_{T}\mcH^{\gamma+1}}^{2}.
		\end{split}
	\end{equation}
	
	Therefore it follows by~\eqref{eq:lolli_regular_semimart_bound}, \eqref{eq:lolli_regular_semimart_bound_I}, \eqref{eq:lolli_regular_semimart_bound_II} and Lemma~\ref{lem:lolli_regular_L2Hgamma} that for every $p\in[2,\infty)$,
	\begin{equation*}
		\begin{split}
			\mbE[\norm{\ti^{\delta}}_{L_{T}^{\infty}\mcH^{\gamma}}^{p}]&\lesssim_{p}(1+\delta^{-1})^{((\gamma\vee0)+1)p/2}\norm{\het}_{C_{T}\mcH^{\gamma\vee0}}^{p/2}\mbE[\norm{\ti^{\delta}}_{L_{T}^{2}\mcH^{\gamma+1}}^{p/2}]+(1+\delta^{-1})^{(\gamma+2)p}\norm{\het}_{L_{T}^{2}\mcH^{\gamma+1}}^{p}\\
			&\lesssim(1+\delta^{-1})^{(\gamma+2)p}\max\{\norm{\het}_{C_{T}\mcH^{\gamma\vee0}},\norm{\het}_{L_{T}^{2}\mcH^{\gamma+1}}\}^{p}.
		\end{split}
	\end{equation*}
	To generalize the bound to $p\in[1,\infty)$, it suffices to apply H\"{o}lder's inequality, which yields the claim.
\end{proof}
In the next proposition we specialize $\het=\srdet$, where $\rdet$ is a solution to the deterministic Keller--Segel equation~\eqref{eq:Determ_KS} with initial data $\rho_{0}$ and chemotactic sensitivity $\chem\in\mbR$ (see Lemma~\ref{lem:Determ_KS_continuity}). A combination of Lemmas~\ref{lem:lolli_regular}, \ref{lem:lolli_regular_L2Hgamma} and~\ref{lem:lolli_regular_unif} with results of Section~\ref{sec:det_PDE} yields the existence and regularity of $\ti^{\delta}=\vdiv\mcI[\srdet\boldsymbol{\xi}^{\delta}]$ in $C_{T}L^{2}(\mbT^{2})\cap L_{T}^{2}\mcH^{1}(\mbT^{2})$ if $\rho_{0}$ is assumed to be continuous, positive and of unit mass.
\begin{proposition}\label{prop:lolli_regular_srdet}
	Let $\rho_{0}\in C(\mbT^{2})$ be such that $\rho_{0}>0$ and $\mean{\rho_{0}}=1$, $\rdet$ be the weak solution to~\eqref{eq:Determ_KS} with initial data $\rho_{0}$ and chemotactic sensitivity $\chem\in\mbR$, and $\Trdet$ be its maximal time of existence (see Lemma~\ref{lem:Determ_KS_continuity}). Then for all $T<\Trdet$ and $\delta>0$ it holds that $\ti^{\delta}\defeq\vdiv\mcI[\sqrt{\rdet}\boldsymbol{\xi}^{\delta}]\in C_{T}L^{2}(\mbT^{2})\cap L_{T}^{2}\mcH^{1}(\mbT^{2})$ almost surely and we can bound for every $p\in[1,\infty)$ and $\gamma\in[-1,0]$,
	\begin{equation}\label{eq:lolli_regular_srdet_CTL2_LT2H1}
		\mbE\Bigl[\norm{\ti^{\delta}}_{C_{T}\mcH^{\gamma}\cap L^{2}_{T}\mcH^{\gamma+1}}^{p}\Bigr]^{1/p}\lesssim_{p}(1+\delta^{-\gamma-2})\norm{\srdet}_{C_{T}L^{2}\cap L_{T}^{2}\mcH^{1}}.
	\end{equation}
\end{proposition}
\begin{proof}
	Let $\rho_{0}\in C(\mbT^{2})$ be such that $\rho_{0}>0$ and $\mean{\rho_{0}}=1$. An application of Lemma~\ref{lem:regularity_srdet} yields $\srdet\in L_{T}^{2}\mcH^{1}(\mbT^{2})$, which combined with Lemma~\ref{lem:lolli_regular_L2Hgamma} allows us to control for every $\gamma\in[-1,0]$,
	\begin{equation}\label{eq:lolli_regular_srdet_LT2H1}
		\mbE[\norm{\ti^{\delta}}_{L^{2}_{T}\mcH^{\gamma+1}}^{p}]^{1/p}\lesssim_{p}(1+\delta^{-\gamma-2})\norm{\srdet}_{L^{2}_{T}\mcH^{\gamma+1}}\lesssim(1+\delta^{-\gamma-2})\norm{\srdet}_{L^{2}_{T}\mcH^{1}}
	\end{equation}
	and deduce $\ti^{\delta}\in L^{2}_{T}\mcH^{\gamma+1}(\mbT^{2})$ almost surely.
	
	To establish $\ti^{\delta}\in C_{T}L^{2}(\mbT^{2})$ almost surely, we use Lemma~\ref{lem:lolli_regular_unif} and an approximation argument. By the density of the smooth functions in $C(\mbT^{2})$, we can find a sequence $(\rho_{0}^{n})_{n\in\mbN}$ such that $\rho_{0}^{n}\in C^{\infty}(\mbT^{2})$, $\rho_{0}^{n}>0$, $\mean{\rho_{0}^{n}}=1$ for every $n\in\mbN$ and $\rho_{0}^{n}\to\rho_{0}$ in $C(\mbT^{2})$ as $n\to\infty$.
	\begin{details}
		Passing to the full space, we have $\rho_{0}\in C_{\per}(\mbR^{2})$, where we abuse notation to identify functions on $\mbT^{2}$ with periodic functions on $\mbR^{d}$. An application of~\cite[App.~C, Thm.~6]{evans_98_PDE} then yields a sequence of mollifications $(\rho_{0}^{n})_{n\in\mbN}$ such that $\rho_{0}^{n}\in C^{\infty}(\mbR^{2})$ for every $n\in\mbN$ and $\rho_{0}^{n}\to\rho_{0}$ in $C([0,1]^{2})$ as $n\to\infty$. Using that the integral of a periodic function over its fundamental domain is constant under arbitrary translations, we can deduce that $\mean{\rho_{0}^{n}}=1$ for every $n\in\mbN$. Since mollifications of periodic functions remain periodic, we can pass back to the torus to obtain a suitable sequence of smooth approximations. 
	\end{details}
	Denote by $(\rdet^{n})_{n\in\mbN}$ the corresponding sequence of solutions to~\eqref{eq:Determ_KS} with initial data $(\rho_{0}^{n})_{n\in\mbN}$ and chemotactic sensitivity $\chem$. It follows by (the proof of) Theorem~\ref{thm:Determ_KS_Well_Posedness} that $\rdet^{n}\to\rdet$ in $C_{T}L^{2}(\mbT^{2})\cap L_{T}^{2}\mcH^{1}(\mbT^{2})$
	\begin{details}
		Using that $(\norm{\rho_{0}^{n}}_{L^{2}})_{n\in\mbN}$ is bounded uniformly in $n\in\mbN$, we can use~\cite[Rem.~1.2]{liu_roeckner_13_local_global}, to conclude that there exists some $\bar{T}>0$ such that both $(\rdet^{n})_{n\in\mbN}$ and $\rdet$ exist on $[0,\bar{T}]$. Further, by~\cite[Lem.~2.4]{liu_roeckner_13_local_global}, we can control $\norm{\rdet^{n}}_{L^{2}_{\bar{T}}\mcH^{1}}$ uniformly in $n\in\mbN$. Hence, the same bounds that lead to the uniqueness of solutions in Theorem~\ref{thm:Determ_KS_Well_Posedness} also yield $\rdet^{n}\to\rdet$ in $C_{\bar{T}}L^{2}(\mbT^{2})\cap L_{\bar{T}}^{2}\mcH^{1}(\mbT^{2})$.
		
		Denote the maximal times of existence in $L^{2}(\mbT^{2})$ by $(\Trdet[\rdet^{n}])_{n\in\mbN}$ and $\Trdet[\rdet]$. It follows that $\Trdet$ is lower semicontinuous (cf.~Lemma~\ref{lem:blow_up_lsc}), i.e.\ $\Trdet[\rdet]\leq\liminf_{n\to\infty}\Trdet[\rdet^{n}]$. Hence, for every $T<\Trdet[\rdet]$ there exists some $N\in\mbN$ such that $T<\Trdet[\rdet^{n}]$ for all $n\geq N$. Iterating the local convergence above, it follows that $\rdet^{n}\to\rdet$ in $C_{T}L^{2}(\mbT^{2})\cap L_{T}^{2}\mcH^{1}(\mbT^{2})$ for all $T<\Trdet[\rdet]$.
	\end{details}
	and an application of~\cite[Cor.~2.91]{bahouri_chemin_danchin_11} combined with a cut-off argument as in Lemma~\ref{lem:chain_rule_fractional} yields $\sqrt{\rdet^{n}}\to\sqrt{\rdet}\in C_{T}L^{2}(\mbT^{2})\cap L_{T}^{2}\mcH^{1}(\mbT^{2})$.
	\begin{details}
		Here we also use $\rdet^{n}\to\rdet\in C([0,T]\times\mbT^{2})$, which follows by $\rho_{0}^{n}\to\rho_{0}\in C(\mbT^{2})$ and~\eqref{eq:determ_KS_continuity_bilinear}. In particular, $\rdet^{n}$ is bounded away from zero uniformly for $n$ sufficiently large.
	\end{details}
	Using that $\rho_{0}^{n}\in C^{\infty}(\mbT^{2})\subset\mcH^{\vartheta}(\mbT^{2})$ for every $\vartheta>0$, we can apply Lemma~\ref{lem:regularity_srdet} to deduce $\sqrt{\rdet^{n}}\in C_{\eta;T}\mcH^{\vartheta+2\eta}(\mbT^{2})$ for each $\eta\in(0,1/2)$. Hence, Lemma~\ref{lem:lolli_regular} yields $\vdiv\mcI[\sqrt{\rdet^{n}}\boldsymbol{\xi}^{\delta}]\in C_{T}L^{2}(\mbT^{2})$ almost surely (where we use that $2\eta<\gamma+2\eta$) and Lemma~\ref{lem:lolli_regular_unif} yields $\vdiv\mcI[\sqrt{\rdet^{n}}\boldsymbol{\xi}^{\delta}]\to\vdiv\mcI[\sqrt{\rdet}\boldsymbol{\xi}^{\delta}]=\ti^{\delta}$ in $L_{T}^{\infty}L^{2}(\mbT^{2})$ as $n\to\infty$ almost surely along a subsequence (which we do not relabel). It follows that $\ti^{\delta}\in C_{T}L^{2}(\mbT^{2})$ almost surely by the completeness of $C_{T}L^{2}(\mbT^{2})$ under the uniform norm.
	
	Finally, another application of Lemma~\ref{lem:lolli_regular_unif} yields
	\begin{equation*}
		\begin{split}
			\mbE[\norm{\ti^{\delta}}_{C_{T}\mcH^{\gamma}}^{p}]^{1/p}&\lesssim_{p}(1+\delta^{-\gamma-2})\max\{\norm{\srdet}_{C_{T}L^{2}},\norm{\srdet}_{L_{T}^{2}\mcH^{\gamma+1}}\}\\
			&\lesssim(1+\delta^{-\gamma-2})\max\{\norm{\srdet}_{C_{T}L^{2}},\norm{\srdet}_{L_{T}^{2}\mcH^{1}}\}
		\end{split}
	\end{equation*}
	which combined with~\eqref{eq:lolli_regular_srdet_LT2H1} proves~\eqref{eq:lolli_regular_srdet_CTL2_LT2H1}.
\end{proof}
We apply the following infinite-dimensional sub-Gaussian upper deviation inequality in the proof of Theorem~\ref{thm:negative_values}.
\begin{lemma}\label{lem:lolli_regular_upper_deviation}
	Let $\rho_{0}\in C(\mbT^{2})$ be such that $\rho_{0}>0$ and $\mean{\rho_{0}}=1$, $\rdet$ be the weak solution to~\eqref{eq:Determ_KS} with initial data $\rho_{0}$ and chemotactic sensitivity $\chem\in\mbR$, and $\Trdet$ be its maximal time of existence (see Lemma~\ref{lem:Determ_KS_continuity}). Then for all $T<\Trdet$, $\delta>0$ and $\lambda\geq0$ it holds that
	\begin{equation*}
		(1+\delta^{-2})^{2}\log\mbP\Bigl(\norm{\ti^{\delta}}_{C_{T}L^{2}\cap L_{T}^{2}\mcH^{1}}-\mbE[\norm{\ti^{\delta}}_{C_{T}L^{2}\cap L_{T}^{2}\mcH^{1}}]\geq\lambda\Bigr)\lesssim-\lambda^{2}\norm{\srdet}_{C_{T}L^{2}\cap L_{T}^{2}\mcH^{1}}^{-2}.
	\end{equation*}
\end{lemma}
\begin{proof}
	Denote $E\defeq C_{T}L^{2}(\mbT^{2})\cap L_{T}^{2}\mcH^{1}(\mbT^{2})$. An application of~\cite[(4.4)]{ledoux_94} yields for every $\lambda\geq0$,
	\begin{equation*}
		\mbP\Bigl(\norm{\ti^{\delta}}_{E}-\mbE[\norm{\ti^{\delta}}_{E}]\geq\lambda\Bigr)\leq\euler^{-\frac{\lambda^{2}}{2\sigma^{2}}}\quad\text{where}\quad\sigma\defeq\sup_{\substack{\xi\in E'\\\norm{\xi}_{E'}\leq1}}\mbE[\inner{\xi}{\ti^{\delta}}^{2}]^{1/2}
	\end{equation*}
	and Proposition~\ref{prop:lolli_regular_srdet} shows the existence of a constant $C>0$ such that
	\begin{equation*}
		\sigma\leq\mbE[\norm{\ti^{\delta}}_{E}^{2}]^{1/2}\leq C(1+\delta^{-2})\norm{\srdet}_{E}.
	\end{equation*}
	Both combined allow us to estimate
	\begin{equation*}
		\log\mbP\Bigl(\norm{\ti^{\delta}}_{E}-\mbE[\norm{\ti^{\delta}}_{E}]\geq\lambda\Bigr)\leq-\frac{1}{2}C^{-2}\lambda^{2}(1+\delta^{-2})^{-2}\norm{\srdet}_{E}^{-2},
	\end{equation*}
	which yields the claim.
\end{proof} 
\section*{Declarations}
\begin{itemize}
	\item \textbf{Data and Material:} N/A
	\item \textbf{Code:} N/A
	\item \textbf{Funding:} Please see the acknowledgements in Section~\ref{sec:introduction}.
	\item \textbf{Competing Interests:} The authors have no relevant financial or non-financial interests to disclose.
\end{itemize}
\printbibliography[heading=bibintoc,title={References}]
\end{document}